\documentclass[12pt,a4paper,titlepage]{report}

\pdfoutput=1

\usepackage[utf8]{inputenc}
\usepackage{lmodern}

\usepackage[twoside,paper=a4paper,left=30mm,right=30mm,top=30mm,bottom=35mm]{geometry}

\usepackage{setspace}
\usepackage[headsepline]{scrlayer-scrpage}
\clearscrheadfoot

\ofoot{\pagemark}
\ohead{\headmark}
\automark[subsection]{section}

\usepackage{amsmath}
\usepackage{amsfonts}
\usepackage{amssymb}
\usepackage{amsthm}

\usepackage{mathrsfs}
\usepackage[shortlabels]{enumitem}

\usepackage[nottoc,notlot,notlof]{tocbibind}

\usepackage{bm}

\usepackage{graphicx}

\usepackage[toc,page]{appendix}

\usepackage{xcolor}

\usepackage{soul}

\usepackage{tikz}
\usetikzlibrary{arrows,shapes,trees}
\usepackage{tikz-cd}

\usepackage{etoolbox}
\apptocmd{\sloppy}{\hbadness 10000\relax}{}{}

\usepackage[bookmarksopen,
bookmarksnumbered,
plainpages=false,
pdfpagelabels,pagebackref]{hyperref}
\hypersetup{
	colorlinks=false,
	linkbordercolor=magenta,          
	citebordercolor=green,       
	filebordercolor=red,      
	urlbordercolor=blue,
	pdfborderstyle={/S/U/W 2}
}

\usepackage{bookmark}

\theoremstyle{definition}
\newtheorem{defi}{Definition}[section]

\theoremstyle{plain}
\newtheorem{prop}[defi]{Proposition}

\theoremstyle{plain}
\newtheorem{thm}[defi]{Theorem}

\theoremstyle{plain}
\newtheorem{cor}[defi]{Corollary}

\theoremstyle{plain}
\newtheorem{lem}[defi]{Lemma}

\theoremstyle{plain}
\newtheorem{conj}[defi]{Conjecture}

\theoremstyle{definition}
\newtheorem{rem}[defi]{Remark}

\theoremstyle{definition}
\newtheorem{ex}[defi]{Example}

\theoremstyle{plain}
\newtheorem*{quest}{Question}

\newenvironment{proofsketch}{
	\proof}{\endproof}

\DeclareMathOperator{\Ima}{Im}

\DeclareMathOperator{\Hom}{Hom}

\DeclareMathOperator{\supp}{supp}

\newcommand{\bigslant}[2]{{\raisebox{.2em}{$#1$}\left/\raisebox{-.2em}{$#2$}\right.}}

\author{Leonhard Katzlinger}
\title{Topological Full Groups}
\date{}

\begin{document}
	
	\hypersetup{pageanchor=false}
	\maketitle
	\hypersetup{pageanchor=true}
	\clearpage
	
	\pagenumbering{roman}
	\begingroup\setlength{\parskip}{0.7pt plus2pt minus2pt}

	\pagestyle{empty}
	
	\newpage
	
	\pdfbookmark[section]{\contentsname}{toc}
	\tableofcontents

	\endgroup
	 
	\newpage

	\section*{Abstract}\addcontentsline{toc}{chapter}{Abstract} 
	
	 Topological full groups originated from the theory of topological dynamical systems and have been having considerable impact on group theory in recent years. This text represents an introduction/survey on topological full groups. After development of the theoretical and historical background, it gives an account of their significance in topological dynamics and discusses their group theoretical aspects.
	 
	\begingroup
	\let\clearpage\relax
	\chapter*{Introduction}\addcontentsline{toc}{chapter}{Introduction}
		
		\section*{Historical remarks}
			\addcontentsline{toc}{section}{Historical remarks}
			
			Initially (measured) full groups grew out of von Neumann algebra theory and measured dynamics. Henry Dye defined and studied full groups in the papers \cite{dye59} and \cite{dye63}. Full groups offer a way to classify automorphism groups of measure algebras. In particular, every measure preserving action of a countable group $G$ on a Lebesgue measure space $(X,\lambda)$ gives rise to a group $[G]$ consisting of all non-singular,\footnote{Preimages of null-sets are null-sets.\\} measurable transformations $\gamma \colon X \to X$ such that $\gamma(x) \in Gx$ for $\lambda$-almost-every $x \in X$. Dye introduced the notion of approximate finitness of such an automorphism group $G$ i.e. every finite subset of elements in $G$ can be approximated by elements in the full group $[G]$, and showed that for ergodic actions of approximately finite groups by measure preserving transformations on $(X,\lambda)$ the associated full groups are complete isomorphism invariants for the arising von Neumann crossed products (\cite{dye59}, Theorem~5). Furthermore, he demonstrated that every isomorphism between the full groups of ergodic actions of countable groups by measure preserving transformations on $(X,\lambda)$ is induced by an orbit equivalence between the systems (\cite{dye63}, Theorem~2).
			
			Similar to the introduction of von Neumann algebras predating the consideration of general C*-algebras, the study of the interplay between measured dynamics and the theory of von Neumann algebras, which goes back to the work of John von Neumann and Francis J. Murray, predates the study of the interplay between topological dynamics and C*-algebra theory.
			Dye's ideas motivated efforts to find analogues in the setting of topological dynamics. One of these efforts is represented by the ample groups of Wolfgang Krieger defined in \cite{kri80}, the definition of which recalls Dye's full groups. Some years later Thierry Giordano, Ian F. Putnam and Christian F. Skau obtained a complete classification of minimal topological dynamical $\mathbb{Z}$-systems over a Cantor space up to orbit equivalence (resp. strong orbit equivalence) in terms of their C*-algebra crossed products and the associated dimension groups \cite{gps95}. Eli Glasner and Benjamin Weiss observed that some of the classification results of \cite{gps95} and weaker versions thereof can be shown \emph{without} relying on C*-algebra theory by using a transferred version of Dye's full group and a smaller, countable cousin of it -- the \emph{topological}\footnote{By no means ``topological" refers to a topology on this group.\\}\emph{ full group $\mathfrak{T}(\varphi)$}.		
				
			This group had already appeared implicitely 6 years earlier in \cite{put89} in terms of unitary normalizers of crossed product C*-algebras. The results of \cite{gw95} foreshadowed that of \cite{gps99}: Full groups (resp. topological full groups) of minimal topological dynamical $\mathbb{Z}$-systems over a Cantor space are complete isomorphism invariants for orbit equivalence (resp. flip conjugacy). Later on the concept of full groups proved to be fruitful in other settings e.g. minimal topological $\mathbb{Z}$-systems over locally compact Cantor spaces in \cite{mat02} or Borel actions of polish groups on polish spaces in \cite{mr07}.
			
			In \cite{mat06} Hiroki Matui defined topological full groups of countable, \'etale equivalence relations on Cantor spaces. Those have the structure of a principal, \'etale Cantor groupoid as such encompassing free actions of discrete groups on Cantor spaces. The scope was even more generalized in \cite{mat12} where a notion of topological full groups of vastly more general \'etale groupoids was introduced. In this much broader context topological full groups locate within the interplay between the theories of C*-algebras, \'etale groupoids and inverse semigroups -- the study of which had been initiated by Jean Renault in \cite{ren80} and developed further i.a. by Alan L.T. Patterson in \cite{pat99}, Ruy Exel in \cite{exe08} and Mark V. Lawson in \cite{law10}.
			
		\section*{The interest of group theory}
			\addcontentsline{toc}{section}{The interest of group theory}

			The first structural results on topological full groups were given in \cite{gps99} by topological dynamical methods: The approximation of minimal Cantor systems by periodic systems in terms of Kakutani-Rokhlin partitions, a fundamental tool in the classification of such systems, induces a factorization of the associated topological full groups as a product of a pair of direct limits of  sequences of finite symmetric groups each of which acts on atoms of Kakutani-Rokhlin partitions by permutation. These factors correspond precisely to the ample groups that have been introduced by Krieger. Furthermore, the topological full group of a minimal Cantor system admits a non-trivial homomorphism to $\mathbb{Z}$ called the index map, which describes the global transfer on orbits. Matui proved that the derived subgroup of the topological full group of a minimal Cantor system is simple and in the case of minimal subshifts even finitely generated in \cite{mat06}. This paper marks the first of a series of works (\cite{mat12}, \cite{mat13}, \cite{mat15}, \cite{mat16}) by Matui dealing with the subject, as well as the awakening of geometric group theorist's interest.
			
			In the first version of \cite{gm14} Rostislav Grigorchuk and Kostya Medynets conjectured topological full groups of minimal $\mathbb{Z}$-subshifts to be amenable, which was subsequently verified by Kate Juschenko and Nicolas Monod in \cite{jm13}. This produces an uncountably infinite number of non-isomorphic examples of infinite, finitely generated, simple, amenable groups, examples of which had hitherto been unknown. This triggered a number of publications dealing with the subject:
			In \cite{mb14a} and \cite{mb14b} Nicolás Matte Bon gave examples of topological full groups of minimal subshifts satisfying the Liouville property, thus providing first examples of infinite, finitely generated, simple groups with that property, and demonstrated that Grigorchuk's groups of intermediate growth admit embeddings into topological full groups of Cantor minimal systems. Volodymyr Nekrashevych dealt with topological full groups in \cite{nek17} and \cite{nek18}, where he generalized Matui's definition of topological full groups, introduced a pair of subgroups consisting of somewhat ``elementary" elements and used these to give first examples of finitely generated, simple groups with intermediate growth. It is moreover worth mentioning that topological full groups associated to groupoids arising from one-sided Markov shifts cast new light on Higman-Thompson groups. We recognize that the study of topological full groups has been a rewarding endeavour and is still ongoing.

		\section*{On the structure}
			\addcontentsline{toc}{section}{On the structure}		
			
			We finish the front matter with a short synopsis of what is to come. In the first two chapters we acquaint ourselves with the territory we are moving in. This is largely provided by the setting of \'etale groupoids. We start however in Chapter~\ref{chap: 1} with an introduction to fundamental concepts of topological dynamical systems over Cantor spaces, which is the surrounding in which topological full groups originate:
			\begin{enumerate}
				\item[] Section~\ref{sec: basic definitions} lists basic definitions from topological and measured dynamics, gives basic results and examples of topological dynamical systems over Cantor spaces and recalls the fundamental features of crossed product C*-algebras associated with such systems.
				
				\item[] Section~\ref{sec: the toolbox of minimal Cantor systems} introduces concepts of the theory of Cantor $\mathbb{Z}$-systems. Keywords are Bratteli-Vershik systems and Kakutani-Rokhlin partitions.
			\end{enumerate}
			
			In Chapter~\ref{chap: 2} we dwell on \'etale groupoids, inverse semigroups and the interrelations between those types of objects:
			\begin{enumerate}
				\item[] Section~\ref{sec: basics of groupoids} contains basic definitions and examples of (topological) groupoids.
				
				\item[] Section~\ref{sec: basics of inverse semigroups} does the same for inverse semigroups.
				
				\item[] Section~\ref{sec: analysis of topological groupoids} introduces locally compact groupoids as setting for (functional) analysis on topological groupoids and recalls the definition of groupoid C*-algebras.
				
				\item[] Section~\ref{sec: etale groupoids} comprises the definition and characterization  of \'etale groupoids and the subclass of ample groupoids, moreover it gives a suitable version of homology of \'etale groupoids. 
				
				\item[] Section~\ref{sec: interweaving inverse semigroups and etale groupoids} recounts a general version of non-commutative Stone duality, which is a duality between the categories of \'etale groupoids and inverse semigroups, and the historical developments leading up to it.
				
				\item[] Section~\ref{sec: resuming cantor dynamics} introduces \'etale Cantor groupoids as the setting of geheralized Cantor dynamics and the important subclasses of expansive groupoids, AF groupoids, almost finite groupoids, purely infinite groupoids and groupoids associated to one-sided Markov shifts.
			\end{enumerate}
			
			Chapter~\ref{chap: 3} represents a survey on topological full groups and finishes with a conceptual result on the irreducibility of certain Koopman representations:
			\begin{enumerate}
				\item[] Section~\ref{sec: basic structure and significance in cantor dynamics} starts with the definition of topological full groups (of Cantor systems and more general of \'etale Cantor groupoids), it recalls the description of topological full groups of minimal Cantor systems in terms of Kakutani-Rokhlin partitions, introduces important subgroups and outlines the significance of these groups in Cantor dynamics.
				
				\item[] Section~\ref{sec: results by matui} gives account of Matui's work on topological full group in that it sketches amongst other things his proof of topological full groups being complete isomorphism invariants for \'etale Cantor groupoids, his proof of simplicity of the commutator subgroup of the topological full groups of a minimal Cantor system and his proof of exponential growth.
				
				\item[] Section~\ref{sec: on a pair of subgroups} deals with the alternating full group introduced by Nekrashevych, in particular, with his proof of simplicity and finite generation.
				
				\item[] Section~\ref{sec: topological full groups in light of non-commutative stone duality} accounts for the fact of topological full groups of \'etale groupoids being groups of units of specific inverse $\wedge$-monoids and as such fall within the scope of non-commutative Stone duality. To this end it represents a dictionary of how principles from the groupoid world, and in very rough terms the isomorphism theorem of Matui, translate to the inverse semigroup setting.
				
				\item[] Section~\ref{sec: Significance for geometric group theory} represents a compilation of facts on topological full groups with regard to different notions from geometric group theory: Local embeddability into finite groups, amenability, growth of groups, etc.
				
				\item[] Section~\ref{sec: the koopman representation of topological full groups} shows how Artem Dudko's notion of measure contracting actions produces a short proof of irreducibility of the Koopman representation associated with the natural action of the topological full group of a minimal, purely infinite, \'etale Cantor groupoid on the space of units.
			\end{enumerate}
			
			The \hyperref[appendices]{Appendix} consists of two parts:
			\begin{enumerate}
				\item[] Appendix~\ref{app: some terms of geometric group theory} lists basic terms of geometric group theory e.g. growth of groups, graphs associated with groups, amenability, Higman-Thompson groups etc.
				
				\item[] Appendix~\ref{app: c*-algebras} is a collection of base vocabulary on C*-algebras, in particular it contains short accounts of Hilbert bundles, von-Neumann algebras and a non-formal adumbration of operator theoretic K-groups.
			\end{enumerate}
			
		\section*{On this text}
			\addcontentsline{toc}{section}{On this text}
			
			As it is the conversion of the author's Master thesis at the University of Vienna, this text represents an undergraduate's attempt to give an introduction on the topic of topological full groups with a focus on their impact in group theory. It was written with the goal in mind of being approachable, hence it includes lots of basic definitions, a big appendix and a dense system of references.
			
			For any crimes against style, grammar or spelling or more severely any misquotations or misrepresentation I offer my earnest apologies, I am thankful for any suggestions of improvement. Feel free to contact me by mail (a01168587@unet.univie.ac.at).
			
			I owe a debt of gratitude especially to Prof. Goulnara Arzhantseva and Martin Finn-Sell, but my thankfulness extends to everyone who supported me in any way during the last years.
			\endgroup
			
	\clearpage\pagenumbering{arabic}
	
	\chapter{Basic terms of dynamics and Cantor systems}\label{chap: 1}
		
		\section{Basic definitions}\label{sec: basic definitions}
			
			 Dynamical systems theory branches out into the areas -- amongst others -- of topological dynamics and ergodic theory. In Subsection~\ref{subs: Basic terms of dynamics} we recall basic definitions from both contexts with an emphasis on the topological setting. In Subsection~\ref{subs: mincasys} we turn to the specific case of Cantor systems. Subsection~\ref{subs: crossprod} comprises basic notions of crossed product C*-algebras.
			
			\subsection{Basic terms of dynamics}\label{subs: Basic terms of dynamics}
			
			For a thourough introduction to topological dynamical systems the reader is refered to \cite{gla03} or \cite{dev93}, E.1.
			
			\begin{defi}\phantomsection\label{defi: topdynsys}
				\begin{enumerate}[(i)]
					\item Let $G$ be a group and let $X$ be a set. An \emph{action of $G$ on $X$} is a map $\alpha \colon G \times X \to X$ such that:
					\begin{enumerate}
						\item $\alpha(1,x)=x$ for all $x \in X$.
						
						\item $\alpha(g,\alpha(h,x))=\alpha(gh,x)$ for all $g,h \in G$ and $x \in X$.
					\end{enumerate}
					If the map $\alpha \colon G \times X \to X$ is continuous, the action is called \emph{continuous}. \emph{As from now we omit $\alpha$ from the notation assuming it to be implicit and write $g \cdot x$ instead of $\alpha(g,x)$.}
				\end{enumerate}
				
				Let a group $G$ act on a set $X$.
				\begin{enumerate}[(i),resume]
					\item Let $x \in X$. The set $Gx:=\{g \cdot x|g \in G\}$ is called the \emph{orbit of $x$}.
					
					\item Let $S$ be a subset of $X$. The group $G_S:=\{g \in G|gS \subseteq S\}$ is called the \emph{stabilizer subgroup of $S$ in $G$}.
					
					\item The action is called \emph{faithful} or \emph{effective} if for every $g,h \in G$ with $g \neq h$ there exists an $x \in X$ such that $gx \neq hx$.
					
					\item The action is called \emph{free} if $gx=x$ for some $x \in X$ and $g \in G$ implies $g = 1$.
					
					\item A continuous action of a group $G$ on a topological space $X$ is called \emph{essentially free}, if for every $g \in G\setminus \{1\}$ the set of fixed points $\{x \in X| gx=x \}$ has empty interior.
					
					\item A \emph{topological dynamical system} or \emph{topological transformation group} is a tuple $(X,G)$ where $G$ is a discrete group, $X$ is a compact Hausdorff space and $G$ acts continuously on $X$.\footnote{Discreetness implies that the continuity of the action is equivalent to $G$ acting by homeomorphisms. Note, that in the literature the definition varies in constraints on $G$ and $X$.\\} We will use \emph{topological $G$-system over $X$} as a synonymic term to put emphasis on the involved topological space and group.
				\end{enumerate}
				Let $(X,G)$ be a topological dynamical system.
				\begin{enumerate}[(i),resume]
					\item A pair $(Y,G)$ is called a \emph{subsystem of $(X,G)$}, if $Y$ is a closed, non-empty, $G$-invariant subset of $X$. Restricting the $G$-action on $X$ to $Y$ makes $(Y,G)$ a topological $G$-system. 
					
					\item The system $(X, G)$ is said to be \emph{minimal}, if it contains no proper subsystem. If $(X,G)$ contains a unique, proper, minimal subsystem, it is said to be \emph{essentially minimal}.
				\end{enumerate}
			\end{defi}
		
			\begin{ex}\label{ex: shift}
				Let $X$ be a compact, metric space and let $G$ be a discrete, countable topological group. The natural continuous action of $G$ on $X^G=\{f\colon G \to X \}$ given by $g \cdot f(h) = f(g\cdot h)$ for all $g,h \in G$ is called a \emph{$G$-shift} or \emph{topological $X$-Bernoulli system}. A closed subsystem of a $G$-shift is called \emph{$G$-subshift}.
			\end{ex}
		
			Minimal systems are fundamental ``building blocks" of topological dynamical systems. Closures of orbits induce subsystems, thus the following is immediate:
			\begin{prop}
				Let $(X,G)$ be a topological dynamical system. The following are equivalent:
				\begin{enumerate}[(i)]
					\item The system $(X,G)$ is minimal.
					
					\item The orbit $Gx$ is dense in $X$ for every $x \in X$.
				\end{enumerate}
			\end{prop}
			
			Compactness and an application of Zorn's Lemma imply a basic result by Garrett Birkhoff:
			
			\begin{prop}\label{prop: ex.minsubs}
				Every topological dynamical system contains a minimal subsystem.
			\end{prop}
			
			For topological $\mathbb{Z}$-systems we make adjustments to the notations  chosen in Definition~\ref{defi: topdynsys}: A topological $\mathbb{Z}$-system $(X,\mathbb{Z})$ is signified by the pair $(X,\varphi)$, where $\varphi$ is the homeomorphism of $X$ corresponding to the generator $1 \in \mathbb{Z}$. We deviate in the notation of orbits by writing $\operatorname{Orb}_\varphi (x)$ instead of $Gx$.
			
			\begin{defi}
				Let $(X,\varphi)$ be a topological $\mathbb{Z}$-system.
				\begin{enumerate}[(i)]
					\item For every $x \in X$ define the \emph{forward orbit of $x$} as the set $\operatorname{Orb}_\varphi^+(x) :=\{\varphi^n|n>0 \}$ and the \emph{backward orbit of $x$} as the set $\operatorname{Orb}_\varphi^-(x):=\{\varphi^n|n \leq 0 \}$.
					
					\item The system $(X,\varphi)$ is called \emph{periodic} if every orbit is finite and it is called \emph{aperiodic} if every orbit is infinite.
				\end{enumerate}	
			\end{defi}
			
			\begin{rem}\label{rem: orbdens}
				For a minimal $\mathbb{Z}$-systems $(X,\varphi)$ forward- and backward orbits are dense in $X$, because their respective sets of accumulation points are closed invariant subsets. By definition minimal $\mathbb{Z}$-systems are always examples of free actions.
			\end{rem}
			
			\begin{defi}
				Let $(X_1, G)$ and $(X_2, G)$ be topological $G$-systems.
				\begin{enumerate}[(i)]
					\item A \emph{homomorphism of topological $G$-systems} $F \colon (X_1, G) \to (X_2, G)$ consists of a continuous map $F \colon X_1 \to X_2$ that intertwines the respective $G$-actions i.e. $F(g\cdot x) = g\cdot F(x)$ for all $x \in X$ and $g \in G$.
				\end{enumerate}
				
				Let $F \colon (X_1, G) \to (X_2, G)$ be a homomorphism of topological $G$-systems.
				\begin{enumerate}[(i),resume]
					\item If the associated map $F \colon X_1 \to X_2$ is surjective, it is called \emph{factor map}. In this case the system $(X_2, G)$ is said to be a \emph{factor} of $(X_1, G)$ or $(X_1, G)$ is said to be an \emph{extension} of $(X_2, G)$.
					
					\item It is called an \emph{isomorphism of topological $G$-systems} if the associated map $F \colon X_1 \to X_2$ is a homeomorphism. In this case the systems $(X_1, G)$ and $(X_2, G)$ are said to be \emph{conjugate}.
				\end{enumerate}
			\end{defi}
			
			The classification of systems up to conjugacy is often intractable and one turns to the following weaker notions: 
			
			\begin{defi}[\cite{gps95}, Definition~1.1 \& 1.2 \& 1.3]
				Let $(X_1, \varphi_1)$ and $(X_2, \varphi_2)$ be topological $\mathbb{Z}$-systems.
				\begin{enumerate}[(i)]
					\item The systems $(X_1, \varphi_1)$ and $(X_2, \varphi_2)$ are said to be \emph{orbit equivalent}, if there exists an \emph{orbit map} $F \colon X_1 \to X_2$, i.e. a homeomorphism $F \colon X_1 \to X_2$ such that $F(\operatorname{Orb}_{\varphi_1}(x)) = \operatorname{Orb}_{\varphi_2}(F(x))$ for all $x \in X_1$.
					
					\item The systems $(X_1, \varphi_1)$ and $(X_2, \varphi_2)$ are said to be \emph{flip conjugate}, if $(X_1, \varphi_1)$ is conjugate to either $(X_2, \varphi_2)$ or $(X_2, \varphi_2^{-1})$.
					
					\item Let $(X_1, \varphi_1)$ and $(X_2, \varphi_2)$ be minimal and let $F \colon X_1 \to X_2$ be an orbit map. Then there exist unique\footnote{This is why minimality is required.\\} $n(x),m(x) \in \mathbb{Z}$ such that $F \circ \varphi_1 (x) = \varphi_2^{n(x)} \circ F(x)$ and $F \circ \varphi_1^{m(x)} (x)=\varphi_2 \circ F (x)$ for all $x \in X_1$. The functions $n,m \colon X_1 \to \mathbb{Z}$ are called the \emph{orbit cocycles associated with $F$}.
					
					\item Let $(X_1, \varphi_1)$ and $(X_2, \varphi_2)$ be minimal. They are said to be \emph{strongly orbit equivalent}, if there exists an orbit map $F \colon X_1 \to X_2$ such that the orbit cocyles associated with $F$ each have at most one point of discontinuity.
				\end{enumerate}
			\end{defi}
			
			All above relations between topological $\mathbb{Z}$-systems are equivalence relations satisfying:
			\begin{equation*}
			\text{conjugacy} \Rightarrow \text{flip conjugacy} \Rightarrow \text{orbit equivalence} \Leftarrow \text{strong orbit equivalence}
			\end{equation*}
			
			We finish this subsection with some basic terms from measured dynamics:
			
			\begin{defi}[\cite{bm00}, Definition~1.1]
				Let $(X,\mu)$ be a measure space such that $\mu$ is a $\sigma$-finite measure and let a locally compact, second countable group $G$ act on $X$ such that for every $g \in G$ the map $x \mapsto g\cdot x$ is measurable. 
				\begin{enumerate}[(i)]
					\item For every $g \in G$ the \emph{pushforward measure} $g_{*}\mu$ is defined by						$g_{*}\mu(A):=\mu(g^{-1}A)$ for every measurable subset $A\subseteq X$.
					
					\item The measure $\mu$ is called \emph{quasi-invariant under the action of $G$} if for all $g \in G$ the measures $\mu$ and $g_{*}\mu$ are mutually absolutely continuous i.e. $\mu(A)=0$ \emph{if and only if} $g_{*}\mu(A)=0$ for every measurable subset $A \subseteq X$. Alternatively one speaks of a \emph{measure class preserving action of $G$ on $(X,\mu)$}.
					
					\item The measure $\mu$ is called \emph{invariant under the action of $G$} if for all $g \in G$ the measures $\mu$ and $g_{*}\mu$ coincide. Alternatively one speaks of a \emph{measure preserving action of $G$ on $(X,\mu)$}.
					
					\item The action is called \emph{ergodic with respect to $\mu$} if every $G$-invariant measurable subset $A \subseteq X$ is either null or conull.
				\end{enumerate}
			\end{defi}
			
			Every measure preserving action of a group $G$ on a probability space $(X,\mu)$ induces a unitary representation\footnote{See Definition~\ref{defi: unitaryrepresentations}.\\} of $G$ on $\mathrm{L}^2(X,\mu)$ by $(\tilde{\kappa}(g)f)(x):=f(g^{-1}x)$ called the \emph{Koopman representation}. For the quasi-invariant setting it is defined as:
			
			\begin{defi}[\cite{dud18}]\label{defi: koopman representation}
				Let $(X,\mu)$ be a measure space such that $\mu$ is a $\sigma$-finite measure. Let $G$ be a locally compact, second countable group acting measure-class-preserving on $(X,\mu)$. The \emph{Koopman representation} $\kappa$ of $G$ on $\mathrm{L}^2(X,\mu)$ is given by:
				\begin{equation*}
					(\kappa(g)f)(x):=\sqrt{\frac{\mathrm{d}g_{*}\mu}{\mathrm{d}\mu}(x)} f(g^{-1}x)
				\end{equation*}
			\end{defi}
			
			\begin{ex}
				Let $G$ be locally compact, second countable group. Then the regular representation\footnote{See Appendix~\ref{app: group c-star-algebras}.\\} of $G$ is just the Koopman representation in the case of $G$ acting on itself by left multiplication.
			\end{ex}
				
			The spectral properties of the Koopman representation allow to characterize ergodicity or mixing properties of the underlying measured dynamical system (see e.g. \cite{gla03}, Chapter~3.2).
				
		\subsection{Minimal Cantor systems}\label{subs: mincasys}
		
		The Cantor space is a paragon of an exciting object from mathematics -- be it for its distinctive topological properties, its historical significance for the development of point-set topology and descriptive set theory or its habit to appear in a variety of contexts and forms in dynamics. 
		
		\begin{defi}[\cite{kec95}, §3.A]
			The topological space $\mathcal{C}:=\{0,1\}^\mathbb{N}$ obtained as $\mathbb{N}$-fold topological product of the discrete space $\{0,1\}$ is called the \emph{Cantor space}. The subsets $C[(y_1, \dots ,y_n)]:=\big\{ \{x_i\}_{i \in \mathbb{N}} \in \mathcal{C} : x_i=y_i, \forall i \in 1, \dots, n  \big\}$ where $(y_1, \dots ,y_n) \in \{0,1\}^n$ for some $n \in \mathbb{N}$ are called \emph{cylinder sets}.
		\end{defi}
		
		The cylinder sets constitute a countable basis of clopens\footnote{Every cylinder set is the complement of a finite union of cylinder sets.\\} of the topology on $\mathcal{C}$ i.e. the space $\mathcal{C}$ is zero-dimensional and second countable.	
		
		\begin{thm}[\cite{kec95}, Theorem 7.4]\label{thm: cantor charac}
			The Cantor space $\mathcal{C}$ is uniquely characterized up to homeomorphisms by being a non-empty, compact, perfect, metrizable and zero-dimensional topological space.
		\end{thm}
		
		We call a topological space ``a" Cantor space if it is homeomorphic to ``the" Cantor space.

		\begin{ex}\phantomsection\label{ex: CS}
			\begin{enumerate}[(i)] 
				
				\item Open compact subsets of Cantor spaces and countably infinite products of Cantor spaces are Cantor spaces. A countably infinite product of any finite discrete space is a Cantor space.
				
				\item Let $d \in \mathbb{N}_{\geq 2}$. Let $T$ be an infinite labelled directed $d$-ary rooted tree with root $r$ i.e. an infinite tree together with a distinguished vertex point $r$ called the \emph{root}, such that every edge is oriented away from $r$ and every vertex has exactly $d$ children. The space $\partial T$ of directed infinite paths in $T$ starting at $r$ endowed with the topology arising by taking classes of paths which are identical on finitely many first vertices as a basis is called the \emph{boundary of $T$}. It is a Cantor space.
				
				\item Let $p$ be a prime number. The $p$-adic integers $\mathbb{Z}_p$ with their conventional topology are a Cantor space.
			\end{enumerate}
		\end{ex}
		
		\begin{rem}\phantomsection\label{rem: cant}
			\begin{enumerate}[(i)]
				\item Cantor spaces are by definition Stone spaces i.e. compact, Hausdorff, totally disconnected topological spaces, and Cantor spaces are Polish spaces i.e. separable and completely metrizable.
				
				\item Under Stone duality the characterization of the Cantor space up to homeomorphism from Theorem~\ref{thm: cantor charac} amounts to the fact that its Boolean algebra of clopen subsets is the up to Boolean isomorphism unique countable, atomless Boolean algebra.
			\end{enumerate}
		\end{rem}
		
		For the sake of abbrevation we define:
		\begin{defi}
			\begin{enumerate}[(i)]
				\item A \emph{Cantor system} is a topological $\mathbb{Z}$-system over a Cantor space.
				
				\item Let $(X,\varphi)$ be a Cantor system. Denote by $M_\varphi$ the set of all $\varphi$-invariant probability measures on $X$.
			\end{enumerate}
		\end{defi}
		
		In the following we recall some fundamental examples of Cantor systems:
		\begin{ex}\phantomsection\label{ex: shift+odo}
			\begin{enumerate}[(i)]
				\item Let $\mathcal{A}$ be a finite alphabet. The $\mathbb{Z}$-shift on $\mathcal{A}^\mathbb{Z}$ is a Cantor system called the \emph{full $\mathbb{Z}$-shift over $\mathcal{A}$}. This system is in general not minimal and the closed subsystems are called \emph{$\mathbb{Z}$-subshifts over $\mathcal{A}$}.
				
				\item Let $\bm{\mathrm{a}}=(a_n)_{n \in \mathbb{N}}$ be a strictly increasing sequence of integers with $a_n \geq 2$ and $a_n$ divides $a_{n+1}$ for all $n \in \mathbb{N}$. This data gives rise to an inverse system of surjective homomorphisms of finite cyclic groups $\rho_n \colon \mathbb{Z}/{a_{n+1}\mathbb{Z}} \to \mathbb{Z}/{a_n\mathbb{Z} }$ by $\rho_n(z)=z\mod a_n$. By endowing every group $\mathbb{Z}/{a_n\mathbb{Z}}$ with the discrete topology, the corresponding inverse limit
				\begin{equation*}
				\begin{gathered}
				G_{\bm{\mathrm{a}}}:= \varprojlim_{n} \mathbb{Z}/{a_n\mathbb{Z}}  =\\
				=\Bigl\{\bm{\mathrm{z}}\in \prod_{n} \mathbb{Z}/{a_n\mathbb{Z}}\big| z_{n+1} \equiv z_n \mod a_n \text{ for all } n \in \mathbb{N} \Bigr\}
				\end{gathered}
				\end{equation*}
				is a compact, metric, zero-dimensional, abelian group called \emph{odometer of type $\bm{\mathrm{a}}$} or \emph{adding machine} which is a Cantor space. Every odometer $G_{\bm{\mathrm{a}}}$ contains the element $\bm{1}=(1,1,\dots)$ inducing a minimal Cantor system $(G_{\bm{\mathrm{a}}},\varphi)$ by $\varphi \colon \bm{\mathrm{z}} \mapsto \bm{\mathrm{z}}+\bm{1}$. A survey on such systems can be found in \cite{dow05}.
			\end{enumerate}
		\end{ex}
		
		We call a Cantor system an \emph{odometer} if it is conjugate to an odometer of type $\bm{\mathrm{a}}$ for some suitable $\bm{\mathrm{a}}$, and call it a \emph{subshift} if it is conjugate to a subsystem of some $\mathbb{Z}$-shift.
		
		\begin{rem}\label{rem: symbolic dynamics}
			Any topological dynamical system $(X,G)$ with a finite clopen partition $\{P_i\}_{i \in I}$ such that the $G$-translates of atoms separate points is conjugate to a $G$-shift in $\mathcal{A}^G$ for some finite alphabet $\mathcal{A}$ -- modelled from the $G$-action on the $G$-translates of atoms. This standard construction gives rise to the field of \emph{symbolic dynamics}.
		\end{rem}
		
		\begin{defi}[\cite{mb14a}, p.1639]\label{defi: complexity}
			Let $\mathcal{A}$ be a finite alphabet and let $(X,\varphi)$ be a $\mathbb{Z}$-subshift over $\mathcal{A}$. The \emph{complexity $\rho$ of $(X,\varphi)$} is the function $\rho \colon \mathbb{N} \to \mathbb{N}$ where $\rho(n)$ is defined as the number of words of length $n$ over $\mathcal{A}$ that appear as subwords in elements of $X$.
		\end{defi}
		
		\begin{rem}(\cite{gla03}, Proposition~14.4.3 \& Corollary~14.7)\phantomsection\label{rem: complex+entro}
			\begin{enumerate}[(i)]
				\item For a subshift $(X,\varphi)$ its \emph{topological entropy}\footnote{For a formal definition of topological entropy see \cite{gla03}, § 14.1.\\} $h_{top}(X,\varphi)$ computes via its complexity $\rho$ in that
				\begin{equation*}
				h_{top}(X,\varphi)=\lim\limits_{n\to \infty} \frac{1}{n} \log \rho(n).
				\end{equation*}
				
				\item Topological entropy is invariant under flip-conjugacy i.e. every topological dynamical $\mathbb{Z}$-system satisfies $h_{top}(X,\varphi)=h_{top}(X,\varphi^{-1})$.
			\end{enumerate}
		\end{rem}
		
		\begin{ex}\label{ex: sturm+toeplitz}
			In the following examples $\varphi$ always denotes the shift operator.
			\begin{enumerate}[(i)]
				\item (\cite{lot02}, Chapter~2) Let $\alpha \in \mathbb{R}_{ >0}$ be irrational. Let the $R_\alpha \colon [0,1) \to [0,1)$ the rotation given by $R_\alpha \colon x \mapsto x + \alpha \mod 1$ and let $\mathcal{A}:=\{a,b\}$. Define a map $\Phi \colon [0,1) \to \{a,b\}$ by
				\begin{equation*}
				\Phi(x):=
				\begin{cases}
				a, & x \in [0,\alpha)\\
				b, & \text{ else}
				\end{cases}
				\end{equation*}
				Furthermore define:
				\begin{equation*}
				\Sigma_\alpha:=\overline{\{ \{\Phi(R_\alpha^i(x))\}_{i \in \mathbb{Z}} | x \in [0,1)  \}} \subset \mathcal{A}^{\mathbb{Z}}
				\end{equation*}
				Then $(\Sigma_\alpha,\varphi)$ is a minimal Cantor system with complexity $\rho(n)=n+1$. Such systems are called \emph{Sturmian shifts}. A look at the spectrum of the associated Koopman operator shows that for irrational $\alpha,\beta \in (0,\frac{1}{2})$ with $\alpha \neq \beta$ the Sturmian shifts $(\Sigma_\alpha,\varphi)$ and $(\Sigma_\beta,\varphi)$ are not flip conjugate.
				
				\item (\cite{dow05}, §7) Let $\mathcal{A}:=\{a_1,\dots,a_k \}$ be a finite alphabet. A sequence $\mathbf{s}= \{s_i\}_{i \in \mathbb{Z}} \in \mathcal{A}^{\mathbb{Z}}$ is called a \emph{Toeplitz sequence} if for every $i \in \mathbb{Z}$ there exists an $p \in \mathbb{N}$ such that $s_i=s_{i+pk}$ for all $k \in \mathbb{Z}$. For such an $\mathbf{s}$ define $T_\mathbf{s}:=\overline{\operatorname{Orb}_\varphi (\mathbf{s})} \subset \mathcal{A}^{\mathbb{Z}}$. If $\mathbf{s}$ is not periodic, the arising subshift $(T_\mathbf{s},\varphi)$ is a minimal Cantor system called a \emph{Toeplitz shift}. For every $\alpha \in \mathbb{R}_{>0}$ there exists a Toeplitz shift $(T_\mathbf{s},\varphi)$ such that the entropy satisfies $h(T_\mathbf{s},\varphi)=\alpha$.
			\end{enumerate}
		\end{ex}
		
		Odometers are in some sense opposites to minimal shifts in that every homomorphism from an odometer to a subshift has finite image. A minimal Cantor system is an odometer \emph{if and only if} it is the inverse limit of a sequence of periodic systems. In particular one has the following:
		
		\begin{lem}[\cite{mat13}, Lemma~2.2]\label{lem: notod.}
			Let $(X,\varphi)$ be a minimal Cantor system. If it is not an odometer, there exists a homomorphism from $(X,\varphi)$ to the full $2$-shift $(\{0,1\}^\mathbb{Z},\sigma)$ with infinite image.
		\end{lem}
		\begin{proof}
			Let $\{C_n\}_{n \in \mathbb{N}}$ be the set of all clopen sets in $X$ -- it is countable since $X$ is a Cantor space -- and define a family of maps $\{\pi_n\colon X \to \{0,1 \}^\mathbb{Z} \}_{n \in \mathbb{N}}$ by $(\pi_n(x)_k)_{k \in \mathbb{Z}}:=\mathbf{1}_{\varphi^{-k}(C_n)}(x)$. Every such map is continuous and $\pi_n \circ \varphi = \sigma \circ \pi_n$ for all $n \in \mathbb{N}$. The conclusion follows, if finiteness of $|\pi_n(X)|$ for all $n \in \mathbb{N}$ implies that $(X,\varphi)$ is conjugate to an odometer $G_{\bm{\mathrm{a}}}$ for some sufficient sequence $\bm{\mathrm{a}}=(a_n)_{n\in \mathbb{N}}$. Hence we need to find a sufficient sequence $\bm{\mathrm{a}}=(a_n)_{n\in \mathbb{N}}$ with associated inverse system  $\rho_n \colon \mathbb{Z}/{a_{n+1}\mathbb{Z}} \to \mathbb{Z}/{a_n\mathbb{Z} }$ defined by $\rho_n(z)=z\mod a_n$ and a family of continuous maps $\{\tilde{\pi}_n \colon X \to  \mathbb{Z}/{a_n\mathbb{Z} }\}_{n \in \mathbb{N}}$ such that $\tilde{\pi}_n \circ \varphi (x)=\tilde{\pi}_n+1$ for all $x \in X$, $\tilde{\pi}_{n+1}=\rho_n \circ \tilde{\pi}_n$ and $\pi_n$ factors through $\tilde{\pi}_n$ for all $n \in \mathbb{N}$. To this end assume $|\pi_n(X)|$ is finite for all $n \in \mathbb{N}$. Let $a_1=|\pi_1|$. Then there exists a bijection $b_1 \colon \pi_1(X) \to \mathbb{Z}/{a_1\mathbb{Z}}$ such that $b_1 \circ \sigma (z) = 1+b_1(z)$ for $z \in \pi_1(X)$. It is easy to show $\tilde{\pi}_1:=b_1 \circ \pi_1$ is sufficient. Assume $\tilde{\pi}_n$ has already been defined, then $\pi_{n+1} \times \tilde{\pi}_n$ is a continuous map. Define $a_{n+1}:=|(\pi_{n+1} \times \tilde{\pi}_n)(X)|$. Then there exists a bijection $b_{n+1} \colon (\pi_{n+1} \times \tilde{\pi}_n)(X) \to \mathbb{Z}/{a_{n+1}\mathbb{Z}}$ such that $\tilde{\pi}_{n+1}:=b_{n+1} \circ (\pi_{n+1} \times \tilde{\pi}_n)$ is sufficient. The map $f\colon X \to G_{\bm{\mathrm{a}}}$ defined by $f(x):=(\tilde{\pi}_n)_{n \in\mathbb{N}}$ for $x \in X$ is continuous and injective -- since the functions $\{\pi_n\}_{n \in \mathbb{N}}$ separate points in $X$ -- and thus $(X,\varphi)$ is conjugate to the odometer of type $\bm{\mathrm{a}}$.
		\end{proof}
		
		Minimal Cantor systems play an exceptional role for minimal $\mathbb{Z}$-systems over compact, metrizable spaces steming from the exceptional role\footnote{The phrase ``exceptional role" is chosen, since we can not speak of a universal property in the precise meaning of category theory.\\} of the Cantor space $\mathcal{C}$ in the category of compact metric spaces:
		\begin{thm}[\cite{kec95}, 4.18]\label{thm: cant}
			Every non-empty, compact, metrizable topological space is a continuous image of $\mathcal{C}$.
		\end{thm}
		
		This implies for minimal Cantor systems:
		\begin{thm}[\cite{gps95}, p.~55]
			Let $(X,\psi)$ be a minimal topological $\mathbb{Z}$-system, such that $X$ is metrizable (and compact). Then $(X,\psi)$ is the factor of a minimal Cantor system. 
		\end{thm}
		\begin{proof}
			By Theorem~\ref{thm: cant} for every non-empy, compact, metrizeable space $X$ there exists a continuous surjection from $\mathcal{C}$ onto $X$. Let $\pi$ be such a surjection. The product space $\mathcal{C}^\mathbb{Z}$ is a Cantor space by Example~\ref{ex: CS} (ii). The subset $C:=\{(c_i)_{i \in \mathbb{Z}}:c_i \in \mathcal{C}, \pi(c_{i+1})=\psi(\pi(c_i)) \}$  is closed, as it is an intersection of cylinder sets, hence compact. It is invariant under the natural $\mathbb{Z}$-shift $\sigma$ on $\mathcal{C}^\mathbb{Z}$. By Proposition~\ref{prop: ex.minsubs}, the system $(C, \sigma|_C)$ must contain a minimal subsystem $(M,\sigma|_M)$. Let $p_0|_M$ denote the restriction of the projection $p_0$ on the $0$'th component in $\mathcal{C}^\mathbb{Z}$ to $M$. Then $(M,\sigma)$ is a minimal Cantor system and $\pi \circ p_0|_M$ is continuous, surjective and $(\pi \circ p_0|_M) \circ \sigma_M = \psi \circ (\pi \circ p_0|_M)$ holds. 
		\end{proof}		
		
		\subsection{Some preliminaries of crossed products}\label{subs: crossprod}
		
		In dynamics often the study of group actions on spaces is replaced by the study of automorphism groups of algebras.
		Crossed product C*-algebras are  operator algebras that arise from C*-dynamical systems and, in particular, from topological dynamical systems, thus providing a rich source of C*-algebras.\footnote{See Appendix~\ref{app: c*-algebras} for basic terms of C*-algebras.\\} 
		
		\begin{defi}[\cite{dav96}, p.~216]
			A \emph{C*-dynamical system} is a triple $(\mathfrak{A}, G, \alpha)$ comprising a C*-algebra $\mathfrak{A}$, a locally compact group $G$ and a continuous group-homomorphism $\alpha \colon G \to \mathrm{Aut}(\mathfrak{A})$ i.e. a continuous action of $G$ on $\mathfrak{A}$ by automorphisms.
		\end{defi} 
		
		\begin{ex}
			Let $(X,G)$ be a topological dynamical system given by an action $\alpha \colon G \to \mathrm{Homeo}(X)$. Then $\alpha$ gives rise to a C*-dynamical system $(C(X), G, \tilde{\alpha})$ by the induced map $\tilde{\alpha} \colon G \to \mathrm{Aut}(C(X))$ given by $\big(\tilde{\alpha}(g)(f)\big)(x):= f(\alpha(g^{-1})x)$.
		\end{ex}
		
		\begin{defi}[\cite{dav96}, p.~216]
			Let $(\mathfrak{A}, G, \alpha)$ be a C*-dynamical system. A \emph{covariant representation of $(\mathfrak{A}, G, \alpha)$} is a pair $(\pi,u)$ for which there exists a Hilbert space $\mathcal{H}$ such that $\pi$ is a $*$-algebra-representation of $\mathfrak{A}$ on $\mathcal{H}$ and $u$ is a unitary representation of $G$ on $\mathcal{H}$ satisfying $\pi(\alpha(g)A)=u(g)\pi(A)u(g)^*$ for all $g \in G, A \in \mathfrak{A}$.
		\end{defi}
		
		Let $(\mathfrak{A}, G, \alpha)$ be a C*-dynamical system. The space of functions $C_c(G,\mathfrak{A})$ becomes a $*$-algebra by defining\footnote{$\Delta$ denotes the modular function of $G$, see Subsection~\ref{app: group c-star-algebras}.\\}:
		\begin{equation*}
		\begin{gathered}
		f^{*}(g)=\Delta(g^{-1}) \alpha(g)(f(g^{-1})^*)\\
		(f_1 \ast f_2)(g):=\int\limits_{h \in G} f_1(h)\alpha(h)\big(f_2(h^{-1}g)\big)\;\mathrm{d}\mu_G(h)
		\end{gathered}
		\end{equation*}
		
		Define  $\mathrm{L}^p(G, \mathfrak{A})$ to be the completion of $C_c(G,\mathfrak{A})$ with respect to the $p$-norm given by
		\begin{equation*}
		\|f\|_p:=\big(\int\limits_{h \in G}\|f(g)\|^p\;\mathrm{d}\mu_G(h)\big)^{1/p}
		\end{equation*}
		
		Any covariant representation $(\pi,u)$ of the system $(\mathfrak{A}, G, \alpha)$ induces a representation of $C_c(G,\mathfrak{A})$ by integration:
		\begin{equation*}
		\pi\rtimes u(f):=\int\limits_{g \in G} \pi\big(f(g)\big) u(g) \mathrm{d}\mu_G(g)
		\end{equation*}

		\begin{defi}[\cite{dav96}, p.~217]
			Let $(\mathfrak{A}, G, \alpha)$ be a C*-dynamical system.The completion of $C_c(G,\mathfrak{A})$ with respect to the norm
			\begin{equation*}
			\|f\|_{C^*}:=\sup \{\|\pi \rtimes u(f)\|: (\pi,u) \text{ is a covariant representation of }(\mathfrak{A}, G, \alpha) \}
			\end{equation*}
			is a C*-algebra, called the \emph{(maximal) crossed product C*-algebra of $(\mathfrak{A}, G, \alpha)$} denoted by $\mathfrak{A} \rtimes_\alpha G$. In the case of topological dynamical system $(X,G)$, we write $C(X) \rtimes G$ and $C(X) \rtimes_{\varphi} \mathbb{Z}$ in case of a $\mathbb{Z}$-system $(X,\varphi)$ for the associated crossed product C*-algebras.\footnote{The definition of reduced crossed product C*-algebras has been left out. For minimal Cantor systems the reduced and the maximal crossed product coincide.\\}

		\end{defi}
		
		\begin{rem}
			 Every locally compact group $G$ gives rise to a \emph{degenerate} C*-dynamical system $(\mathbb{C},G,\operatorname{triv})$ and every C*-algebra $\mathfrak{A}$ a \emph{degenerate} C*-dynamical system $(\mathfrak{A},\{e\},\operatorname{triv})$ where $\operatorname{triv}$ denotes the trivial action. Group C*-algebras\footnote{See Appendix~\ref{app: group c-star-algebras}.\\} are crossed product C*-algebras for a degenerate C*-dynamical system in that $C^*(G)= \mathbb{C} \rtimes_{\operatorname{triv}} G$.
		\end{rem}
		
		\begin{rem}[\cite{dav96}, VIII.2]\label{rem: crossedprod}
			In the case of a topological $\mathbb{Z}$-system $(X,\varphi)$ the crossed product C*-algebra $C(X)\rtimes_\varphi \mathbb{Z}$ is the universal C*-algebra generated by $C(X)$ and a unitary $u$ such that $ufu^*= f\circ \varphi^{-1}$. The set of ``Laurent polynomials" 
			\begin{equation*}
			\{\sum_{k=-n}^{n} f_k u^k|n \in \mathbb{N}, f_k \in C(X) \}
			\end{equation*}
			is a dense subalgebra of $C(X)\rtimes_\varphi \mathbb{Z}$. There is an action of the dual group $\hat{\mathbb{Z}}=\mathbb{T}$ by automorphisms on $C(X)\rtimes_\varphi \mathbb{Z}$: Every $\lambda \in \mathbb{T}$ induces an automorphism $\rho_\lambda$ of $C(X)\rtimes_\varphi \mathbb{Z}$ satisfying $\rho_\lambda |_{C(X)}=\operatorname{id}$ and $\rho_\lambda \colon u^k \mapsto \lambda^k u^k$ for $k \in \mathbb{Z}$.\footnote{The map $\lambda u \colon z \to \lambda^z u^z$ defines a unitary representation of $\mathbb{Z}$. The pair $(\operatorname{id},\lambda u)$ is a covariant representation thus inducing a canonical $*$-algebra representation of $C(X)\rtimes_\varphi \mathbb{Z}$ on itself.\\} This action of $\mathbb{T}$, also denoted by $\hat{\varphi}$, induces a conditional expectation $E \colon C(X)\rtimes_\varphi \mathbb{Z} \to C(X)$ by
			\begin{equation*}
			E(a):=\int_{\lambda\in\mathbb{T}} \rho_\lambda(a)\;\mathrm{d}\mu_{\mathbb{T}}
			\end{equation*}
			
			If the system is essentially free the above constitutes $(C(X),C(X)\rtimes_\varphi \mathbb{Z})$ as a Cartan pair.\footnote{See Definition~\ref{defi: cartan pair}.\\}
				The conditional expectation $E$ correponds to the computation of the zero'th Fourier coefficient of a function on $\mathbb{T}$ in that on the dense subalgebra of ``Laurent polynomials" we have
				\begin{equation*}
				\begin{gathered}
				E(\sum_{k=-n}^{n} f_k u^k)=\int_{\lambda\in\mathbb{T}} \rho_\lambda(\sum_{k=-n}^{n} f_k u^k)\;\mathrm{d}\mu_{\mathbb{T}}=\\
				\sum_{k=-n}^{n} \int_{0}^{1}  \rho_{\exp(2\pi i t)}(f_ku^k)\;\mathrm{d}t=
				\sum_{k=-n}^{n} f_k \int_{0}^{1}  \exp(2\pi i k t)  (u^k)\;\mathrm{d}t\\
				\sum_{k=-n}^{n} f_k \int_{0}^{1}  \exp(2\pi i k t)  u^k\;\mathrm{d}t=f_0
				\end{gathered}
				\end{equation*}
				Analogously, for every $k \in \mathbb{Z}$ the maps $E_k \colon C(X)\rtimes_\varphi \mathbb{Z} \to C(X)$ given by $a \mapsto E(au^{-k})$ compute the ``k'th Fourier coefficient".
		\end{rem}
		
		For a C*-dynamical $\mathbb{Z}$-system $(\mathfrak{A},\mathbb{Z},\alpha)$ the cyclical 6-term exact sequence of $K$-groups\footnote{See Appendix~\ref{app: K-theo}.\\} gives rise to the Pimsner-Voiculescu sequence:\footnote{The cylic 6-term exact sequence is applied to a Toeplitz extension of $\mathfrak{A} \otimes \mathbb{K}$ by $\mathfrak{A} \rtimes_\alpha \mathbb{Z}$. The resulting sequence can be shown to correspond to the Pimsner-Voiculescu sequence via diagram chasing.\\}
		\begin{equation*}
		\begin{tikzcd}
		K_0(\mathfrak{A}) \arrow{r}{\operatorname{id}-\alpha_{*}^{-1}} & K_0(\mathfrak{A}) \arrow{r}{\iota_{*}} & K_0(\mathfrak{A} \rtimes_\alpha \mathbb{Z}) \arrow{d}{\delta_0}\\
		K_1(\mathfrak{A} \rtimes_\alpha \mathbb{Z}) \arrow{u}{\delta_1} & K_1(\mathfrak{A}) \arrow{l}{\iota_{*}} & K_1(\mathfrak{A})  \arrow{l}{\operatorname{id}-\alpha_{*}^{-1}}
		\end{tikzcd}
		\end{equation*}
		
		Here $\iota$ denotes the canonical embedding $\mathfrak{A} \hookrightarrow \mathfrak{A} \rtimes_\alpha \mathbb{Z}$. For the crossed product C*-algebra $C(X)\rtimes_\varphi \mathbb{Z}$ associated with a Cantor system $(X,\varphi)$, it holds that $K_0(C(X)) = C(X,\mathbb{Z})$ and $K_1(C(X))=0$ (Example~\ref{ex: afc*}) and above sequence becomes
		\begin{equation*}
		0 \longrightarrow K_1(C(X)\rtimes_\varphi \mathbb{Z}) \overset{\delta_1}{\longrightarrow} C(X,\mathbb{Z}) \overset{\operatorname{id}-\varphi_{*}}{\longrightarrow} C(X,\mathbb{Z}) \overset{\iota_{*}}{\longrightarrow} K_0(C(X)\rtimes_\varphi \mathbb{Z}) \longrightarrow 0
		\end{equation*}
		which produces isomorphisms
		\begin{equation*}
			K_0(C(X)\rtimes_\varphi \mathbb{Z}) \cong C(X,\mathbb{Z})/ \Ima(\operatorname{id}-\varphi_{*}) \quad \text{and} \quad  K_1(C(X)\rtimes_\varphi \mathbb{Z}) \cong \mathbb{Z}.
		\end{equation*}

		\section{A tool box of minimal Cantor systems}\label{sec: the toolbox of minimal Cantor systems}
		
		Subsection~\ref{subs: brat} brushes the theory of AF C*-algebras and Bratteli diagrams. Subsection~\ref{subs: kak.ro} provides a description of Kakutani-Rokhlin partitions. Nested sequences of such partitions allow to approximate minimal Cantor systems, showing these are conjugate to Bratteli-Vershik systems (see Subsection~\ref{subs: ord.brat}). Subsection~\ref{subs: dimgroups} gives the definition of dimension groups. In Subsection~\ref{subs: amplegroups} cites Krieger's work on ample groups. These notions are important to the classification of minimal Cantor systems developed in \cite{put89}, \cite{hps92}, \cite{gps95} and \cite{gps99} (see Subsection~\ref{subs: classific}) and the works of Matui on topological full groups (see Subsection~\ref{subs: simplicity and finite generation}).
		
		\subsection{AF C*-algebras and Bratteli diagrams}\label{subs: brat}
		
		AF C*-algebras are limits of sequences of finite dimensional C*-algebras. The special case of uniformly hyperfinite algebras i.e. limits of sequences of simple finite dimensional C*-algebras, had already been classified by James G. Glimm in \cite{gli60} and Jaques Dixmier in \cite{dix67} in terms of their dimension range i.e. in terms of their $K_0$-groups. Ola Bratteli related AF C*-algebras with certain diagrams:
		
		\begin{defi}[\cite{hps92}, Definition 2.1 \& 2.2]\phantomsection\label{defi: Bratteli diagram}
			~\begin{enumerate}[(i)]
				\item A \emph{Bratteli diagram} $\Gamma=(V,E)$ consists of the following data:
				\begin{enumerate}
					\item a vertex set $V$, which is a disjoint union $\bigsqcup_{n \in \mathbb{N}} V_n$ of finite, non-empty sets $V_n$ with $V_1$ a singleton
					
					\item a set of edges $E$, which is a disjoint union $\bigsqcup_{n \in \mathbb{N}} E_n$ of finite, non-empty sets $E_n$
					
					\item a pair of maps $r,s \colon E \to V$ called \emph{range resp. source map}, such that $r(E_n) \subseteq V_{n+1}$, $s(E_n) \subseteq V_n$, $s^{-1}(v) \neq \emptyset$ for all $v \in V$ and $r^{-1}(v) \neq \emptyset$ for all $v \in V \setminus V_1$.\footnote{The representation as an arrow diagram follows by interpreting an edge $e$ as an arrow from $s(e)$ to $r(e)$.\\}
				\end{enumerate}
				
				\item An \emph{isomorphism of Bratteli diagrams} consists of a bijection between the vertex sets and a bijection between the sets of edges intertwining the range and source maps such that the grading on these sets is preserved.
			\end{enumerate}
			Let $\Gamma=(V,E)$ be a Bratteli diagram.
			\begin{enumerate}[(i),resume]
				\item Let $v \in V_i$ and $v' \in V_j$ with $i<j$. A \emph{path from $v$ to $v'$} is an edge tuple $\{(e_i,e_{i+1},\dots,e_{j-1})\}$ such that $e_k \in E_k$ for $i \leq k \leq j-1$, $s(e_i)=v$, $r(e_{j-1})=v'$ and $r(e_k)=s(e_{k+1})$ for $i \leq k \leq j-2$. The set of all paths from $v$ to $v'$ is denoted by $\mathcal{P}_{v,v'}$
				An \emph{infinite path} $P$ in $\Gamma$ is a sequence $\{e_n\}_{n \in \mathbb{N}}$ of edges $e_n \in E_n$ such that $r(e_n)=s(e_{n+1})$ for all $n \in \mathbb{N}$. The set of all infinite paths in $\Gamma$ is denoted $\mathcal{P}_\Gamma$.
				
				\item Let $\{m_n \}_{n \in \mathbb{N}}$ be a sequence with $m_n \in \mathbb{N}$ and $m_n < m_{n+1}$ for all $n \in \mathbb{N}$ and $m_1=0$. Define a Bratteli diagram $\tilde{\Gamma}=(\tilde{V},\tilde{E})$ by setting $\tilde{V}_n:=V_{m_n}$ and $\tilde{E}_n$ to be the set of paths from vertices in $V_{m_n}$ to vertices in $V_{m_{n+1}}$ and by defining range and source of such an edge by		
				\begin{equation*}
				\begin{gathered}
				r((e_{m_{n}},e_{m_{n}+1},\dots,e_{m_{n+1}-1})):=r(e_{m_{n+1}-1})\\
				s((e_{m_{n}},e_{m_{n}+1},\dots,e_{m_{n+1}-1})):=s(e_{m_{n}}).
				\end{gathered}
				\end{equation*}
				The Bratteli diagram $\tilde{\Gamma}$ is called a \emph{contraction} or a \emph{telescoping of $\Gamma$}. Two Bratteli diagrams are \emph{equivalent} if one can be obtained from the other via a finite chain of telescopings and isomorphisms.
			\end{enumerate}
		\end{defi}
		
		\begin{rem}
			\emph{Microscoping}, a notion dual to telescoping, iteratively inserts new layers of vertices between consecutive vertex sets $V_n$ and $V_{n+1}$ by splitting each edge between $V_n$ and $V_{n+1}$ by a new vertex.
		\end{rem}
		
		Ola Bratteli introduced Bratteli diagrams to study AF C*-algebras in \cite{bra72}:
		\begin{defi}[\cite{bra72}, Definition 1.1 \& §1.3]
			A C*-algebra is said to be \emph{approximately finite-dimensional} (abbreviated as AF), if it is the direct limit of a sequence of finite dimensional C*-algebras.\footnote{Finite dimensional C*-algebras are finite direct sums of full matrix algebras over $\mathbb{C}$.\\}
		\end{defi}
		
		\begin{rem}
			By continuity of the $K_1$-functor and $K_1(\mathbb{C})=0$ the $K_1$-group of AF C*-algebras always vanishes. In particular the crossed product C*-algebra $C(X)\rtimes_\varphi \mathbb{Z}$ of a Cantor system $(X,\varphi)$ is \emph{never} AF, however, every minimal Cantor system $C(X)\rtimes_\varphi \mathbb{Z}$ admits unital embeddings into AF C*algebras (see Theorem~\ref{thm: cantor-AF}).
		\end{rem}
		
		Every AF C*-algebra admits a class of Bratteli diagrams depending on the directed systems of embeddings of finite dimensional C*-subalgebras. Conversely, every Bratteli diagram $\Gamma=(V,E)$ induces a unique AF C*-algebra $C^*(V,E)$: 
		Denote for every $v \in V$ by $M_v$ the full matrix algebra $M_{|\mathcal{P}_{v_0,v}|}(\mathbb{C})$. The  AF C*-algebra $C^*(V,E)$ is defined as the direct limit of the sequence
		\begin{equation*}
		\mathbb{C} \overset{i_0}{\hookrightarrow} \bigoplus_{v \in V_1} M_v
		\overset{i_1}{\hookrightarrow} \bigoplus_{v \in V_2} M_v \overset{i_2}{\hookrightarrow} \dots
		\end{equation*}
		where $i_k \colon \bigoplus_{v \in V_k} M_v \overset{i_k}{\hookrightarrow} \bigoplus_{v \in V_{k+1}} M_v$ is given on a summand $M_v$ as a sum of diagonal maps $M_v \to M_{v'}$ over $V_{k+1}$ with multiplicity corresponding to the number of connecting edges between $v$ and $v'$. Information on properties of $C^*(V,E)$ like the ideal structure can be read in the underlying Bratteli diagram. Theorem 2.7 of \cite{bra72} essentially implies the following:
		\begin{thm}\label{thm: af-brat}
			The just descibed procedures induce a 1-1 correspondence between isomorphism classes of AF C*-algebras and equivalence classes of Bratteli diagrams.
		\end{thm}
		
		This is a manifestation of a much more general idea where C*-algebras are modelled on combinatorial structure.\footnote{A bit more on this in Subsection~\ref{subs: some historical remarks}.\\} \c{S}erban Str\u{a}til\u{a} and Dan-Virgil Voiculescu showed that AF C*-algebras can be diagonalized, which offers means to study their ideal structure. This was one of the key motivations for \cite{ren80} by Jean Renault (-- see Chapter~\ref{chap: 2}):
		
		\begin{thm}[\cite{sv75}, p.~17]\label{thm: AF}
			For every AF C*-algebra $\mathfrak{A}$ there exist a maximal abelian subalgebra $\mathfrak{C} \subseteq \mathfrak{A}$, a conditional expectation $P \colon \mathfrak{A} \to \mathfrak{C}$ and a locally finite unitary subgroup $U \subseteq U(\mathfrak{A})$ such that:
			\begin{enumerate}[(i)]
				\item $u^*Cu=C$ for all $u \in U$
				
				\item $P(u^*Au)=u^*P(A)u$ for all $A \in \mathfrak{A}$ and $u \in U$
				
				\item The C*-algebra $\mathfrak{A}$ is the closed linear span of the set $\{uC| u \in U, C \in \mathfrak{C}\}$.
				
				\item Let $X$ be the spectrum of $\mathfrak{C}$ i.e. $\mathfrak{C} = C(X)$, and let $\Gamma_U$ be the group of homeomorphisms of $X$ identified\footnote{By Gelfand duality -- see Appendix~\ref{app: basic terms of c*-algebras}.\\} with the group of $*$-automorphisms of $C$ given by the set $\{\gamma_u \colon c \mapsto u^*cu| u \in U \}$. The C*-algebra $\mathfrak{A}$ is $*$-isomorphic to a quotient of $ C(X)\rtimes \Gamma_U$.
			\end{enumerate}
		\end{thm}  
		
		The group $\Gamma_U$ is actually a first glimpse into the realms of topological full groups an thus of Chapter~\ref{chap: 3}.
		
		\subsection{Kakutani-Rokhlin partitions}\label{subs: kak.ro}
		
		Kakutani-Rokhlin partitions are a notion transferred from the measured dynamics setting: The \emph{Rokhlin-Kakutani lemma} is a result in ergodic theory proven independently by Vladimir A. Rokhlin and Shizuo Kakutani and was used by Dye to approximate measured dynamical systems. We conform in our choice of notation with \cite{bk00} and \cite{gm14}.
		
		\begin{defi}[\cite{gm14}, Definition~3.1]
			Let $(X,\varphi)$ be a minimal Cantor system.
			\begin{enumerate}[(i)]
				\item Let $A$ be a non-empty clopen subset of $X$. The \emph{first return function of $A$} is the continuous function $t_{\varphi,A} \colon A \to \mathbb{N}$ defined by $t_{\varphi,A}(x):=\min\{n \in \mathbb{N}| \varphi^n(x)\in A \}$.\footnote{The function $t_{\varphi,A}$ is well-defined as forward-orbits are dense in minimal systems by Remark~\ref{rem: orbdens}. Furthermore, it is continuous, since $t_{\varphi,A}^{-1}(n)=\varphi^{-n}(A \cap \varphi^n(A))\setminus \bigcup_{i=1}^{n-1} t_{\varphi,A}^{-1}(i)$.\\}
				
				\item Let $A$ be a clopen subset of $X$ and $n \in \mathbb{N}$ such that the sets $A, \varphi(A), \dots \varphi^{n-1} (A)$ are mutually disjoint. The collection $\{A, \varphi(A), \dots \varphi^{n-1}(A) \}$ is called a \emph{tower of height $n$}. The set $A$ is called \emph{base of the tower}, the set $\varphi^{n-1}(A)$ its \emph{roof}.
				
				\item A disjoint clopen partition $\mathcal{A}$ of $X$ of the form $\mathcal{A}=\{\varphi^k (A_i)|0 \leq k \leq h_i-1, i \in [1,n] \}$ for some $n \in \mathbb{N}$ is called \emph{Kakutani-Rohklin partition}. Every atom of $\mathcal{A}$ corresponds to a pair $(k,i)$ with $0 \leq k \leq h_i-1, i \in \{1,\dots,n\}$. We set $D_{k,i}=\varphi^k (A_i)$ and we define $U(\mathcal{A})$ to be the set of all such pairs $(k,i)$. Denote by $D(i):= \bigsqcup_{0 \leq k \leq h_i-1} D_{k,i}$ the $i$-th tower and denote the minimal height of a tower in $\mathcal{A}$ by $h_{\mathcal{A}}:=\min \{h_i|i \in \{1,\dots,n\} \}$.
				
				\item Let $\mathcal{A}$ be a Kakutani-Rohklin partition of $X$. Then $B(\mathcal{A}):=\bigsqcup_i A_i=\bigsqcup_i D_{0,i}$ is called the \emph{base of $\mathcal{A}$} and $R(\mathcal{A}):=\bigsqcup_i \varphi^{h_i - 1} (A_i)=\bigsqcup_i D_{h_i-1,i}$ is called the \emph{roof of $\mathcal{A}$}.
			\end{enumerate}
		\end{defi}
		
		\begin{rem}\label{rem: kakropa}
			Such partitions exist: Let $(X,\varphi)$ be a minimal Cantor system and let $A$ be a proper non-empty clopen subset of $X$. By continuity of the first return function, there exists a finite family $\{n_i\}$ of positive integers and a finite clopen partition $A=\bigsqcup_i A_i$ such that $t_{\varphi,A}(A_i)=n_i$. Since $\varphi^k(A_i) \cap \varphi^l(A_j)=\emptyset$ for $0 \leq k \leq n_i-1$, $0 \leq l \leq n_j-1$ and $i \neq j$, The partition $\mathcal{A}=\{\varphi^k (A_i)|0 \leq k \leq n_i-1, i \in \{1,\dots,n\} \}$ is a Kakutani-Rohklin partition of $X$.
		\end{rem}

		Let $\mathcal{A}=\{D_{k,i}|0 \leq k \leq h_i-1, i \in \{1,\dots,n\} \}$ be a Kakutani-Rohklin partition of a minimal Cantor system $(X,\varphi)$. It partitions $X$ into $n$ disjoint towers, where the number $h_i$ signifies the height of the $i$-th tower.
		
		\begin{defi}[\cite{gps95}, Definition 1.5]
			Let $(X,\varphi)$ be a minimal Cantor system and let $A$ be a non-empty clopen subset of $X$.
			\begin{enumerate}[(i)]
				\item The homeomorphism $\varphi_A \in \operatorname{Homeo}(X)$ defined by
				\begin{equation*}
				\varphi_A(x):=
				\begin{cases}
				\varphi^{t_{\varphi,A}(x)}(x), & x \in A\\
				x, & \text{ else}
				\end{cases}
				\end{equation*}
				where $t_A$ is the first return function on $A$, is called the \emph{induced transformation of $\varphi$ on $A$}.
				
				\item The restriction of the induced transformation $\varphi_A$ to $A$ is a minimal homeomorphism of $A$ by the minimality of $\varphi$ and the minimal Cantor system $(A,\varphi_A|_A)$ is called the \emph{induced} or \emph{derivative Cantor system of $(X,\varphi)$ over $A$}.
			\end{enumerate}
		\end{defi}	
		
		As their measure theoretic counterparts Kakutani-Rokhlin partitions are used to approximate dynamical systems. The definition is essentially implicit in Theorem 4.2 of \cite{hps92}.
		
		\begin{defi}
			Let $(X,\varphi)$ be a minimal Cantor system.
			\begin{enumerate}[(i)]
				\item Let $\mathcal{A}=\{D_{k,i}|0 \leq k \leq h_i-1, i \in \{1,\dots,n\} \}$ be a Kakutani-Rokhlin partition. Let $\mathcal{P}$ be a finite partition of $X$. Let $P \in \mathcal{P}$ and let $U_{P,\mathcal{A}} \subseteq U(\mathcal{A})$ be the set off all pairs $(k,i)$ such that $P \cap D_{k,i} \neq \emptyset$ and $P \triangle D_{k,i} \neq \emptyset$. Then $\{ \varphi^j(D_{0,i} \cap \varphi^{-k}(P))|0 \leq j \leq k, (k,i) \in S_{P,\mathcal{A}}, P \in \mathcal{P} \}$ is a Kakutani-Rokhlin partition of $X$ called the \emph{refinement of $\mathcal{A}$ by $\mathcal{P}$}.
				
				\item Let $Y \subseteq X$ such that $Y$ is either clopen or a singleton. A sequence $\{\mathcal{A}_n\}_{n \in \mathbb{N}}=\{D_{k,i}^n|0 \leq k \leq h^n_i-1, i \in \{1,\dots,i_n\} \}$ of Kakutani-Rokhlin partitions is called a \emph{nested sequence of Kakutani-Rokhlin partitions around $Y$} if it satisfies the following properties:
				\begin{enumerate}
					\item $\bigcup_{n} \{\mathcal{A}_n\}$ generates the topology of $X$.
					
					\item $\mathcal{A}_{n+1}$ refines $\mathcal{A}_n$.
					
					\item $B(\mathcal{A}_{n+1}) \subset B(\mathcal{A}_{n})$.
					
					\item $\bigcap_n B(\mathcal{A}_n) = Y$. 
				\end{enumerate}
				
				\item Let $x \in X$. A nested sequence $\{\mathcal{A}_n\}_{n \in \mathbb{N}}$ of Kakutani-Rokhlin partitions around $x$ is said to have \emph{property (H)} if the following holds:
				\begin{enumerate}[(H)]
					\item $h_{\mathcal{A}_n} \to \infty$ for $n \to \infty$.
				\end{enumerate}
			\end{enumerate}
		\end{defi}
		
		\begin{rem}
			Let $Y$ be a subset of $X$ that is either clopen or a singleton and let $\{B_n\}_{n \in \mathbb{N}}$ be a descending filtration of clopen sets $B_n$ with $\bigcap_n B_n = Y$ and $B_1=X$. Applying the construction from Remark~\ref{rem: kakropa} to each set $B_n$ and taking sufficient refinements, generates a nested sequence of Kakutani-Rohklin partitions $\{\mathcal{A}_n\}_{n \in \mathbb{N}}=\{D_{k,i}^n|0 \leq k \leq h^n_i-1, i \in \{1,\dots,i_n\} \}$ around $Y$ . This construction also works for the more general systems considered in \cite{hps92} and by restriction to subsequences, hereafter, property (H) can always be assumed.
		\end{rem}
		
		\begin{ex}\label{ex: kakutani-rokhlin partition of an odometer}
			Let $G_{\bm{\mathrm{a}}}$ be the odometer of type $\bm{\mathrm{a}}$ (see Example~\ref{ex: shift+odo}(ii)) for a sufficient sequence $\bm{\mathrm{a}}=(a_n)_{n \in \mathbb{N}}$. Then
			\begin{equation*}
			\begin{gathered}
			\{\mathcal{A}_n\}_{n \in \mathbb{N}}:= \{D_k^n\}_{n \in \mathbb{N},0\leq k \leq a_n-1}:=\\
			=\{(z_i)_{i \in \mathbb{N}}+k \cdot \bm{1}|(z_i)_{i \in \mathbb{N}} \in G_{\bm{\mathrm{a}}},z_1=\dots=z_n=0 \}_{n \in \mathbb{N},0\leq k \leq a_n-1}=\\
			\{(z_i)_{i \in \mathbb{N}} \in G_{\bm{\mathrm{a}}}|z_n=k \}_{n \in \mathbb{N},0\leq k \leq a_n-1}
			\end{gathered}
			\end{equation*}
			is a nested sequence of Kakutani-Rokhlin partitions\footnote{By definition the atoms $D_k^n$ form a basis of clopens for the topology on $X=G_{\bm{\mathrm{a}}}$.\\} around $(0,0,\dots)$ with property (H) in which every partition $\mathcal{A}_n$ consists of one tower of height $a_n$.
		\end{ex}
		
			\subsection{Ordered Bratteli diagrams and Bratteli-Vershik systems}	\label{subs: ord.brat}
			
			Motivated by work of Anatoly M. Vershik in the measured dynamics context, Cantor systems are described in \cite{hps92} in terms of Bratteli diagrams. When endowed with an order structure Bratteli diagrams are inherently dynamical objects:
			\begin{defi}[\cite{hps92}, Definition 2.3. \& 2.6.]
				Let $\Gamma=(V,E)$ be a Bratteli diagram. 
				\begin{enumerate}[(i)]
					\item If $\Gamma$ is endowed with a partial order $\leq$ on $E$ such that $e,\tilde{e} \in E$ are comparable \emph{if and only if} $r(e)=r(\tilde{e})$, it is called an \emph{ordered Bratelli diagram}. This partial order induces a total order on $r^{-1}(r(e))$ for all $e \in E$.
					
					\item If $\Gamma$ is ordered, an infinite path is called \emph{minimal (resp. maximal)}, if $e_n$ is minimal (resp. maximal) in $r^{-1}(r(e_n))$ for all $n \in \mathbb{N}$ and such paths necessarily exist.
					
					\item Let $\Gamma$ be ordered. $\Gamma$ is said to be \emph{essentially simple}, if it has unique minimal and maximal paths $P_{\min},P_{\max}$ and it is \emph{simple} if in addition $\Gamma$ has a telescoping $\tilde{\Gamma}$ such that for all $n \in \mathbb{N}$ and $v \in \tilde{V}_n,v' \in \tilde{V}_{n+1}$ the set $\{e_n \in \tilde{E}_n| s(e_n)=v,r(e_n)=v' \}$ is non-empty.
				\end{enumerate}
			\end{defi}
			
			Let $\Gamma$ be an essentially simple ordered Bratteli diagram. The space of infinite paths $\mathcal{P}_\Gamma$ carries a topology given by cylinder sets, resulting in a compact, totally disconnected, metrizable space. If $\Gamma$ is simple, $\mathcal{P}_\Gamma$ is either finite or a Cantor space. We need to define a homeomorphism $\varphi_\Gamma \colon \mathcal{P}_\Gamma \to \mathcal{P}_\Gamma$. To this end we set $\varphi_\Gamma(P_{\max})=P_{\min}$. Let $P=\{e_n\}_{n \in \mathbb{N}} \in \mathcal{P}_\Gamma$ with $P\neq P_{\max}$. Then $P$ contains edges $e_i$ not maximal in $r^{-1}(r(e_i))$. Let $e_j$ be the first such edge appearing in $P$. Then there exists an edge $f_j$ that is a successor of $e_j$ in $r^{-1}(r(e_j))$  with respect to the induced total ordering. There exists a unique path $(f_0,\dots,f_{j-1})$ from $v_0$ to $s(f_j)$ such that every contained edge $f_k$ is minimal in $r^{-1}(r(f_k))$. Define $\varphi_\Gamma(P):=(f_0,\dots,f_j,e_{j+1}, \dots)$. The obtained map $\varphi_\Gamma$ is indeed a homeomorphism of $\mathcal{P}_\Gamma$ called the \emph{Vershik map}, named after Vershik, whose slightly different version -- the \emph{adic transformations} in \cite{ver81} -- inspired this construction. The associated topological dynamical system $(\mathcal{P}_\Gamma,\varphi_\Gamma)$ is called \emph{Bratteli-Vershik dynamical system}. Fixing $P_{\max}$, makes it a canonical pointed system.					
			
			\begin{thm}[\cite{hps92}, Theorem 4.7.]\label{thm: hps}
				Conjugacy classes of pointed essentially minimal $\mathbb{Z}$-systems over compact, totally disconnected, metrizable spaces are in 1-1 correspondence with equivalence classes of essentially simple, ordered Bratteli diagrams.
			\end{thm}
			
			The tool to establish the other direction are nested sequences of Kakutani-Rokhlin partitions (see Section 4 of \cite{hps92}):
			Let $(X,\varphi)$ be a minimal Cantor system with distinguished point $x$. There exists a nested sequence $\{\mathcal{A}_n\}_{n \in \mathbb{N}}$ of Kakutani-Rokhlin partitions around $x$ arising from a sufficient descending filtration $\{B_n\}_{n \in \mathbb{N}}$ of clopen sets. The set of vertices in $V_n$ is the set of towers of the Kakutani-Rokhlin partition $\mathcal{A}_n$. Since $B_0=X$, the set $V_0$ consists of a single vertex. By definition $\mathcal{A}_{n+1}$ is a refinement of $\mathcal{A}_n$, following the orbits of points from the base of a tower $D^{n+1}(i)$ in $V_{n+1}$ in positive direction up to the top of $D^{n+1}(i)$ corresponds to following the positive orbits of these points through a sequence of towers in $V_n$. For every passing of a tower $D^n(j)$ in $V_n$, we draw a directed edge from $D^n(j)$ to $D^{n+1}(i)$. The order in which towers are passed induces a linear order on these edges.
			Property (d) in the definition of nested sequences implies that minimal and maximal infinite paths of the resulting ordered Bratteli diagram are unique. Furthermore, the diagram is essentially simple. Choosing another nested sequence of Kakutani-Rokhlin partition yields an equivalent Bratteli diagram, since restricting to subsequences corresponds to telescoping. The Bratteli-Vershik system of the associated ordered Bratteli diagram is conjugate to the initial dynamical system, eventually implying Theorem~\ref{thm: hps}.
			
			\begin{rem}\label{rem: simpkakro}
				Let $(X,\varphi)$ be a minimal Cantor system with a nested sequence of Kakutani-Rokhlin partitions $\{\mathcal{A}_n\}_{n \in \mathbb{N}}=\{D_{k,i}^n|0 \leq k \leq h^n_i-1, i \in \{1,\dots,i_n\} \}$ with property (H). Then for every tower $D$ in some $\mathcal{A}_n$ there exists an $l>n$ such that for every $k \geq n$ every tower in $\mathcal{A}_k$ runs through $D$. This implies that the associated Bratteli diagram is simple.
			\end{rem}

		\subsection{Dimension groups}\label{subs: dimgroups}
		
		Dimension functions are real valued, non-negative functions on the positive elements in a C*-algebra (thus in particular on the projections), that preserve the order structure, that are additive on orthogonal elements and invariant under the equivalence of projections. These functions are in close correspondence with the C*-algebras traces and can in some sense be seen as measuring the supports of positive elements -- see \cite{bla06}, II.6.8. Dimension functions and their ranges had been used in the works Glimm and Dixmier, inspired by which George A. Elliott gave a classification of AF C*-algebras in terms of dimension groups -- these are a particular kind of abelian ordered group in which the dimension ranges embed and which have subsequently become identified as operator K-groups. Elliott's classification in turn was a building block in the classification of essentially minimal pointed Cantor systems (see Subsection~\ref{subs: ord.brat}). Bratteli diagrams as well as minimal Cantor systems give rise to dimension groups. These turn out to be K-groups\footnote{See appendix~\ref{app: K-theo}.\\} in disguise and have been fundamental in the classification of minimal Cantor systems:
		\begin{defi}[\cite{bla06}, §V.2.4]\phantomsection\label{defi: ord.g}
			\begin{enumerate}[(i)]
				\item An \emph{ordered group} $(G,G^+)$ is a pair consisting of an abelian group $G$ and a \emph{positive cone} $G^+$ i.e. a subset of $G$ such that:
				\begin{enumerate}
					\item $G^+ +G^+ \subseteq G^+$
					
					\item $G^+ - G^+ =G$
					
					\item $G^+ \cap (-G^+)=\{0 \}$. 
				\end{enumerate}
				
				The positive cone $G^+$ induces an $G$-invariant partial ordering $\leq$ on $G$ by setting $y \leq x$ for $x,y \in G$ if $x-y \in G^+$.
				
				\item Let $(G_1,G_1^+)$ and $(G_2,G_2^+)$ be a pair of ordered groups. A \emph{positive homomorphism  $\varphi \colon (G_1,G_1^+) \to (G_2,G_2^+)$} is a group-homomorphism $\varphi \colon G_1 \to G_2$ with $\varphi(G_1^+) \subseteq G_2^+$. Ordered groups with positive homomorphisms as morphisms form a category.
			\end{enumerate}
			
			Let $(G,G^+)$ be an ordered group.
			\begin{enumerate}[(i),resume]
				\item  An element $u \in G^+$ is called an \emph{order unit} if for every $x \in G$ there exists an $n \in \mathbb{Z},n>0$ such that $x \leq n \cdot u$ i.e. the order ideal generated by $u$ is all of $G$. 
				
				\item The ordered group $(G,G^+)$ is said to be \emph{simple} if it has no proper order ideals i.e. every non-zero positive element is an order unit. 
				
				\item It is called \emph{unperforated}, if for every $g \in G$ and positive integer $n$, $n\cdot g \in G^+$ implies $g \in G^+$.
				
				\item It is said to have the \emph{Riesz interpolation property}, if for every $g_1,g_2,h_1,h_2 \in G$ with $g_i \leq h_j$ for $j,i \in \{1,2\}$, there exists an $l \in G$ with $g_i \leq l \leq h_j$ for $j,i \in \{1,2\}$.
				
				\item A \emph{dimension group} is an countable, unperforated, ordered group that has the Riesz interpolation property.
				
				\item Let $u$ be a fixed order unit of $(G,G^+)$. The triple $(G,G^+,u)$ is called a \emph{scaled ordered group}. An \emph{isomorphism of scaled ordered groups} is a positive isomorphism that preserves the order unit.
				
				\item Let $(G,G^+,u)$ be a scaled ordered group. A \emph{state on $(G,G^+,u)$} is a positive homorphism $f\colon (G,G^+) \to (\mathbb{R},\mathbb{R}_{\geq 0})$ with $f(u)=1$. Denote by $\mathcal{S}(G,G^+,u)$ the set of all states on $(G,G^+,u)$.
				
				\item Let $(G,G^+,u)$ be a simple scaled ordered dimension group. An element $g \in G$ is called \emph{infinitesimal} if $s(g)=0$ for all $s \in \mathcal{S}(G,G^+,u)$. The set of infinitesimals is denoted by $\operatorname{Inf}(G)$.
			\end{enumerate}
		\end{defi}
		
		\begin{rem}
			\begin{enumerate}[(i)]
				\item For all $n \in \mathbb{N}$ the group $\mathbb{Z}^n$ the standard order is defined by $(\mathbb{Z}^n)^+:=\{(z_1,\dots,z_n)\in \mathbb{Z}^n|z_i \geq 0 \text{ for all }1 \leq i \leq n \}$ for which the element $(1,\dots,1)$ is an order unit. George A. Elliott introduced dimension groups in \cite{ell76} as direct limits of sequences of free abelian groups of finite rank with standard order structure and positive group homomorphisms as maps. The \emph{Effros-Handelman-Shen theorem} (\cite{ehs80}, Theorem 2.2.) characterizes such ordered groups as the class of groups from Definition~\ref{defi: ord.g}(vi).
				
				\item The order structure of a simple dimension group $(G,G^+)$ with order unit $u$ is determined by the set of states, since $G^+=\{0\} \cup \{g \in G| f(g)>0 \text{ for all } f \in \mathcal{S}(G,G^+,u) \}$.
				
				\item We omit the notation of (scaled) order groups as tuples (resp. triples) when the order structure is clear.
			\end{enumerate}
		\end{rem}
		
		Let $\Gamma=(V,E)$ be a Bratteli diagram. The limit of
		\begin{equation*}
		\mathbb{Z}\overset{\iota_0}{\hookrightarrow}\mathbb{Z}^{|V_1|}\overset{\iota_1}{\hookrightarrow}\mathbb{Z}^{|V_2|}\overset{\iota_2}{\hookrightarrow}\dots
		\end{equation*}
		where $\iota_k \colon \mathbb{Z}^{|V_k|} \to \mathbb{Z}^{|V_{k+1}|}$ is given by $\iota_k (v):=\sum_{e \in s^{-1}(v)} r(e)$ for $v \in V_k$ defines a dimension group denoted by $K_0(V,E)$. Since $\mathbb{Z}^{|V_k|} \cong K_0(\bigoplus_{v \in V_k} M_v)$, we have $K_0(V,E) \cong K_0(C^*(V,E))$ by continuity of the functor $K_0$.
		
		Later Elliott would go on to conjecture a classification of a wide class of C*-algebras up to isomorphism by considering both K-groups with some additional information. This program, which was highly influential in C*-algebra theory, was in part motivated by his result on AF C*-algebras established in \cite{ell76}:
		\begin{thm}
			The scaled ordered $K_0$-groups of AF C*-algebras are complete isomorphism invariants.
		\end{thm}
		
		\begin{ex}\label{ex: afc*}
			Let $\Gamma=(V,E)$ be the Bratteli diagram given by an infinite complete lexicographically labelled directed binary rooted tree. Then $\bigoplus_{v \in V_k} M_v = \bigoplus_{2^k} \mathbb{C}$ for all $k \in \mathbb{N}$, which is isomorphic to the C*-algebra of complex-valued functions on the boundary constant on the cylinder sets $C[\{ 0,1\}^k]$. The limit is the C*-algebra $C(X)$ of complex-valued continuous functions on the Cantor space given by the infinite paths. We have $K_0(C(X))=C(X,\mathbb{Z})$ and $K_1(C(X))=0$ by continuity and the positive cone of $C(X,\mathbb{Z})$ is just the subset of non-negative functions.
			
		\end{ex}
		
		\begin{rem}(\cite{dav96}, Example~III.2.5)
			Commutative AF C*-algebras are precisely the commutative C*-algebras with totally disconnected spectrum.
		\end{rem}
		
		Cantor systems admit dimension groups as follows:
		\begin{thm}[\cite{gps95}, p.~59]
			Let $(X,\varphi)$ be a Cantor system. Define 
			\begin{equation*}
			K^0(X,\varphi):=C(X,\mathbb{Z})/\{f-f\circ \varphi^{-1}| f \in C(X,\mathbb{Z})\}
			\end{equation*}
			
			Then $(K^0(X,\varphi),K^0(X,\varphi)^+,\mathbf{1}_X)$ is a scaled ordered dimension group with positive cone $K^0(X,\varphi)^+:= \{[f]| f \geq 0, f \in C(X,\mathbb{Z})\}$. This makes  $K^0(X,\varphi)/\operatorname{Inf}(K^0(X,\varphi))$ a simple scaled ordered dimension group.
		\end{thm}
		
		This group of coinvariants $K^0(X,\varphi)$ is a K-group in disguise, as noted in Subsection~\ref{subs: crossprod} we have $K^0(X,\varphi) \cong K_0(C(X)\rtimes_\varphi \mathbb{Z})$. The states of $K^0(X,\varphi)$ have the following description:
		\begin{thm}[\cite{put89}, Corollary~5.7]\label{thm: tracstat}
			Let $(X,\varphi)$ be a Cantor system. Every $\mu \in M_\varphi$ induces a state $s(\mu)$ of the scaled ordered group $(K^0(X,\varphi),K^0(X,\varphi)^+,\mathbf{1}_X)$ defined by
			\begin{equation*}
			s(\mu) \colon f \mapsto \int_{X} f\;\mathrm{d}\mu
			\end{equation*}
			for $f \in C(X,\mathbb{Z})$ and a tracial state of $C(X)\rtimes_\varphi \mathbb{Z}$ by $s(\mu) \circ E$, where $E$ is the conditional expectation defined in Remark~\ref{rem: crossedprod}(iii), producing a 1-1-1 correspondence between:
			\begin{enumerate}[(i)]
				\item $\varphi$-invariant probability measures on $X$
				
				\item the set of states $\mathcal{S}(K^0(X,\varphi),K^0(X,\varphi)^+,\mathbf{1}_X)$
				
				\item tracial states of $C(X)\rtimes_\varphi \mathbb{Z}$
			\end{enumerate}
		\end{thm}
		
		This implies:
		
		\begin{cor}\label{cor: simple scaled ordered dim group of minimal cantor system}
			There is an isomorphism of simple scaled ordered dimension group between $K^0(X,\varphi)/\operatorname{Inf}(K^0(X,\varphi))$ and $C(X,\mathbb{Z})/\{f \in C(X,\mathbb{Z})| \int_{X}f\; \mathrm{d}\mu = 0 \text{ for all }\mu \in M_\varphi \}$.
		\end{cor}

		\subsection{Ample groups}\label{subs: amplegroups}
		
		Wolfgang Krieger introduced in \cite{kri80} \emph{ample groups} -- of which the definition reminds of Dye's full group, and gave -- in the wake of Elliot's work -- a classification in terms of dimension. 
		
		\begin{defi}[\cite{kri80}, p.~88]\label{defi: ample group}
			Let $X$ be a Cantor space. Let $\mathcal{A}$ be a subalgebra of the Boolean algebra $\mathcal{B}_X$ of clopen sets of $X$ and let $G$ be a countable group of homeomorphisms of $X$ that leaves $\mathcal{A}$ invariant.
			\begin{enumerate}[(i)]
				\item Denote by $[G,\mathcal{A}]$ the group of homeomorphisms $h \in \operatorname{Homeo}(X)$ such that there exists a finite clopen partition $\{X_i\}_{i \in I}$ of $X$ with $X_i \in \mathcal{A}$ for all $i \in I$ and a finite collection $\{g_i\}_{i \in I}$ of elements in $G$ with
				\begin{equation*}
				X=\bigsqcup_{i \in i} X_i = \bigsqcup_{i \in i} g_i(X_i),
				\end{equation*}
				such that $h|_{X_i}=g_i|_{X_i}$ for all $i \in I$.
				
				\item The pair $(G,\mathcal{A})$ is called a \emph{unit system} if $G$ is a locally finite\footnote{A group is called \emph{locally finite} if every finitely generated subgroup is finite. Equivalently it is a group that is a direct limit of a directed system of finite groups.\\} group such that the set of fixed points is in $\mathcal{A}$ for all $g \in G$, the map $g \to g|_\mathfrak{A}$ is an isomorphism of groups and $G=[G,\mathcal{A}]$. A unit system is called \emph{finite} if $\mathcal{A}$ is finite.
				
				\item Let $(\mathcal{A},G)$ and $(\mathcal{B},H)$ be unit systems. The unit system $(\mathcal{A},G)$ is said to be \emph{finer} than $(\mathcal{B},H)$ if $\mathcal{B}\subseteq \mathcal{A}$ and $H \leq G$.
				
				\item A countable, locally finite group $G$ of homeomorphisms of $X$ is called \emph{ample} if $(\mathcal{B}_X,G)$ is a unit system i.e. if the set of fixed points is clopen for all $g \in G$ and for any finite clopen partition $\{X_i\}_{i \in I}$ of $X$ and finite collection $\{g_i\}_{i \in I}$ of elements in $G$ with
				\begin{equation*}
				X=\bigsqcup_{i \in i} X_i = \bigsqcup_{i \in i} g_i(X_i),
				\end{equation*}
				the element $g \in \operatorname{Homeo}(X)$ defined by $g|_{X_i}=g_i|_{X_i}$ is in $G$.
				
				\item The topological dynamical system $(X,G)$ given by the action of an ample group $G$ of homeomorphisms on $X$ is called an \emph{AF-system}.
			\end{enumerate}
		\end{defi}
		
		Every unit system $(\mathcal{A},G)$ can be represented as limit $(\bigcup_{n \in \mathbb{N}} \mathcal{A}_n, \bigcup_{n \in \mathbb{N}} G_n)$ of a refining sequence $\{(\mathcal{A}_n,G_n)\}_{n \in \mathbb{N}}$ of finite unit systems (\cite{kri80}, Lemma~2.1). For every ample group $G$ the associated refining sequence of unit systems $\{(\mathcal{A}_n,G_n)\}_{n \in \mathbb{N}}$ induces an action of $G$ on a Bratteli diagram $(V,E)$ where the level sets $V_n$ are given by the $G_n$-orbits of atoms in $\mathcal{A}_n$ and the edge structure arises from the relations between these orbits. Under Theorem~\ref{thm: af-brat} this Bratteli diagram corresponds to the crossed product C*-algebra $C(X) \rtimes G$ associated to the AF-system $(X,G)$ with associated dimension group $K^0(X,G):=K_0(C(X) \rtimes G)$.
		Drawing from Elliot's classification of AF C*-algebras, Krieger obtained the following:
		
		\begin{thm}[\cite{kri80}, Corollary~3.6]\label{thm: krieg}
			Let $G_1,G_2$ be ample groups of homeomorphisms of a Cantor space $X$. Then $G_1$ and $G_2$ are spatially isomorphic\footnote{See Definition~\ref{defi: spatial homeomorphism + local subgroups}.\\} \emph{if and only if} the dimension groups $K^0(X,G_1)$ and $K^0(X,G_2)$ are isomorphic as scaled ordered groups.
		\end{thm}
		
		\subsection{Classification of Cantor systems}\label{subs: classific}
		
		In \cite{put89} Putnam used a specific subalgebra of $C(X)\rtimes_\varphi \mathbb{Z}$ associated to a minimal Cantor system to obtain information on the K-theory of $C(X)\rtimes_\varphi \mathbb{Z}$.
		
		\begin{lem}[\cite{put89}, Lemma~3.1]\label{lem: findim}
			Let $(X,\varphi)$ be a minimal Cantor system, let $Y$ be a clopen subset of $X$ and let $\mathcal{P}=\{P_i\}_{i \in I}$ be a finite clopen partition of $X$. Then the C*-subalgebra $\mathfrak{A}_{Y,\mathcal{P}}$ of $C(X)\rtimes_\varphi \mathbb{Z}$ generated by $C(\mathcal{P}):= \langle \{\chi_{P_i}\}_{i \in I} \rangle$ and $u\chi_{X\setminus Y}$ is finite dimensional.
		\end{lem}
		\begin{proofsketch}
			Let $\mathcal{A}$ be a Kakutani-Rohklin partition of $(X,\varphi)$ which has $Y$ as its base and let $\mathcal{A}'=\{D_{k,i}|0 \leq k \leq h_i-1, i \in \{1,\dots,n\} \}$ be the refinement of $\mathcal{A}$ by $\mathcal{P}$. Let $\mathcal{P}'$ be the clopen finite partition of $X$ induced by $\mathcal{A}'$. One can show that $\mathfrak{A}_{Y,\mathcal{P}}$ is contained in a subalgebra $\mathfrak{A}(Y,\mathcal{P}')$ of $C(X)\rtimes_\varphi \mathbb{Z}$ isomorphic to $\bigoplus_{i=1}^{n} M_{h_i}(\mathbb{C})$ given by matrix units $E_{i,j}^{(k)}=u^{i-j}\chi_{D_{k,j}} =\chi_{D_{k,i}} u^{i-j}$. Note that in particular diagonal entries generate $C(\mathcal{P}')$.
		\end{proofsketch}
		
		This then implies:
		
		\begin{thm}[\cite{put89}, Theorem~3.3]\label{thm: putnam}
			Let $(X,\varphi)$ be a minimal Cantor system and let $Y$ be a closed subset of $X$. Then the C*-subalgebra $\mathfrak{A}_{Y}$ of $C(X)\rtimes_\varphi \mathbb{Z}$ generated by $C(X)$ and $u C_0(X \setminus Y)$ is an AF C*-algebra.
		\end{thm}
		\begin{proofsketch}
			Let $\{\mathcal{A}_n\}_{n \in \mathbb{N}}$ be a nested sequence of Kakutani-Rokhlin partitions around $Y$ induced by a descending filtration $\{ B_n \}_{n \in \mathbb{N}}$ around $Y$. Let $\mathcal{P}_n$ be the finite clopen partition of $X$ induced by $\mathcal{A}_n$. Then one can show that $\mathfrak{A}_{Y}$ is the closed union of the finite dimensional C*-subalgebras $\mathfrak{A}(B_n,\mathcal{P}_n)$ from Lemma~\ref{lem: findim}.
		\end{proofsketch}
		
		The above observations are fundamental for the classification obtained in \cite{gps95} in that they imply the following theorem:
		
		\begin{thm}[\cite{put89}, Theorem~6.7]\label{thm: cantor-AF}
			Let $(X,\varphi)$ be a minimal Cantor system and let $x \in X$. Then there is a unital embedding $\iota \colon C(X)\rtimes_\varphi \mathbb{Z} \hookrightarrow \mathfrak{A}_{\{x\}}$ of which the induced map in K-theory $K_0(\iota) \colon K_0(C^*(X,\varphi)) \to K_0(\mathfrak{A}_{\{x\}})$ is an isomorphism of ordered groups. 
		\end{thm}
		
		We finish this chapter with a quick recollection of the classification of minimal Cantor systems from \cite{gps95}: 
		
		\begin{thm}[\cite{gps95}, Theorem~2.1]
			Let $(X_1,\varphi_1)$ and $(X_2,\varphi_2)$ be minimal Cantor systems. Then the following are equivalent:
			\begin{enumerate}[(i)]
				\item The systems $(X_1,\varphi_1)$ and $(X_2,\varphi_2)$ are strongly orbit equivalent.
				
				\item The $K_0$-groups $K^0(X_1,\varphi_1)$ and $K^0(X_2,\varphi_2)$
				are isomorphic as scaled ordered groups.
				
				\item $C(X_1)\rtimes_{\varphi_1} \mathbb{Z}$ and $C(X_2)\rtimes_{\varphi_2} \mathbb{Z}$ are isomorphic.
			\end{enumerate}
		\end{thm}
		For $(ii)\Rightarrow (i)$ the orbit map is derived from the representations as Bratteli-Vershik systems. The inverse direction $(i)\Rightarrow (ii)$ is a consequence of the work of Putnam on the AF C*-subalgebras $\mathfrak{A}_Y$. $(ii)\Leftrightarrow (iii)$ follows from work of Elliott.\footnote{He had obtained a characterization of crossed products C*-algebras associated with minimal Cantor systems as the simple circle algebras of real rank zero with $K_1$-group equal to $\mathbb{Z}$ and had shown that $K_0$-groups as scaled ordered groups are a complete isomorphism invariant of simple circle algebras of real rank zero.\\}
		
		Another theorem obtained in \cite{gps95} is directly influenced by results in measured dynamics. Every non-singular transformation $T$ on a Lebesgue space $(X,\lambda)$ gives rise to an $\mathbb{R}$-action on the weights of the associated von Neumann factor and further to a von Neumann crossed product called the \emph{flow of weights}. In \cite{kri76} Wolfgang Krieger showed that orbit equivalence of non-singular transformations $T_1$ and $T_2$ is equivalent to conjugacy of their associated flows. In the topological dynamics setting, the role of the flow of weights is taken by the simple scaled ordered dimension group $K^0(X,\varphi)/\operatorname{Inf}(K^0(X,\varphi))$ -- see Corollary~\ref{cor: simple scaled ordered dim group of minimal cantor system}.
		
		\begin{thm}[\cite{gps95}, Theorem~2.2]\label{thm: orbiequi}
			Let $(X_1,\varphi_1)$ and $(X_2,\varphi_2)$ be minimal Cantor systems. Then the following are equivalent:
			\begin{enumerate}[(i)]
				\item The systems $(X_1,\varphi_1)$ and $(X_2,\varphi_2)$ are orbit equivalent.
				
				\item The groups $K^0(X_1,\varphi_1)/\operatorname{Inf}(K^0(X_1,\varphi_1))$ and $K^0(X_2,\varphi_2)/\operatorname{Inf}(K^0(X_2,\varphi_2))$ are isomorphic as scaled ordered groups.
				
				\item There exists a homeomorphisms $F \colon X_1 \to X_2$ such that $M_{\varphi_1}$ is mapped to $M_{\varphi_2}$.
			\end{enumerate}
		\end{thm}
		
		$(i)\Rightarrow(iii)$ is straightforward and $(iii)\Rightarrow (ii)$ follows from Theorem~\ref{thm: tracstat}.
		Implication $(ii)\Rightarrow (i)$, however, is by far the most technical part in all of \cite{gps95} and requires besides Bratteli-Vershik systems and Putnam's results in \cite{put89} an existence theorem by Cartan-Eilenberg from homological algebra.
		
		\begin{rem}\label{rem: thm 2.3}
			Theorem~2.3 in \cite{gps95} extends Theorem~\ref{thm: orbiequi} to minimal AF-systems.
		\end{rem}
		
		Closest to our concern is the following classification theorem, which clarifies under which assumptions minimal Cantor systems are flip conjugate. Topological full groups (see Chapter~\ref{chap: 3}) -- they are already implicit in the works \cite{put89} and \cite{gps95} -- were later on recognized as complete invariants for flip conjugacy of minimal Cantor systems:
		
		\begin{thm}[\cite{gps95}, Theorem~2.4]\label{thm: flipconju}
			Let $(X_1,\varphi_1)$ and $(X_2,\varphi_2)$ be minimal Cantor systems. Then the following are equivalent:
			\begin{enumerate}[(i)]
				\item The systems $(X_1,\varphi_1)$ and $(X_2,\varphi_2)$ are flip conjugate.
				
				\item There exists an orbit map $F\colon X_1 \to X_2$ such that the orbit cocycles associated with $F$ are continuous.
				
				\item There exists an isomorphism $\alpha \colon C(X_1)\rtimes_{\varphi_1} \mathbb{Z} \to C(X_2)\rtimes_{\varphi_2} \mathbb{Z}$ such that $\alpha(C(X_1))=C(X_2)$.
			\end{enumerate}
		\end{thm}

	\chapter{\'Etale groupoids and inverse semigroups}\label{chap: 2}
	
	In this chapter we roam around the playgrounds of \emph{groupoids} and \emph{inverse semigroups}. While it is the study of Cantor systems from which topological full groups emanated, this setting is embedded in the broader context of \'etale groupoids.

	\section{Basics of groupoids}\label{sec: basics of groupoids}
	
	As references for the basic definitions in this section see \cite{ren80} and \cite{pat99} except where noted.
	Subsection~\ref{subs: groupoids} lists the basic definitions and examples of groupoids and Subsection~\ref{subs: topological groupoids} does the same for topological groupoids.

	\subsection{Groupoids}\label{subs: groupoids}
	
	In the twentieth century especially groupoids have become objects of utmost importance -- naturally in step with the rise of category theory.

	\begin{defi}\phantomsection\label{defi: grpdstuff}
		\begin{enumerate}[(i)]	
			\item A \emph{groupoid} is a set $\mathcal{G}$ with a product map $(g_1,g_2) \mapsto g_1g_2$ defined on a subset $\mathcal{G}^{(2)} \subseteq \mathcal{G} \times \mathcal{G}$ and an inverse map  $g \mapsto g^{-1}$ defined on all of $\mathcal{G}$ satisfying the following conditions for all $g, g_1, g_2, g_3 \in \mathcal{G}$:
			\begin{enumerate}
				\item If $(g_1,g_2),(g_2,g_3) \in \mathcal{G}^{(2)}$, then $(g_1g_2,g_3),(g_1,g_2g_3)\in \mathcal{G}^{(2)}$ and $(g_1g_2)g_3 = g_1(g_2g_3)$
				
				\item $(g,g^{-1}),(g^{-1},g) \in \mathcal{G}^{(2)}$
				
				\item If $(g_1,g_2) \in \mathcal{G}^{(2)}$, then $(g_1,g_2)g_2^{-1}=g_1$ and $g_1^{-1}(g_1,g_2)=g_2$
			\end{enumerate}
		\end{enumerate}
		
		Let $\mathcal{G}$ be a groupoid.
		\begin{enumerate}[(i),resume]	
			\item A subset $\mathcal{H} \subseteq \mathcal{G}$ is called a \emph{subgroupoid of $\mathcal{G}$} if it is closed with respect to the product and inverse map.
			
			\item The structure maps $s \colon \mathcal{G} \to \mathcal{G}^{(0)}, g \mapsto g^{-1}g$ resp. $r \colon \mathcal{G} \to \mathcal{G}^{(0)}, g \mapsto gg^{-1}$ are called \emph{source map} resp. \emph{range map}.
			
			\item The elements in $\mathcal{G}^{(2)}$ are called \emph{composable pairs}. Denote more generally for all $n \in \mathbb{N}$ by $\mathcal{G}^{(n)}$ the set of sequences $(g_1, g_2, \dots, g_n) \in \mathcal{G}^n$ such that the product $g_1g_2 \dots g_n$ is defined, i.e. $s(g_i)=r(g_{i+1})$ for all $i \in {1,2,\dots,n-1}$.
			
			\item The \emph{product of two subsets} $G_1,G_2 \subseteq G$ is defined as the set $G_1G_2:=\{(g_1,g_2)\in \mathcal{G}^{(2)}|g_1 \in G_1,g_2 \in G_2\}$.
			
			\item The elements of $\{gg^{-1}|g \in \mathcal{G}\}$ are called \emph{units}, the set of all units is denoted by $\mathcal{G}^{(0)}$.
			
			\item Let $U$ be a subset of $\mathcal{G}^{(0)}$. We set $\mathcal{G}^U:=r^{-1}(U)$ and $\mathcal{G}_U:=s^{-1}(U)$, in particular the range fiber of a unit $u \in \mathcal{G}^{(0)}$ is denoted by $\mathcal{G}^u$ and its source fiber by $\mathcal{G}_u$. The set $\mathcal{G}_U \cap \mathcal{G}^U$ forms a subgroupoid of $\mathcal{G}$ when endowed with the restricted product- and inverse map called \emph{reduction of $\mathcal{G}$ to $U$} denoted by $\mathcal{G}|_U$ . In particular for any $u \in \mathcal{G}^{(0)}$, the reduction groupoid $\mathcal{G}|_u$ is a group of pairwise composeable elements of $\mathcal{G}$ called the \emph{isotropy group of $u$}. Denote by $\mathcal{G}_{triv}$ the set of units with trivial isotropy groups.
			
			\item A subset $B \subseteq \mathcal{G}$ is called \emph{slice} or \emph{$\mathcal{G}$-bisection}, if $s|_B$ and $r|_B$ are injective maps. Denote the set of all slices in $\mathcal{G}$ by $\mathcal{B}_\mathcal{G}$.
			
			\item The groupoid $\mathcal{G}$ is called \emph{principal}, if $\mathcal{G}_{triv} =\mathcal{G}^{(0)}$ holds.
			
			\item The set $\operatorname{Iso}(\mathcal{G}):=\{g \in \mathcal{G}:s(g)=r(g)\}=\bigcup_{u \in \mathcal{G}^{(0)}}\mathcal{G}|_u$ forms a subgroupoid of $\mathcal{G}$, called the \emph{isotropy bundle of $\mathcal{G}$}.
			
			\item A pair of units $u_1,u_2 \in \mathcal{G}^{(0)}$ is contained in the same \emph{$\mathcal{G}$-orbit}, if there exists a $g \in \mathcal{G}$ such that $s(g)=u_1$ and $r(g)=u_2$. Let $u \in \mathcal{G}^{(0)}$, denote the orbit containing $u$ by $\mathcal{G}(u)$.
			
			\item (\cite{nek15}, p.~31) If a subset $\mathcal{S} \subseteq \mathcal{G}$ intersects every $\mathcal{G}$-orbit it is called a \emph{$\mathcal{G}$-transversal}.
			
			\item Let $\mathcal{G}_1, \mathcal{G}_2$ be groupoids. A \emph{groupoid homomorphism} is a map $\varphi \colon \mathcal{G}_1 \to \mathcal{G}_2 $ such that $\varphi(\mathcal{G}_1^{(0)}) \subseteq \mathcal{G}_2^{(0)}$ and for all $(g,h) \in \mathcal{G}_1^{(2)}$ it holds that $\big( \varphi(g),\varphi(h) \big) \in \mathcal{G}_2^{(2)}$ and $\varphi(gh)=\varphi(g)\varphi(h)$. If $\varphi$ has an inverse, it is a \emph{groupoid isomorphism}. In particular, groupoid homomorphisms intertwine the structure maps $s$ and $r$.

			\item  A subset $\mathcal{S}$ is said to be a \emph{generating set of $\mathcal{G}$} if $\mathcal{G}=\bigcup_{n \in \mathbb{N}} (\mathcal{S} \cup \mathcal{S}^{-1})^{n}$.	
		\end{enumerate}
	\end{defi}
	
	\begin{rem}
		\begin{enumerate}[(i)]
			\item Groupoids are essentially the same as a small category in which every morphism is invertible, generalizing the representation of groups as categories of the same kind with only one object by Cayley's theorem. Groupoid elements translate to morphisms in the category theoretic description, the units to identity morphisms, range- and source map to the respective category theoretic term, groupoid homorphisms to functors between categories of the described kind etc.\footnote{Note however, when interpreted as "arrows", products must be read from right to left to make sense!\\}
			
		\end{enumerate}
		
		Let $\mathcal{G}$ be a groupoid.
		\begin{enumerate}[(i),resume]
			\item A set $B \subseteq \mathcal{G}$ is a slice \emph{if and only if}  $BB^{-1} \subseteq \mathcal{G}^{(0)}$ or equivalently $B^{-1}B\subseteq \mathcal{G}^{(0)}$ holds.
			
			\item Note that if two units $u,v \in \mathcal{G}^{(0)}$ are contained in the same orbit, then their isotropy groups $\mathcal{G}|_u$ and $\mathcal{G}|_v$ are isomorphic.
		\end{enumerate}
	\end{rem}

	\begin{ex}[\cite{ren80}, Examples~1.2.c]\label{ex: grpd}
		Let $\sim$ be an equivalence relation on a set $X$. Then the set $\mathcal{G}_\sim:=\{(x,y) \in X\times X|x\sim y\}$ becomes a groupoid by defining a pair $((x_1,y_1),(x_2,y_2))$ to be composeable if $y_1=x_2$. Their product is defined by $(x_1,y_1)\cdot(x_2y_2)=(x_1,y_2)$ and the inverse by $(x,y)^{-1}=(y,x)$. The arising groupoid $\mathcal{G}_\sim$ is principal and $\mathcal{G}_\sim^{(0)} \cong X$.
	\end{ex}
	
	Equivalence relations on spaces with structure can be studied by looking at the quotient by collapsing equivalence classes, however, if the quotient map is not sufficiently ``nice", the structure of the quotient space might turn out to be ``bad", e.g. for the \emph{Kronecker foliation}, a foliation of the torus $\mathbb{T}^2=S^1\times S^1$ induced by a 1-dimensional subbundle of the tangent bundle with constant irrational slope, the leaf space carries no manifold structure. Groupoids offer a way to encode the data without the -- possibly brutal -- application of a quotient map e.g. in the context of foliations this role is played by the \emph{holonomy groupoid}.
	Some of the terminology from Definition~\ref{defi: grpdstuff} is accounted for by the capability of	groupoids to picture group actions:
	\begin{defi}[\cite{ren80}, Examples~1.2.a]
		Let $\alpha \colon G \curvearrowright X$ be the action of a group $G$ on a set $X$. The set $G \times X$ carries the structure of a groupoid: A tuple of elements $(g_1,x_1),(g_2,x_2) \in G \times X$ is composeable, if $g_2(x_2)=x_1$. Its product is defined by $(g_1,x_1)(g_2,x_2)=(g_1g_2,x_2)$ and the inverse by $(g,x)^{-1}=\big(g^{-1},g(x)\big)$.	The resulting groupoid is called the \emph{groupoid of the action} or \emph{transformation groupoid} and is denoted by $\mathcal{G}_{(X,G)}$ and by $\mathcal{G}_\varphi$ in the case of $\mathbb{Z}$-action generated by a homeomorphism $\varphi$. Its unit space $\mathcal{G}_{(X,G)}^{(0)}$ coincides with $X$. A transformation groupoid is principal \emph{if and only if} the action is free.
	\end{defi}	
	
	\subsection{Topological groupoids}\label{subs: topological groupoids}
	
	The groupoids we consider are more than algebraic objects -- pre-eminently they are topological objects and under sufficient conditions open to the methods of analysis.				
	\begin{defi}\phantomsection\label{defi: grpd}
		\begin{enumerate}[(i)]
			\item Groupoids endowed with a topology such that product map and inverse map are continuous are called \emph{topological groupoids}.\footnote{This implies continuity of the range and source maps.\\}
		\end{enumerate}
		
		Let $\mathcal{G}$ be a topological groupoid.
		\begin{enumerate}[(i),resume]
			\item If $\mathcal{G}_{triv}$ is dense in $\mathcal{G}^{(0)}$, the groupoid $\mathcal{G}$ is called \emph{essentially-principal}.
			
			\item If all $\mathcal{G}$-orbits are dense in $\mathcal{G}^{(0)}$, the groupoid $\mathcal{G}$ is called \emph{minimal}.
			
			\item  Denote the set of all open slices by $\mathcal{B}_\mathcal{G}^o$ and the set of all open, compact slices by $\mathcal{B}_\mathcal{G}^{o,k}$.
		\end{enumerate}
	\end{defi}	
	
	\begin{ex}\phantomsection\label{ex: act.grpd.}
		\begin{enumerate}[(i)]
			\item Topological groups are topological groupoids.
			
			\item For an equivalence relation $\sim$ on a topological space $X$, the groupoid $\mathcal{G}_\sim$ becomes a topological groupoid when endowed with the subspace topology coming from $X \times X$.
			
			\item Let $\alpha \colon G \to \operatorname{Homeo}(X)$ define a continuous action of a group $G$ on a topological space $X$ and let $\{U_i\}_i$ be a basis of the topology on $X$. Then the transformation groupoid can be endowed with a topology: The sets $\{(g,U_i)\}$ form the basis of a topology on $\mathcal{G}_{(X,G)}$ for which the product and inverse maps are trivially continuous.
			
			\item If in the above example $G$ is a discrete group, an equivalence relation can be defined on $\mathcal{G}_{(X,G)}$ by $(g_1,x_1)\sim(g_2,x_2)$ if $x_1=x_2$ and there exists a neighbourhood $U$ of $x_1$ such that $g_1|_U=g_2|_U$. The quotient $\operatorname{Germ}(X,G):= \mathcal{G}_{(X,G)} / {\sim}$ inherits the structure of a topological groupoid. The resulting groupoid is called the \emph{groupoid of germs of the action}.
		\end{enumerate}
	\end{ex}
	
	\begin{rem}\label{rem: unitsclosed}
		Let $\mathcal{G}$ be a topological groupoid.
		\begin{enumerate}[(i)]
			\item If $\mathcal{G}^{(0)}$ is Hausdorff, the isotropy bundle $\operatorname{Iso}(\mathcal{G})$ is closed, since it is the preimage of the diagonal $\Delta \subseteq \mathcal{G}^{(0)} \times \mathcal{G}^{(0)}$ under the continuous map $g \mapsto (s(g),r(g))$.
			
			\item If a topological groupoid $\mathcal{G}$ is Hausdorff, then $\mathcal{G}^{(0)}$ is closed, since it is the preimage of $\Delta \subseteq \mathcal{G} \times \mathcal{G}$ under the continuous map $g \mapsto (g,g^{-1})$.
		\end{enumerate}
	\end{rem}
	
	The following lemma is crucial in the characterization of \'etale groupoids:
	
	\begin{lem}[\cite{res07}, Lemma~5.16]\label{lem: grpd.opensets}
		Let $\mathcal{G}$ be a topological groupoid and let $\Omega(\mathcal{G})$ be its collection of open sets.
		Then every $U \in \Omega(\mathcal{G})$ satisfies:
		\begin{enumerate}[(i)]
			\item $(U \cap \mathcal{G}^{(0)})\mathcal{G} \subseteq \bigcup \{X\cap Y | X,Y \in \Omega(\mathcal{G}), XY^{-1} \subseteq U \} \subseteq \bigcup \{V \in \Omega(\mathcal{G})| VV^{-1} \subseteq U \}$

			\item $\mathcal{G}(U \cap \mathcal{G}^{(0)}) \subseteq \bigcup \{X\cap Y | X,Y \in \Omega(\mathcal{G}), X^{-1}Y \subseteq U \} \subseteq \bigcup \{V \in \Omega(\mathcal{G})| V^{-1}V \subseteq U \}$
		\end{enumerate}
	\end{lem}
	
	\section{Basics of inverse semigroups}\label{sec: basics of inverse semigroups}
	
	With Felix Klein's proposal from the Erlangen program, to study geometries by the means of transformations under which geometric structure is invariant, groups became the predominant algebraic structure to describe symmetries. While they are a proper algebraic tool when dealing with ``global" symmetries, they fail to characterize symmetries of ``local" nature. In differential geometry local structures (e.g. foliations) are defined in terms of atlantes of coordinate charts which are invariant under a family of chart transition maps. The transition maps are partial functions\footnote{Let $X,Y$ be sets. A \emph{partial function $f \colon X \to Y$} is a function $f\colon A \to B$ where $A \subseteq X$ and $B \subseteq Y$.\\} and cannot be abstractly described by groups. Space groups are fit to describe symmetries of crystals and periodic tilings, but fail to do so for quasicrystals and aperiodic tilings. Scaling symmetries of regions in self-similar structures (e.g. Julia sets in holomorphic dynamics) warrant a structure more general then groups. One way to deal with this insufficiency is provided by \emph{pseudogroups}, another by groupoids. More on the motivation and historical background can be found in §1.1 of \cite{law98}. Subsection~\ref{subs: some terms from order theory} lists the required terms from order theory and Subsection~\ref{subs: inverse semigroups} lays down basic definitions and examples of inverse semigroups.
	
	\subsection{Some terms from order theory}\label{subs: some terms from order theory}

	\begin{defi}[\cite{joh82}]
		\begin{enumerate}[(i)]
			\item A \emph{poset} is a set $P$ with a binary relation $\leq$ that is reflexive, anti-symmetric and transitive.
		\end{enumerate}
		Let $P$ be a poset and let $S$ be a subset of $P$.
		\begin{enumerate}[(i),resume]
			\item A \emph{join (resp. meet) of $S$} is an element $a$ such that
			\begin{enumerate}
				\item $s \leq a$ (resp. $a \leq s$) for all $s \in S$
				
				\item $s \leq b$ (resp. $b \leq s$) for all $s \in S$ implies $a \leq b$ (resp. $b \leq a$) for all $b \in B$.
			\end{enumerate}
			If a join (resp. meet) of $S$ exists it is unique and denoted by $\bigvee S$ (resp. $\bigwedge S$).
			
			\item Denote by $S^{\uparrow}$ the set $\{p \in P| p \leq s \text{ for all }s \in S \}$ and by $S^{\downarrow}$ the set $\{p \in P| s \leq p \text{ for all }s \in S \}$. The subset $S$ is called \emph{(downward) directed} if for every pair $a,b \in S$ there exists a $c \in S$ with $c < a$ and $c < b$.
			
			\item If $S=S^{\downarrow}$ holds, $S$ is called an \emph{order ideal}. If $S$ is a directed set with $S=S^{\uparrow}$, it is called a \emph{filter}.
			
			\item A \emph{join (resp. meet) semi-lattice} is a poset in which every finite subset has a join (resp. meet) and thus $\vee$ (resp. $\wedge$) defines a binary operation. Every join (resp. meet) semi-lattice contains a unique least element $0$ (resp. unique greatest element $1$) with $l \vee 0 = l$ (resp. $l \wedge 1 = l$) for all $l \in L$.
			
			\item A \emph{lattice} is a poset in which every finite subset has a join and a meet.
		\end{enumerate}
		
		Let $L$ be a lattice.
		\begin{enumerate}[(i), resume]
			\item The lattice $L$ is said to be \emph{complete} if every subset has a join and a meet. 
			
			\item The lattice $L$ is called \emph{distributive}, if $l\wedge (m \vee n)=(l \wedge m) \vee (l \wedge n)$ holds for all $l,m,n \in L$. It is said to be \emph{infinitely distributive} if for every $x \in L$ and every $S \subseteq L$ the following holds:
			\begin{equation*}
			x\wedge \bigvee_{s\in S}s=\bigvee_{s\in S}(x\wedge s)
			\end{equation*}
			
			\item Let $a \in L$. Any element $b \in L$ with $a \wedge b = 0$ and $a \vee b =1$ is called a \emph{complement of $a$}.
			
			\item A \emph{Boolean algebra} is a distributive lattice $B$ with an unary operation $\neg \colon B \to B$ such that $\neg b$ is a complement of $b$ for all $b \in B$.
			
			\item Let $A,B$ be Boolean algebras. A function $f \colon A \to B$ is called a \emph{homomorphism of boolean algebras} if it satisfies $f(0)=0$, $f(1)=1$, $f(a_1 \wedge a_2)=f(a_1) \wedge f(a_2)$ and  $f(a_1 \vee a_2)=f(a_1) \vee f(a_2)$ for all $a_1,a_2 \in A$.
			
			\item A \emph{frame} is a complete infinitely distributive lattice. Let $F,R$ be frames. A function $f \colon F \to R$ is called a \emph{morphism of frames} if it preserves finite meets and arbitrary joins.
			
			\item The \emph{category of frames $\mathbf{Frm}$} is the category whose class of objects consists of frames and whose morphisms consist of morphisms of frames. Its opposite category is called the \emph{category of locales} denoted by $\mathbf{Loc}$, whose objects are called \emph{locales} and whose morphisms are called \emph{continuous maps}.\footnote{On the level of objects frames and locales are completely synonymous, it is at the level of morphisms at which these notions are different.\\}
		\end{enumerate}
	\end{defi}
	\begin{ex}
		Let $X$ be a topological space. Taking union and intersection as join and meet the set $\Omega(X)$ of open subsets of $X$ becomes a frame. A continuous map $f \colon X \to Y$ induces a morphism of frames $\Omega(f) \colon \Omega(Y)\to \Omega(X)$.
	\end{ex}
	
	\subsection{Inverse semigroups}\label{subs: inverse semigroups}
	
	The abstract setting for the aforementioned pseudogroups is provided by inverse semigroups:
	
	\begin{defi}[\cite{law98}]
		\begin{enumerate}[(i)]
			\item A \emph{semigroup} is a set $S$ together with an associative binary operation.
			
			\item Let $S$ be a semigroup. A \emph{subsemigroup} of $S$ is a subset which is closed with respect to the semigroup operation.
			
			\item A \emph{monoid} is a semigroup $S$, that contains a unique element $1 \in S$ called the \emph{identity} such that $1 s=s=1  s$ holds for all $s \in S$.
			
			\item Let $S$ be a monoid. An element $s \in S$ is called a \emph{unit} if there exists an element $t \in S$ such that $1=st=ts$. Denote by $U(S)$ the set of all units in $S$.
			
			\item A semigroup $S$ is called a \emph{semigroup with zero}, if it contains an element $0$ such that $0\cdot s=0=s \cdot 0$ for all $s \in S$.
		\end{enumerate}
		
		Let $S$ be a semigroup.
		\begin{enumerate}[(i),resume]
			\item It is said to be an \emph{inverse semigroup} if for every $s \in S$ there exists a unique $t \in S$ such that $sts=s$ and $tst=t$. This element is called the \emph{inverse of $s$} and denoted by $s^{-1}$. The map $S \to S$ induced by taking inverses is called \emph{involution}.
			
			\item An element $s \in S$ is called an \emph{idempotent} if $s^2=s$. The set of all idempotents is denoted by $E(S)$.
			
			\item Let $S$ and $T$ be semigroups. A \emph{semigroup homomorphism} is a map $f \colon S \to T$ such that $f(s_1s_2)=f(s_1)f(s_2)$ for all $s_1,s_2 \in S$.
		\end{enumerate}
		
		Let $S$ be an inverse semigroup.
		
		\begin{enumerate}[(i),resume]
			\item For all $s \in S$ define the source- (resp. range) maps $d,r \colon S \to E(S)$ by $d \colon s \mapsto s^{-1}s$ (resp. $r \colon s \mapsto ss^{-1}$).
		\end{enumerate}
	\end{defi}
	
	\begin{rem}
		\begin{enumerate}[(i)]						
			\item The set of units $U(S)$ in a monoid $S$ is a group with respect to the semigroup operation. Topological full groups turn out to be groups of such kind. 
			
			\item \emph{All inverse semigroup are from now on assumed to be inverse semigroups with zero!}
		\end{enumerate}
	\end{rem}	
	
	Inverse semigroups carry a natural poset structure:
	Let $S$ be an inverse semigroup. Define a partial order on $S$ by $s \leq t$ \emph{if and only if} there exists an idempotent $e$ such that $s=et$ or equivalently $s=ss^{-1}t$. For all $s \in S$ the elements $ss^{-1}$ and $s^{-1}s$ are idempotents. The set of idempotents $E(S)$ is a commutative idempotent semigroup. Commutative idempotent semigroups are essentially the same as semilattices and the product of two elements is given by their meet. The following are important basic examples of inverse semigroups: 
	\begin{ex}\phantomsection\label{ex: inv.semigrp}
		\begin{enumerate}[(i)]
			\item (\cite{hae02}, p.276) Historically the motivation for the introduction of abstract inverse semigroups had been pseudogroups in differential geometry. The definition of pseudogroups varies in the literature, we fix a \emph{traditional pseudogroup} to be a set $S$ of partial homeomorphisms between open sets of a topological space $X$ such that the following hold:
			\begin{enumerate}
				\item For every $s_1 \colon U_1 \to V_1$ and $s_2 \colon U_2 \to V_2$ in $S$ the inverse $s_1^{-1}$ and the compostion $s_1s_2$ given by the partial homeomorphism $s_1|_{U_1 \cap V_2} \circ s_2|_{(s_2)^{-1}(U_1 \cap V_2)}$ are contained in $S$.
				
				\item The homeomorphism $\operatorname{id}_X$ is in $S$.
				
				\item Any partial homeomorphism between open subsets of $X$ that is locally in $S$, is in $S$.
			\end{enumerate}
			
			\item The inverse of an element $s$ in an inverse semigroup is often denoted by $s^*$ reminding of the adjoint in C*-algebras. This is not just a coincidence: Let $\mathfrak{A}$ be a C*-algebra. The set $\operatorname{Par}(\mathfrak{A})$ of all partial isometries in $\mathfrak{A}$ is an inverse semigroup with respect to multiplication, where the inverse is given by the adjoint. If $\mathfrak{A}$ is unital, it is an inverse monoid. Any subset of $\operatorname{Par}(\mathfrak{A})$ which is closed with respect to multiplication and involution is an inverse semigroup. In fact every inverse semigroup has a representation as an inverse semigroup of partial isometries over a Hilbert space (\cite{pat99}, Proposition~2.1.4).
			
			\item (\cite{law98}, p.5 \& p.36) Let $X$ be a set. Denote by $\mathcal{I}(X)$ the set of all bijective partial functions of $X$. In particular $\mathcal{I}(X)$ contains all partial empty functions and all partial identities. Let $f\colon X \to Y$ and $g\colon Y \to Z$ be partial functions. Define the partial function $g \circ f\colon X \to Z$ by setting $\operatorname{dom}(g \circ f):=f^{-1}(\operatorname{dom}f \cap \Ima g)$ and $(g \circ f)(x):=g(f(x))$ for all $x \in \operatorname{dom}(g \circ f)$. Endowing $\mathcal{I}(X)$ with the multiplication given by composition makes it an inverse monoid, called the \emph{symmetric inverse monoid of $X$}. Generalizing Cayley's theorem on groups the Wagner-Preston theorem shows that every inverse semigroup $S$ has a representation as a subsemigroup of $\mathcal{I}(S)$.
			
			\item Let $\mathcal{G}$ be a groupoid. Then the set of slices $\mathcal{B}_\mathcal{G}$ is an inverse monoid with respect to the product and inversion of subsets of $\mathcal{G}$, where the identity is given by $\mathcal{G}^{(0)}$. 
			
		\end{enumerate}
	\end{ex}
	
	\begin{defi}[\cite{exe08}, Definition~4.3]
		\begin{enumerate}[(i)]
			\item Let $X$ be a locally compact Hausdorff space and $S$ an inverse semigroup. A \emph{continuous action $\alpha$ of $S$ on $X$} is an inverse semigroup homomorphism $\alpha \colon S \to \mathcal{I}(X)$ such that for every $s \in S$ the partial function $\alpha(s)$ is continuous with open domain and $\{\operatorname{dom}(\alpha(s))|s \in S \}$ covers $X$.
			
			\item Let $\alpha$ be a continuous action of an inverse semigroup $S$ on a compact, Hausdorff space $X$. It is called \emph{faithful}, if $\alpha$ is injective. The action is called \emph{relatively free} if for all $s \in S$ the set of fixed points of $\alpha(s)$ in $\operatorname{dom}(\alpha(s))$ is compact and open.
		\end{enumerate}
	\end{defi}

	The \emph{groupoid of germs} is a standard construction in the theory of foliations. Example~\ref{ex: act.grpd.}(iv) is a particular case of this more general construction:
	
	\begin{defi}[\cite{exe08}, p.~208-213]\label{defi: germgrpd}
		Let $\alpha$ be a continuous action of an inverse semigroup $S$ on a locally compact, Hausdorff space $X$. Define the subset $\Xi:=\{(s,x)\in S\times X| x \in  \operatorname{dom}(\alpha(s))\}$. Define an equivalence relation on $\Xi$ by $(s_1,x_1) \sim (s_2,x_2)$ if $x_1 =x_2$ and there exists an idempotent $e \in E(S)$ such that $x_1 \in \operatorname{dom}(\alpha(e))$ and $s_1e=s_2e$. An equivalence class $[(s,x)]$ in $\Xi/_{\sim}$ is called the \emph{germ of $s$ at $x$}. This set of germs can be endowed with a groupoid structure by defining a tuple of equivalence classes $([(s_1,x_1)], [(s_2,x_2)])$ to be composable if $ x_1=\alpha(s_2)x_2$ and define the product as $[(s_1s_2,x_2)]$. Define the inverse by $[(s,x)]^{-1}=[(s^{-1},\alpha(s)x)]$.\footnote{One has to verify this is well defined i.e. does not depend on the choice of representants.\\} The obtained groupoid is called the \emph{groupoid of germs of the action} and is denoted by $\operatorname{Germ}(X,S)$. It is a topological groupoid with respect to the basis given by sets $\tau(s,U)=\{[(s,x)]: x \in U\}$ where $s \in S$ and $U$ is an open subset of $\operatorname{dom}(\alpha(s))$ and its unit space is homeomorphic to $X$.
	\end{defi}

	\section{Analysis on topological groupoids}\label{sec: analysis of topological groupoids}
	
	In this section we give basic aspects and definitions that allow for analysis on groupoids. Subsection~\ref{subs: locally compact groupoids} explains what is meant by ``locally compact groupoids". Subsection~\ref{subs: haar systems and measures on groupoids} looks at Haar systems. These are a fibered generalization of Haar measures on locally compact groups that enable analysis on locally compact groupoids. Subsection~\ref{subs: groupoid C*-algebras and regular representations} gives the definition of groupoid C*-algebras and the -- compared to groups more fiddly -- regular representations of groupoids. In particular unitary representations of groupoids are defined in terms of \emph{Hilbert bundles and Hilbert modules} -- see Appendix~\ref{app: hilbert bundles and hilbert modules}.

	\subsection{Locally compact groupoids}\label{subs: locally compact groupoids}
	
	As topological spaces topological groupoids are bit ``nastier" than topological groups. The following example may seem affected, but non-Hausdorff groupoids do emerge in the wild e.g. in foliation theory:
	
	\begin{ex}[\cite{ks02}, Example~1.2]\label{ex: nonhaus}
		Let $X$ be a compact space together with a non-isolated point $x_0 \in X$ and let $\Gamma$ be a non-trivial discrete group acting trivially on $X$ i.e $\gamma \cdot x = x$ for all $\gamma \in \Gamma, x \in X$. Consider the equivalence relation $(\gamma_1,x)\sim (\gamma_2,x)$ if $x \neq x_0$. The topological quotient $\mathcal{G}:=(\Gamma \times X)/{\sim}$ of the transformation groupoid carries  the structure of a topological groupoid which is non-Hausdorff, because if $x \to x_0$, then $\overline{(g,x)}=\overline{(e,x)} \to \overline{(g,x_0)}$ for every $g \in \Gamma$ and equivalence classes of points in $\Gamma \times \{x_0\}$ cannot be seperated by open neighbourhoods.
		Continuous functions must be constant on $\Gamma \times \{x_0\}$ and thus if $\Gamma$ is infinite every compactly supported continuous function vanishes on $\Gamma \times \{x_0\}$. This means $C_c(\mathcal{G})$ no longer sufficiently encodes the topological data of $\mathcal{G}$.
	\end{ex}
	
	Alain Connes has shown that in groupoids from foliation theory enough ``Hausdorff-ness" is present to render analysis possible -- such groupoids are ``locally Hausdorff" i.e. every groupoid element has a compact, Hausdorff neighbourhood. With this in mind one makes the following replacement of $C_c(\mathcal{G})$:
	\begin{defi}[\cite{ks02}, p.~52]\label{defi: connes}
		Let $\mathcal{G}$ be a topological groupoid, such that every $g \in \mathcal{G}$ has a compact Hausdorff neighbourhood. Let $C_c(\mathcal{U_{\mathcal{G}}})$ denote the set of complex valued functions $f$ such that there exists an open Hausdorff subset $U_f \subseteq \mathcal{G}$ for which $f|_{U_f}$ is compactly supported and continuous with respect to the subspace topology on $U_f$ and $f$ is zero-valued on the outside of $U_f$. Define $\mathscr{C}(\mathcal{G}):=\langle C_c(\mathcal{U_{\mathcal{G}}}) \rangle$.
	\end{defi}
	
	If $\mathcal{G}$ is Hausdorff, then $\mathscr{C}(\mathcal{G})=C_c(\mathcal{G})$, if $\mathcal{G}$ is non-Hausdorff, functions in $\mathscr{C}(\mathcal{G})$ need not be continuous! The definition of locally compactness for groupoids has to be chosen such that it guarantees the existence of sufficiently many compact Hausdorff neighbourhoods in the non-Hausdorff case enabling the construction from Definition~\ref{defi: connes}.
	
	\begin{defi}[\cite{ks02}, Definition~1.1]\label{defi: locogrpd}
		Let $\mathcal{G}$ be a topological groupoid. It is said to be \emph{locally compact}, if it satisfies the following properties:
		\begin{enumerate}[(i)]
			\item The topological subspace $G^{(0)}$ is Hausdorff.
			
			\item Each $g \in \mathcal{G}$ has a compact, Hausdorff neighbourhood.
			
			\item The groupoid $\mathcal{G}$ is $\sigma$-compact i.e. it is a countable union of compact subspaces.
			
			\item The source and range maps $s$ and $r$ are open.
		\end{enumerate}
	\end{defi}
	
	\begin{rem}
		\begin{enumerate}[(i)]
			\item The definitions of locally compact groupoids vary in the literature e.g. compare above definition with  Definition~2.2.1 of \cite{pat99}.
		\end{enumerate}
		
		Let $\mathcal{G}$ be a locally compact groupoid.
		\begin{enumerate}[(i),resume]
			\item By condition (ii) in the definition every singleton in $\mathcal{G}$ is closed. The groupoid $\mathcal{G}$ is Hausdorff \emph{if and only if} $G^{(0)}$ is closed.
			Compact subsets of $\mathcal{G}$ need not be closed. The range fibers $\mathcal{G}^u$ and source fibers $\mathcal{G}_u$ are Hausdorff in the relative topology for all $u \in \mathcal{G}^{(0)}$, since their respective diagonal $\{(g,g')\in \mathcal{G}^u \times \mathcal{G}^u|r(g)=r(g'), g^{-1}g'=u\}$ (resp. $\{(g,g')\in \mathcal{G}_u \times \mathcal{G}_u|s(g)=s(g'), gg'^{-1}=u\}$) is closed as the preimage of a singleton.
			
			\item Functions in $\mathscr{C}(\mathcal{G})$ are not necessarily continuous, however all of them are bounded Borel functions.
			
			\item Let ${U_j}_{j \in J}$ be an open Hausdorff cover of $G$. Then $\mathscr{C}(\mathcal{G})$ consists of all finite sums $f = \sum_{i \in I}f_i$ where for every $i \in I$ the function $f_i$ is the extension of a function in $C_c(U_i)$ by zero outside of $U_i$.
			
		\end{enumerate}
	\end{rem}
	
	\subsection{Haar systems and measures on groupoids}\label{subs: haar systems and measures on groupoids}
	
	The tool that enables analysis on groupoids is a construction in the vein of the Haar measure for locally compact groups. Instead of a single measure, one needs to work with a family of measures parametrized by the units of the groupoid with supports in the respective source fiber:
	
	\begin{defi}[\cite{ks02}, p.~50]
		Let $\mathcal{G}$ be a locally compact groupoid.
		\begin{enumerate}[(i)]
			\item A \emph{right Haar system on $\mathcal{G}$} is a family $\{\nu_u\}_{u \in \mathcal{G}^{(0)}}$, where every $\nu_u$ is a positive, regular, locally finite Borel measures on $\mathcal{G}_u$, such that:
			\begin{enumerate}
				\item $\supp(\nu_u) = \mathcal{G}_u$ for all $u \in \mathcal{G}^{(0)}$.
				
				\item $\int_{\mathcal{G}_{r(g)}} f(hg) \; \mathrm{d}\nu_{r(g)}(h)=\int_{\mathcal{G}_{s(g)}} f(h) \; \mathrm{d}\nu_{s(g)}(h)$ for all $g \in \mathcal{G}$ and $f \in \mathscr{C}(\mathcal{G})$.

				\item The maps $\alpha_f \colon \mathcal{G}^{(0)} \to \mathbb{R}^+$ defined by $u \mapsto \int_{\mathcal{G}_u} f(g) \; \mathrm{d}\nu_u (g)$ are contained in $C_c(\mathcal{G}^{(0)})$ for all $f \in \mathscr{C}(\mathcal{G})$.
			\end{enumerate}
			
			\item Let $\nu =\{\nu_u\}_{u \in \mathcal{G}^{(0)}}$ be a right Haar system. Denote for $u \in \mathcal{G}^{(0)}$ by $\nu^u$ the measure on $\mathcal{G}^u$ defined by $\nu^u(A)=\nu_u(A^{-1})$ for every Borel subset $A \subseteq \mathcal{G}^u$. Then $\nu^{-1}:=\{\nu^u\}_{u \in \mathcal{G}^{(0)}}$ is the \emph{left Haar system induced by $\nu$}.
		\end{enumerate}
	\end{defi}
	
	This makes $\mathscr{C}(\mathcal{G})$ a convolution algebra:
	
	\begin{thm}[\cite{pat99}, Theorem~2.2.1]
		Let $\mathcal{G}$ be a locally compact groupoid and let $\nu= \{\nu_u\}_{u \in \mathcal{G}^{(0)}}$ be a fixed left Haar system. The space $\mathscr{C}(\mathcal{G})$ can be given the structure of a normed $*$-algebra by setting for all $f,f_1,f_2 \in \mathscr{C}(\mathcal{G})$ and $g,h \in \mathcal{G}$
		\begin{equation*}
		\begin{gathered}
		f^{*}(g):=\overline{f(g^{-1})}\\
		(f_1 \ast f_2)(g):=\int_{\mathcal{G}_{s(g)}} f_1(gh^{-1})f_2(h)\; \mathrm{d}\nu_{s(g)}(h)\\
		\|f\|_I := \sup_{u \in \mathcal{G}^{(0)}} \big\{ \max \big(\int_{\mathcal{G}_u}|f(g)|\; \mathrm{d}\nu_u(g),\int_{\mathcal{G}_u}|f(g^{-1})|\; \mathrm{d}\nu_u(g)\big) \big\}
		\end{gathered}
		\end{equation*}
	\end{thm}
	
	While the Haar measure of a locally compact group necessarily exists and is unique up to multiplication by a constant factor, neither of which is necessarily true for the Haar systems of locally compact groupoids.
	
	\begin{defi}[\cite{pat99}, p.~86]
		Let $\mathcal{G}$ be a locally compact groupoid with a fixed Haar system $\{\nu_u\}_{u \in \mathcal{G}^{(0)}}$ and let $\mu$ be a probability measure on the space $\mathcal{G}^{(0)}$.
		\begin{enumerate}[(i)]
			\item The measure $\mu$ induces a positive regular Borel measure $\nu_\mu$ on $\mathcal{G}$ by
			\begin{equation*}
			\nu_\mu:=\int_{\mathcal{G}^{(0)}} \nu_u \;\mathrm{d}\mu
			\end{equation*}
			
			\item Denote by $\nu_\mu^{-1}$ the positive regular Borel measure on $\mathcal{G}$ defined by $\nu_\mu^{-1}(B)=\nu_\mu(B^{-1})$ for every Borel set $B \subseteq \mathcal{G}$.
			Note that
			\begin{equation*}
			\nu_\mu^{-1}=\int_{\mathcal{G}^{(0)}} \nu^u \; \mathrm{d}\mu
			\end{equation*}
			
			\item Define a regular Borel measure $\nu_\mu^2$ on $\mathcal{G}^{(2)}$ by 
			\begin{equation*}
			\int_{\mathcal{G}^{(2)}}f(g,h) \;\mathrm{d}\nu_\mu^2=\int_{\mathcal{G}^{(0)}}\; \mathrm{d}\mu \int \int f(g,h)\; \mathrm{d}\nu_u(g)\; \mathrm{d}\nu^u(h)
			\end{equation*}
			for $f \in C_c(\mathcal{G}^{(2)})$.
		\end{enumerate}
	\end{defi}
	
	\begin{rem}
		To set up $\nu_\mu$, one needs to show for every open set $U \subseteq \mathcal{G}$ with compact closure in a Hausdorff subset $U' \subseteq \mathcal{G}$ the linear functional $\Phi_U$ on $C_c(U)$ defined by
		\begin{equation*}
		\Phi_U(f):=\int_{\mathcal{G}^{(0)}} \; \mathrm{d}\mu \int_{\mathcal{G}_u} f \; \mathrm{d}\nu_u
		\end{equation*}
		is continuous. By the Riesz representation theorem, there exists a regular Borel measure $\nu_U$ such that
		\begin{equation*}
		\int f \; \mathrm{d}\nu_U = \Phi_U (f)
		\end{equation*}
		It can be shown that the resultant measures $\nu_U$ coincide on intersections and thus there exists a regular Borel measure $\nu_\mu$ such that $\nu_\mu|_U=\nu_U$. Similar arguments must be applied to obtain $\nu_\mu^2$.		
	\end{rem}
	
	\begin{ex}
		Let $\mathcal{G}$ be a locally compact groupoid with a fixed Haar system $\{\nu_u\}_{u \in \mathcal{G}^{(0)}}$ and let $v \in \mathcal{G}^{(0)}$. Let $\delta_v$ denote the Dirac measure at $v$. Its induced Borel measure is given by $\nu_{\delta_v}=\nu_v$.
	\end{ex}
	
		As with the special case of transformation groups, not all probability measures on the unit space are meaningful for the groupoids representation theory, one restricts to probability measures that satisfy a specific invariance property:					
	\begin{defi}[\cite{ren80}, Definition~3.2 \& 3.5 \& 3.7]
		Let $\mathcal{G}$ be a locally compact groupoid with a fixed Haar system $\{\nu_u\}_{u \in \mathcal{G}^{(0)}}$ and let $\mu$ be a measure on $\mathcal{G}^{(0)}$.
		\begin{enumerate}[(i)]
			\item The measure $\mu$ is said to be \emph{quasi-invariant} if the measure $\nu_\mu$ is equivalent to the measure $\nu_\mu^{-1}$ and it is called \emph{invariant} if $\nu_\mu=\nu_\mu^{-1}$.
			
			\item A $\mu$-measurable subset $U \subseteq \mathcal{G}^{(0)}$ is called \emph{almost invariant} if
			\begin{equation*}
			s(g)\in U \Leftrightarrow r(g) \in U \quad \text{for } \nu_\mu\text{-a.e. }g \in \mathcal{G}
			\end{equation*}
			
			\item Let $\mu$ be quasi-invariant. It is called \emph{ergodic} if every almost invariant measurable set $U \subseteq \mathcal{G}^{(0)}$ satisfies either $\mu(U)=0$ or $\mu(\mathcal{G}^{(0)}\setminus U)=0$.
			
			\item The measure $[\mu]:=r_{*}\nu_\mu$ i.e. for a $\mu$-measurable set $E \subseteq \mathcal{G}^{(0)}$ define $r_{*}\nu_\mu(E)=\nu_\mu(r^{-1}(E))$, is called the \emph{saturation of $\mu$}.
		\end{enumerate}	
	\end{defi}

	Ergodic, quasi-invariant measures always exist: 
	
	\begin{prop}[\cite{ren80}, Proposition~3.6]\label{prop: existence of quasi-invariant measures}
		Let $\mathcal{G}$ be a locally compact groupoid with a fixed Haar system $\{\nu_u\}_{u \in \mathcal{G}^{(0)}}$ and let $\mu$ be a measure on $\mathcal{G}^{(0)}$. The saturation $[\mu]$ is a quasi-invariant measure.
	\end{prop}
	
	\begin{prop}[\cite{ren80}, Proposition~3.8]\label{prop: existence of ergodic measures}
		Let $\mathcal{G}$ be a locally compact groupoid with a fixed Haar system $\{\nu_u\}_{u \in \mathcal{G}^{(0)}}$ and let $v \in \mathcal{G}^{(0)}$. Then the saturation $[\delta_v]=r_{*}\nu_v$ is an ergodic measure.
	\end{prop}
	
	The presence of multiple units in a groupoid $\mathcal{G}$ entails that its unitary representations take range over  fibered objects with $\mathcal{G}^{(0)}$ as the base space i.e. Hilbert bundles instead of Hilbert spaces.
	
	\begin{defi}[\cite{pat99}, Definition~3.1.1]
		Let $\mathcal{G}$ be a locally compact groupoid with a fixed Haar system $\{\nu_u\}_{u \in \mathcal{G}^{(0)}}$ and let $\mu$ be probability measure on $\mathcal{G}^{(0)}$. A \emph{unitary representation $L$ of $\mathcal{G}$ on a measured Hilbert bundle $(\mathcal{G}^{(0)},\{H_u\}_{u \in \mathcal{G}^{(0)}},\mu)$} is a function $L \colon \mathcal{G} \to \bigcup_{u,v \in \mathcal{G}^{(0)}} \mathfrak{U}(H_u,H_v)$ with the following properties\footnote{Let $\mathfrak{U}(H_x,H_y)$ denote the set of unitary bounded operators linear operators from $H_x$ to $H_y$.\\}:
		\begin{enumerate}[(i)]
			\item The probability measure $\mu$ is quasi-invariant.
			
			\item $L(g) \in \mathfrak{U}(H_{s(g)},H_{r(g)})$ for every $g \in \mathcal{G}$.
			
			\item $L(u) = \operatorname{id}_{H_u}$ for all $u \in \mathcal{G}^{(0)}$.
			
			\item $L(g)L(h)=L(gh)$ holds $\nu_\mu^2$-almost everywhere in $\mathcal{G}^{(2)}$.
			
			\item $L(g)^{-1}=L(g^{-1})$ holds $\nu_\mu$-almost everywhere in $\mathcal{G}$.
			
			\item For every pair $\sigma_1,\sigma_2 \in \int_{\mathcal{G}^{(0)} }^{\oplus} H_u \;\mathrm{d}\mu$ the function on $\mathcal{G}$ defined by 
			\begin{equation*}
			g \mapsto \langle L(g)\sigma_1(s(g)),\sigma_2(r(g)) \rangle
			\end{equation*}
			is $\nu_\mu$-measurable.
		\end{enumerate}
	\end{defi} 
	
	\begin{ex}[\cite{pat99}, p.~93]
		Let $\mathcal{G}$ be a locally compact groupoid with a fixed Haar system $\nu=\{\nu_u\}_{u \in \mathcal{G}^{(0)}}$ and let $\mu$ be probability measure on $\mathcal{G}^{(0)}$. Then the \emph{right regular representation of $\mathcal{G}$ with respect to $\nu$ and $\mu$} is given on the measured Hilbert bundle $(\mathcal{G}^{(0)},\{\mathrm{L}^2(\mathcal{G}_u,\nu_u) \}_{u \in \mathcal{G}^{(0)}},\mu)$ by:
		\begin{equation*}
		L_{\mathrm{red}}(g)(f)(h) = f(hg)
		\end{equation*}
		for all $g \in \mathcal{G}$, $f \in \mathrm{L}^2(\mathcal{G}_{s(g)}$ and $h \in \mathcal{G}_{r(g)}$. The \emph{left regular representation of $\mathcal{G}$ with respect to $\nu$ and $\mu$} is given on the measured Hilbert bundle $(\mathcal{G}^{(0)},\{\mathrm{L}^2(\mathcal{G}^u,\nu^u) \}_{u \in \mathcal{G}^{(0)}},\mu)$ by:
		\begin{equation*}
		L_{\mathrm{red}}(g)(f)(h) = f(g^{-1}h)
		\end{equation*}
		for all $g \in \mathcal{G}$, $f \in \mathrm{L}^2(\mathcal{G}^{s(g)}$ and $h \in \mathcal{G}^{r(g)}$.
	\end{ex} 
	
	\subsection{Groupoid C*-algebras}\label{subs: groupoid C*-algebras and regular representations}
	
	See \cite{ks02} for a more detailed description.
	
	\begin{defi}
		Let $\mathcal{G}$ be a locally compact groupoid and let $\nu= \{\nu_u\}_{u \in \mathcal{G}^{(0)}}$ be a fixed left Haar system.
		\begin{enumerate}[(i)]
			\item The completion of $\mathscr{C}(\mathcal{G})$ with respect to the norm $\|\cdot\|_I$ is a Banach $*$-algebra of which the enveloping C*-algebra is called the \emph{full groupoid C*-algebra of $\mathcal{G}$ with respect to $\nu$} denoted by $C^*(\mathcal{G},\nu)$.
			
			\item Let $u \in \mathcal{G}^{(0)}$. Define a bounded $*$-representation $\lambda_u$ of $\mathscr{C}(\mathcal{G})$ on $\mathrm{L}^2(\mathcal{G}_u,\nu_u)$ by 
			\begin{equation*}
			\lambda_u (f)\xi(g)=\int_{\mathcal{G}_u} f(gh^{-1})\xi(h)\;\mathrm{d}\nu_u(h)
			\end{equation*}
			for every $\xi \in \mathrm{L}^2(\mathcal{G}_u,\nu_u),f \in \mathcal{C}_c(\mathcal{G})$. It holds that $\|\lambda_u(f)\|\leq \|f\|_I$.
			
			\item The \emph{reduced C*-algebra of $\mathcal{G}$ with respect to $\nu$} denoted by $C_r^*(\mathcal{G},\nu)$ is the completion of $\mathscr{C}(\mathcal{G})$ with respect to the norm $\|\cdot\|_r$ defined by
			\begin{equation*}
			\|f\|_r:=\sup_{u \in \mathcal{G}^{(0)}} \{\|\lambda_u(f)\| \} \text{ for } f \in \mathscr{C}(\mathcal{G}).
			\end{equation*}
		\end{enumerate}	
	\end{defi}
	
	\begin{ex}\label{ex: grpd.c-star}
		Let $\mathcal{G}_\alpha$ be the transformation groupoid associated with an action $\alpha \colon G \to \operatorname{Homeo}(X)$ of a countable, discrete group $G$ on a locally compact, Hausdorff space $X$. Then $C^*(\mathcal{G}_\alpha) \cong C_0(X) \rtimes G$.
	\end{ex}
	
	If $\mathcal{G}$ is a locally compact group the definitions of the adjoint and the norm differ from the conventional definitions for group C*-algebras and consequently set up different normed $*$-algebra structures on $\mathcal{C}_c(\mathcal{G})$ resp. $\mathscr{C}(\mathcal{G})$. Passing over to the C*-completions the differences vanish and the associated (reduced) C*-algebras are isomorphic.	

	\addtocontents{toc}{\protect\newpage}
	\section{\'Etale groupoids}\label{sec: etale groupoids}
	
	In \cite{ren80} Renault gives special attentation to \emph{r-discrete groupouids} which generalize transformation groupoids arising from actions of discrete groups.
	
	\begin{defi}[\cite{ren80}, Definition~2.6]
		Let $\mathcal{G}$ be a locally compact groupoid. It is called \emph{$r$-discrete}, if $\mathcal{G}^{(0)}$ is an open subset of $\mathcal{G}$. 
	\end{defi}
	
	The denotation is justified -- for such a groupoid, the subspaces $\mathcal{G}_u$ and $\mathcal{G}^u$ are discrete for every $u \in \mathcal{G}^{(0)}$. If a Haar system of an $r$-discrete groupoid exists -- this however is not necessarily the case -- it is essentially unique i.e. it is equivalent to the system given by the system of counting measures (see \cite{ren80}, Lemma~2.7). \'Etale groupoids are more convenient in that a Haar system always exists and it is unique up to the equivalence of measures:
	
	\begin{prop}
		Let $\mathcal{G}$ be a locally compact, \'etale groupoid. Then the system of counting measures on source fibers is a unique Haar system (up to equivalence).
	\end{prop}
	
	Subsection~\ref{subs: definition and properties} discusses the structure and characterization of \'etale groupoids. Subsection ~\ref{subs: homology of etale groupoids} recalls a suitable version of homology for \'etale groupoids.
	
	\subsection{Definition and properties}\label{subs: definition and properties}
	
	\begin{defi}[\cite{res07}, p.162]
		A topological groupoid $\mathcal{G}$ is said to be \emph{\'etale}, if the range map $r$ (or equivalenty the source map $s$) is \'etale i.e. $r$ is a local homeomorphism\footnote{A continuous function $f\colon X \to Y$ between topological spaces is said to be a \emph{local homeomorphism}, if every $x \in X$ has an open neighbourhood $U_x$ such that $f(U_x)$ is open and $f|_{U_x}$ is a homeomorphism.\\}.
	\end{defi}
	
	\begin{ex}
		\begin{enumerate}[(i)]
			\item Let $\varphi: \Gamma \to \operatorname{Homeo}(X)$ define a action of a countable, discrete group $\Gamma$ on a locally compact, Hausdorff space $X$. Then the transformation groupoid $\mathcal{G}_\varphi$ is \`etale.
			
			\item Let $X$ be a compact, metrizable, zero-dimensional space and $\sim$ be a countable equivalence relation on $X$. Then the topological groupoid $\mathcal{G}_\sim$ is \'etale. 
			
			\item Let $\alpha$ be a continuous action of an inverse semigroup $S$ on a locally compact, Hausdorff space $X$. The groupoid of germs $\operatorname{Germ}(X,S)$ is \'etale, since the elements $\tau(s,U)$ of its topological basis are slices. 
		\end{enumerate}
	\end{ex}

	\begin{lem}
		Let $\mathcal{G}$ be an \'etale groupoid. Then the set $\mathcal{B}_\mathcal{G}^{o}$ of open slices is a basis of the topology on $\mathcal{G}$.
	\end{lem}
	\begin{proof}
		Let $g \in \mathcal{G}$ be contained in some open set $V$. Then there exist open neighbourhoods $U_{g,s}$ resp. $U_{g,r}$ on which the source resp. range maps are a homeomorphism with open image. The the set $U_{g,s} \cap U_{g,r} \cap V$ is a non-empty open slice containing $g$.
	\end{proof}

	The following lemma follows by plugging in $U=\mathcal{G}^{(0)}$ in Lemma~\ref{lem: grpd.opensets}:
	
	\begin{lem}[\cite{res07}, Lemma~5.17]
		Let $\mathcal{G}$ be a topological groupoid such that $\mathcal{G}^{(0)}$ is open. Then the open slices $\mathcal{B}_\mathcal{G}^o$ form an open cover of $\mathcal{G}$.
	\end{lem}
	
	This lemma implies in particular the following characterization of \'etale groupoids, which shows that \'etale groupoids are precisely those, whose collection of open set is inherently algebraic:
	\begin{thm}[\cite{res07}, Theorem~5.18]\label{thm: etalegrpd}
		Let $\mathcal{G}$ be a topological groupoid. Then the following are equivalent:
		\begin{enumerate}[(i)]
			\item $\mathcal{G}$ is \'etale.
			
			\item $\mathcal{G}^{(0)}$ is an open subspace and the product of any two open subsets is open i.e. the frame of open subsets of $\mathcal{G}$ forms a monoid.
			
			\item $\mathcal{G}^{(0)}$ is an open subspace and the structure maps $s$ and $r$ are open.
		\end{enumerate}
	\end{thm}

	Already Renault had observed in \cite{ren80}, that locally compact, \'etale groupoids give rise to inverse monoids by looking at the sets of open slices with multiplication inherited from the groupoid.
	
	\begin{prop}\label{prop: open.slices}
		Let $\mathcal{G}$ be an \'etale groupoid. Then the following hold:
		\begin{enumerate}[(i)]
			\item The open slices $\mathcal{B}^{o}_{\mathcal{G}}$ form an inverse submonoid of the inverse monoid $\mathcal{B}_{\mathcal{G}}$.
			
			\item The map $\alpha \colon \mathcal{B}^{o}_{\mathcal{G}} \to \mathcal{I}(\mathcal{G}^{(0)})$ given by $\alpha \colon B \mapsto r|_B \circ (s|_B)^{-1}$ defines a continuous action of $\mathcal{B}^{o}_{\mathcal{G}}$ on $\mathcal{G}^{(0)}$ by partial homeomorphisms.
			
		\end{enumerate}
	\end{prop}
	
	This raises the question how the original groupoid $\mathcal{G}$ relates to the induced groupoid of germs $\operatorname{Germ}(\mathcal{G}^{(0)} ,\mathcal{B}^{o}_{\mathcal{G}})$.
	
	\begin{prop}[\cite{ren08}, Proposition~3.2]
		Let $\mathcal{G}$ be an \'etale groupoid. Then the following is a short exact sequence of \'etale groupoids:
		\begin{equation*}
			\mathcal{G}^{(0)} \longrightarrow \operatorname{Iso}(\mathcal{G})^{\circ}  \longrightarrow \mathcal{G}  \overset{\alpha_{*}}{\longrightarrow} \operatorname{Germ}(\mathcal{G}^{(0)} ,\alpha(\mathcal{B}^{o}_{\mathcal{G}}))  \longrightarrow \mathcal{G}^{(0)} 
		\end{equation*}
		Furthermore, the action $\alpha$ is faithful \emph{if and only if} $\operatorname{Iso}(\mathcal{G})^{\circ}=\mathcal{G}^{(0)}$.
	\end{prop}
	\begin{rem}
		The map $\alpha_{*}\colon \mathcal{G}  \longrightarrow \operatorname{Germ}(\mathcal{G}^{(0)} ,\alpha(\mathcal{B}^{o}_{\mathcal{G}}))$ is defined by sending an element $g \in \mathcal{G}$ to the germ $[(\alpha(B),s(g)]$ where $B$ is an open slice that contains $g$. This does not depend on the choice of $B$ and sets up a continuous surjection. 
	\end{rem}

	\begin{defi}[\cite{nek17}, Definition~2.4]
		Let $\mathcal{G}$ be an \'etale groupoid. it is said to be \emph{effective} or a \emph{groupoid of germs}\footnote{There is no consensus on the terminology in the literature. In \cite{ren09} the property of being essentially principal is called \emph{topologically principal}, another term in use is \emph{essentially free}. Sometimes (as in \cite{mat12}) a groupoid $\mathcal{G}$ is  defined to be essentially principal if it is effective in our terms. This stems from the fact that all of this notions are closely related.\\} if every non-empty open slice $B \in \mathcal{B}_\mathcal{G}^o$ for which $B \cap (\mathcal{G}\setminus \mathcal{G}^{(0)}) \neq \emptyset$ holds satisfies $B \cap (\mathcal{G}\setminus \operatorname{Iso}(\mathcal{G})) \neq \emptyset$.
	\end{defi}
	
	\begin{rem}\label{rem: effective}
		If an \'etale groupoid $\mathcal{G}$ is effective, it necessarily satisfies $\operatorname{Iso}(\mathcal{G})^{\circ} = \mathcal{G}^{(0)}$. The converse holds if $\mathcal{G}$ is Hausdorff.\footnote{In this case $\mathcal{G}^{(0)}$ is clopen.\\}
	\end{rem}
	
	Effectiveness is a kind of ``freeness" condition in that it is closely related to being essentially principal:
	
	\begin{prop}[\cite{ren08}, Proposition~3.6.(i)]
		Let $\mathcal{G}$ be an essentially prinipal, \'etale groupoid. If it is Hausdorff, it is effective. 
	\end{prop}
	
	\begin{prop}[\cite{ren08}, Proposition~3.6.(ii)]\label{prop: effective}
		Let $\mathcal{G}$ be an effective, \'etale groupoid. If $\mathcal{G}$ is second-countable and $\mathcal{G}^{(0)}$ is a Baire space, it is essentially principal.
	\end{prop}

	As the definition suggests, \'etalness entails strong ties between the topology of the whole groupoid $\mathcal{G}$ and the subspace topology on $\mathcal{G}^{(0)}$:
	
	\begin{prop}
		Let $\mathcal{G}$ be an \'etale groupoid.  The topology on $\mathcal{G}$ has a basis of compact, open slices \emph{if and only if} $\mathcal{G}^{(0)}$ has a basis of compact, open sets.
	\end{prop}
	
	Ample semigroups were introduced in \cite{ren80} to characterize groupoids which admit a basis of compact open slices, the terminology was inherited by Krieger's definition of ample groups (compare with Definition~\ref{defi: ample group}). Subsequently such groupoids have been termed ample.
	
	\begin{defi}[\cite{ren80}, Definition~2.12 \& \cite{pat99}, Definition~2.2.4]\phantomsection\label{defi: ample}
		~\begin{enumerate}[(i)]
			\item Let $\alpha$ be the continuous action of an inverse subsemigroup $S \subseteq \mathcal{I}(X)$ on a compact, Hausdorff space $X$. Then $S$ is called \emph{ample} if the following hold:
			\begin{enumerate}
				\item For every compact open subset $U \subseteq X$ the partial homeomorphism $\operatorname{id}_U$ is induced by an element in $S$.
				
				\item For every finite collection $\{s_i\}_{i \in I}$ of elements in $S$ with $\Ima(\alpha(s_i)) \cap \Ima(\alpha(s_j))= \emptyset$ and $\operatorname{dom}(\alpha(s_i)) \cap \operatorname{dom}(\alpha(s_j))= \emptyset$ for all $i,j \in I$ with $i \neq j$, there exists an $s \in S$ such that $\alpha(s)(x)=\alpha(s_i)(x)$ for all $x \in \operatorname{dom}(\alpha(s_i))$.
			\end{enumerate}
			
			\item A topological groupoid $\mathcal{G}$ is called \emph{ample}, if $\mathcal{B}_\mathcal{G}^{o,k}$ is a basis of the topology.
		\end{enumerate}
	\end{defi}
	
	Every ample groupoid is by definition \'etale. Conversely we have:
	
	\begin{prop}[\cite{ll13}, Lemma~3.13]
		Let $\mathcal{G}$ be an \'etale groupoid. Then $\mathcal{G}$ is ample \emph{if and only if} the space of units $\mathcal{G}^{(0)}$ has a basis of compact, open sets.
	\end{prop}

	\begin{rem}
		This implies that every \'etale groupoid $\mathcal{G}$ such that $\mathcal{G}^{(0)}$ is a Stone space is ample.
	\end{rem}
	
	\begin{prop}
		Let $\mathcal{G}$ be an ample groupoid. Then the set of compact, open slices $\mathcal{B}^{o,k}_{\mathcal{G}}$ is an inverse subsemigroup of the inverse monoid $\mathcal{B}^o_{\mathcal{G}}$.
	\end{prop}
	
	In Section~\ref{sec: interweaving inverse semigroups and etale groupoids} we give a quick rundown on the connection between \'etale groupoids and certain inverse semigroups and how this connection had played a role in C*-algebra theory. We finish this subsections with some remarks on locally compact, \'etale groupoids which will be helpful in  Subsection~\ref{subs: topfullgroups koopman}:
	
	\begin{prop}[\cite{pat99}, Proposition~3.2.1]\label{prop: integrability}

		$\quad$Let $\mathcal{G}$ be a locally compact, \'etale groupoid and let $\mu$ be a probability measure on $\mathcal{G}^{(0)}$. Let $B \in \mathcal{B}_{\mathcal{G}}^{o}$ be an open slice and let $f$ be a bounded Borel function on $B$. Then $f$ is $\nu_\mu$ integrable \emph{if and only if} $f \circ (s|_B)^{-1}$ is $\nu_\mu$ integrable and we have
		\begin{equation*}
		\int f\;\mathrm{d}\nu_\mu = \int f \circ (s|_B)^{-1} \; \mathrm{d}\mu \quad \text{and} \quad
		\int f\;\mathrm{d}\nu_\mu^{-1} = \int f \circ (r|_B)^{-1} \; \mathrm{d}\mu
		\end{equation*}
		
	\end{prop}
	\begin{proof}
		It follows immediatly from
		\begin{equation*}
		\int_{\mathcal{G}_u} f\;\mathrm{d}\nu_u =
		\begin{cases}
		f \circ (s|_B)^{-1}(u) , & \text{ if } u \in s(B)\\
		0, & \text{ else}
		\end{cases}
		\end{equation*}
	\end{proof}	
	
	\begin{prop}\label{prop: quisiinvariance and slices}
		Let $\mathcal{G}$ be a locally compact, \'etale groupoid and let $\mu$ be a probability measure on $\mathcal{G}^{(0)}$. Then $\mu$ is quasi-invariant \emph{if and only if} for every open slice $B \in \mathcal{B}_\mathcal{G}^{o}$ we have $\mu \circ r|_B \circ (s|_B)^{-1} \sim \mu$  on $r(B)$.
	\end{prop}
	\begin{proof}
		By definition $\mathcal{G}$ is $\sigma$-compact and in consequence Lindelöf\footnote{Every open cover has a countable subcover.\\}. It follows there exists a countable cover of $\mathcal{G}$ by open slices and thus $\nu_\mu \sim (\nu_\mu)^{-1}$ \emph{if and only if} $\nu_\mu|_B \sim (\nu_\mu)^{-1}|_B$ for all $B \in \mathcal{B}_\mathcal{G}^{o}$. Let $B \in \mathcal{B}_\mathcal{G}^{o}$ and let $f$ be a positive $\mu$-measurable function. Then by Proposition~\ref{prop: integrability} we have
		\begin{equation*}
		\int f\;\mathrm{d}\nu_\mu = \int f \circ (s|_B)^{-1} \; \mathrm{d}\mu\\
		\end{equation*}
		and furthermore
		\begin{equation*}
		\begin{gathered}
		\int_{\mathcal{G}} f\; \mathrm{d}\nu_\mu^{-1}=\int_{\mathcal{G}^{(0)}} f \circ (r|_B)^{-1} \; \mathrm{d}\mu=\int_{\mathcal{G}^{(0)}} f \circ (s|_B)^{-1} \circ s|_B \circ (r|_B)^{-1} \; \mathrm{d}\mu=\\
		=\int_{\mathcal{G}^{(0)}} f \circ (s|_B)^{-1} \; \mathrm{d}(\mu \circ s|_B \circ (r|_B)^{-1})
		\end{gathered}
		\end{equation*}
		Thus $\nu_\mu|_B \sim (\nu_\mu)^{-1}|_B$ \emph{if and only if} $\mu \circ s|_B \circ (r|_B)^{-1} \sim \mu$ on $s(B)$ and equivalently $\mu \circ r|_B \circ (s|_B)^{-1} \sim \mu$ on $r(B)$.
	\end{proof}
	
	\begin{rem}
		Correspondingly a probability measure $\mu$ is invariant \emph{if and only if} for every open slice $B \in \mathcal{B}_\mathcal{G}^{o}$ we have $\mu(s(B))=\mu(r(B))$.
	\end{rem}

	\subsection{Homology of \'etale groupoids}\label{subs: homology of etale groupoids}
	
	For \'etale groupoids homology groups with coefficients in a sheaf over the unit space where introduced in	\cite{cm00} -- dual to an already existing cohomology theory of \'etale groupoids. Matui studied the case of groupoid homology with constant coefficients of second countable, locally compact, Hausdorff, \'etale  Cantor groupoids in a series of papers \cite{mat12} and applied those results to get information on their associated topological full groups. The following is a slightly generalized version descibed in \cite{nek17} to include non-Hausdorff groupoids:							
	\begin{defi}[\cite{nek17}, §2.2]
		Let $\mathcal{G}$ be a locally compact, \'etale groupoid and let $A$ a topological, abelian group. Analogous to Definition~\ref{defi: connes}, denote by $\mathscr{C}(\mathcal{G}^{(n)},A)$ the subgroup of $A^{\mathcal{G}^{(n)}}$ generated by all functions which are constant on some compact open Hausdorff subset $S \subseteq \mathcal{G}^{(n)}$ and map to $0$ if restricted to $\mathcal{G}^{(n)} \setminus S$.\footnote{If $\mathcal{G}$ is Hausdorff, this is just the set of compactly supported continuous functions, where $A$ is endowed with the discrete toplogy.\\} The spaces of composable elements $\{\mathcal{G}^{(n)}\}_{n \in \mathbb{N}}$ form simplicial spaces with face maps $d_i \colon \mathcal{G}^{(n)} \to \mathcal{G}^{(n-1)}$ for $i = 0,1, \dots ,n$ defined by:
		\begin{enumerate}[(i)]
			\item If $n=1$:
			\begin{equation*}
			d_i(g):=
			\begin{cases}
			s(g), & \text{for } i=0\\
			r(g), & \text{for } i=1\\
			\end{cases}
			\end{equation*}
			
			\item Otherwise:		
			\begin{equation*}
			d_i(g_1,g_2, \dots ,g_n):=
			\begin{cases}
			(g_2, \dots ,g_n), & \text{for } i=0\\
			(g_1,\dots,g_i 	g_{i+1},\dots,g_n), & \text{for } 1 \leq i \leq n-1\\
			(g_1,\dots ,g_{n-1}), & 	\text{for } i=n\\
			\end{cases}
			\end{equation*}
		\end{enumerate}
		These give families of morphisms $d_{i\ast}\colon\mathscr{C}(\mathcal{G}^{(n)},A)\to \mathscr{C}(\mathcal{G}^{(n-1)},A)$, defined by:
		\begin{equation*}
		d_{i\ast}(f)(\mathbf{g}):=\sum_{\mathbf{g'}\in d_i^{-1}(\mathbf{g})} f(\mathbf{g'})
		\end{equation*}
		
		which in turn give the boundary maps $\delta_n\colon\mathscr{C}(\mathcal{G}^{(n)},A)\to \mathscr{C}(\mathcal{G}^{(n-1)},A)$, defined by:
		
		\begin{equation*}
		\delta_n:=\sum_{i=0}^{n} (-1)^i d_{i\ast}
		\end{equation*}
		
		Then the following defines a chain complex
		\begin{equation*}
		0 \stackrel{\delta_0}{\leftarrow} \mathscr{C}(\mathcal{G}^{(0)},A) \stackrel{\delta_1}{\leftarrow} \mathscr{C}(\mathcal{G}^{(1)},A) \stackrel{\delta_2}{\leftarrow} \mathscr{C}(\mathcal{G}^{(2)},A) \stackrel{\delta_3}{\leftarrow}\cdots
		\end{equation*}
		
		The homology groups $H_n (\mathcal{G},A) :=\ker \delta_n / \Ima \delta_{n+1} $ of this chain complex are called the \emph{homology groups of $\mathcal{G}$ with constant coefficients $A$}. In the case of $A = \mathbb{Z}$, we abbreviate $H_n (\mathcal{G}):=H_n (\mathcal{G},\mathbb{Z})$. Furthermore define $H_0(\mathcal{G})^+:=\{[f] \in H_0(\mathcal{G})|f \geq 0, f \in \mathscr{C}(\mathcal{G}^{(0)},A)\}$.\footnote{This set is not necessarily a positive cone!\\}
	\end{defi}
	\begin{ex}[\cite{mat12}, §4.2]\label{ex: homology}
		Let $\mathcal{G}_\alpha$ be the transformation groupoid associated with an action $\alpha \colon  G \to \operatorname{Homeo}(X)$ of a countable, discrete group $G$ on a Cantor space $X$. Then the homology groups can be expressed in terms of group homology as $H_n(\mathcal{G}_\alpha) \cong H_n(G,C(X,\mathbb{Z}))$ and Poincaré-duality holds -- correspondingly $H^n(\mathcal{G}_\alpha) \cong H^n(G,C(X,\mathbb{Z}))$ holds for the cohomology groups. These groups are respectively termed the \emph{dynamical (co-)homology groups} of the system $(X,G)$. In case of a minimal Cantor system $(X,\varphi)$ we have $H_{*}(\mathcal{G}_\varphi) \cong K_{*}(C(X) \rtimes_{\varphi} \mathbb{Z})$ (-- see \cite{fh99}).
	\end{ex}
	
	\section{Of the relation between inverse semigroups and \'etale groupoids}\label{sec: interweaving inverse semigroups and etale groupoids}
	
	Subsection~\ref{subs: some historical remarks} recollects aspects that tie together some of the content of preceding subsections and that motivated non-sommutative Stone duality. Subsection~\ref{subs: inverse wedge monoids and abstract pseudogroups} introduces an abstract definition of pseudogroups. Subsection~\ref{subs: non-commutative stone duality} recalls a very general version of non-commutative duality i.e. between sober \'etale groupoids and spatial pseudogroups. 
	
	\subsection{Some historical remarks}\label{subs: some historical remarks}
	
	Reconsider Example~\ref{rem: crossedprod} of the crossed product C*-algebra $C(X) \rtimes_\varphi \mathbb{Z}$ associated with a topological $\mathbb{Z}$-system $(X,\varphi)$ or the case of AF C*-algebras in Theorem~\ref{thm: AF}. Both settings comprise a pair $(\mathfrak{A},\mathfrak{B})$ of operator algebras where $\mathfrak{B}$ is an abelian subalgebra of $\mathfrak{A}$ the inclusion of which reflects the nature of the underlying system. The Feldman-Moore theorem\footnote{See \cite{fm75} and \cite{fm77}.\\} gave a description for this in the context of von Neumann algebras, in that Cartan pairs were shown to precisely correspond to  twisted countable standard measured equivalence relations i.e. in combination with a cohomology class on the relation.
	
	In \cite{ren80} Renault took first steps towards a C*-algebra theoretic/topological version and adopted the viewpoint of using topological groupoids as model spaces for C*-algebras. He did not succeed in finding an analog to \cite{fm75} -- in part due to the fact that some definitions coming from the measured context needed an adaption, nevertheless, much of the ingredients were already present. In \cite{kum86} Alex Kumjian showed that C*-diagonals, the definition of which is motivated by Cartan subalgebras in the measured context, translate to twisted \'etale equivalence relations, which correspond to twisted principial, \'etale groupoids.\footnote{Analogous to the case of groups, groupoid extensions are parametrized by second cohomology (\cite{ren80}, Proposition~1.14). A twist is a groupoid extensions arising from a $2$-cocycles with values in the circle group $\mathbb{T}$ -- we won't go into details.\\} This result, however, does not cover some fundamental examples e.g. Cuntz-Krieger algebras. In \cite{ren08} Renault was able to extend this to the setting of effective, \'etale groupoids in \cite{ren08}.
	
	\begin{defi}[\cite{ren08}, Definition~5.1]\label{defi: cartan pair}
		An abelian C*-subalgebra $\mathfrak{B}$ of a C*-algebra $\mathfrak{A}$ is called a \emph{Cartan subalgebra} if it satisfies the following properties:
		\begin{enumerate}[(i)]
			\item $\mathfrak{B}$ contains an approximate unit of $\mathfrak{A}$
			
			\item $\mathfrak{B}$ is a maximal abelian subalgebra
			
			\item $\mathfrak{B}$ is \emph{regular} i.e. the normalizer $N(\mathfrak{B},\mathfrak{A})$ generates $\mathfrak{A}$.
			
			\item There exists a faithful conditional expectation $P$ of $\mathfrak{A}$ onto $\mathfrak{B}$.
		\end{enumerate}
		The pair $(\mathfrak{B},\mathfrak{A})$ is then called a \emph{Cartan pair}.
	\end{defi}
	
	\begin{ex}[\cite{ren80}, §4]\label{ex: groupoid Cstaralgebra as cartan pair}
		Let $\mathcal{G}$ be a locally compact, second countable, Hausdorff, \'etale groupoid. Then $C_0(\mathcal{G}^{(0)})$ is an abelian subalgebra of $C_r^{*}(\mathcal{G})$ that contains an approximate unit. Moreover the map $E \colon C_r^{*}(\mathcal{G}) \to C_0(\mathcal{G}^{(0)})$ induced by the restriction map $C_c(\mathcal{G}) \to C_0(\mathcal{G}^{(0)})$ given by $f \mapsto f|_{\mathcal{G}^{(0)}}$ is a faithful conditional expectation. The pair $(C_0(\mathcal{G}^{(0)}), C_r^{*}(\mathcal{G}))$ constitutes a Cartan pair i.e.  \emph{if and only if} $\mathcal{G}$ is effective.
	\end{ex}

	Let $(\mathfrak{B},\mathfrak{A})$ be such a Cartan pair, let $n \in N(\mathfrak{B},\mathfrak{A})$ and let $X$ denote the spectrum of $\mathfrak{B}$. The adjoint action of $n$ on $\mathfrak{B}$ induces a uniquely defined partial homeomorphism $\alpha(n)$ between the open sets $\operatorname{d}(n):=\{x \in X|n^{*}n(x)>0 \}$ and $\operatorname{r}(n):=\{x \in X|nn^{*}(x)>0 \}$ such that $n^{*}bn(x)=b(\alpha(n)(x))n^{*}n(x)$ for all $b \in \mathfrak{B}$ and $x \in \operatorname{d}(n)$ (\cite{kum86}, Proposition~6). This defines a continuous action of an inverse semigroup on $X$. The associated groupoid of germs $\operatorname{Germ}(X ,\alpha(N(\mathfrak{B},\mathfrak{A})))$ together with a twist constructed from the pair $(\mathfrak{B},\mathfrak{A})$ gives rise to a twisted groupoid C*-algebra which is precisely the original Cartan pair. For an effective groupoid $\mathcal{G}$ as in Example~\ref{ex: groupoid Cstaralgebra as cartan pair} the groupoid $\operatorname{Germ}(X ,\alpha(N(C_0(\mathcal{G}^{(0)}), C_r^{*}(\mathcal{G}))))$ is canonically isomorphic to $\mathcal{G}$.
	
	Already in \cite{ren80} Renault had observed that inverse semigroups play a crucial role in the theory.
	The idea to model a C*-algebra on an \'etale groupoid which comes from an inverse semigroup proved to be very fruitful, in particular it provides means to tackle C*-algebras that arise from combinatorial structure e.g. directed graphs, higher-rank graphs or tilings. Examples of such combinatorial C*-algebras are Cuntz-Krieger algebras (see Definition~\ref{defi: cuntz-krieger algebras}) or AF C*-algebras (see Proposition~\ref{prop: AF-groupoids}). Renault's investigations where developed further by Alan L. T. Paterson in \cite{pat99}. By Theorem~3.3.2 in \cite{pat99} every \emph{localization} $(X,S)$ i.e. a locally compact Hausdorff space $X$ acted upon by a countable inverse semigroup $S$ such that the domains of elements generate the topology of $S$,\footnote{Motivated by Renault's observations, localizations where introduced by Kumjian in \cite{kum84} -- there the inverse semigroup $S$ is assumed to be a traditional pseudogroup.\\} induces an inverse semigroup homomorphism $\psi \colon S \to \mathcal{B}_ {\operatorname{Germ}(X,S)} ^{o}$, which raises the question how an \'etale groupoid can be built from a general inverse semigroup $S$ e.g. by an action that is intrinsic to $S$. Examples of such constructions are the universal groupoid of Paterson described in \cite{pat99} or -- since in many cases this is too big -- a reduction of it, the tight groupoid of Exel from \cite{exe08}. Another viewpoint was provided by Kellendonk in \cite{kel97} in the context of tilings.
	
	This brought the attentation to the relations between groupoids and inverse semigroups and to the question if there is some unified approach subsuming the above constructions. David Lenz observed in \cite{len08} that the approaches are reconciled when considering the order structures on inverse semigroups, in that the constructed groupoids are equivalence classes of downward directed sets.
	Mark V. Lawson realized that in the case of Boolean \'etale groupoids and Boolean inverse monoids the consideration of inverse semigroups of slices and groupoids of filters produces a duality that generalizes classical Stone duality i.e. the Boolean \'etale groupoids are "non-commutative Stone spaces" and Boolean inverse monoids the corresponding "non-commutative Boolean algebra".\footnote{A quick description of this will be given in Subsection~\ref{subs: refined correspondences}.\\} Similar to classical Stone duality this duality generalizes to much broader contexts -- that of pseudogroups and \'etale groupoids.

	Another question relevant in this context is, which inverse subsemigroups of $\mathcal{B}_\mathcal{G}^{o}$ allow a reconstruction of the original groupoid $\mathcal{G}$. This lets us already glimpse at Chapter~\ref{chap: 3}, where the topological full group of an \'etale Cantor groupoid $\mathcal{G}$ turns out to be the unit group of the inverse semigroup of open, compact slices $U(\mathcal{B}_{\mathcal{G}}^{o,k})$. The question of reconstructability of an \'etale Cantor groupoid up to isomorphism from its topological full group is situated in midst of the interplay between \'etale groupoids and inverse semigroups.

	\subsection{Inverse \texorpdfstring{$\wedge$}{wedge}-monoids and abstract pseudogroups}\label{subs: inverse wedge monoids and abstract pseudogroups}

	The following definitions introduce the objects on the ``inverse semigroup side" of non-commutative Stone duality, they are inverse semigroups with a frame-like order. Frames are order structures that in particular describe how open sets of topological spaces join and intersect. The formal dual to the category of frames $\mathbf{Frm}$ is the category of locales $\mathbf{Loc}$, which in some sense are generalized topological spaces that do not necessarily have ``enough points" -- the reason why the theory has been termed pointless topology. One can adopt the view of this formal dual relation as the most general and obviously most superficial form of Stone duality. Results of deeper meaning arise, when one turns to subcategories i.e. dualities between classes of topological spaces and frame subcategories, the most general in this context being the duality between sober spaces and spatial frames. The historically first iteration of such a duality came in form of the Stone representation theorem for Boolean algebras. Marshall H. Stone discovered that the space of ultrafilters in a Boolean algebra carries a topology and moreover that the Boolean algebra of clopen sets in this space is precisely the initial Boolean algebra. He was able to characterize the topological spaces which arise in such a way as compact, Hausdorff, totally disonnected spaces (these are now termed \emph{Boolean spaces} or \emph{Stone spaces}) and showed that the correspondence between these classes of objects lifts to the level of morphism. Stone's remarkable achievements compelled mathematicians to name the subject it inspired in his honour -- Stone's glory engraved in stone. See \cite{joh82} for a monograph on Stone duality.
	
	\begin{defi}[\cite{law16}, §2]
		Let $S$ be an inverse semigroup. If every binary meet in $S$ exists, it is called \emph{inverse $\wedge$-semigroup.}
	\end{defi}
	
	Inverse $\wedge$-monoids thus are inverse monoids with a (meet-)semilattice-like order structure and as such had already been identified as worthwhile to study by Leech in \cite{lee95}. Leech observed that a characteristic feature of inverse $\wedge$-monoids is the presence of a fixed-point operator:
	\begin{defi}[\cite{law16}, §2]
		Let $S$ be an inverse monoid. A map $f \colon S \to E(S)$ is called a \emph{fixed point operator} if it satisfies the following properties:
		\begin{enumerate}[(i)]
			\item $s \geq f(s) \forall s \in S$ 
			
			\item Every $s \in S$ and $e \in E(S)$ such that $e \leq s$ satisfy $e \leq f(s)$.
		\end{enumerate}
	\end{defi}
	
	\begin{prop}[\cite{law16}, Proposition~2.2]
		An inverse monoid is an inverse $\wedge$-monoid \emph{if and only if} it has a fixed point operator. The fixed point operator of an inverse $\wedge$-monoid $S$ is unique and given by $f \colon s \mapsto s \wedge 1$.
	\end{prop}
	
	The abstract frame theoretic definition of pseudogroups goes back to Pedro Resende.
	
	\begin{defi}[\cite{res07}, Definition~2.5 \& 2.8 \& 2.9]
		Let $S$ be an inverse semigroup.	
		\begin{enumerate}[(i)]
			\item A pair of elements $s,t \in S$ is said to be \emph{compatible} if $s^{-1}t$ and $st^{-1}$ are idempotents. A subset $A \subseteq S$ is called \emph{compatible} if all its elements are pairwise compatible. Note, that if $s \vee t$ exists, $s$ and $t$ are necessarily compatible. 
			
			\item It is called \emph{complete} if every compatible subset has a join.
			
			\item It is called \emph{distributive} if $s \vee t$ exists for every compatible pair $s,t \in S$ and every binary join $s \vee t \in S$ and every $u \in S$ the joins $us \vee ut$ and $su \vee tu$ exist and $us \vee ut=u(s \vee t)$ and $su \vee tu=(s \vee t)u$

			\item It is called \emph{infinitely distributive} if for any subset $T \subseteq S$ that has a join and every $s \in S$ the following hold:
			\begin{enumerate}
				\item $\bigvee_{t\in T}st$ and $\bigvee_{t\in T}ts$ exist.
				
				\item $s ( \bigvee_{t\in T}t)=\bigvee_{t\in T}st$.
				
				\item $(\bigvee_{t\in T}t) s=\bigvee_{t\in T}ts$.
			\end{enumerate}

			\item A \emph{pseudogroup} is a complete, infinitely distributive inverse monoid.
			
			\item A semigroup homomorphism between pseudogroups that preserves all compatible joins is called a \emph{pseudogroup homomorphism}.
		\end{enumerate}
	\end{defi}
	
	\begin{prop}[\cite{res07}, Proposition~2.10]
		Let $S$ be a pseudogroup. Then $S$ is an inverse $\wedge$-semigroup and $E(S)$ is a frame.
	\end{prop}
	
	The above definition of pseudogroup is the 'pointless topology'-version of what was historically the motivation to consider abstract inverse semigroups -- the pseudogroups of partial homeomorphisms between open sets in a topological space. Analogous to the bridge from algebra to topology in Stone's representation theorem, the fundamental notion to get from pseudogroups to \'etale groupoids are filters:
	
	\begin{defi}[\cite{ll13}]
		Let $S$ be a pseudogroup.
		\begin{enumerate}[(i)]
			\item Denote by $\mathrm{L}(S)$ the set of all filters in $S$.
			
			\item A filter $F$ in $\mathrm{L}(S)$ is \emph{proper}, if $0 \notin F$.
			
			\item A filter $F$ in $\mathrm{L}(S)$ is called \emph{completely prime} if for every join $\vee_{i \in I} a_i \in F$ there exists an $i \in I$ such that $a_i \in F$.
			
			\item Denote by $\mathrm{G}(S)$ the set of all completely prime filters in $S$.
			
			\item Let $s \in S$. Denote by $X_s$ the set of all completely prime filters that contain $s$.
			
			\item A pseudogroup $S$ is called \emph{spatial} if $X_s=X_t$ implies $s=t$ for all $s,t \in S$.
			
			\item A maximal proper filter $F$ in $\mathrm{L}(S)$ is called an \emph{ultrafilter}.
			
			\item Let $s \in S$. Denote by $V_s$ the set of all ultrafilters that contain $s$. 
			
			\item Let $S$ and $T$ be pseudogroups. A map $f \colon S \to T$ is called \emph{callitic} if:
			\begin{enumerate}
				\item It is a pseudogroup homomorphism that preserves all meets.
				
				\item $F \cap \Ima(f)\neq \emptyset$ for all completely prime filters $F \subseteq T$.
			\end{enumerate}
			
			\item Denote by $\mathbf{Pseu}$ the category which has pseudogroups as objects and callitic maps as morphisms.
		\end{enumerate}
	\end{defi}
	
	\subsection{Non-commutative Stone duality}\label{subs: non-commutative stone duality}
	
	On the level of objects this general duality had already established in \cite{mr10}. Endowed with a sufficient notion of morphisms the class of \'etale groupoids becomes a category:
	
	\begin{defi}[\cite{ll13}]
		\begin{enumerate}[(i)]
			\item Let $\mathcal{G}$ and $\mathcal{H}$ be \'etale groupoids and let $F \colon \mathcal{G} \to \mathcal{H}$ be a groupoid homomorphism. It is called a \emph{covering functor} if $F|_{\mathcal{G}_u}$ is bijective for all $u \in \mathcal{G}^{(0)}$.
			
			\item Denote by $\mathbf{Etale}$ the category which has \'etale groupoids as objects and continuous covering functors as morphisms.
		\end{enumerate}
	\end{defi}
	
	We start with the very general formulation of \cite{ll13}: The aim is to find a pair of functors $\mathrm{B} \colon \mathbf{Etale} \to \mathbf{Pseu}^{\mathrm{op}}$ and $\mathrm{G} \colon \mathbf{Pseu}^{\mathrm{op}} \to \mathbf{Etale}$ between the category of \'etale groupoids $\mathbf{Etale}$ and the category of pseudogroups $\mathbf{Pseu}$ where $\mathrm{G}$ is right adjoint to $\mathrm{B}$. This adjunction restricts to a duality between the categories of spatial pseudogroups $\mathbf{Pseu}_\mathrm{sp}$ and sober \'etale groupoids $\mathbf{Etale}_\mathrm{sob}$. We start with the description of the functors $\mathrm{B} \colon \mathbf{Etale} \to \mathbf{Pseu}^{\mathrm{op}}$ and $\mathrm{G} \colon \mathbf{Pseu}^{\mathrm{op}} \to \mathbf{Etale}$ on the level of objects:
	
	\begin{prop}[\cite{ll13}, Proposition~2.1]
		Let $\mathcal{G}$ be an \'etale groupoid. Then its set of open slices $\mathcal{B}_\mathcal{G}^o$ is a pseudogroup with respect to multiplication of subsets.
	\end{prop}
	
	The functor $B \colon \mathbf{Etale} \to \mathbf{Pseu}^{\mathrm{op}}$ is defined on objects as $B \colon \mathcal{G} \mapsto \mathcal{B}_\mathcal{G}^o$. The definition of $\mathrm{G}$ on objects is given by means of filters:	By setting $s(F)=(F^*F)^{\uparrow}$ and $r(F)=(FF^*)^{\uparrow}$ for $F \in \mathrm{L}(S)$, the set of composable pairs i.e. $\{(F,G)\in \mathrm{L}(S)^2|s(F)= r(G)\}$ admits a product $F \cdot G:= (FG)^\uparrow$ with respect to which $\mathrm{L}(S)$ is a groupoid. All the different groupoids constructed from an inverse semigroup (\cite{pat99}, \cite{kel97}, \cite{exe08}, etc.) are specific subgroupoids of $\mathrm{L}(S)$ (\cite{ll13}, p.6). We focus on the subgroupoid of completely prime $\mathrm{G}(S)$. This groupoid becomes a topological groupoid by taking the sets $\{X_s|s \in S \}$ as a basis.
	
	\begin{prop}[\cite{ll13}, Proposition~2.8]
		Let $S$ be a pseudogroup. Then $\mathrm{G}(S)$ is an \'etale groupoid.
	\end{prop}
	
	One needs to inspect how $\mathrm{B}$ and $\mathrm{G}$ behave in conjunction with morphisms:
	\begin{lem}[\cite{ll13}, Lemma~2.13 \& 2.14]
		Let $f \colon S \to T$ be a callitic map between pseudogroups. Then the inverse image $f^{-1}(F)$ of every completely prime filter $F$ in $T$ is a non-empty completely prime filter and the induced map $f^{-1} \colon \mathrm{G}(T) \to \mathrm{G}(S)$ is a continuous covering functor.
	\end{lem}
	
	The above lemma ``shows" that the first notion of morphisms between pseudogroups one might think of i.e. just combining the notions of frame morphisms and morphisms of inverse semigroups is not sufficient. The second condition in the definition of callitic maps is needed, because it assures that inverse images of completely prime filters are non-empty. In inverse direction one has:

	\begin{lem}[\cite{ll13}, Lemma~2.19]
		Let $f \colon \mathcal{G} \to \mathcal{H}$ be a continuous covering functor between \'etale groupoids. Then the induced map $f^{-1} \colon \mathrm{B}(\mathcal{H}) \to \mathrm{B}(\mathcal{G})$ on open slices is a callitic map.
	\end{lem}
	
	The obtained functors $\mathrm{B} \colon \mathbf{Etale} \to \mathbf{Pseu}^{\mathrm{op}}$ and $\mathrm{G} \colon \mathbf{Pseu}^{\mathrm{op}} \to \mathbf{Etale}$ then set up an adjunction:
	\begin{thm}[\cite{ll13}, Theorem~2.22]\label{thm: adj}
		The functor $\mathrm{G} \colon \mathbf{Pseu}^{\mathrm{op}} \to \mathbf{Etale}$ is right adjoint to the functor $\mathrm{B} \colon \mathbf{Etale} \to \mathbf{Pseu}^{\mathrm{op}}$.
	\end{thm}
	
	This adjunction is a generalization of the adjunction between topological spaces and frames in that every \'etale groupoid that only consists of units is simply a topological space and its dual a pseudogroup that contains only idempotents, hence a frame. From this generalization one derives a duality in the same manner as the duality between sober topological spaces and spatial frames comes from the adjunction between topological spaces and frames.
	\begin{defi}
		\begin{enumerate}[(i)]
			\item An \'etale groupoid $\mathcal{G}$ is said to be \emph{sober} if $\eta \colon \mathcal{G} \to \mathrm{G}(\mathrm{B}(\mathcal{G}))$ is a homeomorphism.
			
			\item Denote by $\mathbf{Pseu}_\mathrm{sp}$ the category which has spatial pseudogroups as objects and callitic maps as morphisms and by $\mathbf{Etale}_\mathrm{sob}$ the category which has sober \'etale groupoids as objects and continuous covering functors as morphisms.
		\end{enumerate}
	\end{defi}
	
	\begin{lem}[\cite{ll13}, Proposition~2.12]
		For every sober \'etale groupoid $\mathcal{G}$ the pseudogroup $\mathrm{B}(\mathcal{G})$ is a spatial pseudogroup and for every spatial pseudogroup $S$ the \'etale groupoid $\mathrm{G}(S)$ is sober.
	\end{lem}
	
	Above lemma then implies:
	\begin{thm}[\cite{ll13}, Theorem~2.23]
		The adjunction in Theorem~\ref{thm: adj} induces an equivalence of categories between $\mathbf{Pseu}_\mathrm{sp}^\mathrm{op}$ and $\mathbf{Etale}_\mathrm{sob}$.
	\end{thm}
	
	In Subsection~\ref{subs: refined correspondences} we turn to a refinements of this correspondence.
	
	\section{Resuming Cantor dynamics}\label{sec: resuming cantor dynamics}
	
	\emph{\large \label{reminder second countability} For the remainder of this text with the exception of Section~\ref{sec: topological full groups in light of non-commutative stone duality} every \'etale groupoid is assumed to be second countable.}
	
	\bigskip
	
	\'Etale groupoids with a Cantor space as space of units provide Cantor dynamics with a broader scope. For the sake of abbreviation we define:
	\begin{defi}
		A topological groupoid $\mathcal{G}$ is called a \emph{Cantor groupoid} if the topological subspace $\mathcal{G}^{(0)}$ is a Cantor space.
	\end{defi}
	
	\begin{rem}
		\'Etale Cantor groupoids are by definition locally compact groupoids.
	\end{rem}

	In Subsection~\ref{subs: homology of etale cantor groupoids} we take a quick look at the homology of \'etale Cantor groupoids. Subsection~\ref{subs: compact generation and expansive groupoids} showcases generalized subshifts. Subsection~\ref{subs: AF-gruopoids and almost finite groupoids}

	\subsection{Homology of \'etale Cantor groupoids}\label{subs: homology of etale cantor groupoids}

	Matui studied \'etale Cantor groupoids and their topological full group in a series of papers (\cite{mat12},\cite{mat15},\cite{mat16}). In \cite{mat16} Matui made a conjecture, which he termed \emph{HK-conjecture}, that relates the groupoid homology with the $K_0$-group of the groupoid C*-algebra:
	\begin{conj}[\cite{mat16}, Conjecture~2.6]\label{conj: hk}
		Let $G$ be an effective, minimal, \'etale Cantor groupoid. Then
		\begin{equation*}
		K_0(C_r^*(\mathcal{G})) \cong \bigoplus_{i = 0}^{\infty} H_{2i}(\mathcal{G}) 
		\end{equation*}
		and
		\begin{equation*}
		K_1(C_r^*(\mathcal{G})) \cong \bigoplus_{i = 0}^{\infty} H_{2i+1}(\mathcal{G}).
		\end{equation*}
	\end{conj}
	
	The following propositions will be of use in later sections:
	
	\begin{prop}[\cite{mat12}, Lemma~7.3]\label{prop: hom1}
		Let $\mathcal{G}$ be an effective, \'etale Cantor groupoid. Then the following holds:
		\begin{enumerate}[(i)]
			\item Let $B,B'$ be compact, open slices of $\mathcal{G}$ with $r(B)=s(B')$ and let $O \subseteq \mathcal{G}^{(2)}$ be the set of composable pairs $(b,b')$ with $b \in B$ and $b' \in B'$. Then $\delta_2(\mathbf{1}_O)=\mathbf{1}_B-\mathbf{1}_{BB'}+\mathbf{1}_{B'}$
		\end{enumerate}
		As immediate consequences of (i), the following hold:
		\begin{enumerate}[(i),resume]
			\item $[\mathbf{1}_U]=0$ in $H_1(\mathcal{G})$ for every compact, open subset $U \subseteq \mathcal{G}^{(0)}$.
			
			\item $[\mathbf{1}_B]+[\mathbf{1}_{B^{-1}}]=0$ in $H_1(\mathcal{G})$ for every compact, open slice $B$ of $\mathcal{G}$. 
		\end{enumerate}
	\end{prop}
	
	In Chapter~\ref{chap: 3}, the group $H_0(\mathcal{G})\otimes (\mathbb{Z}/2\mathbb{Z}) =H_0(\mathcal{G},\mathbb{Z}/2\mathbb{Z})$ will play a role in the description of the abelization of the topological full group of certain \'etale Cantor groupoids. Let $f \in C(\mathcal{G}^{(0)},\mathbb{Z})$. The set $U_f:=\{u \in \mathcal{G}^{(0)}|f(u) \notin 2\mathbb{Z} \}$ is compact, open. The class $[f]+2 H_0(\mathcal{G})$ in $H_0(\mathcal{G})\otimes (\mathbb{Z}/2\mathbb{Z})$ is uniquely determined by the class of $[\mathbf{1}_{U_f}] + 2 H_0(\mathcal{G})$. Thus one obtains the following description:
	\begin{prop}[\cite{nek15}, Proposition~7.1]\label{prop: 2-homo}
		Let $\mathcal{G}$ be an \'etale Cantor groupoid. Let $\tilde{H}(\mathcal{G})$ be the abelian group given by the following representation:
		\begin{enumerate}[(i)]
			\item The set of generators is given by all $\mathbf{1}_{U}$, where $U$ is a clopen subset of $\mathcal{G}^{(0)}$.
			
			\item The set of relations consist of:
			\begin{enumerate}
				\item $\mathbf{1}_U + \mathbf{1}_U = 0$ for all clopen $U \subseteq \mathcal{G}^{(0)}$.
				
				\item For every finite disjoint union of clopen subsets $U=\bigsqcup U_i$, it holds that $\mathbf{1}_U=\sum_{i} \mathbf{1}_{U_i}$.
				\item For every clopen bisection $B \in \mathcal{B}_\mathcal{G}^{o,k}$, it holds that $\mathbf{1}_{s(B)}=\mathbf{1}_{r(B)}$.
			\end{enumerate}
		\end{enumerate}
		Then $H_0(\mathcal{G})\otimes (\mathbb{Z}/2\mathbb{Z}) \cong \tilde{H}(\mathcal{G})$.
	\end{prop}

	\subsection{Compact generation and expansive groupoids}\label{subs: compact generation and expansive groupoids}
	
	The following is a translation of the definition in \cite{hae02} -- there given in the context of pseudogroups:
	\begin{defi}[\cite{nek15}, Definition 2.3.1]					
		Let $\mathcal{G}$ be an \'etale groupoid.
		\begin{enumerate}[(i)]
			\item Let $(S, U)$ be a pair of compact sets $S \subseteq \mathcal{G}$ and $U \subseteq \mathcal{G}^{(0)}$. The pair $(S, U)$ called a \emph{compact generating pair} if $U$ contains an open $\mathcal{G}$-transversal and for all $g \in \mathcal{G}|_U$ there exists an $n \in \mathbb{N}$ such that $\bigcup_{k=1}^n (S \cup S^{-1})^k$ is an open neighbourhood of $g$ in $ \mathcal{G}|_U$.
			
			\item The groupoid $\mathcal{G}$ is said to be \emph{compactly generated} if it contains a compact generating pair. 
		\end{enumerate}
	\end{defi}
	
	The case of an \'etale groupoid with compact unit space, in particular the case of an \'etale Cantor groupoid, greatly simplifies above definition
	\begin{defi}
		Let $\mathcal{G}$ be an \'etale Cantor groupoid. It is called \emph{compactly generated} if it contains a \emph{compact generating set} i.e. there exists a compact subset $S \subseteq \mathcal{G}$ such that $\mathcal{G}= \bigcup_{n=1}^{\infty} (S \cup S^{-1})^n$.
	\end{defi}
	
	By the assumptions $S$ admits a finite cover by open, compact slices, the union of which may replace $S$ in the definition. Expansiveness is given in terms of such a cover and constitutes a class of compactly generated \'etale Cantor groupoids that generalize the notion of dynamical shifts:
	
	\begin{defi}[\cite{nek17}, Definition~5.2]
		Let $\mathcal{G}$ be a compactly generated, \'etale Cantor groupoid.
		\begin{enumerate}[(i)]
			\item Let $S$ be a compact generating set of $\mathcal{G}$ and let $\mathcal{B}$ be a finite cover of $S$ by open, compact slices. The cover $\mathcal{B}$ is called \emph{expansive} if $\bigcup_{n \in \mathbb{N}} (\mathcal{B} \cup \mathcal{B}^{-1})^n$ is a topological basis for $\mathcal{G}$.
			
			\item The groupoid $\mathcal{G}$ is said to be \emph{expansive}, if it has a compact generating set which has an expansive cover.
		\end{enumerate}
	\end{defi}
	
	\begin{rem}[\cite{nek17}, Proposition 5.2]\label{rem: exp}
		If an \'etale Cantor groupoid $\mathcal{G}$ is expansive, all compact generating sets of $\mathcal{G}$ admit expansive covers.
	\end{rem}
	
	Expansive covers can be characterized as follows:
	\begin{prop}[\cite{nek17}, Proposition~5.3]\label{prop: expansivecover}
		Let $\mathcal{G}$ be a compactly generated, \'etale Cantor groupoid and let $S$ be a compact generating set. Let $\mathcal{B}$ be a finite cover of $S$ by open slices. Then the following are equivalent:
		\begin{enumerate}[(i)]
			\item The cover $\mathcal{B}$ is expansive.
			
			\item The set $\{s(F)| F\in \bigcup_{n \in \mathbb{Z}_{+}} (\mathcal{B} \cup \mathcal{B}^{-1} )^n \}$ is a topological basis of $\mathcal{G}^{(0)}$.
			
			\item For every pair $u_1,u_2 \in \mathcal{G}^{(0)}$ with $u_1 \neq u_2$ there exist $F_1,F_2 \in \bigcup_{n \in \mathbb{Z}_{+}} (\mathcal{B} \cup \mathcal{B}^{-1} )^n$ such that $u_1 \in s(F_1)$ and $u_2 \in s(F_2)$ and $s(F_1)\cap s(F_2)=\emptyset$.
			
			\item For every $u \in \mathcal{G}^{(0)}$, it holds that:
			\begin{equation*}
			\bigcap_{n \in \mathbb{Z}_{+}} \bigcap \{s(F)|k \leq n, F \in (\mathcal{B}\cup \mathcal{B}^{-1} )^k,u \in s(F) \}=\{u\}
			\end{equation*}
		\end{enumerate} 
	\end{prop}

	The naming derives from the following class of group actions:
	\begin{defi}
		The action of a finitely generated group $G$ on a uniform space $X$ is called \emph{expansive}, if there exists a neighbourhood of the diagonal $\Delta \subseteq W \subseteq X \times X$ such that $x,y \in X$ and $(g(x),g(y)) \in W$ for all $g \in G$ implies $x=y$.
	\end{defi}
	
	In the case of the action of a finitely generated group $G$ on a Cantor space $X$ above definition allows to construct a finite clopen partition $\mathcal{P}$ of $X$ such that $G$-translates of atoms separate points, by which the system is conjugate to a $G$-shift (Remark~\ref{rem: symbolic dynamics}).
	Expansive Cantor groupoids can be seen as ``generalized subshifts" and in the case of transformation groupoids on has:
	
	\begin{prop} [\cite{nek17}, Proposition 5.5]
		Let a finitely generated group $G$ act by $\alpha \colon G \to \operatorname{Homeo}(X)$ on a Cantor space $X$. Then the following are equivalent:
		\begin{enumerate}[(i)]
			\item The action of $G$ on $X$ is expansive.
			
			\item The groupoid $\mathcal{G}_{(X,G)}$ is expansive.
			
			\item The groupoid $\operatorname{Germ}(X,G)$ is expansive.
			
			\item The dynamical system $(X, G)$ is a subshift.
		\end{enumerate}
	\end{prop}
	
	\begin{ex}
		Let $A$ be a finite set and let $G$ be a group with finite, symmetric generating set $T$. Let $(X,G)$ be a subsystem of $(A^G,G)$ i.e. a $G$-subshift. Then $S:=\{(t,x)|s \in T,x\in X \}$ is a compact generating set of $\mathcal{G}_{(X,G)}$. The set $\mathcal{S}:=\{(t,C(a))|t \in T, a \in A \}$ where $C(a)$ denotes the cylinder set $\{x \in X| x_1=a \}$ is an expansive cover of $S$ by Proposition~\ref{prop: expansivecover}[(iv)$\Rightarrow$(i)].
	\end{ex}

	The following applies in particular to minimal groupoids:
	
	\begin{lem} [\cite{nek17}, Lemma 5.1]\label{lem: comp.gen.}
		Let $\mathcal{G}$ be a compactly generated \'etale Cantor groupoid, such that every $\mathcal{G}$-orbit contains $2$ or more elements.
		Then there exists an open compact generating set contained in $\mathcal{G}\setminus\operatorname{Iso}(\mathcal{G})$.
	\end{lem}
	
	\begin{proof}
		Let $S$ be a open compact generating set containing $\mathcal{G}^{(0)}$.
		By the assumptions $S$ has a finite cover $\mathcal{S}$ of by slices. Let $S_i$ be a component of $\mathcal{S}$ and let $g \in S_i$.
		Given $g$ is contained in $\operatorname{Iso}(\mathcal{G})$, then there exists an element $h \in \mathcal{G}\setminus\operatorname{Iso}(\mathcal{G})$ with $s(g)=s(h)$. Since $\mathcal{G}$ is \'etale, there exists a slice $B$ with $B \subseteq \mathcal{G}\setminus\operatorname{Iso}(\mathcal{G})$ and $h \in B$ such that $s(BS_i) \cap r(BS_i)$. Then the slice $B_g:=B \cup BS_i$ is contained in $\mathcal{G}\setminus\operatorname{Iso}(\mathcal{G})$ and $N_g:=B^{-1}BS_i$ is a neighbourhood of $g$ contained in $\bigcup_{n \in \mathbb{N}} (B_g \cup B_g^{-1})^{n}$ (i.e the subgroupoid of $\mathcal{G}$ generated by $B_g$). Given $g \in \mathcal{G}\setminus\operatorname{Iso}(\mathcal{G})$, let $N_g \subseteq S_i$ be a neighbourhood of $g$ contained in $\mathcal{G}\setminus \operatorname{Iso}(\mathcal{G})$ and define $B_g:=N_g$.	The collection $\{N_g\}_{g \in S_i}$ is a cover of $S_i$, hence there exists a finite subcover $\{N'_{g_k}\}$. Then the set $S':=\bigcup_{k} B_{g_k}$ is compact, it generates $S$ and $S' \subseteq \mathcal{G}\setminus\operatorname{Iso}(\mathcal{G})$.
	\end{proof}
	
	\subsection{AF-groupoids and almost finite groupoids}\label{subs: AF-gruopoids and almost finite groupoids}
	
	The following definition is originally due to Renault:
	
	\begin{defi}[\cite{mat12}, Definition~2.2]
		Let $\mathcal{G}$ be an \'etale Cantor groupoid.
		\begin{enumerate}[(i)]
			\item A subgroupoid $\mathcal{H} \leq \mathcal{G}$ is called an \emph{elementary subgroupoid}, if $\mathcal{H}$ is compact, open and principal with $\mathcal{H}^{(0)}=\mathcal{G}^{(0)}$.
			
			\item The groupoid $\mathcal{G}$ is called an \emph{AF-groupoid} if there is a family of nested elementary subgroupoids $\{ \mathcal{H}_i\}$ such that $\mathcal{G}=\bigcup_i \mathcal{H}_i$.
		\end{enumerate}
	\end{defi}
	
	These are the first examples which arise as an application of Renault's program in \cite{ren80} in that:
	\begin{prop}[\cite{ren80}, Proposition~III.1.15]\label{prop: AF-groupoids}
		Let $\mathfrak{A}$ be a C*-algebra. Then the following are equivalent:
		\begin{enumerate}[(i)]
			\item The C*-algebra $\mathfrak{A}$ is AF.
			
			\item The C*-algebra $\mathfrak{A}$ is isomorphic to $C^*(\mathcal{G})$ of an AF-groupoid $\mathcal{G}$ where $\mathcal{G}$ is uniquely determined up to isomorphism.
		\end{enumerate}
	\end{prop}
	
	Furthermore AF-groupoids are in 1-1 correspondence to inductive limits of compact, \'etale equivalence relations (abbr.: CEER) which are termed \emph{AF-equivalence relations}. In this disguise they are intrinsically studied in \cite{gps04}. The following is the prime example of this correspondence:
	
	\begin{ex}\label{ex: AFgrpd}
		Let $\Gamma=(V,E)$ be a Bratteli diagram\footnote{See Definition~\ref{defi: Bratteli diagram}.\\} and let $\mathcal{P}_\Gamma$ be its space of infinite paths. Define for all $k \in \mathbb{N}$ equivalence relations $\sim_k$ by $(e_n)_{n \in \mathbb{N}}\sim_k (f_n)_{n \in \mathbb{N}}$ \emph{if and only if} $e_n=f_n$ for all $n \geq k$. The sequence of associated groupoids $\{\mathcal{G}_{\sim_k}\}_{k \in }$ is a nested sequence of \'etale groupoids. The inductive limit of this sequence is per definition an AF groupoid. By Theorem~3.9 of \cite{gps04} for every AF-equivalence relation and hence every AF-groupoid there exists a Bratteli diagram upon which it can be modelled.
	\end{ex}
	
	AF-groupoids generalize Krieger's AF-systems (see Subsection~\ref{subs: amplegroups}), satisfy the \hyperref[conj: hk]{HK-conjecture} and fall within Krieger's classification via dimension groups:
	
	\begin{thm}[\cite{mat12}, Theorem~4.10 \& 4.11]
		Let $\mathcal{G}$ be an AF-groupoid. Then the following hold:
		\begin{enumerate}[(i)]
			\item $(H_0(\mathcal{G}),H_0(\mathcal{G})^+, [1_{\mathcal{G}^{(0)}}] )$ is an ordered group with order unit and there exists an isomorphism of ordered groups with order unit
			\begin{equation*}
			(H_0(\mathcal{G}),H_0(\mathcal{G})^+,[1_{\mathcal{G}^{(0)}}] )\cong (K_0(C_r^*(\mathcal{G})), K_0(C_r^*(\mathcal{G}))^+,[1_{C_r^*(\mathcal{G})}])
			\end{equation*}
			
			\item The ordered groups with order unit $(H_0(\mathcal{G}), H_0(\mathcal{G})^+, [1_{\mathcal{G}^{(0)}}] )$ are a complete isomorphism invariant for AF-groupoids.
			
			\item For any topological abelian group $A$, the groups $H_n(\mathcal{G},A)$ are trivial for $n \geq 1$.
		\end{enumerate}
	\end{thm}
	
	In \cite{mat12} Matui introduced the class of almost finite groupoids: 
	
	\begin{defi}[\cite{mat12}, Definition 6.2]
		Let $\mathcal{G}$ be an effective, \'etale Cantor groupoid. It is said to be \emph{almost finite}, if for every compact subset $K \subseteq \mathcal{G}$ and every $\varepsilon >0$ there exists an elementary subgroupoid $\mathcal{H}\leq \mathcal{G}$ such that for all $u \in \mathcal{G}^{(0)}$:
		\begin{equation*}
		\frac{|K\mathcal{H}u\setminus \mathcal{H}u|}{|\mathcal{H}u|}<\epsilon
		\end{equation*}
	\end{defi}

	Almost finite groupoids encompass AF-groupoids and transformation groupoids associated with free $\mathbb{Z}^n$-actions on a Cantor space by Lemma~6.3 in \cite{mat12}.
	
	\begin{prop}[\cite{mat12}, Lemma~6.3]\label{prop: alm.fin.}
		Let $n \in \mathbb{N}$ and let $\varphi$ be a free continuous action of $\mathbb{Z}^n$ on a Cantor space. Then the following hold:
		\begin{enumerate}[(i)]
			\item The transformation groupoid $\mathcal{G}_{\varphi}$ is almost finite.
			
			\item The transformation groupoid $\mathcal{G}_{\varphi}$ satisfies the \hyperref[conj: hk]{HK-conjecture}.
		\end{enumerate}
	\end{prop}
	
	\subsection{Purely infinite groupoids and groupoids of shifts of finite type}\label{subs: Purely infinite groupoids and groupoids of shifts of finite type}
	
	Matui defined purely infinite groupoids in \cite{mat15} based on \cite{rs12} where condition on dynamical systems are given such that the arising crossed product C*-algebra is purely infinite. For a treatment of purely infinite C*-algebras see \cite{bla06}, V.2.2. 
	
	\begin{defi}[\cite{mat15}, Definition 4.9.]
		Let $\mathcal{G}$ be an \'etale Cantor groupoid.
		\begin{enumerate}[(i)]
			\item A clopen subset $U \subseteq \mathcal{G}^{(0)}$ is called \emph{properly infinite}, if there exist compact open slices $S_1,S_2$ satisfying $s(S_1)=s(S_2)=U$, $r(S_1) \cup r(S_2) \subseteq U$ and $r(S_1) \cap r(S_2) =\emptyset$.
			
			\item The groupoid $\mathcal{G}$ is called \emph{purely infinite}, if every clopen subset of $\mathcal{G}^{(0)}$ is properly infinite.
		\end{enumerate}
	\end{defi}
	
	\begin{rem}
		If the unit space $\mathcal{G}^{(0)}$ of a Hausdorff, \'etale Cantor groupoid $\mathcal{G}$ is properly infinite, there exists no $\mathcal{G}$-invariant probability measure on $\mathcal{G}^{(0)}$.
	\end{rem}
	
	The class of purely infinite groupoids contains in particular groupoids that arise from one-sided shifts of finite type.\footnote{For details on such systems see \cite{lm95}, Capter~2.\\} Such groupoids are of significance for operator theory in that they are groupoids upon which Cuntz-Krieger algebras\footnote{See Definition~\ref{defi: cuntz-krieger algebras}.\\} can be modeled. It is important to mention the name of Kengo Matsumoto who has been studying orbit equivalence and full groups in this context -- see \cite{mats10}, \cite{mats13}.
	
	\begin{defi}[\cite{mat15}, §6.1]\label{defi: shifts of finite type}
		Let $n \in \mathbb{N}$ and let $\{A_{i,j}\}_{i,j \in \{1,\dots,n\} } \in \mathbf{M}_{n\times n}(\mathbb{Z}_{\geq 0})$.
		\begin{enumerate}[(i)]
			\item The matrix $A$ is said to \emph{satisfy (Inp)} if it is irreducible i.e. for all $i,j \in \{1,\dots,n\}$ there exists an $n \in \mathbb{N}$ such that $(A^n)_{i,j}>0$, and $A$ is not a permutation matrix.
		\end{enumerate}
		
		Let $A=\{A_{i,j}\}_{i,j \in \{1,\dots,n\} } \in \mathbf{M}_{n\times n}(\mathbb{Z}_{\geq 0})$ be such that it satisfies (Inp).
		
		\begin{enumerate}[(i),resume]
			\item Let $\Gamma_A$ be a finite, directed graph with a finite vertex set $V$ and a finite set of edges $E$ such that $A$ is the adjacency matrix of $\Gamma$.\footnote{This means the vertex set is associated with $\{1,\dots,n\}$ and for every pair of vertices $v,w \in V$ define $A_{v,w}:= \# \{ e \in E| s(e)=v,r(e)=w \}$. Note that property (Inp) implies that the arising directed graph is strongly connected and is not a cycle.\\} The space of one-ended infinite directed paths $X_A:=\{(e_k)_{k \in \mathbb{N}}\in E^{\mathbb{N}}|r(e_k)=s(e_{k+1}) \text{ for all }k\in \mathbb{N} \}$ is a Cantor space. For every finite word $\omega=(\omega_1,\dots,\omega_l) \in E^*$ the cylinder set
			\begin{equation*}
				C_\omega:=\{(e_k)_{k \in \mathbb{N}}\in X_A| e_i=\omega_i \text{ for all } i \in \{ 1,\dots,l \} \}
			\end{equation*}
			is a clopen subset of $X_A$.
			
			\item Define a continuous surjection $\sigma_A \colon X_A \to X_A$ by $\sigma_A (x)_i=e_{i+1}$ for every $i \in \mathbb{N}$ and $x=(e_k)_{k \in \mathbb{N}} \in X_A$. The pair $(X_A,\sigma_A)$ is called an \emph{irreducible one-sided topological Markov shift} or \emph{irreducible one-sided shift of finite type}.
			
			\item If $A=n \in \mathbf{M}_{1\times 1}(\mathbb{Z}_{\geq 0})$, the arising one-sided shift of finite type is called \emph{full}. Since it only depends on $n$ it is denoted by $(X_n,\sigma_n)$.
			
			\item  Let $(X_A,\sigma_A)$ be a one-sided shift of finite type. Denote by $\mathcal{G}_A$ the topological groupoid given on the set
			\begin{equation*}
			\{(x,k,y) \in X_A\times \mathbb{Z} \times X_A|\exists m,n \in \mathbb{Z}_{\geq 0}: \sigma_A^n(x)=\sigma_A^m(y),k=n-m \}
			\end{equation*}
			by defining pairs $((x_1,k_1,y_1),(x_2,k_2,y_2))\in \mathcal{G}_A$ as compatible when $y_1=x_2$, the product by $(x_1,k_1,y_1)\cdot(x_2,k_2,y_2) =(x_1,k_1+k_2,y_2)$, the inverse by $(x,k,y)^{-1}=(y,-k,x)$ for all $(x,k,y) \in 	\mathcal{G}_A$ and by declaring sets of the form $\{(x,k-l,y)|x\in U,y \in V, \sigma_A^k(x)=\sigma_A^l(y)  \}$ where $U,V$ are open subsets of $X_A$ to be a topological basis. We call a groupoid a \emph{SFT-groupoid} if it is $\mathcal{G}_A$ for some one-sided shift of finite type.
		\end{enumerate}
	\end{defi}
	
	Let $\mathcal{G}_A$ be SFT-groupoid. The space of units $\mathcal{G}_A^{(0)}$ given by the elements $(x,0,x) \in \mathcal{G}_A$ canonically corresponds to $X_A$ making the groupoid $\mathcal{G}_A$ a Cantor groupoid.
	For every $x \in X_A$ the orbit $\mathcal{G}_A(x)$ is given by the set $\{y \in X_A|\exists l,m \in\mathbb{Z}_{\geq 0}:\sigma_A^n(x)=\sigma_A^m(y)  \}$. Irreducibility of $A$ then assures that $\mathcal{G}_A$ is minimal.
	For every pair of finite directed paths $\omega=(\omega_1,\dots, \omega_l),\upsilon=(\upsilon_1,\dots,\upsilon_m) \in E^*$ which satisfy $r(\omega_l)=r(\upsilon_m)$ the set
	\begin{equation*}
		B_{\omega,\upsilon}:= \{(x,l-m,y)\in\mathcal{G}_A|\sigma_A^{l}(x)=\sigma_A^{m}(y), x \in C_\omega ,y \in C_\upsilon  \}
	\end{equation*}
	is a compact open slice. Slices of this form generate the topology of $\mathcal{G}_A$ by which it is \'etale.
	Since $A$ satisfies (Inp), the set of non eventually periodic paths in $X_A$ is dense and the groupoid $\mathcal{G}_A$ is essentially principal.
	
	\begin{lem}[\cite{mat15}, Lemma~6.1]\label{lem: sft-groupoids are purely infinite, minimal etc}
		SFT-groupoids are purely infinite, Hausdorff, minimal, effective, \'etale Cantor groupoids.
	\end{lem}
	
	\begin{defi}
		Let $n \in \mathbb{N}$, let $\{A_{i,j}\}_{i,j \in \{1,\dots,n\} } \in \mathbf{M}_{n\times n}(\mathbb{Z}_{\geq 0})$ be such that it satisfies (Inp) and let $\Gamma=(V,E)$ be the associated directed, finite graph. Define $\tilde{A}$ to be the edge matrix of this graph i.e. the matrix $\tilde{A} \in \mathbf{M}_{|E|\times|E|}(\mathbb{Z}/2\mathbb{Z})$ given by \begin{equation*}
		\tilde{A}_{e,f}:=\begin{cases}
		1, & \text{ if } r(e)=s(f)\\
		0, & \text{ else}
		\end{cases}
		\end{equation*}
	\end{defi}
	
	\begin{rem}
		If a matrix $A$ satisfies (Inp), then the matrix $\tilde{A}$ is a matrix satisfying the assumptions in Definition~\ref{defi: cuntz-krieger algebras}.
	\end{rem}

	\begin{thm}[\cite{ren00}, Proposition~4.8]\label{thm: cuntz-krieger algebras are groupoid c*-algebras of one sided shifts of finite type}
		Let $(X_A,\sigma_A)$ be a one-sided shift of finite type. Then $\mathcal{O}_{\tilde{A}} \cong C_r^*(\mathcal{G}_A)$.
	\end{thm}
	
	\begin{ex}
		Let $A=n \in \mathbf{M}_{1\times 1}(\mathbb{Z}_{\geq 2})$. Then the associated graph $\Gamma_A$ is just a ``directed bouquet graph"-- it consists of one vertex with $n$ attached directed loops. Its edge matrix $\tilde{A}$ is a $n \times n$-matrix with all entries $1$. Correspondingly the Cuntz algebra of order $n$ is the groupoid C*-algebra of an SFT-groupoid associated to a full one-sided shift of finite type.
	\end{ex}
	
	All of this considerations embed into the broader context of graph groupoids, graph inverse semigroups and graph C*-algebras.

	\chapter{Topological full groups}\label{chap: 3}
	
	\section{Basic structure and significance in Cantor dynamics}\label{sec: basic structure and significance in cantor dynamics}
	
	We start in Subsection~\ref{subs: definition} with the definition of full groups and topological full groups and some immediate observations. Subsection~\ref{subs: topological full groups in terms of kakutani-rokhlin partitions} contains description of topological full groups of minimal Cantor systems in terms of Kakutani-Rokhlin partitions. Subsection~\ref{subs: sight} traces back the first glimpses on topological full groups in work of Krieger and Putnam and introduces important subgroups.  Subsection~\ref{subs: gla.wei.} cites a lemma by Glasner and Weiss. Subsection~\ref{subs: isom} concludes this section with a look on full groups as isomorphism invariants in topological dynamics. 
	
	\subsection{Definition}\phantomsection\label{subs: definition}

	\begin{defi}\label{defi: tfg1}
		Let $(X, G)$ be a topological dynamical system such that $G$ is countable.
		\begin{enumerate}[(i)]
			\item The \emph{full group of $(X,G)$} denoted by $\mathfrak{F}(G)$ is the subgroup of homeomorphisms $\gamma \in \operatorname{Homeo}(X)$ for which there exists a function $f_\gamma \colon X \to G$ such that $\gamma (x)=f_\gamma(x)\cdot x$. Such a function $f_\gamma$ is called an \emph{orbit cocycle of $\gamma$}.
			
			\item The \emph{topological full group of $(X,G)$} denoted by $\mathfrak{T}(G)$ is the subgroup of homeomorphisms $\gamma \in \operatorname{Homeo}(X)$ for which there exists a continuous orbit cocycle.\footnote{In most references e.g. in \cite{gps99} and all of Matui's works, the full group (resp. topological full group) of a topological dynamical system $(X, G)$ is denoted by $[G]$ (resp. $[[G]]$) or by $[\varphi]$ (resp. $[[\varphi]]$) in the context of $\mathbb{Z}$-systems $(X,\varphi)$. We choose a notation close to the one for topological full groups of \'etale groupoids used in \cite{nek15} and \cite{nek17} to avoid confusion with commutator brackets and to subsequently have a natural notation for important subgroups. The naming ``topological full group" is a little awkward, in that the word ``topological" serves just as a reference to the field of topological dynamics and signifies the difference between $\mathfrak{F}(G)$ and $\mathfrak{T}(G)$, but it does not refer to a topology on $\mathfrak{T}(G)$.\\}
			In the case of a $\mathbb{Z}$-system $(X,\varphi)$ we write $\mathfrak{F}(\varphi)$ (resp. $\mathfrak{T}(\varphi)$) instead.
		\end{enumerate}
	\end{defi}

	We begin with some immediate observations:
	\begin{rem}
		\begin{enumerate}[(i)]
			\item The groups $G$, $\mathfrak{T}(G)$ and $\mathfrak{F}(G)$ have the same orbits.
			
			\item In general orbit cocycles and continuous orbit cocycles are not unique. Note that if the action is free, the orbit cocycle $f_\gamma$ of every element $\gamma \in \mathfrak{F}(G)$ is uniquely determined.
			
			\item The groups $\mathfrak{F}(G)$ and $\mathfrak{T}(G)$ are not necessarily countable \emph{and} interesting to consider. Let $\gamma \in \mathfrak{T}(G)$. Since $G$ was assumed countable, the fibers of $f_\gamma$ induce a countable partition $X= \bigsqcup_{g \in G} X_g$ by closed sets, thus if $G$ acts freely on a connected space $X$ the topological full group is just isomorphic to $G$.
			
			\item In the following we restrict ourselves to the case where the space $X$ is a Cantor space and thus totally disconnected. Note that orbit cocycles in general are not uniquely determined, thus for $\mathfrak{T}(G)$ to be graspable, we require the ``freeness condition" of minimality on the underlying system.
		\end{enumerate}
	\end{rem}
	
	The full group $\mathfrak{F}(\varphi)$ of a minimal Cantor system $(X,\varphi)$ is in general uncountable, whereas we have for topological full groups:
	
	\begin{prop}\label{prop: count}
		Let $(X, \varphi)$ be a minimal Cantor system.
		\begin{enumerate}[(i)]
			\item For every $\gamma \in \mathfrak{F}(\varphi)$ the orbit cocycle $f_\gamma$ is uniquely determined.
			
			\item A homeomorphism $\gamma \in \operatorname{Homeo(X)}$ satisfies $\gamma \in \mathfrak{T}(\varphi)$ \emph{if and only if} there exists a finite clopen partition $X=\bigsqcup_{i \in I \subset \mathbb{Z}} X_i^\gamma$ such that $\gamma|_{X_i^\gamma}=\varphi^i|_{X_i^\gamma}$. 
		\end{enumerate}
	In particular $\mathfrak{T}(\varphi)$ is countable and $\supp(\gamma)=\overline{\{x \in X|\gamma (x) \neq x \}}$ is clopen for all $\gamma \in \mathfrak{T}(\varphi)$.
	\end{prop}
	\begin{proof}
		\begin{enumerate}[(i)]
			\item Assume that both $f_\gamma,g_\gamma$ are orbit cocycles of $\gamma$ such that $f_\gamma(x) \neq g_\gamma(x)$ for some $x \in X$. This is only possible if $x$ is a periodic point.
			
			\item  Let $\gamma$ in $\mathfrak{T}(\varphi)$. Since $f_\gamma$ is continuous, its image must be compact, hence finite. Since $\mathbb{Z}$ is discrete, fibers of the orbit cocycle $f_\gamma$ are clopen and thus $X_i^\gamma = f_\gamma^{-1}(\{i\})$ for $i \in \Ima(f_\gamma) \subset \mathbb{Z}$ is the required partition. Conversely, the function $f_\gamma \colon X \to \mathbb{Z}$ defined by $f_\gamma(x)= i$ for $x \in X_i^\gamma$ is continuous, since every $X_i^\gamma$ is clopen. Since $X$ is zero-dimensional and has a countable basis, there are only countably many clopen sets in $X$ and consequently $\mathfrak{T}(G)$ must be countable.
		\end{enumerate}
	\end{proof}
	\begin{rem}\label{lem: cocy}
		For every $\gamma \in \mathfrak{T}(\varphi)$ the orbit cocycle $f_\gamma$ can be represented as $f_\gamma = \sum_{k \in \mathbb{Z}} k \mathbf{1}_{X_k^\gamma}$ and for $\gamma_1, \gamma_2 \in \mathfrak{T}(\varphi)$ the associated orbit cocycles satisfy true to their names $f_{\gamma_1 \gamma_2}=f_{\gamma_1} \circ \gamma_2 + f_{\gamma_2}$.
	\end{rem}
	
	Already in \cite{mat06} Matui defined topological full groups of \emph{\'etale equivalence relations} on a Cantor space -- the associated equivalence groupoid is a principal, \'etale Cantor groupoid. In \cite{mat12} he generalized the definition to \'etale groupoids with compact unit space as
	\begin{equation*}
	\mathfrak{T}(\mathcal{G}):=\{\gamma \in \operatorname{Homeo}(\mathcal{G}^{(0)})|\exists B \in \mathcal{B}_\mathcal{G}^{o,k}: \gamma = r\circ (s|_B)^{-1} \}
	\end{equation*}
	
	Analogous to the above we mostly restrict to the case of Cantor groupoids:
	
	\begin{defi}[\cite{nek17}, Definition 2.3]\label{defi: tfg2}
		Let $\mathcal{G}$ be an \'etale Cantor groupoid. The set $\{B \in \mathcal{B}_\mathcal{G}^{o,k}: s(B)=r(B)=\mathcal{G}^{(0)} \}$ forms a group with respect to multiplication and inversion of subsets, called the \emph{topological full group of $\mathcal{G}$}, denoted by $\mathfrak{T}(\mathcal{G})$.
	\end{defi}
	
	\begin{rem}\phantomsection\label{rem: topfgroupsdefi}
		\begin{enumerate}[(i)]
			\item Put differently, the topological full group $\mathfrak{T}(\mathcal{G})$ is by definition  the unit group $\operatorname{U}(\mathcal{B}_\mathcal{G}^{o,k})$ of the inverse monoid of compact open slices. 
			
			\item If $\mathcal{G}$ is a second countable, \'etale Cantor groupoid, then $\mathfrak{T}(\mathcal{G})$ is countable per definition.
			
			\item If $\mathcal{G}$ is effective, $\mathfrak{T}(\mathcal{G})$ is isomorphic to the group of Matui's definition i.e. a group of homeomorphisms of $\mathcal{G}^{(0)}$ via the injection $B \mapsto r\circ (s|_B)^{-1}$.
			
			\bigskip
			
			 \emph{\large We indulge ourselves with leaving it to the context in the effective case if we refer to $\mathfrak{T}(\mathcal{G})$ as a group of slices or of homeomorphisms! As an allevation we will stand by the convention of writing upper case letters for slices and lower case letters for homeomorphisms.}
			 
			\bigskip
			
			\item In line with Proposition~\ref{prop: open.slices} we denote the homeomorphism of $\mathcal{G}^{(0)}$ induced by a slice $B \in \mathfrak{T}(\mathcal{G})$ by $\alpha_B$.
			
			\item In the case of a continuous action of a group $G$ on a Cantor space $X$ via $\alpha
			\colon G \to \operatorname{Homeo}(X)$, the groups $\mathfrak{T}(G)$ and $\mathfrak{T}(\mathcal{G}_\alpha)$ coincide. 
		\end{enumerate}
	\end{rem}
	
	As in the reduced context of Cantor systems minimality is key in a lot of results. It allows for the construction of certain elementary elements in the topological full group:
	
	\begin{lem}\label{lem: existence of pencils}
		Let $\mathcal{G}$ be a minimal, \'etale Cantor groupoid and let $U,V \subseteq \mathcal{G}^{(0)}$ be non-empty and clopen. Then there exists a finite family of open, compact slices $\{B_i \}$ with $r(B_i) \subset V$, $U=\bigcup_i s(B_i)$ and $s(B_i)\cap s(B_j)=\emptyset$ for $i\neq j$.
	\end{lem}
	\begin{proof}
		The $\mathcal{G}$-orbit of every element in $U$ is dense in $\mathcal{G}^{(0)}$ by minimality, thus for every $u \in U$ there exist an $g_u \in \mathcal{G}$ with $s(g_u)=u$ und $r(g_u) \in V$. Since $\mathcal{G}$ is \'etale, there exists an open compact slice $B_u$ containing $g_u$ such that $r(B_u) \subseteq V$. The family $\{s(B_u)\}_{u \in U}$ covers $U$. By compactness, there exists a finite family $\{B_i\}$ with $U=\bigcup_i s(B_i)$ and $r(B_i) \subset V$, which we can assume to be disjoint, thus $s(B_i)\cap s(B_j)=\emptyset$ for $i\neq j$.
	\end{proof}
	
	\begin{defi}
		Let $\mathcal{G}$ be an \'etale Cantor groupoid and let $B \in \mathcal{B}_\mathcal{G}^{o,k}$ be non-empty such that $s(B) \cap r(B) = \emptyset$. Denote by $T_B$ (resp. $\tau_B$) the induced involution in $\mathfrak{T}(\mathcal{G})$ given by $B \cup B^{-1} \cup (\mathcal{G}^{(0)} \setminus (s(B)\cup r(B)))$.
	\end{defi}
	
	\begin{cor}\label{cor: existence of involutions}
		Let $\mathcal{G}$ be a minimal, \'etale Cantor groupoid and let $U \subseteq \mathcal{G}^{(0)}$ be non-empty and clopen. Then there exists an element $T \in \mathfrak{T}(\mathcal{G})$ with $T^2=1$ such that $\mathcal{G}^{(0)} \setminus (T \cap \mathcal{G}^{(0)}) \subseteq U$.
	\end{cor}
	\begin{proof}
		Let $x,y \in U$ with $x \neq y$. Since $\mathcal{G}^{(0)}$ is a Cantor space, there exist clopen neighbourhoods $V,W$ separating them. Then for every compact open bisection $B_i$ as chosen in Lemma~\ref{lem: existence of pencils} for the non-empty clopen sets $V \cap U$ and $W \cap U$ the element $T_{B_i}$ is sufficient.
	\end{proof}
	
	\subsection{Topological full groups in terms of Kakutani-Rokhlin partitions}\label{subs: topological full groups in terms of kakutani-rokhlin partitions}
	
	Minimal Cantor systems can be approximated by Kakutani-Rokhlin partitions\footnote{See Subsection~\ref{subs: kak.ro}.\\} and those allow for a convenient representation of $\mathfrak{T}(\varphi)$ as originally demonstrated in \cite{bk00}:	
	
	Let $(X,\varphi)$ be a minimal Cantor system with a fixed Kakutani-Rokhlin partition $\mathcal{A}=\{D_{k,i}|0 \leq k \leq h_i-1, i \in \{1,\dots,n\} \}$. Then $\mathcal{A}$ induces partitions $\alpha(\mathcal{A}),\alpha'(\mathcal{A})$ of $\{1,\dots,n\}$ in the following way: A subset $J \subseteq \{1,\dots,n \}$ is an atom of $\alpha(\mathcal{A})$ \emph{if and only if} there exists a subset $J' \subseteq \{1,\dots,n \}$ such that $\varphi(\bigsqcup_{j \in J} D_{h_j-1,j})=\bigsqcup_{j' \in J'} D_{0,j'}$ and for any proper subset $J_0 \subset J$ the set $\varphi(\bigsqcup_{j \in J_0} D_{h_j-1,j})$ is not a union of tower bases in $\mathcal{A}$. All such sets $J$ resp. $J'$ give rise to the partitions $\alpha(\mathcal{A})$ resp. $\alpha'(\mathcal{A})$ of $\{1,\dots,n\}$ and there is a 1-1 correspondence between their atoms. An atom $J \in \alpha(J)$ corresponds to a collection of towers of which the union of roofs is mapped on the union of floors of a collection of towers corresponding to $J' \in \alpha'(\mathcal{A})$ via $\varphi$. Denote by $h_J:=\min \{h_j|j \in J \}$ resp. $h_{J'}:=\min \{h_{j'}|j' \in J' \}$ the minimal height of a tower associated with an atom $J \in \alpha(\mathcal{A})$ resp. $J' \in \alpha'(\mathcal{A})$. For any $i \in \{1,\dots,n\}$ define $J(i) \in \alpha(\mathcal{A})$ resp. $J'(i) \in \alpha'(\mathcal{A})$ to be the atoms containing $i$. Let $\gamma \in \mathfrak{T}(\varphi)$. There exists a Kakutani-Rokhlin partition $\mathcal{A}=\{D_{k,i}|0 \leq k \leq h_i-1, i \in \{1,\dots,n\} \}$ such that $\mathcal{A}$ refines the finite partitions $\{X_i^\gamma \}_{i \in I}$ and $\{\varphi^{f_i}(X_i^\gamma) \}_{i \in I}$ given in Proposition~\ref{prop: count} and such that $|f_i| \leq h_{\mathcal{A}}$ for all $i \in I$, this means $f_\gamma$ is constant on atoms of $\mathcal{A}$ and $\|f_\gamma\|_{\infty}\leq h_{\mathcal{A}}$. For such a Kakutani-Rokhlin partition the pairs $U(\mathcal{A})$ admit the following set:
	\begin{enumerate}[(i)]
		\item $U_{\text{in}} \subseteq U(\mathcal{A})$ the set of pairs $(k,i)$, such that $\gamma(D_{k,i}) \subset D(i)$.
		
		\item $U_{\text{top}} \subseteq U(\mathcal{A})$ the set of pairs $(k,i)$, such that $f_\gamma|_{D_{k,i}} + k \geq h_i$.
		
		\item $U_{\text{bot}} \subseteq U(\mathcal{A})$ the set of pairs $(k,i)$, such that $f_\gamma|_{D_{k,i}} + k <0$.
	\end{enumerate}
	
	When $\gamma$ is applied, the atoms corresponding to pairs in $U_{\text{in}}$ stay in the same tower, atoms corresponding to pairs in $U_{\text{top}}$ move through the roof of the containing tower and atoms corresponding to pairs in $U_{\text{bot}}$ move through the floor of the containing tower. By the assumption $|f_i| \leq h_{\mathcal{A}}$ for all $i \in I$, atoms corresponding to pairs in $U_{\text{top}}$ must lie at within a distance of $h_\mathcal{A}$ from the top of the containing tower and analogously atoms corresponding to pairs in $U_{\text{bot}}$ must lie at within a distance of $h_\mathcal{A}$ from the floor of the containing tower.
	
	\begin{defi}[\cite{bk00}, Definition 2.1]
		\begin{enumerate}[(i)]
			\item Let $J \in \alpha(\mathcal{A})$ and $J' \in \alpha'(\mathcal{A})$.
			Define for  $r \in \{0,\dots,  h_J-1\}$ resp. $r' \in \{0,\dots,  h_{J'}-1\}$ the sets:
			\begin{equation*}
			F_1(r,J):= \bigsqcup_{j \in J} D_{h_j-h_J+r,j}\quad \text{ and } \quad F_2(r',J'):= \bigsqcup_{j' \in J'} D_{r',j'}
			\end{equation*}
			
			\item A homeomorphism $\gamma \in \mathfrak{T}(\varphi)$ satisfies the \emph{level condition (L)} if the following conditions hold:
			\begin{enumerate}
				\item if $(k,i) \in U_{\text{top}}$ and $D_{k,i} \subseteq X_l^\gamma$, then $F_1(h_{J(i)}-h_i+k,J(i)) \subseteq X_l^\gamma$.
				
				\item if $(k,i) \in U_{\text{bot}}$ and $D_{k,i} \subseteq X_l^\gamma$, then $F_2(k,J'(i)) \subseteq X_l^\gamma$.
			\end{enumerate}
		\end{enumerate}
	\end{defi}
	
	The level condition (L) amounts to the fact that when an atom $D_{k,i}$ is moved through the roof (resp. floor) of the containing tower $D(i)$ by $\gamma$, then $D_{h_j-h_i+k,j}$ is moved through the roof of $D(j)$ for all $j \in J(i)$ (resp. $D_{k,j}$ is moved through the floor of $D(j)$ for all $j \in J'(i)$ by $\gamma$). Note, that if $U_{\text{in}}=U(\mathcal{A})$ holds for a homeomorphism $\gamma \in \mathfrak{T}(\varphi)$ i.e. $\gamma(D(i))=D(i)$ for every tower $D(i)$, then $\gamma$ obviously satisfies the level condition.
	
	\begin{defi}[\cite{bk00}, §2.1 \& \cite{gm14}, Definition 4.1]\label{defi: permut}
		Let $(X,\varphi)$ be a minimal Cantor system and let $\mathcal{A}=\{D_{k,i}|0 \leq k \leq h_i-1, i \in \{1,\dots,n\} \}$ be a Kakutani-Rokhlin partition. Denote by $\mathfrak{T}(\mathcal{A})$ the set of homeomorphisms $\gamma \in \mathfrak{T}(\varphi)$ such that $\mathcal{A}$ refines the finite partitions $\{X_i^\gamma \}_{i \in I}$ and $\{\varphi^{f_i}(X_i^\gamma) \}_{i \in I}$ and $|f_i| \leq h_{\mathcal{A}}$ holds for all $i \in I$ and $\gamma$ satisfies the level condition (L) with respect to $\mathcal{A}$.
	\end{defi}
	
	\begin{rem}
		By definition the set $\mathfrak{T}(\mathcal{A})$ is finite.
	\end{rem}
	
	\begin{defi}
		Denote by $\mathcal{P}(\mathcal{A})$ the subset of homeomorphisms $\gamma \in \mathfrak{T}(\mathcal{A})$ such that $U_{\text{in}}=U(\mathcal{A})$ holds. In case of a nested sequence $\{\mathcal{A}_n\}_{n \in \mathbb{N}}$ of Kakutani-Rokhlin partitions satisfying property (H), we call the elements of $\mathcal{P}(\mathcal{A}_n)$ \emph{$n$-permutations}. Every $n$-permutation $p$ admits a decomposition $p=p_1\dots p_{m}$ such that the support of every component $p_i$ lies within a single tower of the partition $\mathcal{A}_n$.
	\end{defi}
	
	\begin{rem}
		The set $\mathcal{P}(\mathcal{A})$ is a finite subgroup of $\mathfrak{T}(\varphi)$ and  isomorphic to $\bigoplus_{i \in \{1,\dots,n\}} \mathfrak{S}_{h_i}$, where $\mathfrak{S}_{h_i}$ denotes the symmetric group of degree $h_i$. In the case of a nested sequence $\{\mathcal{A}_n\}_{n \in \mathbb{N}}$ of Kakutani-Rokhlin partitions satisfying property (H), it holds that $\mathcal{P}(\mathcal{A}_n) \leq \mathcal{P}(\mathcal{A}_{n+1})$ and $\bigcup_{n \in \mathbb{N}} \mathcal{P}(\mathcal{A}_n)$ is a proper subgroup of $\mathfrak{T}(\varphi)$. Every finite group embedds into $\bigcup_{n \in \mathbb{N}} \mathcal{P}(\mathcal{A}_n)$ and thus into $\mathfrak{T}(\varphi)$. Note that $\mathcal{P}(\mathcal{A})'\cong \bigoplus_{i \in \{1,\dots,n\}} \mathfrak{A}_{h_i}$, where $\mathfrak{A}_{h_i}$ denotes the alternating group of degree $h_i$.
	\end{rem}
	
	It holds that given a sufficient nested sequence $\{\mathcal{A}_n\}_{n \in \mathbb{N}}$ of Kakutani-Rokhlin partitions every element of the topological full group can be described as a homeomorphism that permutes atoms of $\mathcal{A}_n$ for some $n \in \mathbb{N}$: 
	
	\begin{thm}[\cite{bk00}, Theorem 2.2]\label{thm: kakro}
		Let $(X,\varphi)$ be a minimal Cantor system and let $\{\mathcal{A}_n\}_{n \in \mathbb{N}}$ be a nested sequence of Kakutani-Rokhlin partitions satisfying property (H). Then $\mathfrak{T}(\varphi)=\bigcup_{n \in \mathbb{N}} \mathfrak{T}(\mathcal{A}_n)$.
	\end{thm}
	
	In \cite{med07}, Theorem~\ref{thm: kakro} was generalized to aperiodic Cantor systems. Besides proving the above theorems, \cite{bk00} reflects them in the context of different examples of Cantor systems and is devoted to study topological properties of $\mathfrak{F}(\varphi)$ and $\mathfrak{T}(\varphi)$ -- specifically, a topology on $\operatorname{Homeo(X)}$ is introduced with respect to which $\mathfrak{T}(\varphi)$ is dense in $\mathfrak{F}(\varphi)$.
	
	In the following we review the stronger statement from \cite{gm14} in more detail. It elucidates the stucture of topological full groups of minimal Cantor systems in that it shows that their structure is close to a union of permutational wreath products of $\mathbb{Z}$. Note that $\mathfrak{T}(\varphi)$ always admits embeddings of such permutational wreath products:
	
	\begin{prop}\label{prop: permutational wreath products embedd into topological full groups}
		Let $(X,\varphi)$ be a minimal Cantor system. Then for every $n \in \mathbb{N}$ the group $\mathbb{Z} \wr_{n} \mathfrak{S}_{n} := \mathbb{Z}^{n} \rtimes \mathfrak{S}_{n}$ embeds into $\mathfrak{T}(\varphi)$.
	\end{prop}
	\begin{proof}
		Let $n \in \mathbb{N}$. By assumption there exists a tower of height $n$ i.e. a clopen subset $U\subset X$ such that $U,\varphi(U),\dots,\varphi^{n-1}(U)$ are pairwise disjoint. Then the group generated by the induced transformation $\varphi_U$ and the group $\mathfrak{S}_{n}$ of permutations of this tower is isomorphic to $\mathbb{Z} \wr_{n} \mathfrak{S}_{n}$.
	\end{proof}
	
	The stronger result is obtained by strengthening the assumptions on the approximating nested sequence $\{\mathcal{A}_n\}_{n \in \mathbb{N}}$ of Kakutani-Rokhlin partitions:
	
		Let $\{m_n\}_{n \in \mathbb{N}}$ be a fixed sequence of positive integers such that $m_n \to \infty$ for $n \to \infty$. By restriction to a subsequence, we can assume that
		\begin{enumerate}[(i)]
			\item $h_{\mathcal{A}_n} \geq 2m_n+2$
		\end{enumerate}
		
		Any homeomorphism of a compact metric space is uniformly continuous by the \emph{Heine-Cantor theorem}\footnote{Every continuous map from a compact metric space to a metric space is uniformly continuous.\\}, which implies that in the construction of $\{\mathcal{A}_n\}_{n \in \mathbb{N}}$ the base $B(\mathcal{A}_n)$ can be made sufficiently small such that we can assume:
		\begin{enumerate}[(i),resume]
			\item $\operatorname{diam}(\varphi^k(B(\mathcal{A}_n))) \leq \frac{1}{n}$ for $-m_n-1 \leq i \leq m_n$
		\end{enumerate}
		This condition allows to get rid of some of the more extensive vocabulary leading up to Theorem~\ref{thm: kakro}, because the Lebesgue number lemma then automatically assures property(L) for $\gamma \in \mathfrak{T}(\varphi)$ with respect to $\mathcal{A}_n$ for every $n$ big enough.

	\begin{defi}[\cite{gm14}, Definition~4.5]\label{defi: rotation}
		Let $(X,\varphi)$ be a minimal Cantor system and let $\{\mathcal{A}_n\}_{n \in \mathbb{N}}$ be a nested sequence of Kakutani-Rohklin partitions that satisfies (H) and let $n \in \mathbb{N}$.
		\begin{enumerate}[(i)]
			\item Let $D^n(i)$ be a tower. The set of atoms $\{D_{k,i}^n|0 \leq k \leq h_i^n -1 \}$ contained in $D^n(i)$ can be regarded as a cyclic group of order $h_i^n$ endowed with a metric $d_i^n$ by associating  $D_{k,i}^n$ with $e^{2k\pi i/h_i^n}$ and $d_i^n$ correponds to the restriction of the intrinsic metric on the unit circle normalized by $h_i^n/2\pi$.
			
			\item Let $p_i$ and $p_j$ be $n$-permutations with $\supp(p_i)\subseteq D^n(i)$ and $\supp(p_j)\subseteq D^n(j)$ such that \begin{equation*}
				d_r^n(D_{k,r}^n,p_r(D_{k,r}^n))\leq m_n \text{ for all } 0 \leq k \leq h_r^n\text{ and }r \in \{i,j\} 
			\end{equation*}
			i.e. the action of $p_r$ on an atom in $D^n(r)$ can be uniquely interpreted as a rotation in positive or negative direction. Let $k \in \{-m_n,\dots,0,\dots,m_n \}$. Then the notation $p_i(k)=p_j(k)$ signifies, that in the case of $0 \leq k \leq m_n$, the action of $p_i$ moves the atom $D_{k,i}^n$ the same way as the action $p_j$ moves the atom $D_{k,j}^n$ when interpreted as rotation in positive or negative direction and in the case of $-m_n \leq k \leq 0$, the action of $p_i$ moves the atom $D_{h_i^n-|k|,i}^n$ the same way as the action $p_j$ moves the atom $D_{h_j^n-|k|,j}^n$ when interpreted as rotation in positive or negative direction.
			
			\item Define for $0 \leq k \leq h_{\mathcal{A}_n}$ the sets
			\begin{equation*}
			U(k):=\bigsqcup_i D_{h_i-k-1,i}^n \qquad L(k):=\bigsqcup_i D_{k,i}^n
			\end{equation*}
			\item Let $r \in \mathbb{N}$. A homeomorphism $\gamma \in \mathfrak{T}(\varphi)$ is called an \emph{n-rotation with rotation number at most $r$} if there exist a pair of subsets $S_u,S_l \subseteq \{1,\dots,m_n \}$ called \emph{supportive sets} and families of integers $\{l_i\}_{i \in S_u}$ and $\{k_j\}_{j \in S_l}$ with $|l_i|\leq r$, $|k_j| \leq r$ such that:
			\begin{equation*}
			\gamma = \prod_{i \in S_u} (\varphi_{U(i)})^{l_i} \times \prod_{i \in S_l} (\varphi_{L(j)})^{k_j}
			\end{equation*}						
		\end{enumerate}
	\end{defi}	
	
	\begin{rem}
		\begin{enumerate}[(i)]
			\item The set $U(k)$ consists of all atoms which are $k$ levels beneath the top in their containing tower and the set $L(k)$ consists of atoms which are $k$ levels above the bottom of their containing tower.
			
			\item Note, that this is well-defined by assumption (i) on the nested sequence of Kakutani-Rokhlin partitions made in the beginning of this subsection, as it guarantees for $U(k) \cap L(l) = \emptyset$ for every $0 \leq k\leq h_{\mathcal{A}_n}$ and $0 \leq l\leq h_{\mathcal{A}_n}$.
		\end{enumerate}
	\end{rem}
	
	\begin{lem}[\cite{bm08}, Proposition 2.1.(3)]\label{lem: ind.trans}
		Let $(X,\varphi)$ be a minimal Cantor system and let $A$ be a clopen subset of $X$. Then $\varphi_A^{-1}\varphi$ is periodic i.e. every orbit is finite.
	\end{lem}
	\begin{proof}
		Note that $\varphi_A^{-1}|_{A}=\varphi^{t_{\varphi^{-1},A}(x)}(x)$ and $A=\bigsqcup_{n \in\mathbb{N}}\bigsqcup_{i=0}^{n-1} \varphi^i (A \cap t_{\varphi,A}^{-1}(n))$. It follows immediatly from the definitions that for every $n \in \mathbb{N}$ and every $x \in \bigsqcup_{i=0}^{n-1} \varphi^i (A \cap t_{\varphi,A}^{-1}(n))$ we have $(\varphi_A^{-1}\varphi)^n(x)$.
	\end{proof}
	
	We sketch the more powerful result from \cite{gm14}:
	
	\begin{thm}[\cite{gm14}, Theorem~4.7]\label{thm: grig-med}
		Let $(X,\varphi)$ be a minimal Cantor system and let $\mathcal{A}_n$ be a nested sequence of Kakutani-Rokhlin partitions satisfying property (H). Let $\gamma \in \mathfrak{T}(\varphi)$.
		\begin{enumerate}[(i)]
			\item There exists an $n_0 \in \mathbb{Z}, n_0>0$, such that for every $n \geq n_0$ there exists a decomposition $\gamma = p_\gamma r_\gamma$ such that $r_\gamma$ is an $n$-rotation with rotation number at most $1$ and $P_\gamma$ is an $n$-permutation with decomposition $p_\gamma=p_{\gamma,1} \dots p_{\gamma,i_n}$ (see Definition~\ref{defi: permut}) satisfying the following properties:
			\begin{enumerate}
				\item The supportive sets $S_u,S_l$ associated with the $n$-rotation $r_\gamma$ are contained in $\{0,\dots,n_0-1\}$.
				\item $d_i^n(p_{\gamma,i}(D_{k,i}^n),D_{k,i}^n)\leq n_0$ for $i \in \{1,\dots,i_n\}$ and $0 \leq k \leq h_i^n-1$
				\item $p_{\gamma,i}(k)=p_{\gamma,j}(k)$ for all $k \in [-m_n,m_n]$ and $i,j \in \{1,\dots,i_n\}$.
			\end{enumerate}
			\item If $\gamma=p_\gamma'r_\gamma'$ is another decomposition of this kind, then $p_\gamma'=p_\gamma$ and $r_\gamma'=r_\gamma$.
			\item For every finite subset $F \subseteq \mathfrak{T}(\varphi)$, there exists an $n_0'>0$ such that for every $n \geq n_0'$ and every pair of elements $\gamma_1,\gamma_2 \in F$ with $\gamma_1 \neq \gamma_2$ the decomposition satisfies $p_{\gamma_1} \neq p_{\gamma_2}$.
		\end{enumerate}
	\end{thm}
	\begin{proofsketch}
		\begin{enumerate}[(i)]
			\item As mentioned in the beginning of this subsection, we can find an $n_0$ big enough such that the cocylce $f_\gamma$ is constant on the atoms of $\mathcal{A}_{n_0}$ and such that $\|f_\gamma\|_{\infty} \leq n_0$. Furthermore, choose $n_0$ big enough such that $n_0^{-1}$ is a Lebesgue number for the partition $\{X_i^\gamma \}_{i \in \mathbb{Z}}$ induced by $\gamma$. Then assumption (ii) on the nested sequence of Kakutani-Rokhlin partitions in the beginning of this subsection assures that $f_\gamma$ is constant on $\varphi^k(B(\mathcal{A}_n))$ for $-m_n-1 \leq i \leq m_n$ for all $n \geq n_0$. Let $n \geq n_0$. Define
			\begin{equation*}
			R_\gamma(x):=
			\begin{cases}
			\varphi_{U(h_i^n-1-k)}(x), & \text{ if } x \in D_{k,i}^n \text{ and } f_\gamma(x)+k \geq h_i^n\\
			\varphi_{L(k)}^{-1}(x), & \text{ if } x \in D_{k,i}^n \text{ and } f_\gamma(x)+k < 0 \\
			x, & \text{ else}
			\end{cases}
			\end{equation*}
			This is a well defined $n$-rotation with rotation number at most $1$ satisfying property (a).
			Since $f_\gamma$ is constant on $\varphi^k(B(\mathcal{A}_n))$ for $-m_n-1 \leq i \leq m_n$ for $n \geq n_0$, we can produce an $n$-permutation $p_\gamma$ from $\gamma$, by tweaking $\gamma$ in such a way that if an atom $D_{k,i}^n$ is moved over the top of its containing tower (resp. over the bottom of its containing tower) to some atom $D_{l,j}^n$, it is moved to $D_{l,i}^n$ (resp. $D_{h_i^n-h_j^n+l,i}^n$) instead. This is well defined and the obtained permutation $p_\gamma=p_{\gamma,1} \dots p_{\gamma,i_n}$ satisfies the required properties (b) and (c).
			The proof of $\gamma=p_\gamma r_\gamma$ is straightforward, one just has to keep track of atoms under the respective homeomorphisms.
			\item Let $\gamma=p'_\gamma r'_\gamma$ be another decomposition of the described kind. Then $p_\gamma^{-1}p_\gamma'=r_\gamma r_\gamma'^{-1}$. The left-hand side of this equation is an $n$-permutation, the right-hand side an $n$-rotation. By definition, both sides must be the identity and hence $p_\gamma=p_\gamma'$ and $r_\gamma=r_\gamma'$.
			\item Let $\gamma_1,\gamma_2 \in \mathfrak{T}(\varphi)$ with $\gamma_1 \neq \gamma_2$ with corresponding decompositions $\gamma_1=p_{\gamma_1}r_{\gamma_1}$ and $\gamma_2=p_{\gamma_2}r_{\gamma_2}$. There exists a clopen set $C \subset X$ such that $\gamma_1(x)\neq\gamma_2(x)$ for all $x$ in $C$. Property (i) of the deomposition implies that for every $n$ large enough the supports of $r_{\gamma_1}$ and $r_{\gamma_2}$ are contained in $\bigsqcup_{j \in\{1,\dots ,k\}} U(j) \sqcup L(j)$ for some fixed $k \in \mathbb{N}$. By definition of $\mathcal{A}_n$ for every $n$ large enough there exists a subset $C' \subset C$ with $C'\cap \supp(r_{\gamma_1})=\emptyset=C'\cap \supp(r_{\gamma_2})$ and consequently, $p_{\gamma_1}|_{C'}\neq p_{\gamma_2}|_{C'}$.
		\end{enumerate}
	\end{proofsketch}

	For odometer systems above situation is the most lucid in that their topological full groups in fact are unions of permutational wreath products:
	\begin{prop}[\cite{mat13}, Proof of Proposition~2.1]\label{prop: topfull groups of odometers are unions of permutational wreath products}
		Let $(X,\varphi)$ be an odometer system of type $\bm{\mathrm{a}}=(a_n)_{n \in \mathbb{N}}$. Then the topological full group is of the form
		\begin{equation*}
		\mathfrak{T}(\varphi) \cong \bigcup_{n \in \mathbb{N}} \mathbb{Z} \wr_{a_n} \mathfrak{S}_{a_n} \big(:\cong \bigcup_{n \in \mathbb{N}} \mathbb{Z}^{a_n} \rtimes \mathfrak{S}_{a_n}\big)
		\end{equation*}
	\end{prop}
	\begin{proof}
		Let $\{\mathcal{A}_n\}_{n \in \mathbb{N}}=\{D_k^n\}_{n \in \mathbb{N},0\leq k \leq a_n-1}$ be the sequence of Kakutani-Rokhlin partitions with property (H) given in Example~\ref{ex: kakutani-rokhlin partition of an odometer}. Denote by $\Gamma_n$ the group of elements $\gamma \in \mathfrak{T}(\varphi)$ of which the cocycle $f_\gamma$ is constant on atoms of $\mathcal{A}_n$. The family $\{\Gamma_n\}_{n \in \mathbb{N}}$ is a nested sequence of subgroups with $\bigcup_{n \in \mathbb{N}} \Gamma_n = \mathfrak{T}(\varphi)$. Since for every $n \in \mathbb{N}$ the partition $\mathcal{A}_n$ consists of only one tower of height $a_n$ every $\gamma \in \Gamma_n$ induces a well-defined permutation $\sigma_\gamma \in \mathfrak{S}_{a_n}$.
		The induced map $\pi_n \colon \Gamma_n \to \mathfrak{S}_{a_n}$ is a group homomorphism and has a section $\pi_n^{-1}$ defined by $\pi_n^{-1}(\sigma)(x)=\varphi^{\sigma(k)-k}(x)$ for $\sigma \in \mathfrak{S}_{a_n}$ and $x \in D_k^n$. Let $\gamma \in \ker(\pi)$. Then for all $k \in \mathbb{Z}/a_n\mathbb{Z}$ it holds that $\gamma(D_k^n)=D_k^n$ and thus $\gamma|_{D_k^n}=\varphi^{m_{\gamma,k} \cdot a_n}$ for some $m_{\gamma,z} \in \mathbb{Z}$. Then $\ker(\pi_n)\to \mathbb{Z}^{a_n} \colon \gamma \to (m_{\gamma,k})_{k \in \mathbb{Z}/a_n\mathbb{Z}}$ is an isomorphisms and therefore $\Gamma_n \cong \mathbb{Z}^{a_n} \rtimes \mathfrak{S}_{a_n}$.
	\end{proof}

	\subsection{Topological full groups in C*-algebraic terms}\label{subs: sight}
	
	The first glances on topological full groups of minimal Cantor systems were from a perspective informed by topological dynamics and operator algebras. The significance of full groups resp. topological full groups in topological dynamics is, that under sufficient conditions they determine the orbit equivalence class resp. flip conjugacy class of a topological dynamical system. One appearance of $\mathfrak{T}(\varphi)$ can be traced back to \cite{put89}:
	
	Let $(X,\varphi)$ be a minimal Cantor system and let $Y$ be a closed subset of $X$. Then by Theorem~\ref{thm: putnam} the C*-subalgebra $\mathfrak{A}_{Y}$ of $C(X)\rtimes_\varphi \mathbb{Z}$ generated by $C(X)$ and $u C_0(X \setminus Y)$ is an AF C*-algebra, by which it lent itself to the analysis on ideals in the wake of Theorem~\ref{thm: AF} by Stratila and Voiculescu. The maximal abelian subalgebra with associated conditional expectation is given by $C(X)$ and the restriction of the conditional expectation $E$ from Remark~\ref{rem: crossedprod}(iii) to $\mathfrak{A}_{Y}$. Fix the following notations:
	\begin{equation*}
	\begin{gathered}
	\operatorname{UN}(C(X),C(X) \rtimes_{\varphi} \mathbb{Z}):=\{v \in \operatorname{U}(C(X) \rtimes_{\varphi} \mathbb{Z})| vC(X)v^*=C(X) \}\\
	\operatorname{UN}(C(X),\mathfrak{A}_{Y}):=\{v \in \operatorname{U}(\mathfrak{A}_{Y})| vC(X)v^*=C(X) \}
	\end{gathered}
	\end{equation*}
	
	Both sets can be considered as groups of $*$-automorphism of $C(X)$ by their adjoint action -- the latter obviously being a subgroup of the former. Denote the respective quotients by $\operatorname{U}(C(X))$ of this groups by $\Gamma$ (resp. $\Gamma_Y$). The group $\operatorname{UN}(C(X),\mathfrak{A}_{Y})$ plays the role of the locally finite unitary subgroup $U$ from Theorem~\ref{thm: AF}, while $\Gamma_Y$ represents the corresponding group of homeomorphisms $\Gamma_U$.	By Lemma 5.1 in \cite{put89}, every $v \in \operatorname{UN}(C(X), C(X) \rtimes_{\varphi} \mathbb{Z})$ can be written uniquely as $v=f\sum_{n \in \mathbb{Z}}u^n p_n $, where $f \in \operatorname{U}(C(X))$ and $\{p_n \}_{n \in \mathbb{Z}}$ is a collection of pairwise orthogonal projections $p_n \in C(X)$ with $p_n=0$ for almost all $n$ such that
	\begin{equation*}
	\mathbf{1}_{X}=\sum_{n \in \mathbb{Z}} p_n = \sum_{n \in \mathbb{Z}} \varphi^{-n}(p_n)
	\end{equation*} 
	holds. The projection $p_n$ is given by the absolute value of $E_n(v)$ (Remark~\ref{rem: crossedprod}(iii)). The group $\Gamma$ is isomorphic to the group of invertible elements in the monoid $C(X,\mathbb{Z})$ where the monoid operation is given by
	\begin{equation*}
	(f_1\cdot f_2)(x):=f_1(x)+f_2(\varphi^{-f_1(x)}(x))
	\end{equation*}
	This, however, is just the topological full group $\mathfrak{T}(\varphi)$. Every $v \in \operatorname{UN}(C(X), C(X) \rtimes_{\varphi} \mathbb{Z})$ admits an element $\Phi(v) \in \mathfrak{T}(\varphi)$ defined by $\Phi(v)(x):= \varphi^n(x)$ for all $x \in \supp(p_n)$.
	
	\begin{prop}[\cite{put89}, Theorem 5.2]\label{prop: exact}
		Let $(X,\varphi)$ be a minimal Cantor system. Then 
		\begin{equation*}
		1 \longrightarrow \operatorname{U}(C(X)) \overset{\iota}{\longrightarrow} \operatorname{UN}(C(X), C(X) \rtimes_{\varphi} \mathbb{Z}) \overset{\Phi}{\longrightarrow} \mathfrak{T}(\varphi) \longrightarrow 1
		\end{equation*}
		is an exact sequence which is split via the section $\gamma \mapsto v_\gamma := \sum_{n \in \mathbb{Z}} u^n \mathbf{1}_{X_n^\gamma}$, where $\{X_n^\gamma\}_{n \in \mathbb{Z}}$ is the partition defined in Proposition~\ref{prop: count}.
	\end{prop}

	The above considerations are embedded in the groupoid setting, where they are an aspect of Renault's topological version of Feldman-Moore. The following proposition is a special case of \cite{ren08}, Proposition~4.8 (by putting twists aside): 
	\begin{lem}[\cite{mat12}, Lemma~5.5]
		Let $\mathcal{G}$ be a Hausdorff, locally compact, effective, \'etale groupoid. For every partial isometry $v \in C_r^*(\mathcal{G})$ such that $vv^*, v^*v \in C_0(\mathcal{G}^{(0)})$ and $vC_0(\mathcal{G}^{(0)})v^*=vv^*C_0(\mathcal{G}^{(0)})$ there exists a compact open slice $V \in \mathcal{B}_\mathcal{G}^{o,k}$ such that $V=\{g \in \mathcal{G}|v(g) \neq 0  \}$.
	\end{lem}
	
	The assignement $v \mapsto \tau_V$ induces a homomorphism $\Phi \colon \operatorname{UN}(C(\mathcal{G}^{(0)}), C_r^*(\mathcal{G})) \to \mathfrak{T}(\mathcal{G})$ such that $ufu^*=f \circ \Phi(u)$ for $u \in \operatorname{UN}(C(\mathcal{G}^{(0)}), C_r^*(\mathcal{G})$, in particular the following holds:
	\begin{prop}[\cite{mat12}, Proposition~5.6]\label{prop: exact sequence, topological full groups as unitary normalizers}
		Let $\mathcal{G}$ be an effective, \'etale Cantor groupoid. Then the following sequence is exact and splits
		\begin{equation*}
		1 \longrightarrow \operatorname{U}(C(\mathcal{G}^{(0)})) \overset{\iota}{\longrightarrow} \operatorname{UN}(C(\mathcal{G}^{(0)}), C_r^*(\mathcal{G})) \overset{\Phi}{\longrightarrow} \mathfrak{T}(\mathcal{G}) \longrightarrow 1
		\end{equation*}
	\end{prop}
	
	\begin{rem}
		A canonical section of $\Phi$ is provided by mapping $B \in \mathfrak{T}(\mathcal{G})$ onto $\mathbf{1}_B$. Note that then $\mathbf{1}_B f \mathbf{1}_B^*=f \circ \alpha_B$ for all $f \in C(\mathcal{G}^{(0)})$.
	\end{rem}

	\subsection{Locally finite subgroups}
	
	\begin{defi}[\cite{gps99}, §2]
		Let $(X,\varphi)$ be a minimal Cantor system and let $x \in X$.	Define $\mathfrak{T}(\varphi)_{\{x\}}$ to be the stabilizer $\mathfrak{T}(\varphi)_{\operatorname{Orb}_\varphi^+(x)}=\{\gamma \in \mathfrak{T}(G)| \gamma(\operatorname{Orb}_\varphi^+(x)) = \operatorname{Orb}_\varphi^+(x)\}$.
	\end{defi}	
	
	\begin{rem}\label{rem: locfin. conj}
		It follows immediately from this definition that $\varphi \cdot \mathfrak{T}(\varphi)_{\{x\}} \cdot \varphi^{-1}=\mathfrak{T}(\varphi)_{\{ \varphi(x)\}}$.
	\end{rem}
	
	Subgroups of this form can be characterized in terms of approximation by Kakutani-Rokhlin partitions and are necessarily locally finite as a corollary of Theorem~\ref{thm: grig-med}:
	
	\begin{cor}\label{cor: locally finite subgroups}
		Let $(X,\varphi)$ be a minimal Cantor system. Let $x \in X$ and $\{\mathcal{A}_n\}_{n \in \mathbb{N}}$ be a nested sequence of Kakutani-Rokhlin partitions around $x$ satisfying property (H). Then the following hold:
		\begin{enumerate}[(i)]
			\item \begin{equation*}
			\mathfrak{T}(\varphi)_{\{x\}} = \bigcup_{n \in \mathbb{N}} \mathcal{P}(\mathcal{A}_n):\cong\bigcup_{n \in \mathbb{N}} \; \bigoplus_{i \in \{1,\dots, i_n\}} \mathfrak{S}_{h_i^n}.
			\end{equation*}
			
			\item The group $\mathfrak{T}(\varphi)_{\{x\}}$ is locally finite and in consequence amenable.
			
			\item For every $x,y \in X$ the groups $\mathfrak{T}(\varphi)_{\{x\}}$ and $ \mathfrak{T}(\varphi)_{\{y\}}$ are isomorphic.
		\end{enumerate}
		
	\end{cor}
	\begin{proof}
		(i) Let $\gamma \in \mathfrak{T}(\varphi)_{\{x\}}$. Assume $r_\gamma$ is non-trivial, then by the definition of $n$-rotations, there exist points in $\operatorname{Orb}_\varphi^-(x)$ which are mapped to $\operatorname{Orb}_\varphi^+(x)$ by $\gamma$. Conversely, any $\gamma \in \mathcal{P}(\mathcal{A}_n)$ preserves towers and thus forward and backward orbits of $x \in X$.
		
		(ii) and (iii) are immediate consequences of (i).
		
		(iv) follows from (i), Remark~\ref{rem: locfin. conj} and minimality.
	\end{proof}
	
	\begin{cor}[\cite{gm14}, Proposition~2.12 \& \cite{dC13}, FAIT~2.2.4]
		Let $(X,\varphi)$ be a minimal Cantor system. Then the group $\mathfrak{T}(\varphi)_{\{x\}}$ satisfies no group law.\footnote{Let $w \in \mathbb{F}_k$ for some $k \in \mathbb{N}$ be a non-empty word. A group $G$ \emph{satisfies the group law $w$} if every set of elements $g_1,\dots,g_k \in G$ evaluates under $w$ as $w(g_1,\dots,g_k) = 1 \in G$.\\} It follows that the groups $\mathfrak{T}(\varphi)$ and $\mathfrak{T}(\varphi)'$ satisfy no group law.
	\end{cor}
	\begin{proof}
		Assume there exists a sufficient non-empty word $w \in \mathbb{F}_k$ for some $k \in \mathbb{N}$. Then by Corollary~\ref{cor: locally finite subgroups}(i) every finite symmetric group satisfies the group law $w$. Hence by Caley's Theorem every finite group and in consequence every product of finite groups satisfies the law $w$. But since residually finite groups are precisely the groups which embed into direct products of finite groups and free groups are residually finite,\footnote{See e.g. Corollary~2.2.6 and Theorem~2.3.1 of \cite{cc10}.\\} this is a contradiction.
	\end{proof}
	
	\begin{cor}[\cite{gm14}, Proposition~5.2]
		Let $(X,\varphi)$ be a minimal Cantor system. The group $\mathfrak{T}(\varphi)_{\{x\}}$ is a maximal locally finite subgroup for every $x \in X$.
	\end{cor}
	\begin{proof}
		Let $\gamma \in \mathfrak{T}(\varphi)\setminus \mathfrak{T}(\varphi)_{\{x\}}$. Then by Corollary~\ref{cor: locally finite subgroups}, the decomposition of $\gamma$ for a sufficient sequence of Kakutani-Rokhlin partition must always contain a non-trivial $n$-rotation $r_\gamma$. Then $r_\gamma \in \langle\mathfrak{T}(\varphi), \gamma \rangle$ is of infinite order and thus, any subgroup of $\mathfrak{T}(\varphi)$ that contains $\mathfrak{T}(\varphi)_{\{x\}}$ and $\gamma$ cannot be locally finite.
	\end{proof}
	
	These locally finite subgroups had alread appeared implicitely in \cite{kri80} and in \cite{put89}:
	The subgroup $\mathfrak{T}(\varphi)_{\{x\}}$ is just the group $\Gamma_Y$ in the case of $Y=\{x\}$ for $x \in X$ considered by Putnam. For every $x \in X$ the system $(X,\mathfrak{T}(\varphi)_ {\{x\}})$ is a minimal AF-system (see Definition~\ref{defi: ample group}). Since every AF-system is conjugate to a Bratteli-Vershik system $(X,\varphi)$ conversely for every minimal AF-system $\Gamma$ there exists an $x \in X$ such that $\Gamma \cong \mathfrak{T}(\varphi)_{\{x\}}$. A more general result is \cite{jm13}, Lemma~4.1 \& Lemma~4.2 where local finiteness for more general stabilizer subgroups is shown. Matui generalizes Corollary~\ref{cor: locally finite subgroups} to AF-groupoids in \cite{mat06}, §3 (for AF-equivalence relations): The topological full group $\mathfrak{T}(\mathcal{G})$ of the AF-groupoid $\mathcal{G}$ associated to a Bratteli diagram $(V,E)$ as in Example~\ref{ex: AFgrpd} is the increasing union of direct sums of symmetric groups $G_k:= \bigoplus_{v \in V_k} \mathfrak{S}_{|\mathcal{P}_{0,v}|}$ that act by permutation of cylinder sets in the space of infinite paths. Conversely, if $\mathfrak{T}(\mathcal{G})$ of an \'etale Cantor groupoid is locally finite, the groupoid $\mathcal{G}$ is AF.
	
	The commutator subgroups of the locally finite subgroups $\mathfrak{T}(\varphi)_{\{x\}}$ are simple groups: 	
	\begin{cor}\label{cor: ev.permut.}
		Let $(X,\varphi)$ be a minimal Cantor system. Let $x \in X$ and let $\{\mathcal{A}_n\}_{n \in \mathbb{N}}$ be a nested sequence of Kakutani-Rokhlin partitions around $x$ satisfying property (H). Then
		\begin{equation*}
		\mathfrak{T}(\varphi)_{\{x\}}' \cong \bigcup_{n \in \mathbb{N}} \; \bigoplus_{i \in \{1,\dots, i_n\}} \mathfrak{A}_{h_i^n}
		\end{equation*}
		where $\mathfrak{A}_{h_i^n}$ denotes the alternating group of degree $h_i^n$ and furthermore:
		\begin{equation*}
		\mathfrak{T}(\varphi)_{\{x\}}^{\operatorname{ab}}:=\bigslant{\mathfrak{T}(\varphi)_{\{x\}}}{\mathfrak{T}(\varphi)_{\{x\}}'}= \varinjlim_{n} (\mathbb{Z}/2\mathbb{Z})^{i_n}.
		\end{equation*}
	\end{cor}
	
	\begin{rem}\label{rem: af sign}
		We note that for minimal AF-systems\footnote{See \cite{gps04} and \cite{mat06}, §3.\\} the following hold:
		\begin{equation*}
		\varinjlim (\mathbb{Z}/2\mathbb{Z})^{i_n} \cong K_0(C(X) \rtimes \mathfrak{T}(\varphi)_{\{x\}}) \otimes \mathbb{Z}/2\mathbb{Z} \overset{\text{\text{Theorem~\ref{thm: cantor-AF}}}}{\cong} K_0(C(X) \rtimes_\varphi \mathbb{Z}) \otimes \mathbb{Z}/2\mathbb{Z}
		\end{equation*}
		By this we have $\mathfrak{T}(\varphi)_{\{x\}}^{\operatorname{ab}}\cong K_0(C(X) \rtimes_\varphi \mathbb{Z}) \otimes \mathbb{Z}/2\mathbb{Z}$.
	\end{rem}
	
	\begin{prop}\label{prop: loc.fin.simpl.}
		Let $(X,\varphi)$ be a minimal Cantor system and let $x \in X$. Then $\mathfrak{T}(\varphi)_{\{x\}}'$ is a simple group.
	\end{prop}
	\begin{proof}
		Fix a nested sequence $\{\mathcal{A}_n\}_{n \in \mathbb{N}}$ of Kakutani-Rokhlin partitions around $x$ satisfying property (H). Let $H$ be a non-trivial normal subgroup of $\mathfrak{T}(\varphi)_{\{x\}}'$. Let $\gamma$ be a non-trivial element of the subgroup $H_n:=H \cap \bigoplus_{i \in \{1,\dots, i_n\}} \mathfrak{A}_{h_i^n}$. by Remark~\ref{rem: simpkakro}, there exists an $l>n$ such that every summand of $\gamma$ viewed as element in $\bigoplus_{i \in \{1,\dots, i_k\}} \mathfrak{A}_{h_i^k}$ for $k\geq l$ is non-trivial. But since every factor $\mathfrak{A}_{h_i^k}$ is simple, any proper normal subgroup of $\bigoplus_{i \in \{1,\dots, i_k\}} \mathfrak{A}_{h_i^k}$ is of the form $\bigoplus_{i \in I} \mathfrak{A}_{h_i^k}$ for some proper subset $I\subseteq \{1,\dots, i_k\}$. This means the normal closure of $H_n$ in $\bigoplus_{i \in I} \mathfrak{A}_{h_i^k}$ is the whole group for all $k$ large enough and thus $H=\mathfrak{T}(\varphi)_{\{x\}}'$.
	\end{proof}

	\subsection{A lemma by Glasner and Weiss}\label{subs: gla.wei.}
	
	In this Subsection, we give a swift sketch of a lemma in \cite{gw95}. One purpose of \cite{gw95} was to do the classification of minimal Cantor systems by the associated dimension groups obtained in \cite{gps95} \emph{without} relying on C*-algebra theory. In particular, it uses the full group and the topological full group to construct proper equivalences between minimal Cantor systems given isomorphisms of the associated dimension groups -- thus \cite{gw95} foreshadows the results in \cite{gps99}.
	
	\begin{lem}[\cite{gw95}, Lemma 2.5]\label{lem: gw}
		Let $(X,\varphi)$ be a minimal Cantor system and let $A,B$ be clopen subsets of $X$ such that $\mu(B)<\mu(A)$ for every  $\mu \in M_\varphi$. Then there exists an element $\gamma \in \mathfrak{T}(\varphi)$ such that $\gamma(B)\subset A$, $\gamma^2=\operatorname{id}$ and $\gamma|_{X\setminus (B \cup \gamma(B))}=\operatorname{id}$.
	\end{lem}
	\begin{proofsketch}
		If there exists a Kakutani-Rokhlin partition $\mathcal{A}$ that refines
		\begin{equation*}
		\mathcal{P}:=\{B \setminus A,A\setminus B,A\cap B, X \setminus (A \cup B)\}
		\end{equation*}
		such that the number of atoms contained in $A$ is greater than the number of atoms of contained in $B$, a sufficient homeomorphism $\gamma \in \mathfrak{T}(\varphi)$ can easily be given by moving atoms of $\mathcal{A}$.  By the assumptions, $\int \mathbf{1}_A-\mathbf{1}_B \; \mathrm{d}\mu > 0$ holds for every $\mu \in M_\varphi$. Then there exists a $c>0$ such that $\inf\{\int \mathbf{1}_A-\mathbf{1}_B \; \mathrm{d}\mu|\mu \in M_\varphi \}>c$, because otherwise there exists a sequence of measures for which the above integral goes to $0$. By the compactness of $M_\varphi$ in the weak*-topology, a cluster point $\mu$ of this sequence exists. But $\mu$ satisfies $\mu(A)=\mu(B)$, which is a contradiction. In consequence, there exists an $n_0 \in \mathbb{N}$ such that every $n \geq n_0$ satisfies 
		\begin{equation*}
		n^{-1} \cdot \sum_{i=0}^{n-1} (\mathbf{1}_A-\mathbf{1}_B)(\varphi^i(x))\geq c
		\end{equation*}
		for all $x \in X$. Assume otherwise the existence of an increasing sequence $\{n_k\}_{k \in \mathbb{N}}$ of naturals and a sequence of points $\{x_k\}_{k\in \mathbb{N}}$ such that
		\begin{equation*}
		n_k^{-1} \cdot \sum_{i=0}^{n_k-1} (\mathbf{1}_A-\mathbf{1}_B)(\varphi^i(x_k))\leq c.
		\end{equation*}
		Then a weak* cluster point $\tilde{\mu}$ of 
		\begin{equation*}
		n_k^{-1} \cdot \sum_{i=0}^{n_k-1} \varphi^i(\delta_{x_k})
		\end{equation*}
		satisfies $\int \mathbf{1}_A-\mathbf{1}_B \; \mathrm{d}\tilde{\mu} \leq c$ which is a contradiction.
		Let $\mathcal{A}$ be Kakutani-Rokhlin partition that refines $\mathcal{P}$ such that $h_{\mathcal{A}}>n_0$. Then the above inequality implies that $A$ contains more atoms than $B$.
	\end{proofsketch}
	
	We note that Lemma~\ref{lem: gw} is crucial for the proof of Theorem~\ref{thm: spat}. Furthermore, it implies the following lemma, which is required to prove simplicity of $\mathfrak{T}(\varphi)'$ (see Subsection~\ref{subs: simplicity and finite generation}):
	\begin{lem}[\cite{mat06}, Lemma~4.7]\label{lem: 2-div.}
		Let $(X,\varphi)$ be a minimal Cantor system and let $x \in X$. Then every clopen neighbouhood $U$ of $x$ contains a clopen neighbourhood $V$ of $x$, such that $[\mathbf{1}_V]$ is $2$-divisible in $K^0(X,\varphi)$.\footnote{There exists a class $[f]\in K^0(X,\varphi)$ such that $2[f]=[\mathbf{1}_V]$.\\}
	\end{lem}
	\begin{proof}
		Let $V'$ be a clopen neighbourhood of $x$ such that $2\mu(V')<\mu(U)$ for all $\mu \in \mathcal{M}(\varphi)$. Then $\mu(V')<\mu(U\setminus V')$ for all $\mu \in \mathcal{M}(\varphi)$. Let $\gamma$ be the homeomorphism obtained in Lemma~\ref{lem: gw} for the pair $V'$ and $U \setminus V'$. We have to show $[\mathbf{1}_{\gamma(V')}-\mathbf{1}_{V'}]=0$, because then $[\mathbf{1}_{V'}]= [\mathbf{1}_{\gamma(V')}]$, hence $[\mathbf{1}_{V'\cup \gamma(V')}]=2[\mathbf{1}_{V'}]$. Then $C_k:=V'\cap f_\gamma^{-1}(k)$ for $k \in \mathbb{Z}$ is a partition of $V'$ and we can write
		\begin{equation*}
		\begin{gathered}
		\mathbf{1}_{\gamma(V')}-\mathbf{1}_{V'}=\sum_{k \in \mathbb{Z}}\big(\mathbf{1}_{\gamma(C_k)}- \mathbf{1}_{C_k}\big) = \sum_{k \in \mathbb{Z}}\big(\mathbf{1}_{\varphi^k(C_k)}- \mathbf{1}_{C_k}\big)
		\end{gathered}
		\end{equation*}
		Every summand can be written as a telescoping sum of which every summand is trivial in $K^0(X,\varphi)$: Write
		\begin{equation*}
		\mathbf{1}_{\varphi^k(C_k)}- \mathbf{1}_{C_k}=\sum_{i=0}^{k-1}(\mathbf{1}_{\varphi^i(C^k)}\circ \varphi^{-1}-\mathbf{1}_{\varphi^i(C^k)})
		\end{equation*}
		for all $k>0$ and
		\begin{equation*}
		\mathbf{1}_{\varphi^k(C_k)}- \mathbf{1}_{C_k}=\sum_{i=0}^{|k|-1}(\mathbf{1}_{\varphi^{k+i}(C^k)}-\mathbf{1}_{\varphi^{k+i}(C^k)}\circ \varphi^{-1})
		\end{equation*}
		for all $k<0$.
	\end{proof}

	\subsection{The index map}
	
	The paper \cite{gps99} introduces another important subgroup by the means of the \emph{index map}. For minimal Cantor systems the index map measures the transfer from backward orbits to forward orbits. As the terminology suggests one motivation for the index map lies in operator index theory.
	
	\begin{defi}[\cite{gps99}, Definition 5.1 \& 5.2]\label{defi: kap.lam.}
		Let $(X,\varphi)$ be a minimal Cantor system with a nested sequence of Kakutani-Rohklin partitions $\{\mathcal{A}_n\}_{n \in \mathbb{N}}$ that satisfies (H)  and let $x \in X$.
		\begin{enumerate}[(i)]
			\item Let $\gamma \in \mathfrak{T}(\varphi)$. Define $\kappa_x(\gamma):=|\{y \in \operatorname{Orb}_\varphi^-(x)|\gamma(y) \in \operatorname{Orb}_\varphi^+(x) \}|$ and define $\lambda_x(\gamma):=|\{y \in \operatorname{Orb}_\varphi^+(x)|\gamma(y) \in \operatorname{Orb}_\varphi^-(x) \}|$.
			
			\item Let $a \in \mathbb{N}$, $b \in \mathbb{Z}$ and let $m \in \mathbb{N}$ such that $h_{\mathcal{A}_m} > 2|b|+2a$. Define the homeomorphism $\sigma_{k,l} \in \mathfrak{T}(\varphi)$ by:
			\begin{equation*}
			\sigma_{k,l}(y):=
			\begin{cases}
			\varphi^{-2|b|-a}(y) & \text{for } y \in D_{k,i}^{m}\text{ if }|b|+1 \leq k \leq |b|+a \\
			\varphi^{2|b|+a}(y) & \text{for } y \in D_{k,i}^{m}\text{ if }h_i^{m}-|b|-a+1 \leq k \leq h_i^{m}-|b|\\
			y & \text{else}
			\end{cases}
			\end{equation*}
		\end{enumerate}
	\end{defi}
	
	\begin{rem}\label{rem: fin}
		$\kappa_x(\gamma)$ and $\lambda_x(\gamma)$ are necessarily finite. The homeomorphisms $\sigma_{k,l}$ are elements in $\mathfrak{T}(\varphi)$ of finite order.
	\end{rem}
	
	\begin{lem}[\cite{gps99}, Lemma 5.3.]\label{lem: tfgrp-rep}
		Let $(X,\varphi)$ be a minimal Cantor system and $x \in X$. Then the following holds:
		\begin{equation*}
		\mathfrak{T}(\varphi)=\coprod_{k \in \mathbb{N},l \in \mathbb{Z}} \mathfrak{T}(\varphi)_{\{x\}}\; \varphi^l \sigma_{k,l}\;\mathfrak{T}(\varphi)_{\{x\}}.
		\end{equation*}
		
	\end{lem}
	\begin{proof}
		Let $\beta_x \colon \mathfrak{T}(\varphi) \to \mathbb{N}_0 \times \mathbb{N}_0$ be defined by $\beta_x(\gamma):=\big(\kappa_x(\gamma),\lambda_x(\gamma)\big)$. Varying over all possible pairs $k,l$ lets $\beta_x(\varphi^l \sigma_{k,l})$ vary over all of $\mathbb{N}_0 \times \mathbb{N}_0$, thus $\beta_x$ is surjective. Let $\gamma \in \mathfrak{T}(\varphi)$ such that $p=\kappa_x(\gamma), \lambda_x(\gamma)=q$. If $p \geq q$, then $\beta_x(\gamma)=\beta_x(\varphi^{p-q}\sigma_{q,p-q})$ and if $p < q$, then $\beta_x(\gamma)=\beta_x(\varphi^{p-q}\sigma_{p,q-p})$. By applying a homeomorphism $\gamma_2$ in $\mathfrak{T}(\varphi)_{\{x\}}$, we can bring elements in $\operatorname{Orb}_\varphi(x)$ into a position such that applying $\varphi^{p-q}\sigma_{q,p-q}$ resp. $\varphi^{p-q}\sigma_{p,q-p}$ maps orbit points into the same half-orbits as $\gamma$, hence there exists a homeomorphism $\gamma_1$ in $\mathfrak{T}(\varphi)_{\{x\}}$ such that $\gamma = \gamma_1 \varphi^l \sigma_{k,l} \gamma_2$ for sufficient $k,l$.
	\end{proof}
	
	\begin{prop}[\cite{gps99}, Proposition 5.4. \& 5.5.]\label{prop: index}
		Let $(X,\varphi)$ be a minimal Cantor system and let $\mu$ be a $\varphi$-invariant probability measure on $X$. Then $\Hom(\mathfrak{T}(\varphi),\mathbb{Z})$ is isomorphic to $\mathbb{Z}$ and $I_\mu(\gamma):=\int_{X} f_\gamma \mathrm{d}\mu$ defines the unique group homomorphism $I_\mu \colon \mathfrak{T}(\varphi) \to \mathbb{Z}$ with $I_\mu (\varphi)=1$.
	\end{prop}
	\begin{proof}
		By Lemma~\ref{lem: tfgrp-rep} elements of the topological full group can be written as products of elements of finite order and multiples of $\varphi$, thus homeomorphisms in $\Hom(\mathfrak{T}(\varphi) , \mathbb{Z})$ only depend on the image of $\varphi$, hence $\Hom(\mathfrak{T}(\varphi),\mathbb{Z}) \cong \mathbb{Z}$. Let $\gamma_1, \gamma_2 \in \mathfrak{T}(\varphi)$. By Lemma~\ref{lem: cocy}, $\varphi$-invariance of the measure and linearity of the integral, we have $I_\mu(\gamma_1 \gamma_2)=I_\mu(\gamma_1)+I_\mu(\gamma_2)$, thus $I_\mu$ defines a group homomorphism with $\Ima (I_\mu)\subseteq\mathbb{R}$ and $I_\mu (\varphi)=1$. Elements of finite order must necessarily be contained in $\ker(I_\mu)$, thus $\Ima (I_\mu) = \mathbb{Z}$.
	\end{proof}
	
	\begin{defi}\label{defi: index}
		The map $I=I_\mu$ described in Proposition~\ref{prop: index} is called the \emph{index map}.
	\end{defi}
	
	One can find a covariant representation $\rho$ of $C^*(X,\varphi)$ on $\mathcal{H}=\ell^2(\mathbb{Z})$ (the left regular representation) together with a projection $P$ on $\mathcal{H}$, such that
	\begin{equation*}
	I(\gamma)=\dim\ker(P\rho(v_\gamma)P)- \dim\ker((P\rho(v_\gamma)P)^*)= \kappa_x(\gamma)-\lambda_x(\gamma)
	\end{equation*}
	
	Thus $I$ is the manifestation of a Fredholm index. The motivation of this lies in non-commutative differential geometry. The \emph{cyclic cohomology} $HC^*(\mathcal{A})$ of a C*-algebra $\mathcal{A}$ is a non-commutative version of de Rham homology introduced by Connes. It admits a bilinear pairing $\langle\;,\;\rangle\colon HC^1(C(X) \rtimes_{\varphi} \mathbb{Z})\times K_1(C(X) \rtimes_{\varphi} \mathbb{Z})\to \mathbb{C}$. Every $\gamma \in \mathfrak{T}(\varphi)$ induces a canonical unitary $v_\gamma \in U(C^*(X,\varphi))$ by Proposition~\ref{prop: exact} thus further  inducing a class $[v_\gamma]\in K_1(C(X) \rtimes_{\varphi} \mathbb{Z})$. Every $\varphi$-invariant measure $\mu$ induces a class of one-cocycles $[c_\mu]\in HC^1(C(X) \rtimes_{\varphi} \mathbb{Z})$, the pairing induces a map $\pi_\mu\colon\mathfrak{T}(\varphi)\to \mathbb{C}$. It can be shown that $\pi_\mu(\gamma)=I(\gamma)$. A direct proof of $I(\gamma)=\kappa_x(\gamma)-\lambda_x(\gamma)$ follows from Theorem~\ref{thm: grig-med}:
	
	\begin{cor}[\cite{gm14}, Lemma~5.3]
		Let $(X,\varphi)$ be a minimal Cantor system. Then $\Ima(I)=\mathbb{Z}$ and $I=\kappa_x-\lambda_x$.
	\end{cor}
	\begin{proof}
		Let $\gamma \in \mathfrak{T}(\varphi)$ and let $\{\mathcal{A}_n\}_{n\in \mathbb{N}}$ be a sufficient nested sequence of Kakutani-Rokhlin partitions. Let $\gamma=p_\gamma r_\gamma$ be a decomposition into an $n$-permutation $p_\gamma$ and an $n$-rotation $r_\gamma$ with rotation number at most $1$. Then $I(\gamma)=I(r_\gamma)$. By Lemma~\ref{lem: ind.trans}, $I(\varphi_A^{-1})=I(\varphi_A^{-1} \varphi \varphi^{-1})=I(\varphi^{-1})=-1$ for every clopen set $A \subseteq X$. Thus by definition of $r_\gamma$ we have $I(r_\gamma)=|U|-|L|$, where $U$ is the set of numbers $k \in \{0,\dots,m_n\}$ such that $\gamma$ moves $U(k)=\bigsqcup_i D_{h_i-k-1,i}^n$ over the roof of the partition, hence $|U|=\kappa_x(\gamma)$ and $L$ is the set of numbers $k \in \{1,\dots,m_n\}$ such that $\gamma$ moves $L(k)=\bigsqcup_i D_{k,i}^n$ below the base of the partition, hence $|L|=\lambda_x(\gamma)$. 
	\end{proof}

	\begin{defi}[\cite{gps99}, Definition 5.7]
		Let $(X,\varphi)$ be a minimal Cantor system. Denote by $\mathfrak{T}(\varphi)_0$ the kernel of $I$.
	\end{defi}
	
	\begin{rem}\label{rem: com. ind.}
		We have $\mathfrak{T}(\varphi)_{\{x\}} \leq \mathfrak{T}(\varphi)_0$ for all $x \in X$, and since commutators in $\mathfrak{T}(\varphi)$ vanish under the index map, we have $\mathfrak{T}(\varphi)' \leq \mathfrak{T}(\varphi)_0$.
	\end{rem}
	
	\begin{prop}[\cite{mat06}, Lemma 4.1]\label{prop: mat.fac}
		Let $(X,\varphi)$ be a minimal Cantor system and let $x,y \in X$ with $\operatorname{Orb}_\varphi(x) \neq \operatorname{Orb}_\varphi(y)$. Then $\mathfrak{T}(\varphi)_0 = \mathfrak{T}(\varphi)_{\{x\}} \cdot \mathfrak{T}(\varphi)_{\{y\}}$.
	\end{prop}
	\begin{proof}
		Let $\gamma \in \mathfrak{T}(\varphi)_0$. As $0=I(\gamma)=\kappa_x(\gamma)-\lambda_x(\gamma)$, there exists a bijection $b$ between the finite sets:
		\begin{equation*}
		A=\{n \in \mathbb{N}|\gamma(\varphi^n(x)) \in \operatorname{Orb}_\varphi^-(x)\},\;
		B=\{n \in \mathbb{N}|\gamma(\varphi^{1-n}(x)) \in \operatorname{Orb}_\varphi^+(x)\}.
		\end{equation*}
		Let $m=\max(A \cup B)$. Let $U$ be a clopen set containing $x$ such that $U$ does not contain $\varphi^i(x)$ for $i \in \{-m,\dots,m \}$ and
		\begin{equation*}
		V=\bigcup_{n \in A} \varphi^{1-n}(U) \cup \bigcup_{n \in B} \varphi^n(U)
		\end{equation*}
		does not contain $\varphi^i(y)$ for $i \in \{0, \dots , 2m-1\}$. Then $\tilde{\gamma} \in \mathfrak{T}(\varphi)_y$ defined by 
		\begin{equation*}
		\tilde{\gamma}(z):=
		\begin{cases}
		\varphi^{1-b(n)-n}, & z \in \varphi^n(U) \text{ for } n \in A\\
		\varphi^{n-1+b^{-1}(n)}, & z \in \varphi^{1-n}(U) \text{ for } n \in B\\
		z, &  z \in X \setminus V
		\end{cases}
		\end{equation*}
		satisfies $\gamma \tilde{\gamma} \in \mathfrak{T}(\varphi)_{\{x\}}$, thus $\gamma=(\gamma \tilde{\gamma})\tilde{\gamma}^{-1}$ gives the desired factorization.
	\end{proof}
	
	This has immediate consequences:
	
	\begin{cor}
		Let $(X,\varphi)$ be a minimal Cantor system. Then the following hold:
		\begin{enumerate}[(i)]
			\item $\mathfrak{T}(\varphi)_0^{\operatorname{ab}}$ is an elementary abelian $2$-group.
			
			\item $\mathfrak{T}(\varphi)'=\mathfrak{T}(\varphi)_0'$
			
			\item $\mathfrak{T}(\varphi)^{\operatorname{ab}}=\mathbb{Z}\times \mathfrak{T}(\varphi)_0^{\operatorname{ab}}$ 
		\end{enumerate}
		
	\end{cor}
	\begin{proof}
		By Corollary~\ref{cor: locally finite subgroups}(i) and Proposition~\ref{prop: mat.fac} the group $\mathfrak{T}(\varphi)_0$ is generated by its involutions implying (i). By Proposition~\ref{prop: mat.fac} and Remark~\ref{rem: locfin. conj}, we have $\varphi\cdot \mathfrak{T}(\varphi)_0'\cdot \varphi^{-1}=\mathfrak{T}(\varphi)_0'$, which implies
		\begin{equation*}
		\mathfrak{T}(\varphi)/\mathfrak{T}(\varphi)_0'= \mathfrak{T}(\varphi)/\mathfrak{T}(\varphi)_0 \times \mathfrak{T}(\varphi)_0^{\operatorname{ab}}  =\mathbb{Z}\times \mathfrak{T}(\varphi)_0^{\operatorname{ab}}
		\end{equation*}
		
		The group $\mathbb{Z}\times \mathfrak{T}(\varphi)_0^{\operatorname{ab}}$ is abelian, implying (ii) which in turn implies (iii).
	\end{proof}
	
	The notion of the index map generalizes to the setting of groupoids:
	
	\begin{defi}[\cite{nek17}, Definition 2.4]\label{defi: index2}
		Let $\mathcal{G}$ be an \'etale Cantor groupoid. The \emph{index map} is the homomorphism $\operatorname{I} \colon \mathfrak{T}(\mathcal{G}) \to H_1(\mathcal{G})$, which maps $B \in \mathfrak{T}(\mathcal{G})$ onto to the equivalence class of its characteristic function $[\mathbf{1}_B]$.\footnote{$[\mathbf{1}_B]$ is by definition contained in $\ker \delta_1 = \ker (s_\ast - r_\ast)$.\\}
	\end{defi}
	
	\begin{rem}
		The index map is a homomorphism as an immediate consequence of Proposition~\ref{prop: hom1}.
	\end{rem}
	
	For the transformations groupoids associated to a minimal Cantor system, the dynamical homology is isomorphic to the K-theory of the associated crossed product and above map genereralizes the index map of Definition~\ref{defi: index}. 
	
	\begin{defi}
		The kernel of the index map $\operatorname{I}$ is a subgroup of $\mathfrak{T}(\mathcal{G})$ denoted by $\mathfrak{T}(\mathcal{G})_0$.
	\end{defi}
	
	By Proposition~\ref{prop: hom1} we immediately have:
	
	\begin{cor}\label{cor: transpositions are in the index kernel}
		Let $\mathcal{G}$ be an \'etale Cantor groupoid and let $B \in \mathcal{B}_\mathcal{G}^{o,k}$ such that $s(B) \cap r(B)= \emptyset$. Then $T_B \in \mathfrak{T}(\mathcal{G})_0$. 
	\end{cor}
	
	\subsection{Isomorphism theorems}\label{subs: isom}
	
	The chief purpose of \cite{gps99} was the classification of minimal Cantor systems by the means of full groups and their subgroups. 
	
	\begin{defi}[\cite{gps99}, Definition 4.1, 4.10., 5.1. \& 5.2.]\phantomsection\label{defi: spatial homeomorphism + local subgroups}\leavevmode
		\begin{enumerate}[(i)]
			\item Let $(X_1,G_1)$ and $(X_2,G_2)$ be topological dynamical systems. An isomorphism $\alpha \colon G_1 \to G_2$ is called \emph{spatial} if there exists a homeomorphism $\varphi \colon X_1 \to X_2$ such that $\alpha (g)=\varphi \circ g \circ \varphi^{-1}$ for all $g \in G_1$.
			
			\item  Let $(X,\varphi)$ be a Cantor system. Let $A \subseteq X$ and let $G \leq \operatorname{Homeo}(X)$. Denote by $G_A$ the subgroup $\{g \in G| \supp(g) \subseteq A \}$. If $A$ is a clopen subset the resulting group is called a \emph{local subgroup of $G$}.
		\end{enumerate}
	\end{defi}
	
	A crucial ingredient is the abundance of certain involutions:
	\begin{lem}[\cite{gps99}, Lemma~3.3]\label{lem: abundinvo}
		Let $(X,\varphi)$ be a Cantor system and let a group $G \leq \operatorname{Homeo}(X)$ either be $\mathfrak{F}(\varphi)$ resp. $\mathfrak{T}(\varphi)$ resp. $\mathfrak{T}(\varphi)_{x}$\footnote{These groups are termed \emph{of class F} in \cite{gps99}, we reserve this terminology for a more general definition by Matui.\\} for some $x \in X$. Let $U,V$ be clopen subsets of $X$. Then the following are equivalent:
		\begin{enumerate}[(i)]
			\item $[\mathbf{1}_U]=[\mathbf{1}_V]$ in $K^0(X,\varphi)$ resp. $K^0(X,\varphi)/\operatorname{Inf}(K^0(X,\varphi))$ resp. $K^0(X,\mathfrak{T}(\varphi)_{x})$.
			
			\item There exists an element $g \in G$ such that $g(U)=V$.
			
			\item There exists an element $g \in G$ such that $g(U)=V$, $g^2=\operatorname{Id}$ and $g|_{(U \cup V)^{\mathrm{C}}}=\operatorname{Id}_{(U \cup V)^{\mathrm{C}}}$.
		\end{enumerate}
		If the action of $G$ on $X$ is minimal, for every clopen set $U \subseteq X$ and for every $x \in U$ there exists a $g \in G$ such that $g^2=\operatorname{Id}$, $g(x) \neq x$ and $g|_{U^\mathrm{C}}=\operatorname{Id}_{U^\mathrm{C}}$.
	\end{lem}
	\begin{rem}
		The only non-trivial implication is $(i) \Rightarrow (ii)$. In the case of $G=\mathfrak{F}(\varphi)$ it follows by "patching together" elements in $\mathfrak{T}(\varphi)$ that arise as in Lemma~\ref{lem: gw}. In the other cases it follows from the Bratteli-Vershik representations of the respective systems.
	\end{rem}
	
	\begin{thm}[\cite{gps99}, Theorem 4.2]\label{thm: spat}
		Let $(X_1,\varphi_1)$ and $(X_2,\varphi_2)$ be minimal Cantor systems and let $G_i$ either be $\mathfrak{F}(\varphi_i)$, $\mathfrak{T}(\varphi_i)$, or $\mathfrak{T}(\varphi_i)_{x_i}$ for $i \in\{1,2 \}$ and $x_i \in X_i$. Then any group isomorphism $\alpha \colon G_1 \to G_2$ is spatial.
	\end{thm}
	
	By Stone duality homeomorphisms between the Cantor spaces $X_1$ and $X_2$ are in 1-1 correspondence with Boolean isomorphisms between the associated Boolean algebras of clopen sets $\operatorname{CO}(X_1)$ and $\operatorname{CO}(X_2)$, thus it is enough to construct an Boolean isomorphism $a \colon \operatorname{CO}(X_1) \to \operatorname{CO}(X_2)$ such that $\alpha(g)a = a g$ for all $g \in G_1$. The key steps are that local subgroups characterize clopen sets and that local subgroups have an algebraic characterization. Then for every $U \in \operatorname{CO}(X_1)$ the group $\alpha((G_1)_U)$ proves to be a local subgroup of $\Gamma_2$ and thus corresponds to a clopen subset $a(U) \subseteq \operatorname{CO}(X_2)$ setting up a bijection $a \colon \operatorname{CO}(X_1) \to \operatorname{CO}(X_2)$. The proof is accomplished by showing $a$ preserves intersections and satisfies	$\alpha(g)a = a g$ for all $g \in G_1$. This theorem is in the vein of Dye's work, parts of the algebraic characterization of local subgroups are directly transferred from \cite{dye63}.	
	Combining Theorem~\ref{thm: spat} with the classification of Cantor systems obtained in \cite{gps95} (see Theorem~\ref{thm: orbiequi} resp. Theorem~\ref{thm: flipconju} resp. Remark~\ref{rem: thm 2.3}) shows that $\mathfrak{F}(\varphi)$, $\mathfrak{T}(\varphi)$ and $\mathfrak{T}(\varphi)_{\{x\}}$ are invariants for the respective types of equivalence:
	
	\begin{cor}[\cite{gps99}, Corollary 4.6]
		Let $(X_1,\varphi_1)$ and $(X_2,\varphi_2)$ be minimal Cantor systems. The following are equivalent:
		\begin{enumerate}[(i)]
			\item The systems $(X_1,\varphi_1)$ and $(X_2,\varphi_2)$ are orbit equivalent.
			
			\item The groups $\mathfrak{F}(\varphi_1)$ and $\mathfrak{F}(\varphi_2)$ are isomorphic.
		\end{enumerate}
	\end{cor}
	
	\begin{cor}[\cite{gps99}, Corollary 4.4]\label{cor: flipconj1}
		Let $(X_1,\varphi_1)$ and $(X_2,\varphi_2)$ be minimal Cantor systems. The following are equivalent:
		\begin{enumerate}[(i)]
			\item The systems $(X_1,\varphi_1)$ and $(X_2,\varphi_2)$ are flip conjugate.
			
			\item The groups $\mathfrak{T}(\varphi_1)$ and $\mathfrak{T}(\varphi_2)$ are isomorphic.
		\end{enumerate}
	\end{cor}
	
	\begin{cor}[\cite{gps99}, Corollary 4.11]
		Let $(X_1,\varphi_1)$ and $(X_2,\varphi_2)$ be minimal Cantor systems. The following are equivalent:
		\begin{enumerate}[(i)]
			\item The systems $(X_1,\varphi_1)$ and $(X_2,\varphi_2)$ are strong orbit equivalent.
			
			\item The groups $\mathfrak{T}(\varphi_1)_{x_1}$ and $\mathfrak{T}(\varphi_2)_{x_2}$ are isomorphic for all $x_1 \in X_1,x_2 \in X_2$.
		\end{enumerate}
	\end{cor}
	
	\cite{gps99} concludes by showing that isomorphisms between index map kernels are spatial, again by the algebraic characterization of local subgroups, which implies that $\mathfrak{T}(\varphi)_0$ is a complete invariant for flip-conjugacy.  
	\begin{cor}[\cite{gps99}, Corollary 5.18]\label{cor: flipconj2}
		Let $(X_1,\varphi_1)$ and $(X_2,\varphi_2)$ be minimal Cantor systems. The following are equivalent:
		\begin{enumerate}[(i)]
			\item The systems $(X_1,\varphi_1)$ and $(X_2,\varphi_2)$ are flip conjugate.
			
			\item The groups $\mathfrak{T}(\varphi_1)_0$ and $\mathfrak{T}(\varphi_2)_0$ are isomorphic.
		\end{enumerate}
	\end{cor}
	
	In \cite{bm08} the authors found the following adaption of Theorem~\ref{thm: spat} by a different approach:
	\begin{thm}[\cite{bm08}, Theorem 4.2]\label{thm: spat2}
		Let $(X_1,\varphi_1)$ and $(X_2,\varphi_2)$ be minimal Cantor systems and let $\Gamma_i$ either be  $\mathfrak{T}(\varphi_i)$, $\mathfrak{T}(\varphi_i)_0$ or $\mathfrak{T}(\varphi_i)'$ for $i \in\{1,2 \}$ and $x_i \in X_i$. Then any group isomorphism $\alpha \colon \Gamma_1 \to \Gamma_2$ is spatial.
	\end{thm}
	
	As a corollary the above list extends by:
	\begin{cor}\label{cor: flipconj3}
		Let $(X_1,\varphi_1)$ and $(X_2,\varphi_2)$ be minimal Cantor systems. The following are equivalent:
		\begin{enumerate}[(i)]
			\item The systems $(X_1,\varphi_1)$ and $(X_2,\varphi_2)$ are flip conjugate.
			
			\item The groups $\mathfrak{T}(\varphi_1)'$ and $\mathfrak{T}(\varphi_2)'$ are isomorphic.
		\end{enumerate}
	\end{cor}
	
	The above results permit to use dynamical properties invariant under the respective notions of equivalence to distinguish between groups of the respective type up to isomorphism, in particular, invariants of flip conjugacy e.g. topological entropy or ergodicity, allow to distinguish topological full groups up to isomorphy. By Remark~\ref{rem: complex+entro}(ii), Example~\ref{ex: sturm+toeplitz} and Corollaries~\ref{cor: flipconj1} \& \ref{cor: flipconj2} \& \ref{cor: flipconj3} we have:
	
	\begin{prop}\label{prop: uncount}
		There exists an uncountable family of pairwise non-ismorphic topological full groups of minimal subshifts.
	\end{prop}
	
	\section{Results by Matui}\label{sec: results by matui}
	
	In this section we review Matui's contributions to the study of topological full groups and their properties.	Subsection~\ref{subs: isomorphism theorems II} gives a outline of Matui's spatial realization theorem on \'etale Cantor groupoids from \cite{mat15}. Subsection~\ref{subs: simplicity and finite generation} recollects the results on topological full groups of minimal Cantor systems from \cite{mat06}. In Subsection~\ref{subs: generalizations} we turn to generalizations in the context of \'etale Cantor groupoids. Subsection~\ref{subs: on growth} finishes the section with Matui's proof of exponential growth of the topological full groups of minimal subshifts.
	
	\subsection{Isomorphisms theorems II}\label{subs: isomorphism theorems II}
	
	The results in the wake of \cite{dye63} are instance of a larger principle: The reconstruction of structures from their groups of transformations. The proof of Theorem~\ref{thm: spat} shows that isomorphisms between certain automorphism groups of the Boolean algebras $\operatorname{CO}(X_1)$ and $\operatorname{CO}(X_2)$ are generated by isomorphisms between $\operatorname{CO}(X_1)$ and $\operatorname{CO}(X_2)$. In \cite{bm08} the spatial realization of isomorphisms for the groups $\mathfrak{T}(\varphi)$, $\mathfrak{T}(\varphi)'$ and $\mathfrak{T}(\varphi)_0$ of a minimal Cantor system $(X,\varphi)$ is demonstrated by following the proof of Theorem~384D in \cite{fre11}. The same method was applied to the respective groups in the context of general minimal topological dynamical systems over Cantor spaces (see \cite{med11}). The flip conjugacy of Cantor minimal systems is equivalent to isomorphy of the corresponding transformation groupoids (\cite{mat16b}, Theorem 8.1), thus the groups $\mathfrak{T}(\mathcal{G})$, $\mathfrak{T}(\mathcal{G})'$ and $\mathfrak{T}(\mathcal{G})_0$ are complete isomorphism invariants for transformation groupoids. We take a look at Matui's proof of spatial realization modeled after Theorem~384D in \cite{fre11}:
	
	\begin{defi}[\cite{mat15}, Definition~3.1]\label{defi: class F}
		Let $X$ be a Cantor space. A subgroup $G \leq \operatorname{Homeo}(X)$ is said to be of \emph{class F} if it satisfies the following properties:
		\begin{enumerate}[(i)]
			\item [(F1)] For every $g \in G$, $g^2=1$ implies that $\supp(g)$ is clopen.
			
			\item [(F2)] For every clopen set $U \subseteq X$ and for every $x \in U$ there exists a $g \in G\setminus \{ 1\}$ such that $g^2=1$ and $x \in \supp(g)\subseteq U$.
			
			\item [(F3)] For every $g \in G\setminus \{ 1\}$ with $g^2=1$ and for every non-empty clopen set $U \subset \supp(g)$ there exists a $g' \in G \setminus \{ 1\}$ such that $\supp(g')\subset U \cup g(U)$ and $g|_{\supp(g')}= g'|_{\supp(g')}$.
			
			\item [(F4)] For every non-empty clopen subset $U \subseteq X$ there exists a $g \in G$ with $\supp(g) \subseteq U$ and $g^2\neq 1$.
		\end{enumerate}
	\end{defi}
	
	Property (F2) immediately allows to encode the order structure on regular\footnote{A closed subset $A$ of a topological space $X$ is called \emph{regular} if $A=\overline{A^{\circ}}$} closed subsets of $X$ by local subgroups:
	\begin{lem}[\cite{mat15}, Lemma~3.2]\label{lem: local subgroups encode order structure}
		Let $X$ be a Cantor space and let $G \leq \operatorname{Homeo}(X)$ be a group of class F. Then for a pair of regular closed subsets $A,B \subseteq X$ the following are equivalent:
		\begin{enumerate}[(i)]
			\item $A \subseteq B$
			
			\item $G_A \subseteq G_B$
		\end{enumerate}
	\end{lem}
	
	Next, local subgroups over supports of involutions are algebraically characterized:
	
	\begin{defi}[\cite{mat15}, p.7]\label{defi: needed for characterization of local subgroups}
		Let $X$ be a Cantor space, let $G \leq \operatorname{Homeo}(X)$ be a group of class F and let $\tau \in G$ be a non-trivial involution. Define:
		\begin{enumerate}[(i)]
			\item $C_\tau:=\{g \in G| \tau g \tau = g \}$
			
			\item $U_\tau:=\{g \in C_\tau| g^2=1,gcgc^{-1}=cgc^{-1}g \; \forall c \in C_\tau \}$
			
			\item $S_\tau:=\{g^2 \in G|g \in G, u g u^{-1}=g \; \forall u \in U_\tau \}$
			
			\item $W_\tau:=\{g \in G| sgs^{-1}=g \;\forall s \in S_\tau\}$
		\end{enumerate}
	\end{defi}
	
	\begin{lem}[\cite{mat15}, Lemma~3.3]\label{lem: characterization local subgroups}
		Let $X$ be a Cantor space, let $G \leq \operatorname{Homeo}(X)$ be a group of class F and let $\tau \in G$ be a non-trivial involution. Then $W_\tau = G_{\supp(\tau)}$.
	\end{lem}
	
	Lemma~\ref{lem: local subgroups encode order structure} and Lemma~\ref{lem: characterization local subgroups} then imply:
	
	\begin{lem}[\cite{mat15}, Lemma~3.4]\label{lem: containment and intersection of supports}
		Let $X_1, X_2$ be Cantor spaces, let $G_1 \leq \operatorname{Homeo}(X_1)$ and $G_2 \leq \operatorname{Homeo}(X_2)$ be groups of class F and let $\Phi\colon G_1 \to G_2$ be an isomorphism of groups. Let $\tau,\sigma \in G_1$ be involutions. Then the following equivalences hold:
		\begin{enumerate}[(i)]
			\item $\supp(\tau)\subseteq \supp(\sigma) \Leftrightarrow \supp(\Phi(\tau)) \subseteq \supp(\Phi(\sigma))$
			
			\item $\supp(\tau) \cap \supp(\sigma)=\emptyset \Leftrightarrow \supp(\Phi(\tau)) \cap \supp(\Phi(\sigma))=\emptyset$
		\end{enumerate} 
	\end{lem}
	
	This allows to construct a spatial homeomorphism:
	
	\begin{thm}[\cite{mat15}, Theorem~3.5]\label{thm: matuispat}
		For $i \in \{1,2\}$ let $X_i$ be Cantor spaces and let $G_i \leq \operatorname{Homeo}(X_i)$ be groups of class F. Let $\alpha \colon G_1 \to G_2$ be an isomorphism and let $x \in X_1$. Fix the sets
		\begin{equation*}
		T(x):=\{g \in G_1|x \in \supp(g), g^2=1\}
		\end{equation*}
		and
		\begin{equation*}
		P(x):= \bigcap_{g \in T(x)} \supp(\alpha(g)).
		\end{equation*}
		Then $P(x)$ is a singleton for every $x \in X$ and the map $\varphi \colon X_1 \to X_2$ given by $\varphi(x):=P(x)$ for $x \in X_1$ is a homeomorphism such that $\alpha(g)=\varphi \circ g \circ \varphi^{-1}$ for all $g \in G_1$.
	\end{thm}
	
	Matui proceeds by showing:
	
	\begin{prop}[\cite{mat15}, Proposition~3.6]\label{prop: classF}
		Let $\mathcal{G}$ be a minimal, effective, \'etale Cantor groupoid. Then every subgroup $G\leq\mathfrak{T}(\mathcal{G})$ with $\mathfrak{T}(\mathcal{G})' \leq G$ is a subgroup of $\operatorname{Homeo}(\mathcal{G}^{(0)})$ of class F. 
	\end{prop}
	\begin{proof}
		It is sufficient to show that $\mathfrak{T}(\mathcal{G})'$ is of class F.
		\item [(F1)] This holds by definition.
		
		The properties (F2), (F3) and (F4) follow from minimality by Corollary~\ref{cor: existence of involutions}.
		
		\item [(F2)] Let $U \subseteq \mathcal{G}^{(0)}$ be non-empty and clopen and let $x \in U$. Then there exist non-empty, compact, open bisections $W,V \in \mathcal{B}_\mathcal{G}^{o,k}$ such that $s(W) \cup r(W) \subseteq U$, $x \in s(W)$, $s(W) \cap r(W) = \emptyset$, $s(V) \cup r(V) \subseteq s(W)$, $x \in s(V)$ and $s(V) \cap r(V) = \emptyset$. The element $[\tau_W,\tau_V]$ is the required involution.
		
		\item [(F3)] Let $\gamma$ be a non-trivial involution in $\mathfrak{T}(\mathcal{G})$ and let $U \subseteq \supp(\gamma)$ be non-empty and clopen. There exists a non-empty, compact, open bisection $W \in \mathcal{B}_\mathcal{G}^{o,k}$ with $s(W) \cup r(W) \subseteq U$ such that $s(W)$, $r(W)$, $\gamma(s(W))$ and $\gamma(r(W))$ are pairwise disjoint. Define $\tilde{\gamma} \in \mathfrak{T}(\mathcal{G})$ by
		\begin{equation*}
		\tilde{\gamma}(x):=
		\begin{cases}
		\gamma(x), & \text{ if } x \in s(W) \cup \gamma(s(W))\\
		x, & \text{ else}
		\end{cases}
		\end{equation*}
		The element $[\tilde{\gamma},[\gamma,\tau_W]]$ is sufficient.
		
		\item [(F4)] Let $U \subseteq \mathcal{G}^{(0)}$ be non-empty and clopen. Then there exist non-empty, compact, open bisections $W,V \in \mathcal{B}_\mathcal{G}^{o,k}$ such that $s(W) \cup r(W) \subseteq U$, $s(V) \cup r(V) \subseteq U$, the sets $s(W)$, $r(W)$ and $r(V)$ are pairwise disjoint and $s(V) \subseteq s(W)$. Then the element $[(\tau_{(Ws(V))})^{-1},\tau_V]$ has order three.					
	\end{proof}

	\begin{prop}[\cite{mat15}, Proposition~3.8]\label{prop: spatial}
		For $i \in \{1,2\}$ let $\mathcal{G}_i$ be a minimal, effective, \'etale Cantor groupoid and let $G_i$ be subgroups of $\mathfrak{T}(\mathcal{G}_i)$ with $\mathfrak{T}(\mathcal{G}_i)' \leq G_i$. Then every spatial isomorphism $\alpha \colon G_1 \to G_2$ induces an isomorphism of groupoids between $\mathcal{G}_1$ and $\mathcal{G}_2$.
	\end{prop}
	\begin{proofsketch}
		Let $h \colon G_1^{(0)} \to G_2^{(0)}$ such that $\alpha\colon G_1 \to G_2$ is given by $\gamma \mapsto h \circ \gamma \circ h^{-1}$. By minimality every $g \in \mathcal{G}_1$ is contained in an open, compact slice $B \in \mathcal{B}_\mathcal{G}^{o,k}$ such that the associated element $\alpha_B \in \mathfrak{T}(\mathcal{G})$ is in $\mathfrak{T}(\mathcal{G})'$ (This is verified similarily as the proof of (F4) in Proposition~\ref{prop: classF}). For every $g \in \mathcal{G}_1$ the element $g':= (s|_{h \circ B \circ h^{-1} })^{-1} (h(s(g))) \in \mathcal{G}^2$ is well defined and the map induced by $g \to g'$ is a groupoid isomorphisms between $\mathcal{G}_1$ and $\mathcal{G}_2$.
	\end{proofsketch}
	
	Theorem~\ref{thm: matuispat}, Proposition~\ref{prop: classF} and Proposition~\ref{prop: spatial} then imply:
	
	\begin{thm}[\cite{mat15}, Theorem 3.10]\label{thm: isom}
		Let $\mathcal{G}_1, \mathcal{G}_2$ be effective, minimal, Hausdorff, \'etale Cantor groupoids. The following are equivalent:
		\begin{enumerate}[(i)]
			\item $\mathcal{G}_1 \cong \mathcal{G}_2$
			
			\item $\mathfrak{T}(\mathcal{G}_1) \cong \mathfrak{T}(\mathcal{G}_2)$
			
			\item $\mathfrak{T}(\mathcal{G}_1)' \cong \mathfrak{T}(\mathcal{G}_2)'$
			
			\item $\mathfrak{T}(\mathcal{G}_1)_0 \cong \mathfrak{T}(\mathcal{G}_2)_0$
		\end{enumerate}
	\end{thm}
	
	Later on it was noticed that spatial realization follows directly from reconstruction results of Matatyahu Rubin:
	
	\begin{defi}[\cite{mat16}, §5]
		Let $X$ be a topological space.	A subgroup $G \subseteq \operatorname{Homeo}(X)$ is called \emph{locally dense} if for every $x \in X$ and every open set $x \in U \subseteq X$ the set $\{f(x)|f \in G, f|_{X \setminus U} = \operatorname{id}_{X \setminus U} \}$ is somewhere dense.
	\end{defi}
	
	The following theorem is a direct consequence of Corollary 3.5 in \cite{rub89}:
	
	\begin{thm}[\cite{mat16}, Theorem 5.3]\label{thm: rubin}
		Let $X_1,X_2$ be locally compact, Hausdorff topological spaces containing no isolated points. Let $G_1 \subseteq \operatorname{Homeo}(X_1), G_2 \subseteq \operatorname{Homeo}(X_2)$ be isomorphic, locally dense groups of homeomorphisms.
		Then for every isomorphism $\varphi \colon G_1 \to G_2$, there exists a unique homeomorphism $\phi \colon X_1 \to X_2$ such that $\varphi(g)=\phi \circ g \circ \phi^{-1}$.
	\end{thm}
	
	\subsection{Simplicity and finite generation}\label{subs: simplicity and finite generation}
	
	In \cite{mat06}, the first paper primarily concerned with algebraic properties of $\mathfrak{T}(\varphi)$, Matui showed that $\mathfrak{T}(\varphi)'$ is simple by characterization of the abelianization of $\mathfrak{T}(\varphi)$ in terms of $K_0(C^*(X,\varphi))$ and in the case of a subshift finitely generated.
	Matui applies C*-algebraic methods, in that he uses a surjective homomorphism $\mathfrak{T}(\varphi)_0 \to K_0(C^*(X,\varphi)) \otimes (\mathbb{Z}/2\mathbb{Z}) $ which is a modification of the \emph{mod map} (\cite{gps99}, Definition 2.10) -- thus the name \emph{signature map}. To set up the signature map, Matui proves a structural lemma on $\mathfrak{T}(\varphi)_0$ describing it in terms of the locally finite subgroups $\mathfrak{T}(\varphi)_{\{x\}}$.

	In Remark~\ref{rem: af sign}, we mentioned the isomorphism $\mathfrak{T}(\varphi)_{\{x\}}^{\operatorname{ab}} \cong K_0(C^*(X,\varphi)) \otimes (\mathbb{Z}/2\mathbb{Z})$. The preomposition with the quotient homomorphism induces a homomorphism
	\begin{equation*}
		\operatorname{sgn}_x \colon \mathfrak{T}(\varphi)_{\{x\}} \to K_0(C^*(X,\varphi)) \otimes (\mathbb{Z}/2\mathbb{Z}). 
	\end{equation*}The group homomorphism $\operatorname{sgn}$ on $\mathfrak{T}(\varphi)_0$ then is obtained by patching together the maps $\{\operatorname{sgn}_x\}_{x \in X}$ via the factorization obtained in Proposition~\ref{prop: mat.fac}. The homomorphism $\operatorname{sgn}_x$ can be expressed in the following way:
	
	Let $\gamma \in \mathfrak{T}(\varphi)_{\{x\}}$. For every $x \in X$, define $g_\gamma(x)$ to be the smallest positive integer $k$ such that $\gamma^k(x)=x$. This sets up a continuous function $g_\gamma \colon X \to \mathbb{N}$. Let $\{\mathcal{A}_n\}_{n \in \mathbb{N}}$ be a nested sequence of Kakutani-Rokhlin partitions such that the cocycle $f_\gamma$ is constant on atoms, then the function $g_\gamma$ is constant on atoms. Hence for every $k \in \mathbb{N}$ with $g_\gamma^{-1}(k)$ non-empty and there exists a union $\mathcal{D}_k = \bigcup_{i \in I} D_i$ of atoms of the partition $\{\mathcal{A}_n\}_{n \in \mathbb{N}}$ such that the sets $\mathcal{D}_k,\gamma(\mathcal{D}_k),\dots, \gamma^{k-1}(\mathcal{D}_k)$ are pairwise disjoint and their union is $g_\gamma^{-1}(k)$. Define a function $G_\gamma \in C(X,\mathbb{Z})$ by $G_\gamma=\sum_{k \in 2\mathbb{N}} \mathbf{1}_{D_{k}}$. This function allows to characterize $\operatorname{sgn}_x$:
	\begin{lem}[\cite{mat06}, Lemma~3.7]\label{lem: sgnx}
		Let $(X,\varphi)$ be a minimal Cantor system and let $\gamma \in \mathfrak{T}(\varphi)_{\{x\}}$. Then $\operatorname{sgn}_x(\gamma) = [G_\gamma] + 2K_0(C^*(X,\varphi)) \otimes (\mathbb{Z}/2\mathbb{Z})$. 
	\end{lem}
	
	This characterization implies $\operatorname{sgn}_x|_{\mathfrak{T}(\varphi)_{\{x\}} \cap\mathfrak{T}(\varphi)_{\{y\}}}= \operatorname{sgn}_y|_{\mathfrak{T}(\varphi)_{\{x\}} \cap \mathfrak{T}(\varphi)_{\{y\}}}$ for all $x,y \in X$. It holds that the map $\operatorname{sgn}_x$ is surjective on intersecions $\mathfrak{T}(\varphi)_{\{x\}}\cap \mathfrak{T}(\varphi)_{\{y\}}$ (\cite{mat06}, Lemma~4.3). Furthermore, let $\gamma \in \mathfrak{T}(\varphi)_0$. By Proposition~\ref{prop: mat.fac} for any $x,y \in X$ with $\operatorname{Orb}_\varphi(x) \neq \operatorname{Orb}_\varphi(y)$ there exist $\gamma_x \in \mathfrak{T}(\varphi)_{\{x\}}$ and $\gamma_y \in \mathfrak{T}(\varphi)_{\{y\}}$ with $\gamma=\gamma_x\gamma_y$. Define $\operatorname{sgn}(\gamma)=\operatorname{sgn}_x(\gamma_x)+\operatorname{sgn}_y(\gamma_y)$. By Lemma~4.4 in \cite{mat06}, $\operatorname{sgn}(\gamma)$ neither depends on the choice of $x,y \in X$ nor on the choice of factors $\gamma_x,\gamma_y$. Thus one obtains a well-defined surjective map $\operatorname{sgn} \colon \mathfrak{T}(\varphi)_0 \to K_0(C^*(X,\varphi)) \otimes (\mathbb{Z}/2\mathbb{Z})$, which is in fact a group homomorphism (\cite{mat06}, Proposition~4.6).	
	Equipped with this homomorphism, Matui characterizes the abelianization of $\mathfrak{T}(\varphi)_0$:				
	\begin{thm}[\cite{mat06}, Theorem 4.8.]\label{thm: sgn}
		Let $(X,\varphi)$ be a minimal Cantor system. Then the signature map satisfies $\ker(\operatorname{sgn})=\mathfrak{T}(\varphi)_0'$ giving rise to the isomorphism $\mathfrak{T}(\varphi)_0^\mathrm{ab} \cong K_0(C^*(X,\varphi)) \otimes (\mathbb{Z}/2\mathbb{Z})$.
	\end{thm}
	\begin{proofsketch}
		The inclusion $\mathfrak{T}(\varphi)_0'\subseteq \ker(\operatorname{sgn})$ is obvious. Assume $\gamma \in \ker(\operatorname{sgn})$. Let $x,y \in X$ with $\operatorname{Orb}_\varphi(x) \neq \operatorname{Orb}_\varphi(y)$. Then by Lemma~\ref{lem: 2-div.}, we can choose $U$ in the proof of Proposition~\ref{prop: mat.fac} such that $[\mathbf{1}_U]$ is $2$-divisible and thus by Lemma~\ref{lem: sgnx}, the element $\tilde{\gamma} \in \mathfrak{T}(\varphi)$ satisfies $\operatorname{sgn}(\tilde{\gamma}) \in \ker(\operatorname{sgn}_y)=\mathfrak{T}(\varphi)_{\{y\}}'$. In consequence, $\gamma \tilde{\gamma} \in \ker(\operatorname{sgn}_x)=\mathfrak{T}(\varphi)_{\{x\}}'$ holds and therefore $\gamma$ is a product of commutators. 
	\end{proofsketch}
	
	\begin{cor}\label{cor: strong AH-conjecture for systems}
		Let $(X,\varphi)$ be a minimal Cantor system. Then we have
		\begin{equation*}
		\mathfrak{T}(\varphi)^{\mathrm{ab}}=\mathbb{Z} \oplus ( K_0(C^*(X,\varphi)) \otimes (\mathbb{Z}/2\mathbb{Z})).
		\end{equation*}
		
	\end{cor}
	
	Furthermore, Matui proves with the help of the signature map:
	\begin{thm}[\cite{mat06}, Theorem 4.9.]\label{thm: simpel}
		Let $(X,\varphi)$ be a minimal Cantor system. Then $\mathfrak{T}(\varphi)'$ is a simple group.
	\end{thm}
	\begin{proof}
		Let $H$ be a non-trivial normal subgroup of $\mathfrak{T}(\varphi)'=\mathfrak{T}(\varphi)_0'$. Let $\gamma \in H$ be a non-trivial element and define $l:=\max\{f_\gamma(x)|x \in X\}$. Fix an arbitrary $x \in X$. Then there exists a $y \in X$ with $\operatorname{Orb}_\varphi(x) \neq \operatorname{Orb}_\varphi(y)$ such that $y \neq \gamma(y) \neq \varphi(y)$ and furthermore a clopen neighbourhood $U$ of $x$ such that $U \cap \varphi(U)=\emptyset$, such that $U$ does not contain the elements $\varphi(y),\gamma(y),\varphi^{-1}\gamma(y)$ and $\varphi^k(x)$ for $-l-1 \leq k \leq l+1$ and, by Lemma~\ref{lem: 2-div.}, such that $\mathbf{1}_U$ is $2$-divisible in $K^0(X,\varphi)$. Then the homeomorhpism $\tau_U$ defined by 
		\begin{equation*}
		\tau_U(x):=
		\begin{cases}
		\varphi (x), & \text{ if } x \in U\\
		\varphi^{-1}(x), & \text{ if } x \in \varphi(U)\\
		x, & \text{ else}
		\end{cases}
		\end{equation*}
		is by definition contained in $\mathfrak{T}(\varphi)_{\{x\}}$ and by the description of $\operatorname{sgn}_x$ in Lemma~3.7 in \cite{mat06} we have $\tau_u \in \ker(\operatorname{sgn}_x)$. Thus the non-trivial homeomorphism $\tau_U \circ \gamma \circ \tau_U\circ \gamma^{-1}$ is contained in $\mathfrak{T}(\varphi)_{\{x\}}' \cap H$. But since $\mathfrak{T}(\varphi)_{\{x\}}'$ is simple by Proposition~\ref{prop: loc.fin.simpl.}, it holds that $\mathfrak{T}(\varphi)_{\{x\}}' \subseteq H$ for all $x \in X$ and thus $H=\mathfrak{T}(\varphi)_0'$.
	\end{proof}
	
	Theorem~\ref{thm: sgn} and Theorem~\ref{thm: simpel} combined imply:
	\begin{cor}[\cite{mat06}, Corollary 4.10.]
		Let $(X,\varphi)$ be a minimal Cantor system. The group $\mathfrak{T}(\varphi)_0$ is simple \emph{if and only if} $K_0(C^*(X,\varphi))$ is 2-divisible.
	\end{cor}
	
	In \cite{bm08} the authors obtained a different proof of simplicity of $\mathfrak{T}(\varphi)'$ by more approachable techniques. Just as Matui they rely on the Lemma of Glasner and Weiss (see Lemma~\ref{lem: gw}).
	
	Matui demonstrated that in the case of minimal subshifts the group $\mathfrak{T}(\varphi)'$ is finitely generated, that the groups $\mathfrak{T}(\varphi)$, $\mathfrak{T}(\varphi)'$ and $\mathfrak{T}(\varphi)_0$ cannot be finitely presented and gave a sufficient and necessary condition for the finite generation of $\mathfrak{T}(\varphi)$ and $\mathfrak{T}(\varphi)_0$. The proofs will be omited, since, firstly, in Subsection~\ref{subs: finite generation of alternating full groups} a generalization of Theorem~\ref{thm: fingen} will be given and, secondly, finite presentation for minimal subshifts follows as a corollary from results discussed in Subsection~\ref{subs: the lef-property}.
	
	\begin{thm}[\cite{mat06}, Theorem 5.4.]\label{thm: fingen}
		Let $(X,\varphi)$ be a minimal Cantor system. The group $\mathfrak{T}(\varphi)'$ is finitely generated \emph{if and only if} $(X,\varphi)$ a minimal subshift.
	\end{thm}
	
	Combining Theorem~\ref{thm: fingen} with Theorem~\ref{thm: sgn} gives:
	\begin{cor}[\cite{mat06}, Corollary 5.5.]\label{cor: fingen2}
		Let $(X,\varphi)$ be a minimal Cantor system. The following are equivalent:
		\begin{enumerate}[(i)]
			\item $(X,\varphi)$ is a minimal subshift and $K_0(C^*(X,\varphi)) \otimes (\mathbb{Z}/2\mathbb{Z})$ is finite.
			
			\item $\mathfrak{T}(\varphi)_0$ is finitely generated.
			
			\item $\mathfrak{T}(\varphi)$ is finitely generated.
		\end{enumerate}
	\end{cor}
	
	Matui's proof of finite presentation follows by contradiction relying on results on minimal subshifts:
	\begin{thm}[\cite{mat06}, Theorem 5.7]
		Let $(X,\varphi)$ be a minimal subshift. Then $\mathfrak{T}(\varphi)'$ cannot be finitely presented.
	\end{thm}
	
	Using the facts that a finite index subgroup $H$ of a group $G$ is finitely presented \emph{if and only if} the group $G$ is finitely presented and that finite presentability is stable under extensions (see Chapter~V of \cite{dlh00}), gives the following corollary:
	
	\begin{cor}[\cite{mat06}, Corollary 5.8]
		Let $(X,\varphi)$ be a minimal subshift. Then the groups $\mathfrak{T}(\varphi)$ and $\mathfrak{T}(\varphi)_0$ cannot be finitely presented.
	\end{cor}
	
	\subsection{Generalizations}\label{subs: generalizations}
	
	By Proposition~\ref{prop: alm.fin.}, the transformation group of free $\mathbb{Z}^n$ actions on a Cantor space are almost finite. Thus the following simplicity results of Matui in particular generalize Theorem~\ref{thm: simpel}:
	
	\begin{thm}[\cite{mat13}, Theorem~4.7]\label{thm: almsimpl}
		Let $\mathcal{G}$ be an almost finite, minimal, effective, \'etale Cantor groupoid. Any non-trivial subgroup of $\mathfrak{T}(\mathcal{G})$ normalized by $\mathfrak{T}(\mathcal{G})'$ contains $\mathfrak{T}(\mathcal{G})'$. Thus $\mathfrak{T}(\mathcal{G})'$ is simple.
	\end{thm}
	
	\begin{thm}[\cite{mat13}, Theorem~4.16]\label{thm: purely infinite, minimal groupoids, then commutat of topfgrp is simple}
		Let $\mathcal{G}$ be a purely infinite, minimal, effective, \'etale Cantor groupoid. Any non-trivial subgroup of $\mathfrak{T}(\mathcal{G})$ normalized by $\mathfrak{T}(\mathcal{G})'$ contains $\mathfrak{T}(\mathcal{G})'$. Thus $\mathfrak{T}(\mathcal{G})'$ is simple.
	\end{thm}
	
	Both of these simplicity theorems hinge on a generalization of Lemma~\ref{lem: gw}. In the almost finite case this is represented by Lemma~6.7 in \cite{mat12}. In the purely infinite case it comes for free by the structure of the groupoid by the following lemma -- which we state because of its use in Subsection~\ref{subs: topfullgroups koopman}:
	
	\begin{prop}[\cite{mat15}, Proposition 4.11.]\label{prop: purely infinite construction}
		\quad Let $\mathcal{G}$ be an \'etale Cantor groupoid. The following are equivalent:
		\begin{enumerate}[(i)]
			\item $\mathcal{G}$ is purely infinite and minimal.
			
			\item For every pair of clopen sets $A,B \subset \mathcal{G}^{(0)}$ with $B \neq \emptyset$, there exists a compact, open slice $U \in \mathcal{B}_\mathcal{G}^{o,k}$ with $s(U)=A$ and $r(U) \subset B$.
			
			\item For every pair of clopen sets $A,B \subset \mathcal{G}^{(0)}$ with $A \neq \mathcal{G}^{(0)} $ and $B \neq \emptyset$, there exists an $S \in \mathfrak{T}(\mathcal{G})$ such that $r(SA) \subset B$.
		\end{enumerate}
	\end{prop}
	\begin{proof} 
		\emph{(i) $\Rightarrow$ (ii):} Since $B$ is properly infinite, there exist compact open slices $U,V$ satisfying $s(U)=s(V)=B$, $r(U) \cup r(V) \subset B$ and $r(U) \cap r(V) =\emptyset$. The family $\{V_n\}_{n \in \mathbb{N}}$ of compact, open slices defined inductively by $V_1:=U$ and $V_{n+1}=VV_n$ satisfies $s(V_n)=B$, $r(V_n) \subset B$ and $r(V_n) \cap r(V_m)=\emptyset$ for $n\neq m$. Since $\mathcal{G}$ is minimal there exists a finite family of slices $\{W_i\}_{i \in I}$ as chosen in Lemma~\ref{lem: existence of pencils}. Then $U=\bigcup_{i \in I} V_iW_i$ is the required slice. 
		
		\emph{(ii) $\Rightarrow$ (iii):} \emph{$B \setminus A$ is not empty}: By (ii) there exists a compact open slice $U$ with $s(U)=A$ and $r(U) \subset B \setminus A$. The compact, open slice $S := U \cup U^{-1} \cup (\mathcal{G}^{(0)}\setminus (s(U) \cup r(U))$ is sufficient.
		
		\emph{$B \setminus A$ is empty}: We have $B \subset A$. By the previous case, there exist $S_1,S_2 \in \mathfrak{T}(\mathcal{G})$ with $r(S_1 A)\subset \mathcal{G}^{(0)}\setminus A$ and $r(S_2(\mathcal{G}^{(0)}\setminus A)) \subset B$. The element $S:=S_2S_1$ is sufficient.
		
		\emph{(iii) $\Rightarrow$ (i):} Minimality is immediate, as by this property $ \{Su|S \in \mathfrak{T}(\mathcal{G})\}$ is dense in $\mathcal{G}^{(0)}$ for every $u \in \mathcal{G}^{(0)}$. Let $A$ be a non-empty clopen subset of $\mathcal{G}^{(0)}$. Let $B_1,B_2,B_3 \subseteq A$ be non-empty clopen and pairwise disjoint. Let $S_1,S_2 \in \mathfrak{T}(\mathcal{G})$ such that $r(S_1(A\setminus B_1))\subset B_2$ and $r(S_2(B_1 \cup B_2)) \subset B_3$. Then $U=B_1 \cup S_1(A\setminus B_1)$ and $V=S_2U$ are compact open slices with $s(U)=s(V)=A$, $r(U)\subset B_1 \cup B_2$ and $r(V) \subset B_3$, thus $A$ is properly infinite.
	\end{proof}
	
	In light of above simplicity results and of Subsection~\ref{subs: simplicity and finite generation}, it makes sense to study the abelization of $\mathfrak{T}(\mathcal{G})$. Matui formulates the following conjecture, which he calls \emph{AH-conjecture}:
	\begin{conj}[\cite{mat16}, Conjecture~2.9]\label{conj: AH}
		Let $\mathcal{G}$ be a minimal, effective, \'etale Cantor groupoid. Then the following sequence is exact
		\begin{equation*}
		H_0(\mathcal{G})\otimes (\mathbb{Z}/2\mathbb{Z}) \overset{j}{\longrightarrow}\mathfrak{T}(\mathcal{G})^{\mathrm{ab}} \overset{\overline{I}}{\longrightarrow} H_1(\mathcal{G}) \longrightarrow 0
		\end{equation*}
	i.e. $\mathfrak{T}(\mathcal{G})^{\mathrm{ab}}$ is an extension of $H_1(\mathcal{G})$ by a quotient of $H_0(\mathcal{G})\otimes (\mathbb{Z}/2\mathbb{Z})$.
	\end{conj}	
	
	\begin{rem}[\cite{mat16}, Remark~2.10]
		At first Matui considered a stronger version where the sequence is short exact. While it is true for some classes of groupoids, there are counterexamples highlighted by Volodymyr V. Nekrashevych. This stronger version is termed the \emph{strong AH-conjecture}.
	\end{rem}
	
	Since $\mathfrak{T}(\mathcal{G})' \leq \mathfrak{T}(\mathcal{G})_0$, the homomorphism $\overline{I}$ induced by the index map $I$ is well-defined. Of course we need to clarify how the homomorphism $j$ comes about: By the description of $H_0(\mathcal{G})\otimes (\mathbb{Z}/2\mathbb{Z})$ given in Proposition~\ref{prop: 2-homo}, it is sufficient to define the images $j(\mathbf{1}_U)$ for clopen subsets $U \subset \mathcal{G}^{(0)}$. By minimality Lemma~\ref{lem: existence of pencils} applies and there exists a finite family $\{B_i\}_{i \in I}$ of compact open slices such that $r(B_i) \subset \mathcal{G}^{(0)}\setminus U$, $U=\bigcup_i s(B_i)$ and $s(B_i)\cap s(B_j)=\emptyset$ for $i\neq j$. Define $j(\mathbf{1}_U)$ to be the class $[\prod_{i \in I} T_{B_i}] \in \mathfrak{T}(\mathcal{G})^{\mathrm{ab}}$. Verifying the AH-conjecture requires to show that this map is even well-defined and that the arising sequence is exact, in particular, that the index map $I$ is surjective. Matui was able to show surjectivity of $I$ in the purely infinite and in the almost finite case:
	\begin{thm}[\cite{mat12}, Theorem~7.5 \& \cite{mat15}, Theorem~5.2]\label{thm: indexsurj}
		Let $\mathcal{G}$ be an effective, \'etale Cantor groupoid. If $G$ is either almost finite or purely finite, the index map $\operatorname{I} \colon \mathfrak{T}(\mathcal{G}) \to H_1(\mathcal{G})$ is surjective.
	\end{thm}
	
	With regards to Conjecture~\ref{conj: AH} Matui obtained:
	
	\begin{thm}
		The AH-conjecture is satisfied by \'etale Cantor groupoids which are:
		\begin{enumerate}[(i)]
			\item \emph{(\cite{mat16}, Theorem~3.6)} principal, minimal and almost finite.
			
			\item \emph{(\cite{mat16}, Theorem~5.8)} a finite product of SFT-groupoids.
			
			\item \emph{(Corollary~\ref{cor: strong AH-conjecture for systems})} the transformation groupoid of a minimal Cantor system.\footnote{For transformation groupoids associated to minimal $\mathbb{Z}^n$-actions on Cantor spaces with $n>1$ it is only known if the action is free -- by Case (i).\\}
		\end{enumerate}
	\end{thm}
	
	Assume that $j$ is well-defined, then Corollary~\ref{cor: transpositions are in the index kernel} implies $\Ima j \subseteq \ker \bar{I}$. One obtains $\ker \bar{I} \subseteq \Ima j$ in particular if $\mathfrak{T}(\varphi)$ is generated by transpositions of the form $T_B$. In the case of minimal Cantor systems this is represented by Proposition~\ref{prop: mat.fac}. In \cite{mat12}, Matui considered the subgroup of \emph{elementary homeomorphisms} in $\mathfrak{T}(\mathcal{G})$ i.e. the subgroup of homeomorphisms $\gamma \in \mathfrak{T}(\mathcal{G})$ of finite order such that $\{x \in \mathcal{G}^{(0)}| \gamma^n(x)=x \}$ is clopen for all $n \in \mathbb{N}$ and related this subgroup with the index map kernel $\mathfrak{T}(\mathcal{G})_0$ in the almost finite, principal case:
	
	\begin{thm}[\cite{mat12}, Theorem~7.13]\label{thm: ker in fin}
		Let $\mathcal{G}$ be an almost finite, effective, \'etale Cantor groupoid. Then every element of $\mathfrak{T}(\mathcal{G})_0$ can be written as product of four elements in $\mathfrak{T}(\mathcal{G})$ of finite order.
	\end{thm}
	
	In principal, \'etale Cantor groupoids an element in $\mathfrak{T}(\mathcal{G})$ is elementary \emph{if and only if} it is of finite order.
	
	\subsection{On growth}\label{subs: on growth}
	
	In \cite{mat13}, Matui shows that if $\mathfrak{T}(\varphi)'$ is finitely generated, the group $(X,\varphi)$ necessarily has exponential growth.\footnote{See Appendix~\ref{app: growth of groups} for definition of the growth rate of a finitely generated group.\\} This follows from the fact that minimal Cantor systems which are not odometers -- these in particular encompasses all minimal subshifts -- are precisely those of which $\mathfrak{T}(\varphi)'$ contains a copy of the \emph{lamplighter group $L$}.
	
	\begin{defi}\label{defi: lamp}
		The semi-direct product $L:=(\bigoplus_\mathbb{Z} \mathbb{Z}/2\mathbb{Z}) \rtimes \mathbb{Z}$ where $\mathbb{Z}$ acts on the copies of $\mathbb{Z}/2\mathbb{Z}$ by the shift is called the \emph{lamplighter group}. A standard presentation of the lamplighter group is
		\begin{equation*}
		L=\big\langle x,y|x^2, [x,y^n x y^{-n}]_{n \in \mathbb{N}} \big\rangle.
		\end{equation*}
	\end{defi}
	\begin{prop}[\cite{dlh00}~VII.A.1]
		Any finitely generated group that contains a free sub-semigroup on 2 generators has exponential growth.
	\end{prop}
	\begin{rem}\label{rem: expgro}
		As the subsemigroup in $L$ with above presentation generated by $y$ and $xy$ is free, $L$ has exponential growth.
	\end{rem}
	
	The following is immediate from the structure of topological full groups of odometer systems given by Proposition~\ref{prop: topfull groups of odometers are unions of permutational wreath products}:
	
	\begin{cor}[\cite{mat13}, Proposition~2.1]
		Let $(X,\varphi)$ be a minimal Cantor system. If $(X,\varphi)$ is an odometer, then every finitely generated subgroup of $\mathfrak{T}(\varphi)$ is virtually abelian and thus has polynomial growth and contains no free sub-semigroups.
	\end{cor}
	
	This implies that for any odometer $(X,\varphi)$, the groups $\mathfrak{T}(\varphi)$ and $\mathfrak{T}(\varphi)'$ can not contain copies of the lamplighter group $L$.

	\begin{lem}[\cite{mat13}, Proposition~2.3]\label{lem: no-odo.}
		Let $(X,\varphi)$ be a minimal Cantor system. If it is not an odometer, then there exists a clopen subset $C \subset X$ such that for every finite subset $F$ of $\mathbb{Z}$ the following holds for all $x \in X$:
		\begin{equation*}
		f_{C,F}(x):=\sum_{k \in F} \mathbf{1}_C \circ \varphi^k(x) \equiv 0{\pmod 2}
		\end{equation*}
	\end{lem}
	\begin{proof}
		By Lemma~\ref{lem: notod.}, there exists continuous function $g \colon X \to \{0,1\}^\mathbb{Z}$ with $g(X)$ infinite that satisfies $g \circ \varphi = \sigma \circ g$. Define $C$ to be the set of elements $x$ such that $g(x)_0=1$. By definition, we have $g(x)_k=\mathbf{1}_C(\varphi^k(x))$ Let $F$ be a finite subset of $\mathbb{Z}$ such that $f_{C,F}$ is $0{\pmod 2}$ everywhere and let $l=\min F$ and $m=\max F$. Let $x,y \in X$:
		\begin{equation*}
		g(x)_{l-1}-g(y)_{l-1}\equiv \sum_{k \in F\setminus{l}} g(x)_{k-1}-g(y)_{k-1}{\pmod 2}
		\end{equation*}
		Thus if $g(x)_k=g(y)_k$ for all $l\leq k\leq m$, then $g(x)_{l-1}=g(y)_{l-1}$. By iteration, it follows that every value $g(x)_k$ with $k \in \mathbb{Z}$ depends only on the values of $g(x)_k$ for $l\leq k\leq m$ contradicting the assumption of infinity of $g(X)$.
	\end{proof}
	
	\begin{thm}[\cite{mat13}, Theorem~2.4]
		Let $(X,\varphi)$ be a minimal Cantor system. If it is not an odometer, then the group $\mathfrak{T}(\varphi)'$ contains a copy of the lamlighter group $L$.
	\end{thm}
	\begin{proof}
		Assume $(X,\varphi)$ is not an odometer. Let $U \subset X$ be a non-empty clopen subset such that $U,\varphi(U),\varphi^2(U),\varphi^3(U)$ are pairwise disjoint. For any non-empty clopen subset $V\subset U$ define a homeomorphism $\tau_V \in \mathfrak{T}(\varphi)$ by:
		\begin{equation*}
		\tau_V(x):=
		\begin{cases}
		\varphi (x), & \text{ if } x \in V\\
		\varphi^{-1}(x), & \text{ if } x \in \varphi(V)\\
		x, & \text{ else}
		\end{cases}
		\end{equation*}
		Note, that all homeomorphisms of this kind pairwise commute and that for any finite family of non-empty clopen $V_1,\dots,V_n \subset U$ the homeomorphism $\tau_{V_1}\circ \dots \circ \tau_{V_n}$ is the identity \emph{if and only if} $\sum_{i=1}^{n} \mathbf{1}_{V_i}(x)=0{\pmod 2}$ for all $x \in X$. Define a homeomorphism $r=\varphi_U \circ \varphi \circ \varphi_U \circ \varphi^{-1} \in \mathfrak{T}(\varphi)$, where $\varphi_U$ is the induced transformation associated with $U$ (Definition~\ref{defi: rotation}(iii)). It holds, that $\supp(r)\subset U \cup \varphi(U)$ and $r \circ \tau_V \circ r^{-1}=\tau_{\varphi_U(V)}$. The derivative $(U,\varphi_U|_U)$ of $(X,\varphi)$ over $U$ can not be an odometer and thus by Lemma~\ref{lem: no-odo.} there exists a non-empty clopen subset $C$ such that for every finite subset $F \subset \mathbb{Z}$
		\begin{equation*}
		f_{C,F}(x):=\sum_{k \in F} \mathbf{1}_C \circ \varphi_U^k(x) \equiv 0{\pmod 2}
		\end{equation*}
		holds for all $x \in U$. It holds that $\supp(\varphi_C) \subset U \cup \varphi(U)$. The group generated by $r$ and $\tau_C$ is isomorphic to the lamplighter group, because $r^k \circ \tau_C \circ r^{-k}=\tau_{\varphi_U^k(C)}$ for any $k \in \mathbb{Z}$ and $\sum_{k \in F} \mathbf{1}_{\varphi_U^k(C)}\neq 0$ for every finite $F \subset \mathbb{Z}$, it holds that $\prod_{k \in F} r^k \circ \tau_C \circ r^{-k}$ is not the identity for every finite $F \subset \mathbb{Z}$. Since $\supp(r)\cap\supp(\varphi^2 \circ r\circ \varphi^{-2})=\emptyset$ and $\supp(\tau_C)\cap\supp(\varphi^2 \circ \tau_C \circ \varphi^{-2})=\emptyset$, the group generated by $[ r^{-1} , \varphi^{-2}]$ and $[\tau_C^{-1}, \varphi^{-2}]$ is contained in $\mathfrak{T}(\varphi)'$ and isomorphic to the lamplighter group $L$. 
	\end{proof}
	
	\begin{cor}[\cite{mat13}, Corollary~2.5]
		Let $(X,\varphi)$ be a minimal Cantor system. If $\mathfrak{T}(\varphi)'$ is finitely generated, it has exponential growth.
	\end{cor}
	\begin{proof}
		By Theorem~\ref{thm: fingen}, the group $\mathfrak{T}(\varphi)'$ associated with a minimal Cantor system $(X,\varphi)$ is finitely generated \emph{if and only if} $(X,\varphi)$ is conjugate to a minimal subshift, thus it can not be an odometer and must contain a copy of $L$. By Remark~\ref{rem: expgro}, the growth of the lamplighter group is exponential and thus $\mathfrak{T}(\varphi)'$ has exponential growth.
	\end{proof}
	
	\subsection{Higman-Thompson groups as topological full groups}
	
	For the definition of Higman-Thompson groups see Appendix~\ref{app: higman-thompson groups}, of Cuntz-Krieger algebras Definition~\ref{defi: cuntz-krieger algebras} and for SFT-groupoids Subsection~\ref{subs: Purely infinite groupoids and groupoids of shifts of finite type}. In \cite{mat15} Matui mentions the fact that under the representation of Cuntz algebras as groupoid C*-algebras -- see Theorem~\ref{thm: cuntz-krieger algebras are groupoid c*-algebras of one sided shifts of finite type} -- the Higman-Thompson groups turn out to be topological full groups of SFT-groupoids. For a more precise description of correspondences see \cite{mm17}.
	
	\begin{defi}
		Let $n \in \mathbb{N}$ and let $\{A_{i,j}\}_{i,j \in \{1,\dots,n\}} \in \mathbf{M}_{n\times n}(\mathbb{Z}/2\mathbb{Z})$ be a matrix of which every row and column is non-zero and which is not a permutation matrix. Let $\{S_i\}_{i \in \{1,\dots,n\} }$ be a collection of non-zero partial isometries over some Hilbert space $\mathcal{H}$ that generate a faithful representation of the Cuntz-Krieger algebra $\mathcal{O}_{A}$. Let $w=(w_1,\dots, w_k) \in \{1,\dots,n\}^k$ be a finite word of length $k \in \mathbb{N}$. Denote by $S_w$ the product $S_{w_1}\dots S_{w_k}$.
	\end{defi}
	
	\begin{rem}\label{rem: admissible words}
		Note that $S_w \neq 0$ \emph{if and only if} $A_{w_i,w_{i+1}}=1$ for all $i \in \{1, \dots, n-1 \}$.
	\end{rem}
	
	Matui's observation is based on a result obtained in \cite{bir04} and in \cite{nek04}: The Higman-Thompson group $G_{n,1}$ has a faithful representation as unitary subgroup of the Cuntz algebra $\mathcal{O}_n$.
	
	\begin{thm}[\cite{nek04}, Proposition~9.6]\label{thm: unitary representation of thompson-higman groups in Cuntz-algebras}
		Let $n \in \mathbb{N}\setminus \{1\}$. Then the map $\pi \colon G_{n,1} \to \mathcal{O}_n$ given for every $g \in G_{n,1}$ represented as a table of cofinite bases in $X_1\mathcal{A}_n^*$ by
		\begin{equation*}
		g=\begin{pmatrix}
		x_1w_1 & \dots & x_1w_k\\
		x_1u_1 & \dots & x_1u_k
		\end{pmatrix}
		\mapsto S_{u_1}S_{w_1}^*+ \dots + S_{u_k}S_{w_k}^*
		\end{equation*}
		is a faithful unitary representation. Moreover $G_{n,1}$ is isomorphic to the group generated by unitary elements $S \in \mathcal{O}_n$ which are of the form $S=S_{u_1}S_{w_1}^*+ \dots + S_{u_k}S_{w_k}^*$ for some $k \in \mathbb{N}$ and $w_i,u_i \in \mathcal{A}_n^*$.
	\end{thm}

	Let $\mathcal{G}_A$ be an SFT-groupoid. As in Subsection~\ref{subs: Purely infinite groupoids and groupoids of shifts of finite type} an SFT-groupoid $\mathcal{G}_A$ comes with the data of a matrix $A$ that satisfies (Inp), an associated finite directed graph $\Gamma_A=(V,E)$ with induced edge shift $\sigma_A$ on the space $X_A$ of one-ended infinite directed paths and $\mathcal{G}_A$ is the equivalence groupoid induced by the tail equivalence of paths. Under the identification of $E$ with $\{1, \dots,n \}$ the condition of Remark~\ref{rem: admissible words} is equivalent to $\omega=(\omega_1,\dots, \omega_k) \in E^*$ being a finite directed path in $\Gamma_A$.
	
	\begin{defi}
		Denote by $\mathcal{D}_{\tilde{A}}$ the subalgebra of $\mathcal{O}_{\tilde{A}}$ generated by elements of the form $S_\omega S_\omega^*$ for $\omega \in E^*$.
	\end{defi}
	
	The C*-subalgebra $\mathcal{D}_{\tilde{A}}$ is commutative and isomorphic to $C(X_A)$ via the assignment $S_\omega S_\omega^* \mapsto \mathbf{1}_{C_\omega}$, in fact, it is a Cartan subalgebra of $\mathcal{O}_{\tilde{A}}$ isomorphic $C(\mathcal{G}_A^{(0)})$ under the isomorphism of Theorem~\ref{thm: cuntz-krieger algebras are groupoid c*-algebras of one sided shifts of finite type}.\footnote{See §6.3 of \cite{ren08}.\\} 
	
	\begin{lem}
		Let $\mathcal{G}_A$ be SFT-groupoid. Denote by $G_A$ the group of all homeomorphisms $h \in \operatorname{Homeo}(X_A)$ such that there exists a pair of continuous maps $k_h,l_h \colon X_A \to \mathbb{Z}_{\geq 0}$ such that $\sigma_A^{k_h(x)}(x)= \sigma_A^{l_h(x)}(h(x))$. Then $G_A \cong \mathfrak{T}(\mathcal{G}_A)$.
	\end{lem}
	
	 For inverse fibers of $k_h$ and $l_h$ there exist finite clopen partitions into cylinder sets and by this every $B \in \mathfrak{T}(\mathcal{G}_A)$ decomposes as compact open slice into a union $\bigsqcup_{i \in I} B_{\omega_i,\upsilon_i}$ of compact open slices given by a finite collection of pairs of finite directed paths $(\omega_i=(\omega_{i,1},\dots, \omega_{i,l_i}),\upsilon_i=(\upsilon_{i,1},\dots,\upsilon_{i,m_i})_{i \in I}$ which satisfy $r(\omega_{i,l_i})=r(\upsilon_{i,m_i})$, such that $X_A = \bigsqcup_{i \in I} C_{\omega_i} = \bigsqcup_{i \in I} C_{\upsilon_i}$. With this description the isomorphism of unitary normalizers induced by the isomorphism of Theorem~\ref{thm: cuntz-krieger algebras are groupoid c*-algebras of one sided shifts of finite type} is provided by the assignment $B \to \sum_{i \in I} S_{\omega_i} S_{\upsilon_i}^*$. In the case of a full shift of order $n$ the image precisely amounts to the copy of $G_{n,1}$ in $\mathcal{O}_n$ described in Theorem~\ref{thm: unitary representation of thompson-higman groups in Cuntz-algebras}. More generally -- see §6.7.1 of \cite{mat15} -- the groups $G_{n,r}$ turn out to correspond to $\mathfrak{T}(\mathcal{G}_{A_{n,r}})$ for the matrix
	 \begin{equation*}
	 	A_{n,r}:=
	 	\begin{bmatrix}
	 	0		&	\dots	&	\dots 	&	0 		& 	n				\\
	 	1		&	\ddots	&	 	&	 	& 	0				\\
	 	0		&	\ddots	&	\ddots 	& 	\ 	& 	\vdots			\\
	 	\vdots 	&	\ddots	&	\ddots 	& 	\ddots	& 	\vdots			\\
	 	0 		&  	\dots	& 	0 		& 	1 		& 	0 
	 	\end{bmatrix}
	 \end{equation*}
	 
	 Matui thus considered topological full groups $\mathfrak{T}(\mathcal{G}_A)$ associated to irreducible one-sided shifts of finite type $(X_A,\sigma_A)$ as generalized Higman-Thompson groups and showed properties of the groups $G_{n,r}$ generalize to this context:
	 
	 \begin{thm}\label{thm: sft-groupoids are blablablabla}
	 	Let $\mathcal{G}_A$ be an SFT-groupoid. Then the following hold:
	 	\begin{enumerate}[(i)]
	 		\item The group $\mathfrak{T}(\mathcal{G}_A)$ is not amenable.
	 		
	 		\item \emph{(\cite{mat15}, Theorem~6.7)} The group $\mathfrak{T}(\mathcal{G}_A)$ satisfies the Haagerup property.\footnote{The \emph{Haagerup property} is a notion weaker than amenability and similarily is an opposite to Property(T). It is of significance in connection with the \emph{Baum-Connes conjecture}.\\}
	 		
	 		\item \emph{(\cite{mat15}, Theorem~6.21)} The group $\mathfrak{T}(\mathcal{G}_A)$ is of type $F_\infty$.\footnote{\cite{geo08} Let $G$ be a group. A CW-complex $X$ is called an \emph{Eilenberg-MacLane complex of type $K(G,1)$}, if $\pi_1(X)\cong G$ and $\pi_k(X)$ is trivial for all $k \neq 1$. Let $n \in \mathbb{N}$. A group is said to be of \emph{type $F_n$} if there exists an Eilenberg-MacLane complex of type $K(G,1)$ with finite $n$-skeleton. It is said to be of type $F_\infty$ if it is of type $F_k$ for all $k \in \mathbb{N}$. It holds that a group $G$ is finitely generated (resp. finitely presented) \emph{if and only if} it is of type $F_1$ (resp. of type $F_2$), which is why this properties are termed \emph{finiteness properties}.\\}
	 	\end{enumerate}
	 \end{thm}  
	 
	 \begin{rem}
	 	Theorem~\ref{thm: sft-groupoids are blablablabla}(i) is a consequence of being purely infinite and minimal -- see Lemma~\ref{lem: sft-groupoids are purely infinite, minimal etc} and Remark~\ref{rem: non-amenable topological full groups}(ii).	 	
	 \end{rem}

	\section{On a pair of subgroups}\label{sec: on a pair of subgroups}
	
	Nekrashevych defined a pair of subgroups of $\mathfrak{T}(\mathcal{G})$ denoted by $\mathfrak{S}(\mathcal{G})$ and $\mathfrak{A}(\mathcal{G})$ in terms of slices. These have subsequently been termed the symmetric- and alternating full groups. In the effective case $\mathfrak{S}(\mathcal{G})$ is just Matui's subgroup of elementary elements mentioned at the end of Subsection~\ref{subs: generalizations}. The subgroups $\mathfrak{S}(\mathcal{G})$ (resp. $\mathfrak{A}(\mathcal{G})$) are in the vein of the description of the locally finite subgroups $\mathfrak{T}(\varphi)_{\{x\}}$ (resp. $\mathfrak{T}(\varphi)_{\{x\}}'$) in terms of permutations (resp. even permutations) of atoms in an associated nested sequence of Kakutani-Rokhlin partitions (see Corollary~\ref{cor: locally finite subgroups} \& \ref{cor: ev.permut.}). Similar to the situation of Proposition~\ref{prop: mat.fac}, these new subgroups are closely related to the subgroups $\mathfrak{T}(\mathcal{G})_0$ (resp. $\mathfrak{T}(\mathcal{G})'$), in fact, they are equal for certain classes of groupoids. Subsection~\ref{subs: symmetric - and alternating full groups} gives the definition of $\mathfrak{S}(\mathcal{G})$ and $\mathfrak{A}(\mathcal{G})$. In Subsection~\ref{subs: simplicity of the alternating full groups} we recount Nekrashevych's proof of simplicity of $\mathfrak{A}(\mathcal{G})$ for minimal, effective, \'etale Cantor groupoid and in Subsection~\ref{subs: finite generation of alternating full groups} his proof of finite generation of $\mathfrak{A}(\mathcal{G})$ for minimal, expansive, \'etale Cantor groupoids.
	
	\subsection{Symmetric- and alternating full groups}\label{subs: symmetric - and alternating full groups}
	
	\begin{defi}[\cite{nek17}, Definition 3.1]
		Let $\mathcal{G}$ be an \'etale Cantor groupoid and let $d$ be a positive integer.
		A family of compact, open slices $\mathcal{M}=\{M_{i,j}\}_{1 \leq i,j \leq d}$ with the properties
		\begin{enumerate}[(i)]
			\item $M_{j,k}M_{i,j}=M_{i,k}$ for all $1\leq i,j,k \leq d$
			
			\item $M_{i,i} \subseteq \mathcal{G}^{(0)}$ for all $1\leq i \leq d$
			
			\item $M_{i,i}\cap M_{j,j}=\emptyset$ if $i \neq j$ for all $1\leq i, j\leq d$
		\end{enumerate}
		is called \emph{multisection of degree $d$}. The set $\bigcup_i M_ {i,i}$ is called the \emph{domain of $\mathcal{M}$}.	Let $U$ be a clopen subset contained in some component $M_{i',i'}$ of the domain. Then $\mathcal{M}|_U:=\{M_{i,j}r(M_{i',i}U)\}_{i,j=1}^d$ is a multisection of degree $d$ called \emph{the restriction of $\mathcal{M}$ to $U$}.
	\end{defi}
	
	The following lemma assures us to have enough multisections:
	
	\begin{lem}[\cite{nek17}, Lemma 3.1]\label{lem: mult}
		Let $\mathcal{G}$ be an \'etale Cantor groupoid. Let $\{u_1, \dots ,u_d\}$ be a finite set of elements in $\mathcal{G}^{(0)}$ which is contained in a single $\mathcal{G}$-orbit and  in an open neighbourhood $U$. There exists a multisection $\mathcal{M}=\{M_{i,j}\}_{1 \leq i,j \leq d}$ of degree $d$, such that $u_i \in M_{i,i}$, $x_j=r(M_{i,j}u_i)$ and $U$ contains the domain of $\mathcal{M}$.
	\end{lem}	
	\begin{proof}
		Since all $u_i$ are contained in one $\mathcal{G}$-orbit, there exist elements $g_1, g_2 \dots g_d$ such that $s(g_i)=u_1$ and $r(g_i)=u_i$, in particular set $g_1=u_1$. Since $\mathcal{G}$ is an \'etale Cantor groupoid, for every $i \in \{1,\dots,d\}$ there exists a compact, open slice $B_i \in \mathcal{B}_\mathcal{G}^{o,k} $ such that $g_i \in B_i$. Since $\mathcal{G}^{(0)}$ is a Cantor space, there exists a clopen neighbourhood $u_1 \in W$ with $W \subseteq s(B_i) \cap U$ such that the sets $r(B_iW)$ are disjoint and $B_1W=W$ (by assuming $W \subseteq s(B_1)$). Then $M_{i,j}:=(B_jW)(B_iW)^{-1}$ gives the desired multisection.
	\end{proof}
	
	\begin{defi}[\cite{nek17}, Defintion 3.2]\label{defi: symmetr, alternating full group}
		Let $\mathcal{G}$ be an \'etale Cantor groupoid and let $\mathcal{M}$ be a multisection of degree $d$ with domain $U$. There is an embedding of the symmetric group $\mathfrak{S}_d$ of degree $d$ into $\mathfrak{T}(\mathcal{G})$ given by $\pi \mapsto \mathcal{M}_\pi:=\bigcup_{i=1}^d M_{i,\pi(i)} \cup (\mathcal{G}^{(0)}\setminus U)$ and accordingly an embedding of the alternating group $\mathfrak{A}_d$ of degree $d$. Denote the image of the symmetric group under this embedding by $\mathfrak{S}(\mathcal{M})$ and by $\mathfrak{A}(\mathcal{M})$ the image of $\mathfrak{A}_d$.
		The subgroup of $\mathfrak{T}(\mathcal{G})$ generated by the union of the subgroups $\mathfrak{S}(\mathcal{M})$ for all multisections $M$ of degree $d$ is denoted by $\mathfrak{S}_d(\mathcal{G})$, the subgroup of $\mathfrak{T}(\mathcal{G})$ generated by the union of the subgroups $\mathfrak{A} (\mathcal{M})$ for all multisections $M$ of degree $d$ is denoted by $\mathfrak{A}_d (\mathcal{G})$.	The subgroup $\mathfrak{S}(\mathcal{G}):= \mathfrak{S}_2(\mathcal{G})$ is called the \emph{symmetric full group} and the subgroup $\mathfrak{A}(\mathcal{G}):= \mathfrak{A}_3(\mathcal{G})$ the \emph{alternating full group}.				
	\end{defi}
	
	\begin{rem}\phantomsection\label{rem: SA}
		\begin{enumerate}[(i)]
			\item The subgroups $\mathfrak{S}_d(\mathcal{G})$ and $\mathfrak{A}_d(\mathcal{G})$ are normal, since conjugates of multisections of degree $d$ are multisections of degree $d$. It holds, that $\mathfrak{S}_d(\mathcal{G}) \geq \mathfrak{S}_{d+1}(\mathcal{G})$ and $\mathfrak{A}_d(\mathcal{G}) \geq \mathfrak{A}_{d+1}(\mathcal{G})$, since for a multisection $\mathcal{M}$ of degree $d+1$ the groups $\mathfrak{S}(\mathcal{M})$ resp. $\mathfrak{A}(\mathcal{M})$ are generated by all the subgroups $\mathfrak{S}(\mathcal{M}')$ resp. $\mathfrak{A}(\mathcal{M}')$ for all multisections $\mathcal{M}'$ of degree $d$ in $\mathcal{M}$.
			
			\item Since every transposition in $\mathfrak{S}(\mathcal{G})$ is in $\mathfrak{T}(\mathcal{G})_0$, it is immediate that $\mathfrak{S}(\mathcal{G}) \leq \mathfrak{T}(\mathcal{G})_0$. Since $3$-cycles are generated by commutators, it holds that $\mathfrak{A}(\mathcal{G}) \leq \mathfrak{T}(\mathcal{G})'$.
		\end{enumerate}
	\end{rem}

	\begin{prop} [\cite{nek17}, Proposition 3.6 \& Corollary 3.7]\label{prop: SdG}
		Let $\mathcal{G}$ be an \'etale Cantor groupoid. If all $\mathcal{G}$-orbits have $n$ or more points, $\mathfrak{S}(\mathcal{G})=\mathfrak{S}_d(\mathcal{G})$ for all $2 \leq d \leq n$ and $\mathfrak{A}(\mathcal{G})= \mathfrak{A}_d(\mathcal{G})$ for all $3 \leq d \leq n$. Thus, if all $\mathcal{G}$-orbits are infinite, $\mathfrak{S}(\mathcal{G})=\mathfrak{S}_d(\mathcal{G})$ for all $2 \leq d$ and $\mathfrak{A}(\mathcal{G})= \mathfrak{A}_d(\mathcal{G})$ for all $3 \leq d$.
	\end{prop}
	\begin{proof}
		By Remark~\ref{rem: SA}(i) it holds that $\mathfrak{S}_d(\mathcal{G}) \geq \mathfrak{S}_{d+1}(\mathcal{G})$. It remains to show $\mathfrak{S}_d(\mathcal{G}) \leq \mathfrak{S}_{d+1}(\mathcal{G})$ for $2 \leq d \leq n-1$. Let $\mathcal{M}$ be a multisection of degree $d$. Let $\{x_1,x_2, \dots x_d\}$ be a set of points with $x_i \in M_{i,i}$ contained in the same $\mathcal{G}$-orbit. By assumption there exists an additional point $x_{d+1}$ in the same $\mathcal{G}$-orbit. Let $U$ be 				
		a clopen neighbourhood of $\{x_1,x_2, \dots ,x_{d+1} \}$ containing the domain of $\mathcal{M}$. Lemma~\ref{lem: mult} allows to construct a multisection $\mathcal{M}'$ of degree $d+1$ with $M'_{i,j} \subseteq M_{i,j}$ for $i,j \in 1,2,\dots,d$. Since the domain of $\mathcal{M}$ is compact, it is possible to apply this construction to get a finite family $\mathcal{M}^k$ of multisections of degree $d+1$ with disjoint domains  such that $\bigcup_k M_{i,j}^k \supseteq M_{i,j}$. This implies $\mathfrak{S}(\mathcal{M}) \subseteq \mathfrak{S}_{d+1}(\mathcal{G})$. The proof for $\mathfrak{A}(\mathcal{G})$ is analogous and the second statement is  a simple corollary of the first statement.
	\end{proof}
	
	By \cite{nek17} Theorem~\ref{thm: rubin} can be used to extend the list in Theorem~\ref{thm: isom} by:
	
	\begin{enumerate}
		\item[(v)] $\mathfrak{S}(\mathcal{G}_1) \cong \mathfrak{S}(\mathcal{G}_2)$
		\item[(vi)] $\mathfrak{A}(\mathcal{G}_1) \cong \mathfrak{A}(\mathcal{G}_2)$	
	\end{enumerate}
	
	\subsection{Simplicity of the alternating full group}\label{subs: simplicity of the alternating full groups}
	
	The purpose of this subsection is to recount the proof of simplicity of $\mathfrak{A}(\mathcal{G})$ for minimal, effective, \'etale Cantor groupoids obtained in \cite{nek15}. As a consequence one obtains a verification of the AH-conjecture for minimal, almost finite, principal, \'etale Cantor groupoids. 
	
	\begin{lem} [\cite{nek17}, Lemma 3.3]\label{lem: alt}
		Let $\mathfrak{A}_d$ be the alternating group of degree $d$.	If $d \geq 5$, it holds that $\langle \tilde{\mathfrak{A}}_1 \cup \tilde{\mathfrak{A}}_2 \rangle = \mathfrak{A}_d \times \mathfrak{A}_d \times \mathfrak{A}_d$, where $\tilde{\mathfrak{A}}_1:=\{(\pi,\pi, 1)|\pi \in \mathfrak{A}_d\}$ and $\tilde{\mathfrak{A}}_2:=\{(1,\pi,\pi)|\pi \in \mathfrak{A}_d\}$.
	\end{lem}
	
	\begin{proof}
		Since $[\tilde{\mathfrak{A}}_1, \tilde{\mathfrak{A}}_2] = \{(1,[\mathfrak{A}_d,\mathfrak{A}_d],1)\}$ holds and because the assumption $d \geq 5$ implies $[\mathfrak{A}_d,\mathfrak{A}_d] = \mathfrak{A}_d$, it follows that $\{(1,\mathfrak{A}_d,1)\} \subseteq \langle \tilde{\mathfrak{A}}_1 \cup \tilde{\mathfrak{A}}_2 \rangle$ and consequently
		\begin{equation*}
			\{(\mathfrak{A}_d,1,1)\},\{(\mathfrak{A}_d,1,1)\} \subseteq \langle \tilde{\mathfrak{A}}_1 \cup \tilde{\mathfrak{A}}_2 \rangle
		\end{equation*} 
	\end{proof}
	
	\begin{prop}[\cite{nek17}, Proposition 3.2]\label{prop: multcov}
		Let $\mathcal{G}$ be an \'etale Cantor groupoid and $\mathcal{M}$ be a multisection of degree $d \geq 5$. If there is a collection of multisections $\mathcal{M}^k$ for $k=1,2,\dots,n$ for some  $n \in \mathbb{N}$ such that $\bigcup_{k=1}^n M_{i,j}^k=M_{i,j}$, then $\mathfrak{A}(\mathcal{M}) \subseteq \langle\bigcup_{k=1}^n \mathfrak{A}(\mathcal{M}^k)\rangle$.
	\end{prop}
	\begin{proof}
		Let $n=2$. The families $\mathcal{I}:=\{M_{i,j}^1 \cap M_{i,j}^2\}_{i,j=1}^d$, $\mathcal{D}^1:=\{M_{i,j}^1 \setminus M_{i,j}^2\}_{i,j=1}^d$ and $\mathcal{D}^2:=\{M_{i,j}^2 \setminus M_{i,j}^1\}_{i,j=1}^d$ are multisections. The domains of $\mathcal{I}$ and $\mathcal{D}^i$ are invariant under elements of $\mathfrak{A}(\mathcal{M}^i)$ and the restriction of an element $\mathcal{M}_\pi^i$ to the domain of $\mathcal{I}$ (resp. $\mathcal{D}^i$) equals $\mathcal{I}_\pi$ (resp. $\mathcal{D}_\pi^i$) for every $\pi \in \mathfrak{A}_d$. Restrictions of elements in $\mathfrak{A}(\mathcal{M}^1)$ to the domain of $\mathcal{D}^2$ are trivial and so are restrictions of elements in $\mathfrak{A}(\mathcal{M}^2)$ to the domain of $\mathcal{D}^1$. Thus, we can apply Lemma~\ref{lem: alt} and $\mathfrak{A}(\mathcal{I}), \mathfrak{A}(\mathcal{D}^i) \leq \langle \mathfrak{A}(\mathcal{M}^1) \cup \mathfrak{A}(\mathcal{M}^2) \rangle$ holds as desired. The other cases follow by induction.
	\end{proof}
	
	\begin{thm} [\cite{nek17}, Theorem 4.1]\label{thm: simpl}
		Let $\mathcal{G}$ be a minimal, effective, \'etale Cantor groupoid.
		Any non-trivial subgroups of $\mathfrak{T}(\mathcal{G})$ normalized by $\mathfrak{A}(\mathcal{G})$ must contain $\mathfrak{A}(\mathcal{G})$. Thus $\mathfrak{A}(\mathcal{G})$ is simple.
	\end{thm}
	
	\begin{proof}
		Let $H \leq \mathfrak{T}(\mathcal{G})$ be a non-trivial subgroup normalized by $\mathfrak{A}(\mathcal{G})$. Let $B \in H \subset \mathcal{B}_\mathcal{G}^{o,k}$ be a non-trivial element. By effectiveness, there exists an element $g \in \mathcal{G}$ with $g \in B$ and such that $s(g) \neq r(g)$. Let $U$ be a clopen neighbourhood of $s(g)$ such that $U \cap BU = \emptyset$. Let $\mathcal{M}$ be a multisection of degree $\geq 3$  with its domain contained in $U$ -- this exists by Lemma~\ref{lem: mult} -- and let $A_1,A_2 \in \mathfrak{A}(\mathcal{M}) \subseteq \mathfrak{A}(\mathcal{G})$. Then $B \mathcal{M} B^{-1}$ is a multisection of which the domain is contained in $BU$ and such that $BA_iB^{-1} \in \mathfrak{A}(B \mathcal{M} B^{-1})$ for $i \in \{1,2\}$. The computation of $[B^{-1},A_1]$ on the sets $(U \cup BU)^{\mathrm{C}}$, $U$ and $BU$ shows $[[B^{-1},A_1]A_2]=[A_1,A_2]$. The assumption of $H$ being normalized by $\mathfrak{A}(\mathcal{G})$ implies  $[[B^{-1},A_1]A_2 ] \in H$, hence $[\mathfrak{A}(\mathcal{M}),\mathfrak{A}(\mathcal{M})] \subseteq H$. Since $\mathfrak{A}(\mathcal{G}) \leq \mathfrak{T}(\mathcal{G})'$ by Remark~\ref{rem: SA}(ii), it follows $\mathfrak{A}(\mathcal{M})=[\mathfrak{A}(\mathcal{M}),\mathfrak{A}(\mathcal{M})]$
		$\subseteq H$ and in consequence $\mathfrak{A}(\mathcal{G}|_U) \leq H$\footnote{$\mathfrak{A}(\mathcal{G}|_U)$ naturally embeddeds into $\mathfrak{A}(\mathcal{G})$ by extending with units outside of $U$.\\}, where $\mathcal{G}|_U$ is the restriction of $\mathcal{G}$ to $U$ as defined in Definition~\ref{defi: grpd}. The proof is completed by showing that the conjucate closure of $\mathfrak{A}(\mathcal{G}|_U)$ in $ \mathfrak{A}(\mathcal{G})$ is $\mathfrak{A}(\mathcal{G})$. By Proposition~\ref{prop: SdG} it is sufficient to show that for any multisection $\mathcal{M}$ of degree $5$, it holds that $\mathfrak{A}(\mathcal{M})$ is contained in the conjugate closure of $ \mathfrak{A}(\mathcal{G}|_U)$ in $ \mathfrak{A}(\mathcal{G})$.	To this end let $\mathcal{M}$ be a multisection of degree $5$ and let $x_1 \in F_{1,1}$ be arbitrary. The points $x_i:=r(F_{1,i}x_1)$ are contained in one $\mathcal{G}$-orbit. By the assumption of minimality, there exist elements $\rho_i,\sigma_i \in \mathcal{G}$ for all $i \in \{1,\dots,5\}$ with $r(\rho_i)=s(\sigma_i)=x_i$ such $r(\sigma_i) \in U$ and since all orbits are infinite the elements $x_1, \dots x_5, s(\rho_1), \dots, s(\rho_5), r(\sigma_1),\dots, r(\sigma_5)$ can be assumed pairwise different. By Lemma~\ref{lem: mult} there exist families of compact, open slices $\{R_i\}_{i=1}^5$ and $\{S_i\}_{i=1}^5$  such that $\rho_i \in R_i$, $\sigma_i \in S_i$, $r(R_i)\subseteq U$ and the sets $s(R_1), \dots ,s(R_5),r(R_1), \dots ,r(R_5),r(S_1), \dots ,r(S_5)$ are pairwise disjoint, and there exists a multisection $\mathcal{M}': =\{M'_{i,j}\}_{i,j=1}^5$ such that $M'_{i,j} \subseteq M_{i,j}$ and $M'_{i,i}=r(R_i)=s(S_i)$. Then the collection $\mathcal{U}:=\{S_jM'_{i,j}S_i^{-1}\}_{i,j=1}^5$ forms a multisection of which the domain is contained in $U$. The sets $C_i:= R_i \cup S_i \cup (S_iR_i)^{-1}$ are compact, open slices such that $s(C_i)=r(C_i)$ and the sets $s(C_1), \dots ,s(C_5)$ are pairwise disjoint.	The compact, open slice $A:=\bigcup_{i=1}^5 C_i \cup \big(\mathcal{G}^{(0)} \setminus \bigcup_{i=1}^5 s(C_i)\big)$ is an element of $\mathfrak{T}(\mathcal{G})$ and in particular of $\mathfrak{A}(\mathcal{G})$ such that $\mathcal{U} = A \mathcal{M}' A^{-1}$. Thus $\mathfrak{A}(\mathcal{M}') = A^{-1} \mathfrak{A}(\mathcal{U}) A \leq A^{-1} \mathfrak{A}(\mathcal{G}|_U) A$. Since the domain of $\mathcal{M}$ is compact, there exists a finite collection $\mathcal{M}^k$ of restrictions of $\mathcal{M}$ constructed in the same way as $\mathcal{M}'$ and satisfying the same properties. Applying Proposition~\ref{prop: multcov} shows $\mathfrak{A}(\mathcal{M}) \leq \langle \mathfrak{A}(\mathcal{G}|_U) \rangle ^{\mathfrak{A}(\mathcal{G})}$ and the proof is finished.
	\end{proof}
	
	Theorem~\ref{thm: simpl}, Theorem~\ref{thm: almsimpl} and Theorem~\ref{thm: purely infinite, minimal groupoids, then commutat of topfgrp is simple} imply: 
	
	\begin{cor}\label{cor: A=D}
		Let $\mathcal{G}$ be an minimal, effective, \'etale Cantor groupoid. If $G$ is either almost finite or purely finite, then $\mathfrak{A}(\mathcal{G}) =\mathfrak{T}(\mathcal{G})'$ holds.
	\end{cor}
	
	Corollary~\ref{cor: A=D} allowed Nekrashevych to proof Conjecture~\ref{conj: AH} in the case of minimal, almost finite, principal, \'etale Cantor groupoids: For principal, \'etale Cantor groupoids, every element $\gamma \in \mathfrak{T}(\mathcal{G})$ is of finite order \emph{if and only if} $\gamma \in \mathfrak{S}(\mathcal{G})$ holds, thus Theorem~\ref{thm: ker in fin} and Remark~\ref{rem: SA}(ii) imply $\mathfrak{S}(\mathcal{G})=\mathfrak{T}(\mathcal{G})_0$. If $\mathcal{G}$ is minimal and almost finite, then $\mathfrak{A}(\mathcal{G}) =\mathfrak{T}(\mathcal{G})'$ by Corollary~\ref{cor: A=D}. This implies the assignement of $j(\mathbf{1}_{s(S)})$ as the class of $S \cup S^{-1} \cup (\mathcal{G}^{(0)} \setminus(s(S)\cup r(S))$ in $ \mathfrak{T}(\mathcal{G})^{\mathrm{ab}}$ from Conjecture~\ref{conj: AH} has its range in $\mathfrak{S}(\mathcal{G})/\mathfrak{A}(\mathcal{G})$ and one even has:
	\begin{thm}[\cite{nek15}, Theorem~7.2]\label{thm: j wd sur}
		Let $\mathcal{G}$ be a minimal, almost finite, principal, \'etale Cantor groupoid. Then $j$ is a well-defined grouphomomorphism and $j(H_0(\mathcal{G})\otimes (\mathbb{Z}/2\mathbb{Z}))=\mathfrak{S}(\mathcal{G})/\mathfrak{A}(\mathcal{G})$.
	\end{thm}
	
	The following theorem, a direct proof of which can be found in Theorem~3.6 from \cite{mat16}, then follows from Theorem~\ref{thm: indexsurj} and Theorem~\ref{thm: j wd sur}:				
	\begin{thm}[\cite{nek15}, Proposition~7.5]
		Let $\mathcal{G}$ be a minimal, almost finite, effective, \'etale Cantor groupoid. Then the following sequence is exact:
		\begin{equation*}
		H_0(\mathcal{G})\otimes (\mathbb{Z}/2\mathbb{Z}) \overset{j}{\longrightarrow}\mathfrak{T}(\mathcal{G})^{\mathrm{ab}} \overset{I}{\longrightarrow} H_1(\mathcal{G}) \longrightarrow 0
		\end{equation*}
	\end{thm}
	
	\subsection{Finite generation of alternating full groups}\label{subs: finite generation of alternating full groups}
	
	In Subsection~\ref{subs: compact generation and expansive groupoids} expansive groupoids where descibed as generalizations of transformation groupoids associated with shifts. Nekrachevych obtained for these a generalization of Theorem~\ref{thm: fingen}.
	
	\begin{lem}[\cite{nek17}, Lemma 3.5]\label{lem: intersect}
		Let $\mathfrak{A}_d$ be the alternating group of degree $d$. Let $X_1,X_2$ be sets with $3 \leq |X_i| \leq \infty$ and $|X_1 \cap X_2|=1$. Then the conjugate closure of $[\mathfrak{A}_{X_1}, \mathfrak{A}_{X_2}]$ in $\langle \mathfrak{A}_{X_1} \cup \mathfrak{A}_{X_2} \rangle$ equals $\mathfrak{A}_{X_1 \cup X_2}$.

	\end{lem}
	
	\begin{proof}
		Denote by $\{s\}:=X_1 \cap X_2$. The group $\mathfrak{A}_{X_1 \cup X_2}$ is generated by $3$-cycles. If the $3$ elements permutated by a $3$-cycle are all contained in one of the sets $X_i$, then the corresponding $3$-cycle is trivially in $\mathfrak{A}_{X_1 \cup X_2}$. A cycle containing $s$ is also contained in $\mathfrak{A}_{X_1 \cup X_2}$, since $(x_1,s,x_2)$ with $x_1 \in X_1,x_2 \in X_2$ is $[(x_2,s,x'_2),(x_1,s,x'_1)]$ for $x'_1 \in X_1,x'_2 \in X_2$. A cycle $(x_1,x'_1,x_2)$ with $x_1,x'_1 \in X_1 \setminus \{s\}$ and $x_2 \in X_2 \setminus \{s\}$ is just the conjugate of the cycle $(x_1,x'_1,s) \in \mathfrak{A}_{X_1}$ by the cycle $(x_2,x'_2,s)$. Thus $\langle \mathfrak{A}_{X_1} \cup \mathfrak{A}_{X_2} \rangle = \mathfrak{A}_{X_1 \cup X_2}$.
		Cycles of the form $(x_1,s,x_2)$ with $x_1 \in X_1,x_2 \in X_2$ are in $[\mathfrak{A}_{X_1} , \mathfrak{A}_{X_2}]$ and by conjugation with elements in $\mathfrak{A}_{X_1 \cup X_2}$ all $3$-cycles can be obtained.
	\end{proof}
	
	\begin{lem} [\cite{nek17}, Proposition 3.4]\label{lem: multiintersect}
		Let $\mathcal{G}$ be an \'etale Cantor groupoid and $\mathcal{B}^1, \mathcal{B}^2$ be multisections with degrees $d_1$ resp. $d_2$ greater than $3$ such that the intersection of their domains equals $I:=B_{1,1}^1 \cap B_{1,1}^2$. Denote by $\tilde{\mathcal{B}}$ the multisection with domain equal to the domain of $\mathcal{B}^1|_I \cup \mathcal{B}^2|_I$ and its collection of slices given by compositions of slices in $\mathcal{B}^1|_I \cup \mathcal{B}^2|_I$.
		Then $\mathfrak{A}(\tilde{\mathcal{B}}) \leq \langle \mathfrak{A}(\mathcal{B}^1) \cup \mathfrak{A}(\mathcal{B}^2) \rangle$.
	\end{lem}
	
	\begin{proof}
		The elements of $[\mathfrak{A}(\mathcal{B}^1), \mathfrak{A}(\mathcal{B}^2)]$ act trivially on $(B_{1,1}^1 \cup B_{1,1}^2) \setminus I$ are precisely the elements in $[\mathfrak{A}(\mathcal{B}^1|_I), \mathfrak{A}(\mathcal{B}^2|_I)]$ and accordingly are contained in  $\mathfrak{A}(\tilde{\mathcal{B}})$. Similarily conjugating elements in  $\mathfrak{A}(\tilde{\mathcal{B}})$ by elements in $v{A}(\mathcal{B}^1) \cup \mathfrak{A}(\mathcal{B}^2)$ is the same as conjugating by the corresponding elements in $\mathfrak{A}(\mathcal{B}^1|_I) \cup \mathfrak{A}(\mathcal{B}^2|_I)$. This means the conjugate closure of $[\mathfrak{A}(\mathcal{B}^1), \mathfrak{A}(\mathcal{B}^2)]$ in $\langle \mathfrak{A}(\mathcal{B}^1) \cup \mathfrak{A}(\mathcal{B}^2) \rangle$ is contained in $\mathfrak{A}(\tilde{\mathcal{B}})$	and by applying Lemma~\ref{lem: intersect} equality follows.
	\end{proof}
	
	We will use subdivisions to ease the navigation through Nekrashevych's proof of finite generation:
	
	\begin{thm} [\cite{nek17}, §5.2]\label{thm: alternating full groups of minimal expansive groupoids are finitely generated}
		Let $\mathcal{G}$ be an expansive, \'etale Cantor groupoid such that every $\mathcal{G}$-orbit contains $5$ or more elements. Then $\mathfrak{A}(\mathcal{G})$ is finitely generated.
	\end{thm}
	
	\begin{proof}
		Fix for the duration of the proof a metric $d$ on $\mathcal{G}^{(0)}$.
		\begin{enumerate}[(i)]
			\item \label{item: 1} By assumption for every $u \in \mathcal{G}^{(0)}$ there exist elements $g_1, g_2, g_3, g_4 \in \mathcal{G}$ with $s(g_i)=u$ and such that $u,r(g_1),r(g_2),r(g_3),r(g_4)$ are pairwise different, by which there exist open, compact slices $G_1,G_2,G_3,G_4$ with identical source such that $g_i \in G_i$ and the sets $s(G_1)$, $r(G_1)$, $r(G_2)$, $r(G_3)$ and $r(G_4)$ are pairwise disjoint. Hence, there exists an $\varepsilon_u > 0$ such that all units in the $\varepsilon_u$-neighbourhood of $u$ have $\mathcal{G}$-orbits containing at least $5$ elements which have a distance of at least $\varepsilon_u$ from each other. Since $\mathcal{G}^{(0)}$ is compact, there exists a countable family of elements in $\mathcal{G}^{(0)}$ with neighbourhoods chosen as above such that $\mathcal{G}^{(0)}$ is covered. Hence there exists an $\varepsilon > 0$ such that \emph{all} units have $\mathcal{G}$-orbits containing at least $5$ elements which have a distance of at least $\varepsilon$ from each other. Thus any finite cover of $\mathcal{G}^{(0)}$ by disjoint clopen sets with components of diameter less than $\varepsilon$ produces a finite, clopen partition $\mathcal{P}:=\{P_i\}$ of $\mathcal{G}^{(0)}$ such that every $\mathcal{G}$-orbit intersects at least $5$ components of $\mathcal{P}$.
		\end{enumerate}	
		
		Let $S$ be a compact, open generating set of $\mathcal{G}$ and let $\mathcal{P}$ be a clopen partition of $\mathcal{G}^{(0)}$ as chosen in \ref{item: 1}.
		
		\begin{enumerate}[(i),resume]	
			\item \label{item: 2}  We can assume $S$ to be symmetric -- by taking the union with its inverse. In addition by Lemma~\ref{lem: comp.gen.} we can assume $S$ to be contained in $\mathcal{G}\setminus \operatorname{Iso}(\mathcal{G})$. This implies $d(s(g),r(g)) \neq 0$ for all $g \in S$ and since $S$ is compact there exists an $\varepsilon > 0$ such that $d\big(s(g),r(g)\big) \geq \varepsilon$ for all $g \in S$. In consequence $s(g)$ and $r(g)$ lie in different components of $\mathcal{P}$ for all $g \in S$ by choosing a sufficient refinement. Moreover it can be assumed that $|Sg| \geq 4$ for all $g \in \mathcal{G}^{(0)}$ just by taking the union with sufficient open, compact slices contained in $\mathcal{G}\setminus \operatorname{Iso}(\mathcal{G})$. This implies that for all $g \in S$ there exist elements $g_1,g_2,g_3,g_4 \in S$ such that $s(g_i)=g$ and the units $g,r(g_1),r(g_2) ,r(g_3),r(g_4)$ are contained in pairwise different components of $\mathcal{P}$. Since $\mathcal{G}$ was assumed to be expansive, by Remark~\ref{rem: exp} there exists an expansive cover $\mathcal{S}$ of $S$ and by choosing a sufficient refinement it can be assumed that $s(F)$ and $r(F)$ are contained in different components of $\mathcal{P}$ for all $F \in \mathcal{S}$. Denote by $T$ the set of elements $g \in \bigcup_{k=1}^3 S^k$ for which $s(g)$ and $r(g)$ are contained in different components of $\mathcal{P}$. For every $g \in T$ and every $V \in \bigcup_{k=1}^3 \mathcal{S}^k$ there exists a multisection $\mathcal{M}$ of degree $5$	with $g \in M_{1,2} \subseteq V$ and all $M_{i,i}$ are contained in pairwise different components of $\mathcal{P}$. The set $T$ is compact, hence there exists a finite collection $\{\mathcal{M}^k\}$ of multisections choosen as $\mathcal{M}$ such that $\bigcup_k M_{1,2}^k$ covers $T$. Let $\mathfrak{A}$ denote the group $\langle \bigcup_k \mathfrak{A}(\mathcal{M}_k)\rangle$. The proof is accomplished by showing $\mathfrak{A}(\mathcal{G})=\mathfrak{A}$.
			
			\item \label{item: 4} Let $g \in \mathcal{G}$ be such that $s(g)$ and $r(g)$ sit in different components of $\mathcal{P}$. Then for every open, compact slice $B$ with $g \in B$ there exists a multisection $\mathcal{N}$ of degree $5$, such that $\mathfrak{A}(\mathcal{N}) \subseteq \mathfrak{A}$, $g \in N_{1,2} \subseteq B$ and the sets $N_{i,i}$ are contained in pairwise different components of $\mathcal{P}$:
			Since $\mathcal{S}$ is an expansive cover of the generating set $S$, the proof reduces to the case where $B$ is a finite product of open compact slices $B_1B_2 \dots  B_n$ with $B_i \in \mathcal{S}$. This allows to proof the statement by induction on $n$.
			If $n \leq 3$ this follows immediatly by definition of the family $\mathcal{M}_k$ in \ref{item: 2}.
			Suppose the statement is true for products of length $n$. Let $B=B_1B_2 \dots B_{n+1}$. Then $g$ can be uniquely written as a product $g=g_1g_2 \dots g_{n+1}$ with $g_i \in B_i$. By the assumptions on $S$ there exist an element $g' \in S$ such that $s(g')=s(g_2)$ and $r(g'),s(g),r(g)$ are contained in pairwise different components of $\mathcal{P}$. Then there exists an open, compact slice $B'$ with $g' \in B'$. Denote $g'':=g_1g_2g'^{-1}$ and $B'':=B_1B_2B'^{-1}$. It holds that $g \in B''B'B_3 \dots B_{n+1} \subseteq B$. By the induction hypothesis there exists a multisection $\mathcal{N}'$ of degree $5$ with $N'_{i,i}$ contained in pairwise different components of $\mathcal{P}$ such that $g'g_3 \dots g_{n+1} \in N'_{1,2} \subseteq B'B_3 \dots B_{n+1}$, $\mathfrak{A}(\mathcal{N'}) \leq \mathfrak{A}$ and we can additionaly assume, that $N'_{3,3}$ and $r(g)$ are contained in pairwise different components of $\mathcal{P}$. Denote by $\mathcal{N}^3$ the induced multisection of degree $3$ on $\{N'_{i,j}\}_{i,j=1}^3$. Since $B''$ is by definition contained in $\mathcal{S}^3$ and $g''$ is contained in $T$, there exists an $\mathcal{M}^j$ (as defined in \ref{item: 2}) with $g'' \in M_{1,2} \subseteq G''$. Choosing a sufficient submultisection of $\mathcal{M}^j$ gives a multisection $\mathcal{M}^3$ of degree $3$ with $g'' \in M^3_{1,2} \subseteq G''$ and such that the only intersecting components of the domains of $\mathcal{N}^3$ and $\mathcal{M}^3$ are $N^3_{2,2}$ and $M^3_{1,1}$. Applying Lemma~\ref{lem: multiintersect} to $\mathcal{N}^3$ and $\mathcal{M}^3$ gives the desired multisection.
			
			\item \label{item: 5} Let $g \in \mathcal{G}\setminus \operatorname{Iso}(\mathcal{G})$ be such that $s(g)$ and $r(g)$ are contained in the same component of $\mathcal{P}$. Then for every open, compact slice $B$ with $g \in B$ there exists a multisection $\mathcal{N}$ of degree $5$, such that $\mathfrak{A}(\mathcal{N}) \subseteq \mathfrak{A}$, $g \in N_{1,2} \subseteq B$ and the sets $N_{1,1} \cup N_{2,2},N_{3,3},N_{4,4},N_{5,5}$ are contained in pairwise different components of $\mathcal{P}$:
			To this end let $P \in \mathcal{P}$ such that $s(g),r(g) \in P$. There exists a $g' \in \mathcal{G}$ with $s(g')=s(g)$ and $g'\notin P$, hence by applying the result obtained in \ref{item: 4} there exists a multisection $\mathcal{N'}$ of degree $3$ with $\mathfrak{A}(\mathcal{N}') \leq \mathfrak{A}$, $g' \in N'_{1,2}$ and the components of its domain contained in pairwise different components of $\mathcal{P}$. Similarily replacing $g'$ by $g'g^{-1}$ gives a multisection $\mathcal{N}''$ of degree $3$ with the corresponding properties. It can be assumed that $N'_{3,3}$ and $N''_{3,3}$ lie in different components of $\mathcal{P}$. Lemma~\ref{lem: multiintersect} applied to $\mathcal{N}'$ and $\mathcal{N}''$ verifies the statement.	
			
			\item \label{item: 6} Let $g_1,g_2 \in \mathcal{G}$ be such that $s(g_1)=s(g_2)$ and $s(g_1),r(g_1),r(g_2)$ pairwise different. Let $\mathcal{N}^1$ be a multisection of degree $5$ given by \ref{item: 4} or \ref{item: 5} for $g_1$ -- whichever applies. It can be assumed that $r(g_2)$ is not in $N^1_{1,1}$ or $N^1_{2,2}$. There are at most $3$ components of $\mathcal{P}$ intersecting $\{s(g_1),r(g_1),r(g_2)\}$, hence there exists a $k \in 3,4,5$ with $r(g_2) \notin N^1_{k,k}$. Consequently $r(g_2)$ is not contained in the domain of the submultisection $\tilde{\mathcal{N}^1}$ of degree $3$ induced by $N^1_{1,1},N^1_{2,2},N^1_{k,k}$.
			Let $\mathcal{N}^2$ be a multisection of degree $5$ given by \ref{item: 4} or \ref{item: 5} for $g_2$. It can be assumed the component of the domain containing $r(g_2)$ does not intersect the domain of $\tilde{\mathcal{N}^1}$. There exists an $l \in 3,4,5$ such that $N^2_{l,l}$ is contained in a component of $\mathcal{P}$ which does not contain $r(g_2)$ and does not intersect the domain of $\tilde{\mathcal{N}^1}$. Denote by $\tilde{\mathcal{N}^2}$ the submultisection of degree $3$ induced by $N^2_{1,1},N^2_{2,2},N^2_{l,l}$.
			The multisection arising from the application of Lemma~\ref{lem: multiintersect} to $\tilde{\mathcal{N}^1}$ and $\tilde{\mathcal{N}^2}$ contains a sufficient submultisection $\tilde{\mathcal{N}}$ of degree $3$ such that $g_1,g_2$ are contained in respective slices contained in $\tilde{\mathcal{N}}$ and $\mathfrak{A}(\tilde{\mathcal{N}}) \leq \mathfrak{A}$.
			
			\item The multisection $\tilde{\mathcal{N}}$ that has just been constructed in \ref{item: 6} shows that  for every multisection $\mathcal{M}$ of degree $3$ and any element $g$ in its domain $\mathcal{M}$ there exists a restriction $\mathcal{N}$ to some subset containing $g$ such that $\mathfrak{A}(\mathcal{N}) \leq \mathfrak{A}$. By compactness of the domain of $\mathcal{M}$ there exists a finite family $\{\mathcal{N}^k\}$ of multisections of degree $3$ such that their domains are a finite cover of the domain of $\mathcal{M}$ and $\mathfrak{A}(\mathcal{N}^k) \leq \mathfrak{A}$. 	Proposition~\ref{prop: multcov} then implies $\mathfrak{A}(\mathcal{G})=\mathfrak{A}_3 (\mathcal{G}) = \mathfrak{A}$.
		\end{enumerate}
	\end{proof}
	
	\section{In terms of non-commutative Stone duality}\label{sec: topological full groups in light of non-commutative stone duality}
	
	In \cite{law16} and \cite{law17} Lawson undertook a translation of Matui's isomorphism theorems on \'etale groupoids under non-commutative Stone duality to the world of inverse semigroups and gave a purely algebraic proof. Subsection~\ref{subs: refined correspondences} contains refined versions of non-commutative Stone duality which include the groupoids Matui considered and Subsection~\ref{subs: translating the reconstruction} deals with the isomorphism theorem. All proofs will be omitted and we only give a suggestive description of the correspondences.
	
	\subsection{Refined correspondences}\label{subs: refined correspondences}
	
	The sufficient duality is in vein of Stone's classical duality between the category of Boolean algebras and the category of Stone spaces.\footnote{Note that as in the ``commutative" setting the discovery of the refined version predates and motivates the general, frame theoretic version.\\}
	\begin{defi}[\cite{law16}]
		\begin{enumerate}[(i)]
			\item An \'etale groupoid $\mathcal{G}$ is called a \emph{Boolean} if its subspace of units $\mathcal{G}^{(0)}$ is a Stone space.
			
			\item An inverse monoid $S$ is called \emph{Boolean} if its idempotents $E(S)$ form a Boolean algebra with the induced order.
		\end{enumerate}
	\end{defi}
	
	\begin{thm}[\cite{law16}, Theorem~3.4]
		There is a duality of categories between the category of Boolean inverse $\wedge$-monoids is dually equivalent to the category of Hausdorff Boolean groupoids.
	\end{thm}
	
	\begin{rem}
		 Since a proper filter of a Boolean inverse $\wedge$-monoid is prime \emph{if and only if} it is an ultrafilter (\cite{law16}, Lemma~3.2), on the level of objects in the passing from a Boolean inverse $\wedge$-monoid $S$ to a Boolean groupoid $\mathrm{G}(S)$ the completely prime filters can be replaced by ultrafilters and the obtained groupoid is topologized by taking the sets $V_s$ as a basis. In the other direction one associates to a Hausdorff Boolean groupoid $\mathcal{G}$ the inverse monoid of compact, open slices $\mathcal{B}_{\mathcal{G}}^{o,k}$.
	\end{rem}
	
	We do not dwell upon what the respective correct definitions for morphisms are. Under this duality the countable Boolean inverse $\wedge$-monoids translate to second countable Hausdorff Boolean groupoids.
	
	\begin{defi}[\cite{law16}, p. 6]
		Let $S$ be a Boolean inverse monoid.
		\begin{enumerate}[(i)]
			\item An ideal $I$ of $S$ is called \emph{$\vee$-closed}, if for all $a,b \in I$ existence of $a \vee b$ implies $a \vee b \in I$.
			
			\item It is said to be \emph{$0$-simplifying}, if it contains no non-trivial $\vee$-closed ideals.
		\end{enumerate}
	\end{defi}
	
	The $\vee$-closed ideals in a Boolean inverse $\wedge$-monoid $S$ correspond to unions of $\mathrm{G}(S)$-orbits (see \cite{law12}, §5), which implies:
	
	\begin{thm}[\cite{law16}, Corollary~4.8]
		Let $S$ be a Boolean inverse $\wedge$-monoid. The groupoid of ultrafilters $\mathrm{G}(S)$ is minimal \emph{if and only if} $S$ is $0$-simplifying.
	\end{thm}
	
	In Lemma~\ref{lem: existence of pencils} we saw that for \'etale groupoids minimality corresponds to the existence of a certain family of slices which translates as follows:
	\begin{lem}\label{lem: existence of pencils - monoids}
		Let $S$ be a Boolean inverse monoid and let $e,f \in E(S)$ with $e,f \neq 0$. If $S$ is $0$-simplifying, there exists a finite set $\{s_1,\dots,s_k \}\subset S$ such that $e=\bigvee_{i=1}^k d(s_i)$ and $r(s_i) \leq f$ for all $i \in \{1,\dots,k \}$.
	\end{lem}
	
	The algebraic characterization of effectiveness is established through the following correspondence:
	\begin{lem}[\cite{law16}, Lemma~4.9]\label{lem: lemma fundamental - effective}
		Let $S$ be a Boolean inverse $\wedge$-monoid and $\mathrm{G}(S)$ be the associated groupoid of ultrafilters. For every $s \in S$ the following are equivalent:
		\begin{enumerate}[(i)]
			\item $a \in Z(E(S))$
			
			\item $V_a \subseteq \operatorname{Iso}(\mathrm{G}(S))$
		\end{enumerate}
	\end{lem}
	
	The following definition is a specialized version of a notion from classical inverse semigroup theory (see \cite{law98}):
	
	\begin{defi}[\cite{law98}, p.139f]
		Let $S$ be an inverse semigroup.
		\begin{enumerate}[(i)]
			\item The \emph{centralizer $Z(E(S))$ of idempotents in $S$} is the set of elements $s \in S$ such that $se=es$ for all $e \in E(S)$.
			
			\item It is said to be \emph{fundamental}, if $E(S)=Z(E(S))$ holds.
		\end{enumerate}
	\end{defi}
	
	The following follows from the characterization in Lemma~\ref{lem: lemma fundamental - effective}:
	\begin{thm}[\cite{law16}, Theorem~4.10]\label{thm: fundamental-effective}
		Let $S$ be a Boolean inverse $\wedge$-monoid. The groupoid of ultrafilters $\mathrm{G}(S)$ is effective \emph{if and only if} $S$ is fundamental.
	\end{thm}
	
	We thus have a class of Boolean inverse $\wedge$-monoids for which one can state the algebraic counterpart to the isomorphism theorem:
	
	\begin{defi}[\cite{law17}, p.381f]
		\begin{enumerate}[(i)]
			\item A Boolean inverse monoid is said to be \emph{simple} if it is $0$-simplifying and fundamental.
			
			\item A \emph{Tarski monoid} $S$ is a countable Boolean inverse $\wedge$-monoid $S$ of which the semilattice of idempotents $E(S)$ is a countable, atomless Boolean algebra.\footnote{See Remark~\ref{rem: cant}(ii).\\}
		\end{enumerate}
	\end{defi}
	
	The effective, minimal, Hausdorff, \'etale Cantor groupoids considered by Matui thus correspond under non-commutative Stone duality to simple Tarski monoids. As mentioned in Remark~\ref{rem: topfgroupsdefi}~(ii) the topological full group of an \'etale Cantor groupoid $\mathcal{G}$ is by definition the unit group of the inverse monoid of compact open slices and thus Theorem~\ref{thm: isom} translates under non-commutative Stone duality to:
	
	\begin{thm}[\cite{law17}, Theorem~2.10]\label{thm: isom-monoid version}
		Let $S_1$ and $S_2$ be simple Tarski monoids. Then the following are equivalent:
		\begin{enumerate}[(i)]
			\item $S_1 \cong S_2$
			
			\item $U(S_1) \cong U(S_2)$
		\end{enumerate}
	\end{thm}
	
	\subsection{Translating the reconstruction}\label{subs: translating the reconstruction}
	
	\begin{defi}[\cite{law17}, p.382]
		Let $S$ be a Boolean inverse monoid with associated Boolean algebra of idempotents $E(S)$.
		\begin{enumerate}[(i)]
			\item Denote by $X(S)$ the Stone space dual to $E(S)$. 
			
			\item Let $e \in E(S)$. Denote by $U_e$ the set of ultrafilters of $E(S)$ that contain $e$.
		\end{enumerate}
	\end{defi}
	
	The elements of $X(S)$ are given by ultrafilters of $E(S)$ and the open sets of $X(S)$ are of the form $U_e$ for some $e \in E(S)$. The group of units $U(S)$ acts on $E(S)$ by $e \mapsto geg^{-1}$. The following is no surprise considering Theorem~\ref{thm: fundamental-effective}:
	\begin{prop}[\cite{law17}, Proposition~3.1]
		Let $S$ be a Boolean inverse monoid. Then $S$ is fundamental \emph{if and only if} the action of $U(S)$ on $E(S)$ is faithful.
	\end{prop}
	
	The algebraic proof of Theorem~\ref{thm: isom-monoid version} is a consequence of the following proposition:
	
	\begin{prop}[\cite{law17}, Lemma~3.10]\label{prop: spatial iso induces isomorphism of monoids}
		Let $S_1$ and $S_2$ be fundamental Boolean inverse monoids and let $G_1 \leq U(S_1)$ and $G_2 \leq U(S_2)$ be subgroups such that $S_i=(G_i^{\downarrow})^{\vee}$ for $i \in \{1,2\}$.\footnote{This is the closure of $G_i^{\uparrow}$ with respect to the join operation.\\} Assume there exists a group isomorphism $\alpha \colon G_1 \to G_2$ and an isomorphism of Boolean algebras $\gamma \colon E(S_1)\to E(S_2)$ such that for all $g \in G_1$ and $e \in E(S_1)$ the following is satisfied
		\begin{equation*}
		\gamma(geg^{-1})=\alpha(g)\gamma(e)\alpha(g^{-1}).
		\end{equation*}
		Then there exists a unique isomorphism $\Theta \colon S_1 \to S_2$ that extends $\alpha$ and $\gamma$.
	\end{prop}
	
	In light of the requirement $S_i=(G_i^{\downarrow})^{\vee}$ Lawson introduces the notion of piecewise factorizability:
	\begin{defi}[\cite{law17}, p.391]
		Let $S$ be a Boolean inverse monoid and let $G \leq U(S)$. Then $S$ is said to be \emph{piecewise factorizable with respect to $G$} if for every $s \in S$ there exist finite collections $\{s_1,\dots,s_k \}\subset S$ and $\{g_1,\dots,g_k \}\subset G$ with $s_i \leq g_i$ for all $i \in \{1,\dots,k\}$ such that $s$ is of the form $s=\bigvee_{i=1}^k g_i d(s_i)$.
	\end{defi}

	A crucial ingredient in the proofs of spatial realization modeled after the proof of Theorem~384D in \cite{fre11} e.g. Matui's proof, is an ``abundance of \emph{involutions}" -- see Definition~\ref{defi: class F}. Lawson demonstrates that \emph{infinitesimals} in $S$ give rise to involutions:
	\begin{defi}[\cite{law17}, p.383f]
		Let $S$ be an inverse monoid.
		\begin{enumerate}[(i)]
			\item An element $s \in S$ is called an \emph{infinitesimal} if $s \neq 0$ and $s^2 = 0$.
			
			\item A pair $(s,t)$ of infinitesimals $s,t \in S$ is called a \emph{$2$-infinitesimal of $S$} if $d(s)=r(t)$ and $st$ is an infinitesimal.
		\end{enumerate}
	\end{defi}
	
	An element $s$ of an inverse monoid $S$ satisfies $s^2=0$ \emph{if and only if} $r(s)d(s)=0$.
	
	\begin{lem}[\cite{law17}, Lemma~2.5 \& 2.7]\label{lem: involutions and 3-cycles}
		Let $S$ be a Boolean inverse monoid.
		\begin{enumerate}[(i)]
			\item Let $s \in S$ be an infinitesimal. Then $g:=s \vee s^{-1} \vee \neg (d(s)\vee r(s) )$ is an involution in $U(S)$ with $s \leq g$.\footnote{The symbol $\neg$ is the negation for the Boolean algebra $E(S)$.\\}
			
			\item Let $(s,t)$ be a $2$-infinitesimal of $S$. Then the element
			\begin{equation*}
			g:= s \vee t \vee (st)^{-1} \vee \neg(r(s) \vee d(s) \vee d(t) )
			\end{equation*}
			is an element of $U(S)$ of order $3$.
		\end{enumerate}
	\end{lem}
	
	\begin{defi}[\cite{law17}, p.383f]
		Let $S$ be a Boolean inverse monoid.
		\begin{enumerate}[(i)]
			\item Every involution $g \in U(S)$ that is constructed as in Lemma~\ref{lem: involutions and 3-cycles}(i) is called a \emph{special involution} and the subgroup of $U(S)$ generated by the set of all special involutions is denoted by $\mathfrak{S}(S)$.
			
			\item Every $g \in U(S)$ of order $3$ that is constructed as in Lemma~\ref{lem: involutions and 3-cycles}(ii) is called a \emph{special $3$-cycle} and the subgroup of $U(S)$ generated by the set of all special $3$-cycles is denoted by $\mathfrak{A}(S)$.
		\end{enumerate}
		
	\end{defi}
	
	In the groupoid picture for a Tarski monoid the above construction of special involutions resp. special $3$-cycles corresponds precisely to the construction of involutions resp. $3$-cycles from multisections as in Definition~\ref{defi: symmetr, alternating full group} and thus this groups are precisely the groups $\mathfrak{S}(\mathcal{G})$ and $\mathfrak{A}(\mathcal{G})$ defined by Nekrashevych in the Tarski monoid picture.
	
	\begin{prop}[\cite{law17}, §3.3]\label{prop: tarski monoids are p.f.}
		Let $S$ be a simple Tarski monoid. Then $S$ is piecewise factorizable with respect to $\mathfrak{S}(S)$ and in consequence $S=(\mathfrak{S}(S)^{\downarrow})^{\vee}$.
	\end{prop}
	
	Another important aspect in the proof of the isomorphism theorem in the groupoid setting are the supports of partial homeomorphisms. The inverse monoid counterpart comes in the following shape:
	\begin{defi}[\cite{law17}, p.393]
		Let $S$ be a Boolean inverse $\wedge$-monoid. The \emph{support operator} $\sigma \colon S \to E(S)$ is given by $\sigma \colon s \mapsto \neg(s \wedge 1)s^{-1}s$.
	\end{defi} 
	
	This definition indeed corresponds to the support in the topological sense in that by Proposition~3.16 of \cite{law17} we have for every $g \in U(S)$ in a fundamental Boolean inverse $\wedge$-monoid $S$
	\begin{equation*}
	U_{\sigma(g)}=\overline{\{F \in X(S)| gFg^{-1}\neq F \}}.
	\end{equation*}
	
	This allows furthermore to give a translation of local subgroups\footnote{See Definition~\ref{defi: spatial homeomorphism + local subgroups}(ii).\\} of $U(S)$:
	\begin{defi}[\cite{law17}, p.396]
		Let $S$ be a Tarski monoid. Let $G \leq U(S)$ such that $G$ contains $\mathfrak{S}(S)$. Let $e \in E(S)$. The \emph{local subgroup} $G_e$ is given by
		\begin{equation*}
		G_e=\{g \in G| \sigma(g) \leq e \}
		\end{equation*}
	\end{defi}
	Analogous to the groupoid setting the order on $E(S)$ corresponds precisely to the containment of local subgroups.
	In \cite{mat15} Matui introduces groups of class F by a set of axioms, demonstrates that ismorphisms of such groups induce spatial homeomorphisms and shows that subgroups of $\mathfrak{T}(\mathcal{G})$ that contain $\mathfrak{T}(\mathcal{G})'$ have class F (-- see Subsection~\ref{subs: isomorphism theorems II}). The approach in \cite{law17} is a little different in that class F is translated to a list of axioms on Booleam inverse $\wedge$-monoids (and the associated groups of special involutions resp. $3$-cycles).
	
	\begin{defi}[\cite{law17}, §3.5]
		Let $S$ be an Boolean inverse $\wedge$-monoid. It is said to be \emph{of class F} if it satisfies the following properties:
		\begin{enumerate}
			\item [(F1)] For every $e \in E(S)$ with $e \neq 0$ there exists a finite set of special involutions $\tau_1,\dots,\tau_k \in \mathfrak{S}(S)$ such that $e=\bigvee_{i=1}^k \sigma(\tau_i)$.
			
			\item [(F2)] For every involution $g \in U(S)$ and $e \in E(S)$ with $e \neq 0$ and $e \leq \sigma(g)$ there exists a special involution $\tau \in \mathfrak{S}(S)$ such that $\sigma(\tau) \leq d(ge)\vee r(ge)$ and $ \sigma(\tau) \leq 1 \wedge g\tau$.
			
			\item [(F3)] For every $e \in E(S)$ with $e \neq 0$ there exists a special $3$-cycle $c$ such that $\sigma(c)\leq e$.
		\end{enumerate} 
	\end{defi}
	
	Considering Lemma~\ref{lem: involutions and 3-cycles} the existence of special involutions resp. $3$-cycles follows from the existence of infinitesimals resp. $2$-infinitesimals -- this is guaranteed by the property of being $0$-simplifying:
	\begin{lem}[\cite{law17}, Lemma~3.5]\label{lem: existence of infinitesimals}
		Let $S$ be a $0$-simplifying Tarski monoid. Let $F \subseteq E(S)$ be an ultrafilter and let $e \in F$. Then there exists an infinitesimal $s \in S$ with $d(s) \in F$ and $s \in eSe$.
	\end{lem}
	\begin{rem}
		As for Lemma~\ref{lem: existence of infinitesimals}: Note that submonoids of the form $eSe$ correspond to restriction groupoids. Analogous to the construction of an involution in $\mathfrak{T}(\mathcal{G})$ following Lemma~\ref{lem: existence of pencils} the existence of a sufficient infinitesimal follows here from Lemma~\ref{lem: existence of pencils - monoids}. Furthermore Lemma~\ref{lem: existence of infinitesimals} induces the existence of $2$-infinitesimals.
	\end{rem}
	
	With the help of the existence of sufficent infinitesimals and the construction of involutions from infintesimals resp. of $3$-cycles from $2$-infinitesimals Lawson shows:
	\begin{prop}[\cite{law17}, Proposition~3.20]
		Simple Tarski monoids are of class F.
	\end{prop}
	
	Let $S$ be a simple Tarski monoid and let $G \leq U(S)$ such that $G$ contains $\mathfrak{S}(S)$. Lawson shows that local subgroups of the form $G_{\sigma(\tau)}$ where $\tau$ is an involution in $G$ can be described algebraically analogous to Definition~\ref{defi: needed for characterization of local subgroups} and Lemma~\ref{lem: characterization local subgroups} (\cite{law17}, Theorem~3.23). Furthermore analogous to Lemma~\ref{lem: containment and intersection of supports} the following holds:
	\begin{lem}[\cite{law17}, Lemma~3.24]
		Let $S_1$ and $S_2$   be simple Tarski monoids and for $i \in \{1,2\}$ let $G_i \leq U(S_i)$ be such that $\mathfrak{S}(S_i) \leq G_i$. Let $\alpha \colon G_1 \to G_2$ be an isomorphism and let $s,t \in G$ be involutions. Then the following hold:
		\begin{enumerate}[(i)]
			\item $\sigma(s) \leq \sigma(t) \Leftrightarrow \sigma(\alpha(s)) \leq \sigma(\alpha(t))$
			
			\item $\sigma(s) \sigma(t) =0 \Leftrightarrow \sigma(\alpha(s)) \sigma(\alpha(t))=0$
		\end{enumerate}
	\end{lem}
	
	The analogue of Theorem~\ref{thm: matuispat} is given by:
	\begin{thm}[\cite{law17}, Proposition~3.6]\label{thm: spat monoid}
		Let $S_1,S_2$   be simple Tarski monoids and for $i \in \{1,2\}$ let $G_i \leq U(S_i)$ be such that $\mathfrak{S}(S_i) \leq G_i$. Let $\alpha \colon G_1 \to G_2$ be an isomorphism. Then there exists a homeomorphism $\varphi \colon X(S_1) \to X(S_2)$ such that $\varphi(gFg^{-1})=\alpha(g)\varphi(F)\alpha(g^{-1})$.
	\end{thm}
	
	The function $\varphi$ is given by $\varphi(F):=(\{\sigma(\tau)|\tau \in G,\tau^2=1, \sigma(t)\in F \})^{\uparrow}$.\footnote{It needs thus to be verified that $\varphi(F)$ is an ultrafilter of $E(S_2)$, that the arising map is indeed a homeomorphism and that $\varphi$ satisfies the required property.\\} The homeomorphism $\varphi$ of Stone spaces corresponds to an isomorphism of Boolean algebras $\gamma \colon E(S_1) \to E(S_2)$ with the following properties:
	\begin{enumerate}[(i)]
		\item $\gamma(\sigma(t))=\sigma(\alpha(t))$ for all involutions in $G_1$
		
		\item $\gamma(geg^{-1})=\alpha(g)\gamma(e)\alpha(g^{-1})$ for all $g \in G_1$ and $e \in E(S_1)$
	\end{enumerate}
	
	Putting together Proposition~\ref{prop: spatial iso induces isomorphism of monoids}, Proposition~\ref{prop: tarski monoids are p.f.} and Theorem~\ref{thm: spat monoid} concludes the proof of Theorem~\ref{thm: isom-monoid version}.
	
	\section{Significance for geometric group theory}\label{sec: Significance for geometric group theory}
	
	We have already talked about combinatorial aspects and growth of topological full groups and their finitely generated subgroups in preceding subsections, however, it is now that we push forward to the aspects that constitute the peculiarity of topological full groups in the context of group theory. Subsection~\ref{subs: the lef-property} starts with approximation properties as corollaries of the factorization result by Grigorchuk and Medynets reviewed in Subsection~\ref{subs: topological full groups in terms of kakutani-rokhlin partitions}. In Subsection~\ref{subs: extensively amenable actions and amenability of topological full groups} we discuss the amenability of $\mathfrak{T}(\varphi)$ for minimal subshifts. In Subsection~\ref{subs: liouville property, quantifying amenability and subgroups of intermediate growth} we take a quick look at \cite{mb14a} where the existence of topological full groups with the Liouville property was demonstrated. We conclude the section with a Subsection~\ref{subs: intermediate growth} on topological full groups as a supply of new examples of groups with intermediate growth.  
	
	\subsection{The LEF-property and infinite presentation}\label{subs: the lef-property}
	
	Some of the main corollaries of Theorem~\ref{thm: grig-med} (Theorem~4.7 of \cite{gm14}) have not yet been mentioned: Topological full groups of minimal Cantor systems are LEF-groups\footnote{See Appendix~\ref{app: lef-groups}.\\} and cannot be finitely presented.
	
	\begin{thm}[\cite{gm14}, Theorem~5.1]\label{thm: topfullgroup is LEF group}
		Let $(X,\varphi)$ be a minimal Cantor system. Then $\mathfrak{T}(\varphi)$ is a LEF-group.
	\end{thm}
	\begin{proof}
		Let $F \subset \mathfrak{T}(\varphi)$ be a finite subset. There exists an $n_0 \in \mathbb{N}$, such that the following properties hold:
		\begin{enumerate}[(i)]
			\item For every $n\geq n_0$ and every pair of elements $\gamma_1,\gamma_2 \in (F \cup \{\operatorname{id}\})^2$ with $\gamma_1 \neq \gamma_2$ the decomposition satisfies $P_{\gamma_1} \neq P_{\gamma_2}$.
			\item There exists an $m \in \mathbb{N}$ such that for every $\gamma \in F$ and $n \geq n_0$  the supportive set of the $n$-rotation $R_\gamma$ is contained in $\{0,\dots,m \}$.
			\item $P_{\gamma,i}^{-1}(k)=P_{\gamma,j}^{-1}(k)$ for all $\gamma \in F$, $n\geq n_0$, $k \in \{-m,\dots,0,\dots,m \}$ and $i,j \in \{1,\dots,i_n\}$.
		\end{enumerate}
		
		Property (i) follows from Theorem~\ref{thm: grig-med}(iii), property (ii) follows from property (a) in Theorem~\ref{thm: grig-med}(i) and property (iii) follows from property (c) in Theorem~\ref{thm: grig-med}(i). Fix an $n \geq n_0$. Then by (i)-(iii) the homeomorphism $P_{\gamma_1}^{-1} R_{\gamma_2} P_{\gamma_1}$ is an $n$-rotation for every $\gamma_1,\gamma_2 \in F$. Consider the finite group $\mathcal{P}(\mathcal{A}_n)$ of $n$-permutations.
		The map $\alpha \colon (F \cup \{\operatorname{id}\})^2 \to \mathcal{P}(\mathcal{A}_n)$ defined by $\alpha\colon \gamma \mapsto P_\gamma$ is well-defined by uniqueness of the decomposition (see Theorem~\ref{thm: grig-med}(ii)). It is injective by property (i). Since $P_{\gamma_1}^{-1} R_{\gamma_2} P_{\gamma_1}$ is an $n$-rotation for every $\gamma_1,\gamma_2 \in F$, we have:
		\begin{equation*}
		\begin{gathered}
		\alpha(\gamma_1\gamma_2)=\alpha(P_{\gamma_1}R_{\gamma_1}P_{\gamma_2}R_{\gamma_2})=\alpha(P_{\gamma_1}(P_{\gamma_2}P_{\gamma_2}^{-1})R_{\gamma_1}P_{\gamma_2}R_{\gamma_2})=\\
		\alpha((P_{\gamma_1}P_{\gamma_2})(P_{\gamma_2}^{-1}R_{\gamma_1}P_{\gamma_2}R_{\gamma_2}))=P_{\gamma_1}P_{\gamma_2}
		\end{gathered}
		\end{equation*}
		and thus $\alpha$ is multiplicative on $F$. Let $\tilde{\alpha}$ be the extension of $\alpha$ to all of $\mathfrak{T}(\varphi)$ defined by $\tilde{\alpha}(\gamma)=\operatorname{id}$ for $\gamma \in \mathfrak{T}(\varphi)\setminus (F \cup \{\operatorname{id}\})^2$. Then $\tilde{\alpha}$ is the required map for the LEF-property.
	\end{proof}
	
	The LEF-property entails some immediate consequences\footnote{See \cite{cc10}, Chapter~7 for terms not defined here.\\}:
	
	\begin{cor}[\cite{gm14}, Corollary~2.7]
		Let $(X,\varphi)$ be a minimal Cantor system. Then its topological full group $\mathfrak{T}(\varphi)$ is
		\begin{enumerate}[(i)]
			\item an LEA-group i.e. locally embeddable into the class of amenable groups.
			
			\item sofic.
		\end{enumerate}
	\end{cor}
	
	In Theorem~\ref{thm: simpel}, $\mathfrak{T}(\varphi)'$ of a minimal Cantor system is simple. By Corollary~\ref{cor: no infinite, finitely presented simple LEF-group} there exists no infinite, finitely presented, simple LEF-group thus Theorem~\ref{thm: topfullgroup is LEF group} implies:
	\begin{cor}
		Let $(X,\varphi)$ be a minimal Cantor system. Then $\mathfrak{T}(\varphi)'$ is never finitely presented.
	\end{cor}
	
	\subsection{Amenability of topological full groups}\label{subs: extensively amenable actions and amenability of topological full groups}
	
	In the first version of \cite{gm14} it was conjectured that topological full groups of minimal Cantor systems are amenable.\footnote{For the definition of amenability of groups and their actions and characterizations see Appendix~\ref{app: amenability}.\\} This was confirmed by Juschenko and Monod in \cite{jm13}. Together with the results of Matui this produces an uncountable family of pairwise non-isomorphic, finitely generated, infinite, simple, amenable groups. Examples of such groups were previously unknown! 
	
	Amenability of $\mathfrak{T}(\varphi)$ is obtained via an embedding into the wobbling group $\mathbf{W}(\mathbb{Z})$:
	\begin{defi}
		Let $(X,d)$ be a metric space. The \emph{wobbling group} $\mathbf{W}(X)$ is the group of all bijections $f$ of $X$ such that $\sup\{d(x,f(x))|x \in X \}<\infty$ holds.
	\end{defi}
	While the wobbling group $\mathbf{W}(\mathbb{Z})$ itself is not amenable,\footnote{One can show there is a copy of $\mathbb{Z}/2\mathbb{Z}\ast \mathbb{Z}/2\mathbb{Z} \ast \mathbb{Z}/2\mathbb{Z}$ in $\mathbf{W}(\mathbb{Z})$.\\} it allows for the following amenability criterion:
	
	\begin{thm}\label{thm: amen-crit}
		Let $G$ be a group that admits a morphism $\iota \colon G \to \mathbf{W}(\mathbb{Z})$. If the stabilizer subgroup $G_\mathbb{N}$ of the action induced by $\iota$ and the natural action of $\mathbf{W}(\mathbb{Z})$ is amenable, then the group $G$ is amenable. 
	\end{thm}
	
	The demonstration of Theorem~\ref{thm: amen-crit} requires some set-up: Let $X$ be a set. Denote by $\mathcal{P}_f(X)$ the set of finite subsets of $X$.\footnote{By definition $\emptyset \in \mathcal{P}_f(X)$.\\} Note, that $\mathcal{P}_f(X)$ is a discrete abelian group with the symmetric difference $\Delta$ as the operation. Note that $\mathcal{P}_f(X)$ is naturally isomorphic to $\bigoplus_{X} \mathbb{Z}/2\mathbb{Z}$ and is subgroup of the group of all subsets of $X$ with $\Delta$ as operation, which is naturally isomorphic to $(\mathbb{Z}/2\mathbb{Z})^X$. The dual group $\widehat{\mathcal{P}_f(X)}$ of characters of $\mathcal{P}_f(X)$ is isomorphic to $(\mathbb{Z}/2\mathbb{Z})^X$ and the canonical duality pairing $\widehat{\mathcal{P}_f(X)}\times \mathcal{P}_f(X) \to \mathbb{C}$ computes for $z=\{z_x\}_{x \in X} \in (\mathbb{Z}/2\mathbb{Z})^X$ and $A \in \mathcal{P}_f(X)$ as
	\begin{equation*}
	\langle z,A \rangle = \exp(\mathrm{i}\pi \sum_{x \in A} z_x)
	\end{equation*}
	The normalized Haar measure $\lambda$ on $\widehat{\mathcal{P}_f(X)}$ corresponds to the symmetric Bernoulli measure on $(\mathbb{Z}/2\mathbb{Z})^X$ and gives rise to the Hilbert space $\mathrm{L}^2(\widehat{\mathcal{P}_f(X)},\lambda)$. The action of a group $G$ on a set $X$ naturally induces an action of $G$ on $\bigoplus_{X} \mathbb{Z}/2\mathbb{Z}$ by permutation of summands producing the semi-direct product $\mathcal{P}_f(X)\rtimes G \cong (\mathbb{Z}/2\mathbb{Z}) \wr_X G$ which is termed the \emph{permutational wreath product} or \emph{lamplighter group of $\mathbb{Z}/2\mathbb{Z}$ and $G$}\footnote{This is a wide generalization of the lamplighter group $L$ from Definition~\ref{defi: lamp}.\\}. The permutational wreath product admits  action on $\mathcal{P}_f(X)$ by $(A,g) \cdot B:= g(A) \Delta B$ for all $A,B \in \mathcal{P}_f(X)$ and $g \in G$. The proof of Theorem~\ref{thm: amen-crit} requires two fundamental ingredients, Theorem~\ref{thm: amenact} and the fact that the action $\mathcal{P}_f(\mathbb{Z})\rtimes \mathbf{W}(\mathbb{Z}) \curvearrowright \mathcal{P}_f(\mathbb{Z})$ is amenable. This fact is the technical core of \cite{jm13} and motivated the following definition:
	\begin{defi}[\cite{jmms18}, Definition~1.1]
		Let $G$ be a group that acts on a set $X$ via $\alpha \colon G \curvearrowright X$. The action is called \emph{extensively amenable}\footnote{Its name comes from the fact that it is stable under a specific notion of extension of actions (\cite{jmms18}, Proposition~2.4 \& Corollary~2.5).\\} if there exists a $G$-invariant mean $m$ on $\mathcal{P}_f(X)$ such that
		\begin{equation*}
		\mu(\{A \in \mathcal{P}_f(X) | B \subseteq A \})=1 \text{ for every }B \in \mathcal{P}_f(X).
		\end{equation*}
	\end{defi}
	
	This definition is justified by the following lemma:
	
	\begin{lem}[\cite{jns16}, Lemma~2.7]\label{lem: extamen}
		Let $G$ be a group that acts transitively on a set $X$ and let $y \in X$. Then the following are equivalent:
		\begin{enumerate}[(i)]
			\item The action of $G$ on $X$ is extensively amenable.
			
			\item There exists a net of unit vectors $\{f_n\}$ of in $\mathrm{L}^2(\widehat{\mathcal{P}_f(X)},\lambda)$ such that $\|g\cdot f_n-f_n\|_2\to 0$ and $\|f_n|_{\{ \{z_x\}_{x \in X}|z_y=0 \} }\|_2 \to 1$.
			
			\item The action of $\mathcal{P}_f(X)\rtimes G$ on $\mathcal{P}_f(X)$ is amenable.
		\end{enumerate}
	\end{lem}
	
	The most intricate step in the proof of amenability for $\mathfrak{T}(\varphi)$ from \cite{jm13} is the following result:
	\begin{thm}[\cite{jm13}, Theorem~2.1]\label{thm: wobl}
		The action of the wobbling group $\mathbf{W}(\mathbb{Z})$ on $\mathbb{Z}$ satisfies condition (ii) from Lemma~\ref{lem: extamen}.
	\end{thm}
	
	Extensive amenability of group actions is a property considerably stronger than amenability of group actions:
	
	\begin{lem}[\cite{jmms18}, Lemma~2.1]
		Let $G$ be a group that acts on a set $X$. If $G$ is amenable, the action is extensively amenable and if the action is extensively amenable and $X$ is non-empty the action is amenable.
	\end{lem}
	\begin{proofsketch}
		The action of a group $G$ on a set $X$ induces an action of $G$ on the set $M$ of means on $\mathcal{P}_f(X)$ such that $\mu(\{A \in \mathcal{P}_f(X) | B \subseteq A \})=1 \text{ for every }B \in \mathcal{P}_f(X)$. The set of $M$ can be shown to be a non-empty, compact, convex subset of $\ell^\infty(X)^*$.  Since $G$ is amenable, the action of $G$ on $M$ is amenable i.e. it has a fixed point $\mu$, which guarantees extensive amenability. For the second statement, by extensive amenability there exists a $G$-invariant mean $\tilde{\mu}$ on $\mathcal{P}_f(X)\setminus \emptyset$ producing a $G$-invariant mean on $X$ as follows:
		\begin{equation*}
		f \in \ell^\infty(X) \mapsto \int_{\mathcal{P}_f(X)\setminus \emptyset} |A|^{-1}\sum_{x \in A} f(x) \; \mathrm{d}\tilde{\mu} (A).
		\end{equation*}
	\end{proofsketch}
	
	Extensive amenability passes over to actions of subgroups on subgroup orbits:
	\begin{lem}\label{lem: extamensubgrp}
		Let a group $G$ act extensively amenable on a set $X$. Then for every subgroup $H \leq G$ and every $H$-orbit $Y$ the induced action is extensively amenable.
	\end{lem}
	\begin{proof}
		By Lemma~\ref{lem: extamen} there exists an $\mathcal{P}_f(X)\rtimes G$-invariant mean $m$ on $\mathcal{P}_f(X)$.	Let $H$ be a subgroup and $Y$ an $H$-orbit. The map $\alpha \colon\mathcal{P}_f(X)\to \mathcal{P}_f(Y)$  given by $A\mapsto A \cap Y$ is a $\mathcal{P}_f(Y)\rtimes H$-equivariant map, since for all $A \in \mathcal{P}_f(X), h \in H$ $B \in \mathcal{P}_f(Y)$ we have 
		\begin{equation*}
		(B,h)\cdot (A \cap Y)=B \Delta h(A \cap Y)=(B\cap Y) \Delta (h(A) \cap Y)=((B,h)\cdot A) \cap Y.
		\end{equation*}
		The pushforward $\alpha_{*}m$ is then a $\mathcal{P}_f(Y)\rtimes H$-invariant mean on $\mathcal{P}_f(Y)$.
	\end{proof}
	
	Ideas of \cite{jm13} where developed further in \cite{jns16}, in particular a sufficient condition for condition (ii) from Lemma~\ref{lem: extamen} was given in terms of stochastics on Schreier graphs\footnote{See Appendix~\ref{app: graphs of actions}.\\}:
	\begin{defi}[\cite{jns16}, §2.3]
		Let a group $G$ with finite, symmetric generating set act transitively on a set $X$. A function $f \in \mathrm{L}^2(\widehat{\mathcal{P}_f(X)},\lambda)$ is said to be a \emph{product of independent variables} or \emph{p.i.r.-function} if there exists a family $\{f_x\}_{x \in X}$ of functions in $\mathbb{R}^{(\mathbb{Z}/2\mathbb{Z})}$ such that $f(z)=\prod_{x \in X} f_x(z_x)$ for every $z=\{z_x\}_{x \in X} \in (\mathbb{Z}/2\mathbb{Z})^X$.
	\end{defi} 
	
	\begin{thm}[\cite{jns16}, Theorem~2.8]
		Let a group $G$ with finite, symmetric generating set act transitively on a set $X$ and let $y \in X$. Then the following are equivalent:
		\begin{enumerate}[(i)]
			\item There exists a sequence of p.i.r.-functions $\{f_n\}\in \mathrm{L}^2(\widehat{\mathcal{P}_f(X)},\lambda)$ satisfying $\|g\cdot f_n-f_n\|_2\to 0$ and $\|f_n|_{\{ \{z_x\}_{x \in X}|z_y=0 \} }\|_2 \to 1$.
			
			\item The Schreier graph $\Gamma(X,G,S)$ is recurrent.
		\end{enumerate}
	\end{thm}
	
	Assuming Theorem~\ref{thm: wobl} we take a quick look at a proof of Theorem~\ref{thm: amen-crit} following the path of \cite{dC13}, §4.2:
	
	\begin{proof}(Theorem~\ref{thm: amen-crit})
		First we note that we may assume $G \leq \mathbf{W}(\mathbb{Z})$: Since $\ker(\iota)\leq G_\mathbb{N}$ is amenable by assumption and the class of amenable groups is stable under extensions and quotients, the group $G$ is amenable \emph{if and only if} $\iota(G)$ is amenable. The first step of the proof is to show that $G_{A \Delta \mathbb{N}}$ is amenable for all $A \in \mathcal{P}_f(\mathbb{Z})$. The stabilizer $G_{k + \mathbb{N}}$ is amenable for all $k \in \mathbb{Z}$ by the following argument: Assume $G_{k + \mathbb{N}}$ is not amenable for some $k \in \mathbb{Z}$. Since the action of $\mathbf{W}(\mathbb{Z})$ on $\mathbb{Z}$ is extensively amenable, the transitive action of $G_{k + \mathbb{N}}$ on the orbit $G_{k + \mathbb{N}}\cdot j$ is extensively amenable for all $j \in \mathbb{Z}$ by Lemma~\ref{lem: extamensubgrp} and thus amenable by Lemma~\ref{lem: extamen}. By Theorem~\ref{thm: amenact} thus $(G_{k + \mathbb{N}})_i$ is non-amenable for some and hence all $i \in \mathbb{Z}$. But since $(G_{k + \mathbb{N}})_k\leq G_{k+1 + \mathbb{N}}$ and $(G_{k + \mathbb{N}})_{k-1}\leq G_{k-1 + \mathbb{N}}$, the containing groups are non-amenable. By iteration of this argument $G_{j + \mathbb{N}}$ is non-amenable for all $j \in \mathbb{Z}$, which is a contradiction to the assumption that $G_\mathbb{N}$ is amenable. Let $A \in \mathcal{P}_f(\mathbb{Z})$. Then there exists a unique $k \in \mathbb{Z}$ such that $A \Delta \mathbb{N}=\sigma(k+\mathbb{N})$ for a finitely supported permutation $\sigma \in \mathfrak{S}(\mathbb{Z})$. Since $G_{A\Delta\mathbb{N}}=\sigma G_{k+\mathbb{N}}\sigma^{-1}$, the stabilizer $G_{A\Delta\mathbb{N}}$ is amenable.	Let $\Phi \colon G \to \mathcal{P}_f(\mathbb{Z})\rtimes \mathbf{W}(\mathbb{Z})$ be the embedding given by $g \mapsto (g(\mathbb{N})\Delta \mathbb{N},g)$.\footnote{This is well-defined since $g \in \mathbf{W}(\mathbb{Z})$ implies $g(\mathbb{N})\Delta \mathbb{N}\in \mathcal{P}_f(\mathbb{Z})$.\\} The action of $G$ on $\mathcal{P}_f(\mathbb{Z})$ induced by $\Phi$ is amenable, since the action of $\mathcal{P}_f(\mathbb{Z})\rtimes \mathbf{W}(\mathbb{Z})$ is amenable by Theorem~\ref{thm: wobl}. Since for all $A \in \mathcal{P}_f(\mathbb{Z})$ the point stabilizer $G_A$ of the action induced by $\Phi$ is precisely the stabilizer $G_{A \Delta \mathbb{N}}$ of the wobbling group action -- which has been shown to be amenable for all $A \in \mathcal{P}_f(\mathbb{Z})$ -- thus $G$ must be amenable by Theorem~\ref{thm: amenact}.
	\end{proof}
	
	\begin{cor}\label{cor: topological full groups of minimal cantor systems are amenable}
		Let $(X,\varphi)$ be a minimal subshift. Then $\mathfrak{T}(\varphi)$ is amenable.
	\end{cor}
	\begin{proof}
		Let $(X,\varphi)$ be a minimal Cantor system. The group $\mathfrak{T}(\varphi)$ injects into $\mathbf{W}(\mathbb{Z})$ by its action on orbits i.e. since cocycles of elements in $\mathfrak{T}(\varphi)$ are bounded, every point $x \in X$ induces a map $\iota_{x} \colon \mathfrak{T}(\varphi) \to \mathbf{W}(\mathbb{Z})$ defined by 
		\begin{equation*}
		\iota_{x}(\gamma)(k)=f_\gamma(\varphi^k(x))+k
		\end{equation*}
		This map is a group homomorphism as for $\gamma_1,\gamma_2 \in \mathfrak{T}(\varphi)$ by Lemma~\ref{lem: cocy} we have:
		\begin{equation*}
		\begin{gathered}
		\iota_x(\gamma_1\gamma_2)(k)=f_{\gamma_1\gamma_2}(\varphi^k(x))+k=f_{\gamma_1}(\gamma_2\varphi^k(x))+f_{\gamma_2}(\varphi^k(x))+k=\\
		=f_{\gamma_1}(\varphi^{f_{\gamma_2}(\varphi^k(x))+k}(x))+\iota_x(\gamma_2)=\iota_x(\gamma_1)\iota_x(\gamma_2)(k)
		\end{gathered}	
		\end{equation*}
		
		By minimality every $\gamma \in \mathfrak{T}(\varphi)$ is uniquely determined by $f_\gamma|_{\operatorname{Orb}_\varphi (x)}$, hence this homomorphism is injective. This implies the group $\iota_x(\mathfrak{T}(\varphi)_x)=\iota_x(\mathfrak{T}(\varphi))_\mathbb{N}$ is amenable by Corollary~\ref{cor: locally finite subgroups}(ii), thus Theorem~\ref{thm: amen-crit} applies.
	\end{proof}
	
	\begin{rem}
		By Corollary~2.4 in \cite{cho80} every finitely generated, simple, elementary amenable group is finite, thus by Theorem~\ref{thm: fingen} and Theorem~\ref{thm: simpel} the topological full group of a minimal subshift is an amenable, non-elementary amenable group.
	\end{rem}

	Putting together Proposition~\ref{prop: uncount},  and Corollary~\ref{cor: topological full groups of minimal cantor systems are amenable} implies:
	
	\begin{cor}
		There exists an uncountable family of pairwise non-isomorphic, infinite, finitely generated, simple, amenable groups.
	\end{cor}
	
	\begin{rem}\phantomsection\label{rem: non-amenable topological full groups}
		\begin{enumerate}[(i)]
			\item Corollary~\ref{cor: topological full groups of minimal cantor systems are amenable} does not generalize to minimal $\mathbb{Z}^d$-actions on a Cantor space. In \cite{em13} an example of a $\mathbb{Z}^2$-action was given, where $\mathfrak{T}(\mathbb{Z}^2)$ contains a copy of the free group on two generators.
			
			\item In \cite{mat15}, Proposition~4.10 it is demonstrated, that the topological full group $\mathfrak{T}(\mathcal{G})$ of a Hausdorff, effective, \'etale Cantor groupoid $\mathcal{G}$ in which $\mathcal{G}^{(0)}$ is properly infinite contains a subgroup isomorphic to the free product $(\mathbb{Z}/2\mathbb{Z})\ast(\mathbb{Z}/3\mathbb{Z})$ and, hence, is not amenable. 
		\end{enumerate}
	\end{rem}
	
	Extensive amenable actions $G \curvearrowright X$ were further studied (still implicitly) for more general $X$ in \cite{dlSJ15}. In \cite{jns16} a powerful amenability criterion was developed encompassing Theorem~\ref{thm: amen-crit} which applies to a diverse class of amenable, non-elementary amenable groups:
	\begin{thm}[\cite{jns16}, Theorem~1.1]
		Let $X$ be a topological space. Let $G \leq \operatorname{Homeo}(X)$ be finitely generated and let $S$ be a generating set. Let $\mathcal{H}$ be a subgroupoid of the groupoid of germs $\operatorname{Germ}(X,\operatorname{Homeo}(X))$ such that the following conditions hold:
		\begin{enumerate}[(i)]
			\item For every $g \in G$ the set of $x \in X$ such that the germ $[(g,x)] \notin \mathcal{H}$ is finite.
			
			\item For every $s \in S$ the set $x \in X$ such that $[(s,x)] \neq \mathcal{H}$ is finite.
			
			\item The action $G \curvearrowright X$ is extensively amenable.
			
			\item For every $x \in \operatorname{Germ}(X,G)^{(0)}$ the isotropy group $\operatorname{Germ}(G,X)|_x$ is amenable.
			
			\item The group $\mathfrak{T}(\mathcal{H})$ is amenable.
		\end{enumerate}
		Then the group $G$ is amenable.
	\end{thm}
	
	Extensive amenability was explicitely introduced and studied in \cite{jmms18}. In particular it was shown that extensive amenability of an action $G \curvearrowright X$ induces amenability of the action $F(X) \rtimes G \curvearrowright F(X)$ i.e. $\mathcal{P}_f(-)$ is replaced by a functor $F$ of certain type.
	
	\subsection{Liouville property and quantifying amenability}\label{subs: liouville property, quantifying amenability and subgroups of intermediate growth}
	
	The findings of \cite{jm13} required further analysis of topological full groups of minimal subshifts. Both papers \cite{mb14a} and \cite{mb14b} by Nicolás Matte Bon are concerned with topological full groups. In \cite{mb14a} Matte Bon obtains an upper bound\footnote{For the definition of $\rho$ see Definition~\ref{defi: complexity}, for the definition of the entropy $H$ see Definition~\ref{defi: entropy}(ii).\\} for random walk entropy of shifts with bounded complexity:
	\begin{thm}[\cite{mb14a}, Theorem~1.2]\label{thm: entropy}
		Let $(X,\varphi)$ be a subshift with no isolated periodic points of which the complexity $\rho$ satisfies
		\begin{equation*}
		\lim_{n\to \infty} \left(\frac{\log n}{n}\right)^2 \rho(n)=0
		\end{equation*}
		Then for every finitely supported symmetric probability measure $\mu$ on $\mathfrak{T}(\varphi)$ there exists a $C>0$ such that for every $n \geq 1$ the following holds:
		\begin{equation*}
		H(\mu^{\ast n})\leq C \rho(\lceil C \sqrt{n \log n} \rceil)\log n
		\end{equation*}
	\end{thm}
	This estimate implies in particular that $\mathfrak{T}(\varphi)$ has zero random walk entropy thus implying by the characterization in Theorem~\ref{thm: charac.liouville}:
	
	\begin{cor}
		Let $(X,\varphi)$ be a subshift as chosen in Theorem~\ref{thm: entropy}. Then $\mathfrak{T}(\varphi)$ has the Liouville property.
	\end{cor}
	
	Subshifts that satisfy the criteria of Theorem~\ref{thm: entropy} do exist -- see Example~\ref{ex: sturm+toeplitz}(i). Since the Lioville property implies amenability, this produces a proof of amenability for such shifts independent of \cite{jm13}. It is of note that for this proof minimality of the subshift is not required.
	Combined with the isomorphism theorems this implies:
	\begin{thm}[\cite{mb14a}, Theorem~1.1]
		There exists an uncountable family of pairwise non-isomorphic, infinite, finitely generated, simple groups with the Liouville property.
	\end{thm}
	
	The upper bound of Theorem~\ref{thm: entropy} produces a lower bound for $\mu^{\ast 2n}(e)$ i.e the return probability of the random walk, and in consequence upper bounds of Følner functions of finitely generated subgroups of $\mathfrak{T}(\varphi)$:
	\begin{cor}[\cite{mb14a}, Corollary~1.7]
		Let $(X,\varphi)$ be a subshift as chosen in Theorem~\ref{thm: entropy}.
		\begin{enumerate}[(i)]
			\item For every finitely supported symmetric probability measure $\mu$ on $\mathfrak{T}(\varphi)$ there exists a $C>0$ such that for every $n \geq 1$ the following holds:
			\begin{equation*}
			\mu^{\ast 2n}(e)\geq \frac{1}{C} \exp(-C \rho(\lceil C \sqrt{n \log n} \rceil)\log n)
			\end{equation*}
			
			\item If in addition there exists a $C'>0$ and a $\alpha \in [0,\frac{1}{2})$ such that $\rho(n)\leq C' n^{\alpha}$,\footnote{This is satisfied for example by Sturmian shifts.\\} then for every finitely generated subgroup $G \leq \mathfrak{T}(\varphi)$, every symmetric finite generating set $S$ of $G$ and for every $\varepsilon > 0$, there exists a $C > 0$ such that for every $n \geq 1$ the following holds:
			\begin{equation*}
			\operatorname{Fol}_{G,S}(n)\leq C \exp(C n^{\frac{2\alpha}{2-\alpha}+\varepsilon})
			\end{equation*}
		\end{enumerate}
	\end{cor}

	\subsection{New groups of Burnside type and intermediate growth}\label{subs: intermediate growth}
	
	In 1980 the general Burnside problem had been resolved long ago as a consequence of the Golod-Shafarevych theorem. The mathematical world was still looking for simple constructions of infinite, periodic, finitely generated group. A striking example was given in \cite{gri80} by Grigorchuk as a group of Lebesgue measure preserving interval transformations now called the (first) Grigorchuk group\footnote{See Appendix~\ref{app: grigorchuk groups} for the definition of Grigorchuk groups.\\}. Today a representation in terms of automorphisms of an infinite, binary, rooted tree is more common. In \cite{gri83} Grigorchuk demonstrated that his group has intermediate growth resolving a longstanding question of John Milnor. As such it provided the first example of a non-elementary amenable group and initiated the study of branch groups and automata groups. Over the years more examples of groups with intermediate growth have been found in particular non-residually finite such groups. However, for a long time there were no examples of a simple, infinite, finitely generated group of intermediate growth.
	
	In \cite{mb14b} Matte Bon shows that Grigorchuk groups embed into topological full groups of shifts, which by \cite{jm13} produces a new proof of amenability of such groups.
	
	\begin{prop}[\cite{mb14b}, Proposition~3.1]
		Let $(X,\varphi)$ be a minimal Cantor system. Then for every $\omega \in \{0,1,2\}^{\mathbb{N}_+}$ which is eventually constant the Grigorchuk group $G_\omega$ embedds into $\mathfrak{T}(\varphi)$.
	\end{prop}
	
	This is an immediate consequence of Proposition~\ref{prop: permutational wreath products embedd into topological full groups} and Proposition~\ref{prop: grigorchuk groups embedding into permutational wreath products}. Of much more interest is the question for the counterparts with intermediate growth:
	
	\begin{thm}[\cite{mb14b}, Proposition~3.7]
		Let $\omega \in \{0,1,2\}^{\mathbb{N}_+}$ such that it is not eventually constant. There exists a minimal subshift $(X_\omega,\varphi_\omega)$ such that the Grigorchuk group $G_\omega$ embedds into $\mathfrak{T}(\varphi_\omega)$.
	\end{thm}
	
	We do not dwell on the nature of the associated shift and the embedding for now, but immediatly turn to more general ideas by Nekrashevych: In \cite{nek18} Nekrashevych showed that topological full groups of fragmentations of dihedral group actions produce examples of infinite, finitely generated, simple, periodic groups of intermediate growth. In this subsection we take a quick look at \cite{nek18}. We dispense with proofs except for a sketch of the embedding of fragmentations of dihedral groups into topological full groups of Cantor systems. For definitions of graphs associated to actions see Appendix~\ref{app: graphs of actions}.
	
	The minimality of an action of a finitely generated group on a compact metrizable space enables to study the action via the study of orbital Schreier graphs. In particular it entails that orbital Schreier graphs are \emph{repetitive} i.e. for every finite subgraph $\Sigma$ a copy can be found in bounded distance -- depending on the size of $\Sigma$  -- of every vertex:
	\begin{prop}[\cite{nek18}, Proposition~2.5]\label{prop: isomorphism of balls in orbital schreier graphs}
		Let a group $G$ generated by a finite symmetric set $S$ act minimally on a compact metrizable space $X$ by homeomorphisms. Then for every $r \in \mathbb{N}$, there exists an $R(r) \in \mathbb{N}$ such that for every $G$-regular point $x \in X$ and every $y \in X$, there exists a vertex $z \in \Gamma(y,G,S)$ such that $d(z,y) \leq R(r)$ and such that the balls $B_{x}(r) \subset \Gamma(x,G,S)$ and $B_z(r) \subset \Gamma(y,G,S)$ are isomorphic as rooted graphs.
	\end{prop}
	\begin{proof}
		Let $r \in \mathbb{N}$ and let $x$ be a $G$-regular point. Then the ball $B_x(r) \subset \Gamma(x,G,S)$ is determined by the values of the function $\delta_x \colon B_{1}(r) \times B_{1}(r) \to \{0,1\}$ given by
		\begin{equation*}
			\delta_x(g,h):=
			\begin{cases}
			1, & \text{for } g(x)=h(x) \text{ in }B_x(r) \\
			0, & \text{else}
			\end{cases}
		\end{equation*}
		
		Since $x$ was assumed to be $G$-regular there exists a open neighbourhood $U_x$ of $x$ such that for all $y \in U_x$ we have $\delta_x=\delta_y$ and thus the balls $B_x(r) \subset \Gamma(x,G,S)$ and $B_y(r) \subset \Gamma(y,G,S)$ are isomorphic as rooted labelled graphs. By minimality the $G$-translates of $U_x$ form an open cover of $X$. By compactness there exists a finite collection $g_1,\dots,g_n$ such that the family $\{g_i(U_x)\}_{i \in \{1,\dots,n\}}$ covers $X$. Let $R_x(r):=\max_{i \in \{1,\dots,n\}} \mathcal{L}_{G,S}(g_i)$. Then for every $y \in X$ there exists an $i \in \{1,\dots,n\}$ such that $g_i^{-1}(y) \in U_x$. Then the balls $B_x(r)\subset \Gamma(x,G,S)$ and $B_{g_i^{-1}(x)}(r)\subset \Gamma(x,G,S)$ are isomorphic as rooted labelled graphs and $d(x,g_i^{-1}(x)) \leq R_x(r)$.
		To get rid of the dependency on the choice of the $G$-regular point $x$, note that the set of isomorphism classes of balls in orbital graphs of the form $B_x(r)$ where $x$ is some arbitrary $G$-regular point $x$ is finite. Then any complete set of representatives $\{B_{x_i}(r)\}_{i \in I}$ provides an estimate independent of the choice of a $G$-regular point by $R(r)= \max_{i \in I} R_{x_i}(r)$.
	\end{proof}
	
	\begin{defi}[\cite{nek18}, Definition~2.6]
		Let a group $G$ generated by a finite symmetric set $S$ act minimally on a compact metrizable space $X$ by homeomorphisms. Then the action is called \emph{linearly repetitive} if there exists a $K \in \mathbb{R}_{>1}$ such that $R(r)<Kr$.
	\end{defi}
	
	Nekrashevych introduces a new class of groups by fragmentations of actions of dihedral groups on Cantor spaces. This class in particular includes the Grigorchuk groups:
	
	\begin{defi}[[\cite{nek18}, §3]
		Let $h$ be a homeomorphism on a Cantor space $X$ with period $2$. 
		\begin{enumerate}[(i)]
			\item A finite group $G$ of homeomorhisms of $X$ is called a \emph{fragmentation of $h$} if for every $g \in G$ and $x \in X$ either $h(x)=x$ or $h(x)=g(x)$ hold, and if for every $x \in X$ there exists a $g \in G$ such that $h(x)= g(x)$.
		\end{enumerate}
		Let $G$ be a fragmentation of $h$.
		\begin{enumerate}[(i),resume]
			\item Let $g \in G$. Denote by $E_{g,1}$ the set of fixed points of $g$ and by $E_{g,h}$ the set of elements $x \in X$ such that $g(x)=h(x)$.
			
			\item Let $h$ be such that its set of fixed points has empty interior.\footnote{This assures that for all $g \in G$ we have $(E_{g,1})^\circ \cap (E_{g,h})^\circ = \emptyset$.\\} Denote by $\mathcal{P}_G$ the family of subsets of $X$ of the form $\bigcap_{g \in G} (E_{g,i_g})^\circ$ where $i_g \in \{1,h\}$ for all $g \in G$. The elements of $\mathcal{P}_G$ are called the \emph{pieces of the fragmentation $G$}.
			
			\item The \emph{infinite dihedral group $D_{\infty}$} is given by the presentation $D_{\infty}:=\langle a,b|a^2,b^2 \rangle$.
			
			\item Let the infinite dihedral group $D_{\infty}:=\langle a,b|a^2,b^2 \rangle$ act on a Cantor set by homeomorphisms. A \emph{fragmentation of the dihedral group $\langle a,b \rangle$} is a group $G$ of homeomorphisms generated by $A \cup B$ where $A$ resp. $B$ are fragmentations of $a$ resp. $b$. 
		\end{enumerate}
	\end{defi}
	
	\begin{ex}[\cite{nek18}, Example~3.3]
		The first Grigorchuk group is a fragmentation of of a dihedral group acting on a Cantor space.
	\end{ex}
	
	By an elaborate analysis on the graphs associated to the action Nekrashevych obtains:
	
	\begin{thm}[\cite{nek18}, Theorem~4.1]
		Let $G$ be the fragmentation of a minimal action of the dihedral group $D_\infty$ on a Cantor space $X$ such that there exists a purely non-Hausdorff singularity. Then the group $G$ is periodic.
	\end{thm}
	
	\begin{rem}
		By Lemma~3.2 in \cite{nek18} for every fixed point $x$ of a non-free, minimal action of a dihedral group on a Cantor space, there exists a fragmentation $G$ such that $x$ is a purely non-Hausdorff singularity.
	\end{rem}
	
	For the refined case of linearly repetive actions Nekrashevych obtains an upper bound for the growth rate. Since any infinite, finitely generated, periodic group is not virtually nilpotent,\footnote{The nilpotent finite-index subgroup would be periodic and finitely generated. Being nilpotent and periodic is equivalent to being locally finite and nilpotent. Locally finite, finitely generated groups are necessarily finite, which is a contradiction.\\} which by Gromov's theorem (Theorem~\ref{thm: gromovs theorem}) implies that it has superpolynomial growth: 
	\begin{thm}[\cite{nek18}, Theorem~6.6]\label{thm: fragmentations with linearly repetitive schreier graphs have intermediate growth}
		Let $G$ be the fragmentation of a minimal action of the dihedral group $D_\infty$ on a Cantor space $X$ such that the orbital Schreier graphs of $G$ are linearly repetitive and there exists a purely non-Hausdorff singularity. Then $G$ has intermediate growth.
	\end{thm}
	
		The following was already observed in the case of Grigorchuk groups by Matte Bon in \cite{mb14b}:
		
		\begin{prop}[\cite{nek18}, Proposition~5.1]
			Every fragmentation $G$ of a minimal action of the dihedral group $D_\infty$ on a Cantor space $X$ embeds into the topological full group of a minimal subshift.
		\end{prop}
		\begin{proofsketch}
			Let $x \in X$ be a $G$-regular point. Then the orbital Schreier graph $\Gamma(x,\langle a,b \rangle)$ is a two-ended infinite path. If two vertices $v_1,v_2$ of the orbital Schreier graph $\Gamma(x,\langle a,b \rangle)$ are connected by an edge labelled by ``$a$" (resp. ``$b$"), then for every $g \in A$ and $i_g \in \{1,a\}$ (resp. $g \in B$ and $i_g \in \{1,b\}$) we have $v_1 \in (E_{g,i_g})^\circ$ \emph{if and only if} $v_2 \in (E_{g,i_g})^\circ$, hence they lie in the same piece $P \in \mathcal{P}_A$ (resp. $P \in \mathcal{P}_B$).\footnote{They are connected in $\Gamma(x,G)$ by an edge for every element $g \in A$ (resp. $g \in B$) such that $g|_P=a$ (resp. $g|_P=b$). Hence orbital Schreier graphs of fragmentations of dihedral groups are one- or two-ended infinite paths ``fattened up" by loops and multiple edges.\\} By associating adjacent edges in the orbital Schreier graph $\Gamma(x,\langle a,b \rangle)$ with successive integers this gives rise to a sequence $\omega_x \in (\mathcal{P}_A \cup \mathcal{P}_B)^{\mathbb{Z}}$.
			Define $\mathcal{W} \subseteq (\mathcal{P}_A \cup \mathcal{P}_B)^{\mathbb{Z}}$ to be the subset of seqences $\omega \in (\mathcal{P}_A \cup \mathcal{P}_B)^{\mathbb{Z}}$ such that every finite substring of $\omega$ is a substring of $\omega_x$. Then $\mathcal{W}$ is a closed shift-invariant subset of $(\mathcal{P}_A \cup \mathcal{P}_B)^{\mathbb{Z}}$ and Proposition~\ref{prop: isomorphism of balls in orbital schreier graphs} implies that $(\mathcal{W},\sigma)$ is minimal. Every $g \in A$ (resp. $g \in B$) induces an element $\iota(g) \in \mathfrak{T}(\sigma)$ by defining for $\omega \in \mathcal{W}$:
			\begin{equation*}
			\iota(g)(\omega):=
			\begin{cases}
			\sigma(\omega), & \text{ if }g|_{\omega_0}=a|_{\omega_0}\qquad \text{(resp. }g|_{\omega_0}=b|_{\omega_0}\text{)}\\
			\sigma^{-1}(\omega), & \text{ if }g|_{\omega_{-1}}=a|_{\omega_{-1}}\quad\text{(resp. }g|_{\omega_{-1}}= b|_{\omega_{-1}} \text{)}\\
			\omega, & \text{else}
			\end{cases}
			\end{equation*}
		\end{proofsketch}
		
		Thus Corollary~\ref{cor: topological full groups of minimal cantor systems are amenable} implies:
		\begin{cor}
			Every fragmentation $G$ of a minimal action of the dihedral group $D_\infty$ on a Cantor space $X$ is amenable.
		\end{cor}
		
		In the case of expansive actions of dihedral groups the following propositions hold:
		\begin{prop}[\cite{nek18}, Proposition~5.2]\label{prop: actions of fragmentations of expansive dihedral actions are expansive}
			Let $G$ be the fragmentation of an action of the dihedral group $D_\infty$ on a Cantor space. If the action of $D_\infty$ is expansive, then the action of $G$ is expansive.
		\end{prop}
		
		The following is a consequence of Theorem~\ref{thm: alternating full groups of minimal expansive groupoids are finitely generated}:
		\begin{prop}\label{prop: topfullgrps of fragmentations are contained in fragmentations}
			Let $G=\langle A \cup B \rangle$ be the fragmentation of an expansive, minimal action of the dihedral group $D_\infty$ on a Cantor space and let $\mathcal{G}$ be its associated transformation groupoid. Then there exists a fragmentation $\tilde{G}$ of this action such that $\mathfrak{A}(\mathcal{G}) \leq \tilde{G}$ and the groups of germs $G_x/G_{(x)} $ and $\tilde{G}_x/\tilde{G}_{(x)}$ are isomorphic for all $x \in X$.
		\end{prop}
		
		Combining the results on alternating full groups\footnote{See Subsection~\ref{subs: simplicity of the alternating full groups} and Subsection~\ref{subs: finite generation of alternating full groups}.\\} with Theorem~\ref{thm: fragmentations with linearly repetitive schreier graphs have intermediate growth}, Proposition~\ref{prop: actions of fragmentations of expansive dihedral actions are expansive} and Proposition~\ref{prop: topfullgrps of fragmentations are contained in fragmentations} implies:
	
		\begin{thm}[\cite{nek18}, Theorem~1.2]
			Let $G$ be the fragmentation of an expansive, minimal action of the dihedral group $D_\infty$ on a Cantor space such that the orbital Schreier graphs of $G$ are linearly repetitive and there exists a purely non-Hausdorff singularity and let $\mathcal{G}$ be the transformation groupoid associated with the action of $G$. Then the group $\mathfrak{A}(\mathcal{G})$ is a infinite, finitely generated, simple, periodic group of intermediate growth.
		\end{thm}
		
	\section{Irreducibility of Koopman representations}\label{sec: the koopman representation of topological full groups}

	Let $(X,\mu)$ be a measure space such that $\mu$ is a finite measure. Let $G$ be a locally compact, second countable group acting ergodically and measure class preserving on $(X,\mu)$. The following question was put forward i.a. by Vershik:
	\begin{quest}[\cite{ver11}, Problem~4]
		In which of the above settings is the associated Koopman representation\footnote{See Definition~\ref{defi: koopman representation}.\\} irreducible? \footnote{In the measure preserving case the space of constant functions $\mathbb{C}\cdot \mathbf{1}_X$ provides a closed invariant subspace thus one restricts to $(\mathbb{C}\cdot \mathbf{1}_X)^{\perp}$.\\}
	\end{quest}
	
	The list of actions for which this question has been answered affirmatively is quite short.\footnote{For a list see \cite{dud18}.\\}
	
	\subsection{Measure contracting actions}
	
	Artem Dudko provides in \cite{dud18} a unified strategy of proving irreducibility for
	 the natural actions of Higman-Thompson groups as groups of piecewise linear homeomorphisms and for the natural actions of weakly branch groups on tree boundaries.
	
	\begin{defi}[\cite{dud18},§2]\phantomsection\label{defi: mca}
		\begin{enumerate}[(i)]
			\item Let $(X,\mu)$ be a probability space and let $G$ be a locally compact, second countable group acting measure class preserving on $(X,\mu)$. This action is called \emph{measure contracting} if for every measurable subset $A$ and $M,\epsilon>0$ there exists a $g \in G$ such that:\footnote{In this subsection let  $\supp(g)$ denote the set $\{x \in X|gx \neq x  \}$.\\}
			\begin{enumerate}
				\item $\mu(\supp(g)\setminus A) < \epsilon$
				\item $\mu(\{x \in A| \sqrt{\frac{\mathrm{d}\mu(g x)}{\mathrm{d}\mu(x)}} < M^{-1} \})>\mu(A)-\epsilon$
			\end{enumerate}
			
			\item Let $(X,\mu)$ be a measure space. Every function $f \in \mathrm{L}^\infty(X,\mu)$ constitutes a multiplication operator $T_f \in \mathcal{B}(\mathrm{L}^2(X,\mu))$ by $T_f (g)(x)=f(x)g(x)$ for all $g \in \mathrm{L}^2(X,\mu)$. Denote by $\mathcal{L}^\infty$ the von Neumann algebra generated by $\{T_f|f \in \mathrm{L}^\infty(X,\mu) \}$ in $\mathcal{B}(\mathrm{L}^2(X,\mu))$.
		\end{enumerate}
	\end{defi}
	
	\begin{rem}\label{rem: l-infty is a masa}
		Let $(X,\mu)$ be a finite measure space. Then the von Neumann algebra $\mathcal{L}^\infty$ defined in Definition~\ref{defi: mca}(ii) is a masa in $\mathcal{B}(\mathrm{L}^2(X,\mu))$ -- see e.g. Proposition~III.1.5.16 of \cite{bla06}.
	\end{rem}
	
	\begin{lem}[\cite{dud18}, Theorem~3]\label{lem: irr}
		Let $G$ be a locally compact, second countable group acting ergodically and measure class preserving on a standard Borel space $(X,\mu)$. Let $\mathfrak{A}$ be the von-Neumann-algebra generated by $\mathcal{M}_\kappa \cup \mathcal{L}^\infty$ in $\mathcal{B}(\mathrm{L}^2(X,\mu))$. Then $\mathfrak{A}=\mathcal{B}(\mathrm{L}^2(X,\mu))$.
	\end{lem}
	\begin{proof}
		If we show $\mathfrak{A}'\subseteq \mathbb{C}\cdot \operatorname{Id}$, we are done, because then $\mathfrak{A}'' = \mathcal{B}(\mathrm{L}^2(X,\mu))$ and by the von Neumann bicommutant theorem\footnote{See Remark~\ref{rem: vNma}(i).\\} $\mathfrak{A}=\mathcal{B}(\mathrm{L}^2(X,\mu))$. Accordingly, let $A \in \mathfrak{A}'$. By Remark~\ref{rem: vNma}(ii), $\mathfrak{A}' \subseteq (\mathcal{L}^\infty)'$ and $\mathfrak{A}' \subseteq \mathcal{M}_\kappa'$ hold. By Remark~\ref{rem: l-infty is a masa} the abelian von-Neumann-algebra $\mathcal{L}^\infty$ is maximal, which is equivalent to $(\mathcal{L}^\infty)'=\mathcal{L}^\infty$. Thus $A$ is the multiplication operator $T_f$ associated with a function $f \in \mathrm{L}^\infty(X,\mu)$. Since $T_f=A \in \mathcal{M}_\kappa'$, we have
		\begin{equation*}
		\begin{split}
		\sqrt{\frac{\mathrm{d}g_{*}\mu}{\mathrm{d}\mu}(x)}f(g^{-1}x)h(g^{-1}x)= \kappa(g)T_f (h)(x)=\\
		=T_f	(\kappa(g)h)(x)= f(x) \sqrt{\frac{\mathrm{d}g_{*}\mu} {\mathrm{d}\mu}(x)} h(g^{-1}x)
		\end{split}
		\end{equation*}
		for all $g \in G$ and $h \in \mathrm{L}^2(X,\mu)$, thus $f$ is $G$-invariant almost everywhere. By ergodicity this implies $f$ must be constant almost everywhere and thus the associated operator satisfies $T_f \in \mathbb{C}\cdot \operatorname{Id}$ -- see \cite{gla03}, Theorem~3.10.
	\end{proof}
	
	\begin{thm}[\cite{dud18}, Theorem~4]\label{thm: measurecontractioninducesapproximation}
		Let $(X,\mu)$ be a probability space and let $G$ be a locally compact, second countable group acting measure class preserving, ergodically and measure contracting on $(X,\mu)$. Then $\mathcal{L}^\infty = \mathcal{M}_\kappa$ holds.
	\end{thm}
	\begin{proof}
		 Let $A$ be a measurable subset of $X$. Then $\mathcal{H}_A:=\{f \in \mathrm{L}^2(X,\mu)|\supp(f)\subset X \setminus A\}$ is a subspace of $\mathrm{L}^2(X,\mu)$. Denote by $P_A$ the orthogonal projection unto this subspace. Since the action of $G$ is measure contracting, we can find a sequence of elements $(g_m^A)_{m \in \mathbb{N}}$ satisfying conditions \ref{defi: mca}(i) and \ref{defi: mca}(ii) for $M=m$ and $\epsilon=m^{-1}$. Set
		 \begin{equation}
		 A_m=\{x \in A| (\frac{\mathrm{d}\mu(g_m^A x)} {\mathrm{d}\mu(x)})^{1/2}<m^{-1} \},B_m:=\supp(g_m^a)\setminus A_m
		 \end{equation}
		 By definition, $A_m \subseteq \supp(g_m^A)$ holds and we have:
		 \begin{equation}
		 \begin{gathered}
		 \mu(A\setminus A_m)=\mu(A)-\mu(A_m)<\mu(A)-(\mu(A)-m^{-1})=m^{-1}\\
		 \mu(B_m)=\mu(\supp(g_m^a)\setminus A_m) \leq \mu(\supp(g_m^a)\setminus A) + \mu(A\setminus A_m) < 2m^{-1}
		 \end{gathered}
		 \end{equation}
		 In the following we varify, that for all $f_1,f_2 \in \mathrm{L}^2(X,\mu)$ we have $\langle\kappa(g_m^A)f_1,f_2\rangle \longrightarrow \langle P_A f_1,f_2\rangle$ for $m \to \infty$ i.e. $P_A$ is in the weak operator closure of $\kappa(G)$ and thus in $\mathcal{M}_\kappa$. Since $\mathrm{L}^2(X,\mu) \cap \mathrm{L}^\infty(X,\mu)$ is dense in $\mathrm{L}^2(X,\mu)$, we may assume $f_1,f_2 \in \mathrm{L}^\infty(X,\mu)$. Let therefore $\|f_i\|_\infty=M_i$ and let us fix the following abbrevations:
		 \begin{equation*}
		 \bm{\mathrm{I}_m^A}(x):=\sqrt{ \frac {\mathrm{d}\mu(g_m^A x)}{\mathrm{d}\mu(x)}}f_1(g_m^Ax)\overline{f_2(x)},\quad
		 \bm{\mathrm{II}}(x):=f_1(x)\overline{f_2(x)}
		 \end{equation*}
		 Then we have:
		 \begin{equation*}
		 |\langle\kappa((g_m^A)^{-1})f_1,f_2\rangle - \langle P_A f_1,f_2\rangle| = |\int_{X}\bm{\mathrm{I}_m^A}(x)\;\mathrm{d}x-\int_{X \setminus A}\bm{\mathrm{II}}(x)\;\mathrm{d}x|
		 \end{equation*}
		 Note that $\bm{\mathrm{I}_m^A}$ and $\bm{\mathrm{II}}$ coincide on $X \setminus (\supp (g_m^A) \cup A)$, thus this expression is equal to:
		 \begin{equation*}
		 \begin{gathered}
		 |\int_{A_m}\bm{\mathrm{I}_m^A}(x)\;\mathrm{d}x+
		 \int_{B_m}\bm{\mathrm{I}_m^A}(x)\;\mathrm{d}x+
		 \int_{A\setminus \supp(g_m^A)}\bm{\mathrm{II}}(x)\;\mathrm{d}x-
		 \int_{\supp(g_m^A)\setminus A}\bm{\mathrm{II}}(x)\;\mathrm{d}x|\\
		 \overset{(3.1)}{<}m^{-1}\int_{A_m}|f_1(g_m^Ax)\overline{f_2(x)}|\;\mathrm{d}x+
		 |\int_{B_m}\bm{\mathrm{I}_m^A}(x)\;\mathrm{d}x|+
		 \int_{A \triangle \supp(g_m^A)}|\bm{\mathrm{II}}(x)|\;\mathrm{d}x
		 \leq\\
		 \leq m^{-1} M_1M_2+
		 |\int_{B_m}\bm{\mathrm{I}_m^A}(x)\;\mathrm{d}x|+
		 \int_{A \triangle \supp(g_m^A)}M_1M_2\;\mathrm{d}x
		 \end{gathered}
		 \end{equation*}
		 The first summand is obtained, since $\|f_1\|_\infty=\|f_1 \circ g_m^A\|_\infty$ by measure class preservance and it obviously goes to $0$ for $m \to \infty$. Since
		 \begin{equation*}
		 \mu(A \triangle \supp(g_m^A)) \leq \mu(A \setminus A_m)+\mu(B_m)\overset{(3.2)}{<}3m^{-1}
		 \end{equation*}
		 holds, the third summand can be bounded by $3m^{-1}M_1M_2$. By Cauchy-Schwartz the second summand can be bounded by:
		 \begin{equation*}
		 \sqrt{\|\kappa((g_m^A)^{-1})f_1\|_2\|P_{X\setminus B_m}f_2\|_2}
		 \end{equation*}
		 By unitarity of $\kappa$ this is just
		 \begin{equation*}
		 \sqrt{\|f_1\|_2\|P_{X\setminus B_m}f_2\|_2}
		 \end{equation*}
		 which tends to $0$ for $m \to \infty$ because:
		 \begin{equation*}
		 \|P_{X\setminus B_m}f_2\|_2^2 \leq \int_{B_m}|f_2(x)|^2\;\mathrm{d}x < 2m^{-1}M_2^2
		 \end{equation*}
		 We have established $P_A \in \mathcal{M}_\kappa$ for every measurable subset $A$. Since by assumption $(X,\mu)$ is a probability space, $\mu$ is in particular finite and thus $\langle \{\mathbf{1}_A|A \subseteq X \text{ open}\} \rangle$ is dense in $\mathrm{L}^\infty(X,\mu)$ with respect to the $\mathrm{L}^2$-Norm. This means that $\langle \{P_A|A \subseteq X \text{ open}\}\rangle$ is dense in $\mathcal{L}^\infty$ with respect to the strong operator topology. This implies $\mathcal{L}^\infty = \mathcal{M}_\kappa$.
	\end{proof}
	
	\begin{thm}[\cite{dud18}, Theorem~4]\label{thm: measurecontractioninducesirreducibility}
		Let $(X,\mu)$ be a probability space and let $G$ be a locally compact, second countable group acting measure class preserving, ergodically and measure contracting on $(X,\mu)$. Then the associated Koopman representation $\kappa$ is irreducible.
	\end{thm}
	\begin{proof}
		Theorem~\ref{thm: measurecontractioninducesapproximation} and Lemma~\ref{lem: irr} imply $\mathcal{M}_\kappa=\mathcal{B}(\mathrm{L}^2(X,\mu))$. But then $\kappa(G)'= \kappa(G)'''=\mathcal{M}_\kappa'=\mathcal{B}(\mathrm{L}^2(X,\mu))' = \mathbb{C}\cdot \operatorname{Id}$. It follows that $\kappa$ is irreducible, since any projection $P_H$ on a $\kappa(G)$-invariant subspace $H$ in $\mathcal{B}(\mathrm{L}^2(X,\mu))$ is contained in $\kappa(G)'$ and thus $P$ must be either $0$ or $\operatorname{Id}$.
	\end{proof}
	
	\subsection{Topological full groups with irreducible Koopman representation}\label{subs: topfullgroups koopman}
	
	We finish the main body of this text with a quick note on the irreducibility of Koopman representations associated to the natural action of the topological full group $\mathfrak{T}(\mathcal{G})$ of a purely infinite, minimal, effective, \'etale Cantor groupoid $\mathcal{G}$ on the space $\mathcal{G}^{(0)}$.
	Dudko applied the notion of measure contraction to show that natural actions of Higman-Thompson groups $F_{n,r}$ and $G_{n,r}$ on $([0,r],\lambda)$ admit irreducible Koopman representations. The arguments for ergodicity and measure contraction relies on the particular setting e.g. on the fact that the action is by piecewise linear homeomorphisms. Contrasting to this ``rigid" setting, in the groupoid setting the measure contraction comes almost for free from the structure of the groupoid and the abundance of sufficient elements in $\mathfrak{T}(\mathcal{G})$. By Proposition~\ref{prop: existence of quasi-invariant measures} and Proposition~\ref{prop: existence of ergodic measures} locally compact, \'etale groupoids always admit quasi-invariant, ergodic measures and it is easy to see that such measures are quasi-invariant, ergodic measures with respect to the natural action by the topological full group:
	
	\begin{cor}
		Let $\mathcal{G}$ be an effective, \'etale Cantor groupoid and let $\mu$ be a quasi-invariant measure on $\mathcal{G}^{(0)}$. Then $\mu$ is quasi-invariant with respect to the natural action of $\mathfrak{T}(\mathcal{G})$ on $\mathcal{G}^{(0)}$.
	\end{cor}
	\begin{proof}
		This follows immediately from Proposition~\ref{prop: quisiinvariance and slices}.
	\end{proof}
	
	\begin{cor}
		Let $\mathcal{G}$ be an \'etale Cantor groupoid and let $\mu$ be a quasi-invariant, ergodic probability measure on $\mathcal{G}^{(0)}$. Then the natural action of $\mathfrak{T}(\mathcal{G})$ on $\mathcal{G}^{(0)}$ is ergodic with respect to $\mu$.
	\end{cor}
	\begin{proof}
		Let $U \subseteq \mathcal{G}^{(0)}$ such that $r(BU)=U$ for all $B \in \mathfrak{T}(\mathcal{G})\subset \mathcal{B}_\mathcal{G}^{o,k}$. Then $U$ is necessarily almost invariant and hence conull or null.
	\end{proof}
	
	\begin{lem}\label{lem: existence of small opens}
		Let $\mathcal{G}$ be an \'etale Cantor groupoid and let $\mu$ be a quasi-invariant probability measure on $\mathcal{G}^{(0)}$. If $\mathcal{G}$ is minimal, the measure $\mu$ is strictly positive\footnote{Every non-empty open subset of $\mathcal{G}^{(0)}$ has positive measure.\\} and hence for every clopen subset $U$ of $\mathcal{G}^{(0)}$ and every $\varepsilon > 0$ there exists a clopen subset $V \subset U$ with $0 < \mu(V) < \varepsilon$.
	\end{lem}
	\begin{proof}
		Assume there exists an open subset $O \subset \mathcal{G}^{(0)}$ with $\mu(O)=0$. Let $U \subset O$ be a clopen subset. By minimality and Lemma~\ref{lem: existence of pencils} there exists a finite family of compact, open slices $\{B_i\}_{i \in I}$ such that $r(B_i) \subset U$, $\mathcal{G}^{(0)}\setminus U =\bigcup_i s(B_i)$. Then $\bigcup_i s(B_i) \cup U$ is a finite, open cover of $\mathcal{G}^{(0)}$. Then by quasi-invariance of $\mu$ and Proposition~\ref{prop: quisiinvariance and slices} this is a cover of $\mathcal{G}^{(0)}$ by nullsets hence by subadditivity $\mu(\mathcal{G}^{(0)}) \leq \sum_{i \in I} \mu(S_i) + \mu(U)=0$, which is a contradiction. The second statement follows then immediately from additivity.
	\end{proof}
	
	\begin{thm}\label{thm: topfullgroupsactmeasurecontracting}
		Let $\mathcal{G}$ be a purely infinite, minimal, \'etale Cantor groupoid and let $\mu$ be a quasi-invariant, ergodic probability measure on $\mathcal{G}^{(0)}$. Then the natural action of $\mathfrak{T}(\mathcal{G})$ on $\mathcal{G}^{(0)}$ is measure contracting. 
	\end{thm}
	\begin{proof}
		Let $U \subseteq \mathcal{G}^{(0)}$ be clopen and let $M,\epsilon>0$. By Lemma~~\ref{lem: existence of small opens} for every $\varepsilon'>0 $ there exists a clopen subset $B \subset U$ such that $0<\mu(B) <\varepsilon' $. Fix the notation $A:= U \setminus B$. Since $\mathcal{G}$ is purely infinite and minimal there exists by Proposition~\ref{prop: purely infinite construction} an open compact slice $S$ with $s(S)=A$ and $r(S) \subseteq B$. Denote by $f_S(x):=\frac{\mathrm{d}\mu(T_Sx)}{\mathrm{d}\mu(x)}$ Then we have
		\begin{equation*}
		\begin{gathered}
		\mu(\{x \in U| \sqrt{f_S(x)} \geq M^{-1} \}) \leq \mu(\{x \in A| \sqrt{f_S(x)} \geq M^{-1} \})+\varepsilon'=\\
		=\mu(\{x \in A| 1 \leq M^2f_S(x) \})+\varepsilon'= \int \mathbf{1}_{\{x \in A| 1 \leq M^2f_S(x) \}}\; \mathrm{d}\mu+\varepsilon'\leq\\
		M^2 \int_A  f_S(x) \; \mathrm{d}\mu+\varepsilon'=M^2 \mu(r(SA))+\varepsilon'=\\
		 M^2 \mu(B)+\varepsilon' \leq M^2 \varepsilon'+\varepsilon'
		\end{gathered}
		\end{equation*}
	
		When chosing $\varepsilon' \leq \frac{\varepsilon}{1+M^2}$ this quantity is smaller or equal to $\varepsilon$, moreover $\supp(T_S) \subseteq U$ holds. Since measurable subsets of $\mathcal{G}^{(0)}$ can be approximated by clopen subsets arbitrarily well, the action is measure contracting.
	\end{proof}
	
	Combining Theorem~\ref{thm: topfullgroupsactmeasurecontracting} with Theorem~\ref{thm: measurecontractioninducesirreducibility} thus implies:
	\begin{cor}
		Let $\mathcal{G}$ be a purely infinite, minimal, \'etale Cantor groupoid and let $\mu$ be a quasi-invariant, ergodic probability measure on $\mathcal{G}^{(0)}$. Then Koopman representation associated to the action of $\mathfrak{T}(\mathcal{G})$ on $\mathcal{G}^{(0)}$ is irreducible.
	\end{cor}
	
	\addtocontents{toc}{\protect\newpage}
	\begin{appendices}\label{appendices}

		\chapter{Some terms of geometric group theory}\label{app: some terms of geometric group theory}
		
		\section{Growth of groups}\label{app: growth of groups}
		
		The contents of this section can be found in § 6.1 \& § 6.4 of \cite{cc10}. The growth of a finitely generated group is given by the growth of the volume of balls around the identity:
		
		\begin{defi}
			Let $G$ be a group generated by a symmetric, finite subset $S=S^{-1}$.
			\begin{enumerate}[(i)]
				\item The \emph{word length $\mathcal{L}_{G,S}(g)$ of $g \in G$ with respect to $S$} is the smallest integer $n \in \mathbb{Z}$ such that there exist generators $s_1,\dots,s_n \in S$ such that $g=s_1 \dots s_n$. By convention $\mathcal{L}_{G,S}(1)=0$.
				
				\item The \emph{word metric $d_{G,S}$ on $G$ with respect to $S$} is given by $d_S(g,h):=\mathcal{L}_S(g^{-1}h)$ for $g,h \in G$.
				
				\item The \emph{growth function $\gamma_{G,S} \colon \mathbb{N} \to [0,+\infty)$ of $G$ with respect to $S$} is given by
				\begin{equation*}
				\gamma_{G,S}(n):=|\{g \in G|\mathcal{L}_{G,S}(g)\leq n \}|
				\end{equation*}
				
				\item Let $f,g \colon \mathbb{N}\to [0,+\infty)$ be non-decreasing functions. We write $f \preceq g$ if there exist constants $C,\alpha > 0$ such that $f(n)\leq C g(\alpha n)$ for all $n \in \mathbb{N}$. We write $f \sim g$ if $f \preceq g$ and $g \preceq f$ holds.\footnote{It is easy to see that this is an equivalence relation.\\}
			\end{enumerate}
		\end{defi}
		
		\begin{rem}
			\begin{enumerate}[(i)]
				\item For finite, symmetric generating sets $S_1,S_2$ of a group $G$ we have $\gamma_{G,S_1}\sim \gamma_{G,S_2}$, henceforth we omit the generating set from the notation.
				
				\item Every finite group $G$ satisfies $\gamma_G \sim 1$.
				
				\item Every infinite, finitely generated group $G$ satisfies $n \preceq \gamma_G \preceq e^n$.
			\end{enumerate}
		\end{rem}
		
		\begin{defi}
			Let $G$ be a finitely generated group. The $\sim$-equivalence class of $\gamma_G$ is called the \emph{growth type of $G$}. The group $G$ is said to have
			\begin{enumerate}[(i)]
				\item \emph{polynomial growth} if $\beta_G(n) \sim n^\alpha$ for some $\alpha > 0$.
				
				\item \emph{superpolynomial growth} if $\frac{\ln \beta_G(n)}{\ln n} \underset{n \to \infty}{\longrightarrow} \infty$.
				
				\item \emph{exponential growth} if $\beta_G(n) \sim e^n$.
				
				\item \emph{subexponential growth} if $\frac{\ln \beta_G(n)}{n} \underset{n \to \infty}{\longrightarrow} 0$.
				
				\item \emph{intermediate growth} if it has superpolynomial and subexponential growth.
			\end{enumerate}
		\end{defi}
		
		\begin{rem}
			One has the implications $(i) \Rightarrow (iv)$ and $(iii) \Rightarrow (ii)$, moreover $(i)$ and $(ii)$ as well as $(iii)$ and $(iv)$ are mutually exclusive. Every finitely generated group either has polynomial, exponential or intermediate growth.
		\end{rem}
		
		\begin{ex}
			Every finitely generated, free abelian group $\mathbb{Z}^d$ has polynomial growth, every free group $\mathbb{F}_d$ with $d\geq 2$ has exponential growth.
		\end{ex}
		
		A milestone of geometric group theory is provided by Gromov's theorem which characterizes finitely generated groups of polynomial growth:
		\begin{thm}[\cite{dlh00}, Theorem~VII.29]\label{thm: gromovs theorem}
			A finitely generated group has polynomial growth \emph{if and only if} it is virtually nilpotent i.e. it contains a nilpotent subgroup of finite index.
		\end{thm}
		
		\section{Grigorchuk groups}\label{app: grigorchuk groups}
		
		For along time the existence of finitely generated groups with intermediate growth was conjectural. A first example was found by Rostislav I. Grigorchuk. He introduced a class of groups -- later termed Grigorchuk groups -- indexed by infinite $\{0,1,2\}$-sequences:
		\begin{defi}[\cite{mb14b}, §2.1]
			Let $T$ be a infinite, labelled, binary, rooted tree, such that vertices are labelled by the set $\{0,1\}^*$ by labelling the root with $\emptyset$ and the children of a vertex $v$ are iteratively labelled by $0v$ and $1v$. Let $\Omega$ be the set $\{0,1,2\}^{\mathbb{N}_+}$ endowed with the product topology and the natural surjective map $\sigma \colon \Omega \to \Omega$ given by the shift. Let $\epsilon$ denote the transposition of $0$ and $1$ and for $i,j \in \{0,1,2\}$ let the symbol $\epsilon_{i,j}$ denote the identity permutation on $0$ and $1$ if $i=j$ and let $\epsilon_{i,j}=\epsilon$ if $i \neq j$. Define a tree automorphism\footnote{Tree automorphisms are graph automorphism that fixes the root.\\} $a$ by $a(xv):=\epsilon(x)v$ for all $x \in \{0,1\}$ and $v \in \{0,1\}^*$. Let $\omega \in \Omega$. Define further tree automorphisms $b_\omega,c_\omega,d_\omega$ recursively by
			\begin{equation*}
			\begin{gathered}
			b_\omega(0xv):=0 \epsilon_{2,\omega_1}(x)v; \quad b_\omega(1v)=1 b_{\sigma(\omega)}v\\
			b_\omega(0xv):=0 \epsilon_{1,\omega_1}(x)v; \quad c_\omega(1v)=1 c_{\sigma(\omega)}v\\
			b_\omega(0xv):=0 \epsilon_{0,\omega_1}(x)v; \quad d_\omega(1v)=1 d_{\sigma(\omega)}v
			\end{gathered}
			\end{equation*}
			for all $x \in \{0,1\}$ and $v \in \{0,1\}^*$. The group generated by the tree automorphisms $a,b_\omega,c_\omega,d_\omega$ is called the \emph{Grigorchuk group $G_\omega$}. The \emph{first Grigorchuk group} is the Grigorchuk group $G_\omega$ where $\omega=012012012\dots$
		\end{defi}
		
		\begin{prop}\label{prop: grigorchuk groups embedding into permutational wreath products}
			Let $\omega$ be sequence in $\{0,1,2\}^{\mathbb{N}_+}$ which is eventually constant i.e. there exists an $N \in \mathbb{N}$ such that $\omega_n=\omega_{n+1}$ for all $n \geq N$. Then there exists an $n \in \mathbb{N}$ such that the group $G_\omega$ embedds into $ \mathbb{Z} \wr_{n} \mathfrak{S}_{n} \cong \mathbb{Z}^n \rtimes \mathfrak{S}_{n}$.
		\end{prop}
		
		\begin{rem}
			In \cite{gri85} Grigorchuk actually proves that $G_\omega$ embedds into $D^\infty \wr_{n} \mathfrak{S}_{n}$ for some $n \in \mathbb{N}$, this group however admits an embedding into $\mathbb{Z} \wr_{2n} \mathfrak{S}_{2n}$.
		\end{rem}
		
		\begin{cor}
			Let $\omega$ be sequence in $\{0,1,2\}^{\mathbb{N}_+}$ which is eventually constant. Then the group $G_\omega$ is virtually abelian\footnote{It contains a finite index abelian subgroup.\\} and in consequence has polynomial growth.
		\end{cor}
		
		The far more interesting Grigorchuk groups however are provided by the rest:
		
		\begin{thm}[\cite{gri85}, Theorem~2.1 \& Corollary~3.2]
			Let $\omega$ be sequence in $\{0,1,2\}^{\mathbb{N}_+}$ which is not eventually constant. Then the group $G_\omega$ has intermediate growth. Moreover $G_\omega$ is periodic \emph{if and only if} every element of $\{0,1,2\}$ appears infinite times in $\omega$.
		\end{thm}

		\section{Graphs of actions}\label{app: graphs of actions}
		
		The graphs in this section are oriented graphs, which are in general not simple. Except where noted the definitions can be found in \cite{nek18}, § 2.1.
		
		\begin{defi}
			Let $G$ be a group with finite, symmetric generating set $S$.
			\begin{enumerate}[(i)]
				\item Let $G$ act transitively\footnote{Transitivity of the action assures that the arising graph is connected.\\} on a set $X$. The \emph{Schreier graph $\Gamma(X,G,S)$} of this action is given by the following data:
				\begin{enumerate}
					\item The vertex set is given by $X$.
					
					\item A pair $(x,y) \in X\times X$ is an edge \emph{if and only if} there exists an $s \in S$ such that $y=sx$. 
				\end{enumerate}
				
				\item Let $G$ act on a set $X$. For every $x \in X$ the \emph{orbital Schreier graph} $\Gamma(x,G,S)$ is the Schreier graph of the induced action of $G$ on the orbit $G \cdot x$.
				
				\item The Schreier graph associated with the left-multiplication of $G$ on itself is the \emph{Cayley graph $\Gamma(G,S)$}.
			\end{enumerate}
		\end{defi}
		
		\begin{rem}
			\begin{enumerate}[(i)]
				\item We omit the generating set from the notion if it is apparent from the context.
				
				\item The orbital Schreier graph $\Gamma(x,G,S)$ is naturally isomorphic to the Schreier graph $\Gamma(G/G_x,G,S)$ associated to the action of $G$ on the set of cosets $G/G_x$.
			\end{enumerate}
			
		\end{rem}
		
		\begin{defi}
			Let a group $G$ with finite, symmetric generating set $S$ act on a compact metrizable space $X$ by homeomorphisms. Let $x \in X$.
			\begin{enumerate}[(i)]
				\item Let $G_{(x)}$ denote the subgroup of elements $g \in G_x$, such that there exists an open neighbourhood $U_x$ of $x$ for which $g|_{U_x}=\operatorname{Id}_{U_x}$.
				
				\item The group\footnote{The group $G_{(x)}$ is indeed a normal subgroup of $G_x$.\\} $G_x/G_{(x)}$ is called the \emph{group of germs of the action at the point $x$}.
				
				\item The \emph{graph of germs $\tilde{\Gamma}(x,G,S)$} is the Schreier graph $\Gamma(G/G_{(x)},G,S)$ associated to the action of $G$ on the set of cosets $G/G_{(x)}$.
			\end{enumerate}
		\end{defi}
		
		\begin{rem}
			There is a natural graph covering $\pi \colon \tilde{\Gamma}(x,G,S) \twoheadrightarrow \Gamma(x,G,S)$ where the group of germs $G_x/G_{(x)}$ acts as deck transformations.
		\end{rem}
		
		\begin{defi}
			Let a group $G$ with finite, symmetric generating set $S$ act on a compact metrizable space $X$ by homeomorphisms.
			\begin{enumerate}[(i)]
				\item An element $x \in X$ is called a \emph{$G$-regular point} if $G_x = G_{(x)}$ i.e. for every $g \in G$ that leaves $x$ invariant there exists a neighbourhood $U_x$ of $x$ that is pointwise fixed by $g$.
				
				\item An element $x \in X$ is called a \emph{singular point} if it is not $G$-regular.
			\end{enumerate}
			
			Let $x \in X$ be a singular point.
			\begin{enumerate}[(i),resume]
				\item It is called a \emph{Hausdorff singularity} if for every $g \in G_x \setminus G_{(x)}$ the element $x$ is not an accumulation point of the interior of the set of fixed points of $g$.
				
				\item It is called a \emph{non-Hausdorff singularity} if there exists a $g \in G_x \setminus G_{(x)}$ such that $x$ is an accumulation point of the interior of the set of fixed points of $g$.
				
				\item It is called a \emph{purely non-Hausdorff singularity} if $x$ is an accumulation point of the interior of the set of fixed points of $g$ for every $g \in G_x$.
			\end{enumerate}
		\end{defi}
		
		One aspect of graphs of actions is that they allow to study finitely generated groups via stochastics on graphs.
		
		\begin{defi}[\cite{woe00}, §I.1.B \& §I.1.C ]
			\begin{enumerate}[(i)]
				\item Let $S$ be a countable set. A \emph{Markov chain with state space $S$} arises from the data of an initial probability distribution $\lambda=(\lambda_s)_{s \in S}$ on $S$ and a stochastic matrix $P=\{p(s,t) \}_{s,t \in S}$ as sequence of random variables $\{X_n\}_{n \in \mathbb{N}}$ for which the following holds for all $n \in \mathbb{N}$ and $s_0,\dots,s_n \in S$:
				\begin{equation*}
				\operatorname{P}[X_0=s_0,X_1=s_1,\dots,X_n=s_n]=\lambda_{s_0}\prod_{i=0}^{n-1} p(s_{i},s_{i+1})
				\end{equation*}
				
				\item Let $\Gamma = (V,E)$ be locally finite graph. The \emph{simple random walk on $\Gamma$ with starting point $v \in V$} is the Markov chain with state space $V$, initial probability distribution $\delta_v$ and stochastic matrix $P=\{p(v,u) \}_{v,u \in V}$ given by
				\begin{equation*}
				p(v,u)=
				\begin{cases}
				|\{(v,w) \in E\}|^{-1}, & \text{if } (v,u) \in E\\
				0, & \text{else}
				\end{cases}
				\end{equation*}
				The graph $\Gamma = (V,E)$ is called \emph{recurrent} if for every $v \in V$ the simple random walk on $\Gamma$ with starting point $v \in V$ is recurrent i.e. 
				\begin{equation*}
				\sum_{i=0}^{\infty} \operatorname{P}[\min\{n \geq 1|X_n=v\}=n]=1.
				\end{equation*}
			\end{enumerate}
			
		\end{defi}
		
		\section{Amenability}\label{app: amenability}
		
		For encompassing treatments on amenability of groups the reader be refered to Chapter~4 of \cite{cc10} or \cite{jus15}. 
		Amenability is a pivotal notion in geometric group theory. While historically amenability (or more precisely non-amenability) already appeared intrinsically in Felix Hausdorff's work on paradoxical decompositions of the sphere, it was John von Neumann who isolated the property which posed an obstruction to existence of paradoxical decompositions:			 
		\begin{defi}
			Let $G$ be a discrete group. Then $G$ is said to be \emph{amenable} if there exists a finitely additive probability measure $\mu$ on $G$ such that $\mu(gA)=\mu(A)$ for every $A \subseteq G$ and $g \in G$.
		\end{defi}
		
		\begin{rem}
			\begin{enumerate}[(i)]
				\item Already von Neumann knew, that finite groups and abelian groups are amenable and that the class of amenable groups is closed under the operations of forming subgroups, forming quotients and extensions.
				
				\item The class of amenable groups is closed under direct limits and thus every locally finite group is amenable.
				
				\item The free group $\mathbb{F}_2$ is non-amenable and in consequence every discrete group containing it. It was a long-standing conjecture if containment of $\mathbb{F}_2$ characterizes non-amenable groups. A counterexample was provided by Aleksandr Y. Olshanskii in 1980 by the construction of \emph{Tarski monsters}. 
			\end{enumerate}
		\end{rem}
		
		Mahlon M. Day formulated the following question: Is the class of amenable groups the smallest class of groups that contains all finite and abelian groups which is closed with respect to subgroups, quotients and extension? This class has been termed the class of \emph{elementary amenable} groups and is denoted by $\mathrm{EG}$. Fundamental in the resolution of this problem was Ching Chou's proof of the following growth type dichotomy on finitely generated $\mathrm{EG}$-groups -- a strict generalization of the theorem restricted to finitely generated solvable groups obtained by John Milnor and Joseph A. Wolf:
		\begin{thm}
			Every finitely generated, elementary amenable group, has either polynomial or exponential growth.
		\end{thm}
		
		Moreover one can establish the following fact:
		\begin{thm}
			Finitely generated groups of subexponential growth are amenable.
		\end{thm}
		
	By this the Grigorchuk groups of intermediate growth provided first counterexamples to Day's question. Since then many examples of non-elementary amenable, amenable groups have been found. There is a number of characterizations of amenability:
		
		\begin{defi}
			\begin{enumerate}[(i)]
				\item Let $X$ be a set. A \emph{mean on $X$} is a linear functional $m \in \mathrm{L}^{\infty}(X)^*$ such that the following conditions hold:
				\begin{enumerate}
					\item $m(\mathbf{1}_X)=1$
					
					\item $m(f) \geq 0$ for all $f \in \mathrm{L}^\infty(X)$ with $f \geq 0$.
				\end{enumerate}
				
				\item Let $\alpha \colon G \curvearrowright X$ be an action of a discrete group $G$ on a set $X$. A mean $m$ on $X$ is called \emph{$G$-invariant} if  $m(g\cdot f)(x):=m(f(g^{-1}\cdot x))=m(f)(x)$ for all $f \in \mathrm{L}^\infty(X)$, $x \in X$ and $g \in G$ i.e. $m$ is a fixed point of the induced action of $G$.
			\end{enumerate}
		\end{defi}
		
		By an application of the \emph{Hahn-Banach theorem}, there is a 1-1 correspondence between means and finitely additive probability measures on a set producing the following characterization:
		\begin{prop}
			Let $G$ be a discrete group. Then $G$ is amenable \emph{if and only if} $G$ carries a $G$-invariant mean (with respect to left multiplication).
		\end{prop}  
		
		Moreover the following hold:
		\begin{thm}
			A discrete group $G$ is amenable \emph{if and only if} $C^*(G)\cong C_r^*(G)$.
		\end{thm}

		\begin{thm}
			Let $G$ be a finitely generated group with finite generating set $S$. The following are equivalent:
			\begin{enumerate}[(i)]
				\item The group $G$ is amenable.
				
				\item There exists a sequence $\{F_i\}_{i\in \mathbb{N}} $ of finite sets called \emph{Følner sets} contained in $G$ such that there exists a sequence $\{\varepsilon_i\}_{i \in \mathbb{N}}$ of positive reals with $\lim\limits_{i \to \infty} \varepsilon_i = 0$ such that $|gF_i\Delta F_i |\leq \varepsilon_i F$ for all $g \in S$ and $i \in \mathbb{N}$.
			\end{enumerate}
		\end{thm}
		
		Følner sets allow to quantify amenability of finitely generated groups by the growth of Følner functions:
		\begin{defi}
			Let $G$ be a finitely generated amenable group with symmetric finite generating set $S$.
			\begin{enumerate}[(i)]
				\item Let $F \subseteq G$. Define  $\partial_S F:=\{g \in G|\exists s \in S:sg \notin F \}$.
				
				\item The \emph{Følner function} $\operatorname{Fol}_{G,S}\colon \mathbb{N} \to \mathbb{N}$ is given by:
				\begin{equation*}
				\operatorname{Fol}_{G,S}(n):=\min\{|F|: F \subseteq G, |\partial_S F| \leq \frac{1}{n}|F|\}
				\end{equation*}
			\end{enumerate}
			
		\end{defi}
		
		A fundamental tool to establish amenability of a group is to study its actions:
		
		\begin{defi}
			Let $\alpha \colon G \curvearrowright X$ be an action of a discrete group $G$ on a set $X$. The action is called \emph{amenable} if $X$ carries a $G$-invariant mean.
		\end{defi}
		
		\begin{thm}
			Every action of an amenable group is amenable.
		\end{thm}
		
		The converse is not true in general, one needs to add a crucial condition:
		\begin{thm}\label{thm: amenact}
			Let $\alpha \colon G \curvearrowright X$ be an amenable action of a  group $G$ on a set $X$. If the point-stabilizer $G_x$ is amenable for every $x \in X$, then $G$ is amenable.
		\end{thm}
		
		\section{LEF groups}\label{app: lef-groups}
		
		The LEF property, where ``LEF" is an abbrevation for ``\textbf{l}ocally \textbf{e}mbeddable into \textbf{f}inite groups", is, as the name suggests, an approximation property:
		
		\begin{defi}[\cite{cc10}, Definition~7.1.3]
			Let $\mathcal{C}$ be a class of groups. A group $G$ is said to be \emph{locally embeddable into $\mathcal{C}$} if for every finite subset $F \subset G$ there exists a group $H_F$ in the class $\mathcal{C}$ and a map $\phi \colon G \to H_F$ such that $\phi(xy)=\phi(x)\phi(y)$ for all $x,y \in F$ and $\phi|_F$ is injective.
			If $\mathcal{C}$ is the class of finite groups, we call such a group a \emph{LEF-group}.
		\end{defi}
		
		For countable groups the LEF-property translates as follows:				
		\begin{thm}[\cite{gv98}, Corollary~1.3]
			Let $G$ be a countable group. The following are equivalent:
			\begin{enumerate}[(i)]
				\item $G$ is a LEF-group
				\item There exists a sequence $\{F_n\}_{n \in \mathbb{N}}$ of finite groups and a sequence of maps $\{\pi_n \colon G \to F_n \}_{n \in \mathbb{N}}$ such that for all $g,h \in G$ the following properties hold:
				\begin{enumerate}
					\item $(g \neq h) \Rightarrow (\exists N \in \mathbb{N}:\forall n > N, \; \pi_n(g)\neq \pi_n(h))$.
					
					\item $\exists N \in \mathbb{N}:\forall n > N, \; \pi_n(gh) = \pi_n(g)\pi_n(h)$.
				\end{enumerate}
			\end{enumerate}
		\end{thm}
		
		Finitely generated LEF-groups can be characterized in the following way:
		
		\begin{prop}[\cite{gv98}, p.53]\label{prop: marked}
			\quad A group with a fixed finite generating set $\{g_1, \dots,g_n\}$ has the LEF-property \emph{if and only if} it is the limit of a sequence of finite groups in the space of marked groups. 
		\end{prop}
		If the sequence of maps $\{\pi_n \colon G \to F_n \}_{n \in \mathbb{N}}$ is made up of group homomorphism, condition (ii) results in the definition of \emph{residually finite groups} and thus residually finite groups form a subclass of LEF-groups. In particular free groups and profinite groups\footnote{A profinite group is a topological group that is inverse limit of a directed system of discrete finite groups.\\} are LEF-groups. In the inverse direction finite presentation is a sufficient condition:
		\begin{prop}[\cite{gv98}, p.58]\label{prop: finitely presented lef group is residually finite}
			Every finitely presented LEF-group is residually finite.
		\end{prop}
		
		As a consequence we have:
		\begin{cor}[\cite{gv98}, Corollary~3]\label{cor: no infinite, finitely presented simple LEF-group}
			There exists no infinite, finitely presented, simple LEF-group.
		\end{cor}
		
		\section{The Liouville property}
		
		The content of this section can be found in \cite{mb14a}. A poperty of finitely generated groups that lies ``between" subexponential growth and amenability is the Liouville property.
		\begin{defi}
			Let $G$ be a finitely generated group.
			\begin{enumerate}[(i)]
				\item Let $\mu$ be a probability measure. A function $f\colon G \to \mathbb{R}$ is said to be \emph{$\mu$-harmonic} if $f(g)=\sum_{h \in G} f(gh)\mu(h)$ for every $g \in G$.\footnote{This condition is also written as $f=f \ast \mu$. The values of $f$ are in some sense given by a weighted average.\\}
				
				\item Let $\mu$ be a probability measure on $G$. The pair $(G, \mu)$ is said to have the \emph{Liouville property} if every $\mu$-harmonic function is constant on the subgroup $\langle \supp(\mu) \rangle \leq G$.
				
				\item The group $G$ is said to have the \emph{Liouville property} if for every symmetric, finitely supported probability measure $\mu$ on $G$ the pair $(G,\mu)$ has the Liouville property.
			\end{enumerate}
		\end{defi}
		
		\begin{thm}\label{thm: liou}
			Every finitely generated group of subexponential growth has the Liouville property and every finitely generated group with the Liouville property is amenable.
		\end{thm}
		
		A fundamental criterion in proving the Liouville property is the study of random walks on groups:
		
		\begin{defi}\label{defi: entropy}
			Let $G$ be an countable, infinite, discrete group and let $\mu$ be a probability measure on $G$.
			\begin{enumerate}[(i)]
				\item The \emph{right random walk on $(G,\mu)$} is the Markov chain with state space $G$ and transition probabilities $p(g,h):=\mu(h^{-1}g)$ for every $g,h \in G$.
				
				\item The \emph{entropy of the probability measure $\mu$} is the number
				\begin{equation*}
				H(\mu):=-\sum_{g \in \supp (\mu)} \mu(g) \log \mu (g)
				\end{equation*}
				
				\item The \emph{entropy of the random walk on $(G,\mu)$} is the number
				\begin{equation*}
				h(G,\mu):= \lim\limits_{n \to \infty} \frac{1}{n} H(\mu^{*n})
				\end{equation*}
			\end{enumerate}
		\end{defi}
		
		The Liouville property can be characterized in terms of the random walk entropy:
		
		\begin{thm}\label{thm: charac.liouville}
			Let $G$ be an countable, infinite, discrete group and let $\mu$ be a probability measure on $G$. The $(G,\mu)$ has the Liouville property \emph{if and only if} $h(G,\mu)=0$.
		\end{thm}

		\section{Higman-Thompson groups}\label{app: higman-thompson groups}
		
		In 1965 Richard J. Thompson introduced a triple of infinite, finitely presented groups $F \subseteq T \subseteq G$. The group $F$, now called \emph{Thompson's group}, has been a historically first candidate for a possible counterexample to the von-Neumann conjecture i.e. for a non-amenable, finitely generated group that contains no free subgroup. It has since gained notoriety for defying any proof of non-amenability or amenability. The groups $T$ and $G$ gave rise to first examples of infinite, finitely presented, simple groups. Later on manifestations of this groups arose in quite diverse contexts. The \emph{Higman-Thompson groups $G_{n,r}$} are generalizations of the group $G=G_{2,1}$ considered by Higman in \cite{hig74}. The Higman-Thompson group $G_{n,r}$ arises as the automorphism group of the \emph{free Jónsson–Tarski algebra of type $n$ on $r$ generators}. We follow \cite{sco84}, \cite{bro87} and \cite{par11}:
		
		Let $n \in \mathbb{N}\setminus \{1 \}$ and $r \in \mathbb{N}$. Let $\mathcal{A}_n$ be the alphabet given by $\{1,\dots,n \}$ and let $\mathcal{A}_n^*$ denote the free monoid generated by $\mathcal{A}_n$ i.e. the set  of finite (possibly empty) words with cocatenation as operation. Let $X_r=\{x_i\}_{i \in \{1,\dots,r \} }$ be a finite set. Denote by $X_r\mathcal{A}_n^*$ the set of finite words of the form $x_iw$ where $x_i \in X_r$ and $w \in \mathcal{A}_n^*$. This structure can be represented by a labelled forest $\mathcal{F}_{n,r}$ consisting of $r$ ordered, infinite, complete, $n$-ary, rooted trees $\{T_i\}_{i \in \{1,\dots,r \} }$. A tree $T_i$ is labelled by labelling its root with $x_i$ and recursively labelling the children of a node $x_iw$ by $x_iw1, \dots, x_iwn$ with the order inherited from the natural order of $\{1,\dots,n \}$. Let $a,b \in X_r\mathcal{A}_n^*$. Define $a \leq b$ to mean there exists an $w \in \mathcal{A}_n^*$ such that $b=aw$ In the tree picture this is equivalent to $b$ being a descendant of $a$. This induces a partial order on $X_r\mathcal{A}_n^*$. A subset $S \subseteq X_r\mathcal{A}_n^*$ is called \emph{independent} if all elements in $S$ are pairwise incomparable with respect to $\leq$. A non-empty subset $V \subseteq X_r\mathcal{A}_n^*$ is called a \emph{subspace of $X_r\mathcal{A}_n^*$} if it is closed with respect to right cocatenation by elements in $\mathcal{A}_n^*$.  Let $V$ be a subspace of $X_r\mathcal{A}_n^*$. It is called \emph{cofinite} if $|X_r\mathcal{A}_n^* \setminus V|$ is finite.  A subset $B \subseteq V$ is called a \emph{basis of $V$} if it is independent and $V=B\mathcal{A}_n^*$ i.e. $V$ is the upward closure of $B$ with respect to $\leq$. Note that every subspace $V \subseteq X_r\mathcal{A}_n^*$ has a canonical basis consisiting of the elements which are minimal with respect to $\leq$. A subset $B \subseteq X_r\mathcal{A}_n^*$ is a \emph{basis} if it is basis of some subspace $W$ of $X_r\mathcal{A}_n^*$. A basis $B$ is called \emph{cofinite} if the subspace $B\mathcal{A}_n^*$ is cofinite. Note that a basis is cofinite \emph{if and only if} it is finite and maximal. Moreover every finite basis is contained in some cofinite basis. Let $B$ be a basis and let $b \in B$. The set $(B \setminus \{b\}) \cup \{b1,\dots,bn \}$ is a basis called a \emph{simple expansion of $B$}. A basis $C$ that is obtained by a finite chain of simple expansions from $B$ is called an \emph{expansion of $B$}. In the tree picture a subspace $V$ consist of a family $\{T_a\}_{a \in A}$ of infinite complete subtrees with roots $a$ where $A$ corresponds to the canonical basis of $V$. Let $V,U$ be subspaces of $X_r\mathcal{A}_n^*$. A map $ \alpha \colon V \to U$ is called a \emph{homomorphism of subspaces} if $\alpha(vw)=\alpha(v)w$ holds for all $v\in V$ and for all $w \in \mathcal{A}_n^*$. In the tree picture it is a map between disjoint families of infinite complete subtrees $\alpha \colon \bigsqcup_{a \in A} T_a \to \bigsqcup_{b \in B} T_b$ with respective indpendents subsets $A,B \subset X_r\mathcal{A}_n^*$ of roots such that the ordered rooted tree structure of trees $T_a$ in the domain is preserved. If in addition a homomorphism of subspaces is bijective, it is called an \emph{isomorphism of subspaces}. An isomorphism of subspaces $ \alpha \colon V \to W$ between cofinite subspaces $V,W$, is called a \emph{cofinite isomorphism}. Let $ \alpha \colon V \to W$ be a cofinite isomorphism. A cofinite isomorphism $\beta \colon \tilde{V} \to\tilde{W}$ is called an \emph{extension of $\alpha$} if $V \subseteq \tilde{V}$ and $\beta(v)=\alpha(v)$ for every $v \in V$. A cofinite isomorphism which has no non-trivial extensions is called \emph{maximal}.
		\begin{ex}
			Let $r,n=2$. Let $\alpha$ be the cofinite isomorphism given by $T_{x_1i} \mapsto T_{x_2i}$ and $T_{x_2i} \mapsto T_{x_1i}$ for $i \in \{1,2\}$. Then the cofinite isomorphism $\tilde{\alpha}$ given by $T_{x_1} \mapsto T_{x_2}$ and $T_{x_2} \mapsto T_{x_1}$ is an extension of $\alpha$.
		\end{ex}
		
		\begin{lem}[\cite{sco84}, Lemma~1]
			Let $n \in \mathbb{N}\setminus \{1 \}$ and $r \in \mathbb{N}$. Every cofinite isomorphism $\alpha$ of subspaces of $X_r\mathcal{A}_n^*$ has a unique maximal extension.
		\end{lem}
		
		Denote the unique maximal extension of a cofinite isomorphism $\alpha$ of subspaces of $X_r\mathcal{A}_n^*$ by $\alpha^*$. Let $ \alpha \colon V \to W$ and $\beta \colon \tilde{V} \to \tilde{W}$ be cofinite isomorphisms of subspaces of $X_r\mathcal{A}_n^*$. Then
		
		\begin{equation*}
		\alpha \odot \beta:=(\beta|_{W \cap \tilde{V}} \circ \alpha|_{\alpha^{-1}(W \cap \tilde{V})})^*
		\end{equation*}
		is a well-defined cofinite isomorphism of subspaces of $X_r\mathcal{A}_n^*$.
		
		\begin{lem}[\cite{sco84}, Lemma~2]
			Let $n \in \mathbb{N}\setminus \{1 \}$ and $r \in \mathbb{N}$. The set of maximal cofinite isomorphisms of subspaces of $X_r\mathcal{A}_n^*$ forms a group with respect to $\odot$.
		\end{lem}
		
		\begin{defi}
			The \emph{Higman-Thompson group $(G_{n,r},\odot)$} is the group of maximal cofinite isomorphisms of subspaces of $X_r\mathcal{A}_n^*$.
		\end{defi}
		
		It is immediate that the groups $G_{n,r}$ admit a natural action on a Cantor space by their natural action on the pathspace of the forest $\mathcal{F}_{n,r}$, furthermore cofinite bases allow for a description of $G_{n,r}$ by \emph{tables}, which will be needed to state Theorem~\ref{thm: unitary representation of thompson-higman groups in Cuntz-algebras}. Let $B= \{b_1,\dots,b_N\}$ and $C=\{c_1,\dots,c_N\}$ be cofinite bases of the same cardinality. Then any bijection $\beta$ between those bases extends naturally to a cofinite isomorphism of subspaces between $B\mathcal{A}_n^*$ and $C\mathcal{A}_n^*$, thus inducing a unique element in $G_{n,r}$ denoted by
		\begin{equation*}
		\begin{pmatrix}
		b_1 & \dots & b_N\\
		\beta(b_1) & \dots & \beta(b_N)
		\end{pmatrix}
		\end{equation*}
		Conversely, for every $g \in G_{n,r}$ there exists a $N \in \mathbb{N}$ and cofinite bases $B= \{b_1,\dots,b_N\}$ and $C=\{c_1,\dots,c_N\}$ such that
		\begin{equation*}
		g=
		\begin{pmatrix}
		b_1 & \dots & b_N\\
		c_1 & \dots & c_N
		\end{pmatrix}
		\end{equation*}
		
		If a table is obtained from another by permutation of columns, they are said to be \emph{equivalent}. If a table is obtained from another by replacing a column
		\begin{equation*}
		\begin{pmatrix}
		b_i\\
		c_i
		\end{pmatrix}
		\quad \text{with} \quad
		\begin{pmatrix}
		b_i1 & \dots & b_i n\\
		c_i1 & \dots & c_i n 
		\end{pmatrix}
		\end{equation*}
		we call it a \emph{simple expansion} of the initial table.  If a table is obtained from another by a finite string of simple expansions it is called an \emph{expansion} of the initial table.
		\begin{lem}
			Two tables have expansions that are equivalent \emph{if and only if} they induce the same element of $G_{n,r}$.
		\end{lem}
		
		Let $g \in G_{n,r}$ given by the table
		\begin{equation*}
		g=
		\begin{pmatrix}
		b_1 & \dots & b_N\\
		c_1 & \dots & c_N
		\end{pmatrix}
		\end{equation*}
		The natural homeomorphisms associated with $g$ on the pathspace of the forest $\mathcal{F}_{n,r}$ is given by mapping the open sets of infinite paths $\{b_iw|w \in \mathcal{A}_n^{\mathbb{N}} \}$ onto the open sets of infinite paths $\{c_iw|w \in \mathcal{A}_n^{\mathbb{N}} \}$ via $g(b_iw)=c_iw$ for all $i \in \{1,\dots,N\}$.
		
		The Higman-Thompson fall within a class of groups of piecewise linear interval transformations, the following description is essentially due to \cite{bro87}:
				
		\begin{thm}
			Let $n \in \mathbb{N}\setminus \{1 \}$ and $r \in \mathbb{N}$. Denote by $\tilde{G}_{n,r}$ the group of right-continuous, piecewise linear bijections $f \colon [0,r) \to [0,r)$ with finitely many singularities, such that all singularities of $f$ are contained in $\mathbb{Z}[\frac{1}{n}]$, such that $f(\mathbb{Z}[\frac{1}{n}] \cap[0,r) )= \mathbb{Z}[\frac{1}{n}] \cap[0,r)$ and for every point $x \in [0,r)$ where $f$ is not singular, there exists a $k \in \mathbb{Z}$ such that $f'(x)=n^k$. Then $\tilde{G}_{n,r} \cong G_{n,r}$
		\end{thm}

		\chapter{C*-algebras}\label{app: c*-algebras}

		\section{Basic terms}\label{app: basic terms of c*-algebras}
		
		The definitions of this section are taken from Chapter I of \cite{dav96} or Chapter 2 of \cite{mur90} except where noted.
		
		\begin{defi}\phantomsection\label{app: banalg}
			\begin{enumerate}[(i)]
				\item A \emph{normed algebra} is an algebra $\mathfrak{A}$ endowed with a submultiplicative norm $\|\cdot \|$ (i.e. $\forall a,b \in \mathfrak{A}: \|ab\| \leq \|a\|\|b\|$).
				If $\mathfrak{A}$ has a multiplicative unit $1$, it is said to be \emph{unital}. A complete normed algebra is called \emph{Banach algebra}. 
				
				\item Let $a$ be an element in a unital Banach algebra $\mathfrak{A}$. The \emph{spectrum of $a$} is the set $\sigma(a):=\{\lambda \in \mathbb{C}: \lambda 1 - a \text{ is non-invertible} \}$.
				The \emph{spectral radius of $a$}, denoted by $\mathrm{spr}(a)$, is defined as $\mathrm{spr}(a):=\underset{\lambda \in \sigma(a)}{\sup}|\lambda|$.
			\end{enumerate}
		\end{defi}
		
		In some contexts it is useful to have a unit around:
		
		\begin{defi}\label{defi: unitiz}
			Let $\mathfrak{A}$ be a Banach algebra. The vectorspace $\mathfrak{A} \times 	\mathbb{C}$ endowed with the product $(a,\lambda)(b,\mu)=(ab+\lambda b + \mu a, \lambda\mu)$ is a unital algebra denoted by $\mathfrak{A}^{+}$ with $(0,1)$ as unit and $\mathfrak{A}$ as a maximal ideal via the embedding $a \mapsto (a,0)$. The norm of $\mathfrak{A}$ naturally extends to a complete norm on $\mathfrak{A}^{+}$ by $\|(a,\lambda)\|=\|a\|+|\lambda|$. The Banach algebra $\mathfrak{A}^{+}$ is called the \emph{unitization of $\mathfrak{A}$}.
		\end{defi}
		
		Banach $*$-algebras (pronounced "star-algebra") are Banach algebras with an isometric involution. C*-algebras are an important subclass:
		
		\begin{defi}\phantomsection\label{defi: *}
			\begin{enumerate}[(i)]
				\item A \emph{$*$-algebra} is a $\mathbb{C}$-algebra $\mathfrak{A}$ with  anti-involution $^* \colon \mathfrak{A} \to \mathfrak{A}$ called the \emph{adjoint}, i.e. a map $^*:a \mapsto a^*$ which satisfies: 
				
				\begin{enumerate}
					\item $\forall a,b \in \mathfrak{A}: \forall z_1,z_2 \in \mathbb{C}: (z_1a+z_2b)^*=\overline{z_1}a^* + \overline{z_2}b^*$ (conjugate-linear)
					
					\item $\forall a \in \mathfrak{A}: a^{**} = a$ (self-inverse)
					
					\item $\forall a,b \in \mathfrak{A}: (ab)^* = b^*a^*$ (anti-multiplicative).
				\end{enumerate}
				
				\item A \emph{$*$-homomorphism} is an algebra-homomorphism $\varphi \colon \mathfrak{A} \to \mathfrak{B}$ between $*$-algebras $\mathfrak{A}, \mathfrak{B}$ such that $\varphi(-^*)=\varphi(-)^*$.
				
				\item A \emph{Banach $*$-algebra} is a Banach algebra $\mathfrak{A}$ which is a $*$-algebra such that $\|a^*\|=\|a\|$ for all $a \in \mathfrak{A}$.

				\item Let $\mathfrak{A}$ be a unital Banach $*$-algebra. An element $a \in \mathfrak{A}$ is called
				\begin{enumerate}
					\item \emph{self-adjoint} if $a^*=a$.
					
					\item \emph{normal} if $a^*a=a^*$.
					
					\item an \emph{isometry} if $a^*a=1$ and a \emph{coisometry} if $aa^*=1$
					
					\item a \emph{partial isometry} if $aa^*a=a$. Denote by $\operatorname{Par}(\mathfrak{A})$ the set of all partial isometries in $\mathfrak{A}$.
					
					\item \emph{unitary} if it is an isometry and a coisometry.
					
					\item \emph{positive}, write $a \geq 0$, if it is self-adjoint and $\sigma(a) \subseteq [0;\infty[$.
					
					\item an \emph{idempotent} if $a=a^2$.
					
					\item a \emph{projection}, if it is a self-adjoint idempotent. Denote by $\operatorname{Proj}(\mathfrak{A})$ the set of all projections of $\mathfrak{A}$.
				\end{enumerate}
				The self-adjoint elements in $\mathfrak{A}$ can be endowed with an order relation by defining $ a\leq b$ if $b - a \geq 0$. Denote the set of all positive elements in $\mathfrak{A}$ by $\mathfrak{A}_{\mathrm{pos}}$.
				
				\item A \emph{C*-algebra} is a Banach $*$-algebra $\mathfrak{A}$ such that $\|a^*a\|=\|a\|^2$ for all $a \in \mathfrak{A}$.
				
				\item Let $\mathfrak{A}$ be a C*-algebra. A \emph{C*-subalgebra of $\mathfrak{A}$} is a norm-closed subalgebra closed with respect to the adjoint. An \emph{ideal of $\mathfrak{A}$} is a norm-closed two sided ideal of the underlying $\mathbb{C}$-algebra.
				
				\item Let $\mathfrak{A}$ be a unital C*-algebra and let $a \in \mathfrak{A}$. The \emph{C*-subalgebra generated by $a$}, denoted by  $\mathrm{C}^*(a)$, is the norm-closure of the linear span of $I$ and products of $a$ and $a^*$.
				
				\item Let $\mathfrak{A}$ be a C*-algebra. A net $(i_\lambda)_{\lambda \in \Lambda}$ of elements in $\mathfrak{A}$ is called \emph{approximate unit} if:
				\begin{enumerate}
					\item $i_\lambda \geq 0, \|i_\lambda\| \leq 1$ for all $\lambda \in \Lambda$
					\item $\lambda \leq \kappa$ implies $i_\lambda \leq i_\kappa$
					\item $\underset{\lambda \in \Lambda}{\lim}\; i_\lambda a = \underset{\lambda \in \Lambda}{\lim} \; a i_\lambda= a$ for all $a \in \mathfrak{A}$.
				\end{enumerate}
				
				\item Let $\mathfrak{A}$ be a C*-algebra and let $\mathfrak{B}$ be a C*-subalgebra of $\mathfrak{A}$. The normalizer of $\mathfrak{B}$ in $\mathfrak{A}$ is the set $N(\mathfrak{B},\mathfrak{A}):=\{a \in \mathfrak{A}|a^*\mathfrak{B}a \subseteq \mathfrak{B}, a\mathfrak{B}a^* \subseteq \mathfrak{B} \}$.
			\end{enumerate}
		\end{defi}

		Let $\mathfrak{A}$ be a C*-algebra, then the set $\{a \in \mathfrak{A}: a \geq 0, \|a\|<1 \}$ is a directed set with respect to the order defined on the self-adjoint elements. Therefore we have:
		
		\begin{thm}[\cite{dav96}, Theorem I.4.8]\label{thm: approxu}
			Every C*-algebra has an approximate unit.
		\end{thm}
		
		As with general Banach algebras it is often useful to have a unit around:
		
		\begin{thm}[\cite{mur90}, Theorem 2.1.6.]
			If $\mathfrak{A}$ is a C*-algebra, the extended norm makes $\mathfrak{A}^+$ into a C*-algebra.
		\end{thm}
		
		After this verbiage we look at the first concrete examples of such algebras:		
		
		\begin{ex}\phantomsection\label{ex: cstar}
			\begin{enumerate}[(i)]
				\item The field $\mathbb{C}$ with its usual norm is a unital C*-algebra with conjugation as $^*$-morphism.
				
				\item \label{item: contfunct} For any locally compact Hausdorff space $X$ the space $C_0 (X)$ of continuous $\mathbb{C}$-valued functions vanishing at infinity\footnote{A $\mathbb{C}$-valued function on a locally compact space $X$ is said to \emph{vanish at infinity} if the set $\{x \in X: |f(x)| \geq \epsilon\}$ is compact for every $\epsilon > 0$ (\cite{rud87}, Definition 3.16).\\} endowed with the supremum norm and $f \mapsto \overline{f}$ as adjoint is a commutative C*-algebra. If $X$ is compact, $C_0 (X)$ contains $\mathbf{1}_X$ as a unit. If $X$ is non-compact, $C_0 (X)$ is non-unital.
				
				\item Let $\mathcal{H}$ be a Hilbert space. Then the space $\mathfrak{B}(\mathcal{H})$ of bounded linear operators on $\mathcal{H}$ is a C*-algebra with respect to the operator norm and the $^*$-morphism given by the adjoint of operators.
			\end{enumerate}
		\end{ex}
		
		\begin{defi}
			\begin{enumerate}[(i)]
				\item Let $\mathfrak{A}$ be a commutative Banach algebra. A \emph{character of $\mathfrak{A}$} is a non-zero algebra-homomorphism $\mathfrak{A} \to \mathbb{C}$. Denote by $\Omega(\mathfrak{A})$ the set of all characters of $\mathfrak{A}$.
				
			\end{enumerate}
			
			Let $\mathfrak{A}$ be a Banach $*$-algebra.
			\begin{enumerate}[(i),resume]
				
				\item A linear functional $\varphi \colon \mathfrak{A} \to \mathbb{C}$ is called \emph{positive}, if $\varphi(a) \geq 0$ for every positive $a \in \mathfrak{A}$. Positive linear functionals of norm $1$ are called \emph{states}. Denoted by $\mathcal{S}(\mathfrak{A})$ the set of all states of $\mathfrak{A}$. A state is said to be \emph{pure} if it is an extreme point\footnote{A point $x$ contained in a convex subset $C$ of some $K$-vectorspace $V$ is called extreme point if it can not be represented as non-trivial convex combination i.e. if there exist $\lambda \in [0;1]$ and $y,z \in K$ with $x=\lambda y + (1-\lambda)z$, it follows that $x=y$ or $x=z$ (\cite{rud91}, §3.22).\\} in $\mathcal{S}(\mathfrak{A})$.
				
				\item Let $\mathfrak{A}$ be a C*-algebra. A \emph{weight on $\mathfrak{A}$} is a function $\phi \colon \mathfrak{A}_{\mathrm{pos}} \to [0,\infty]$ such that $\phi(\lambda a) = \lambda \phi(a)$ for all $a \in \mathfrak{A}_{\mathrm{pos}}$ and $\lambda \in \mathbb{R}_{\geq 0}$ and $\phi(a+b)=\phi(a)+\phi(b)$ for all $a,b \in \mathfrak{A}_{\mathrm{pos}}$.
			\end{enumerate}
		\end{defi}
		
		Abstract C*-algebras become accessible through their representations:
		
		\begin{defi}
			 Let $\mathcal{H}$ be a Hilbert space and let $\mathfrak{A}$ be a Banach $*$-algebra. A \emph{$*$-algebra-representation of $\mathfrak{A}$ on $\mathcal{H}$} is a $*$-homomorphism $\mathfrak{A} \to \mathcal{B}(\mathcal{H})$.
				A $*$-algebra-representation $\phi \colon \mathfrak{A} \to \mathcal{B}(\mathcal{H})$ is called
		
				\begin{enumerate}[(i)]
					\item \emph{trivial} if $\phi(\mathfrak{A})=0$
					\item \emph{topologically irreducible} if $\phi(\mathfrak{A})$ contains no proper closed invariant subspace
					\item \emph{algebraically irreducible} if $\phi(\mathfrak{A})$ contains no proper invariant subspace
					\item \emph{faithful} if it is injective
					\item \emph{non-degenerate} if $\overline{\phi(\mathfrak{A})\mathcal{H}}=\mathcal{H}$
					\item \emph{cyclic} if it has a \emph{cyclic vector} i.e. there exists an element $x\in \mathcal{H}$ such that $\phi(\mathfrak{A})x$ is dense in $\mathcal{H}$.
				\end{enumerate}
		\end{defi}
		
		\begin{rem}
			(\cite{dav96}, Theorem I.9.3). In the case of C*-algebras topological- and algebraical irreducibility are equivalent and one just speaks of \emph{irreducible} representations.
		\end{rem}
		
		Originally, C*-algebras were defined as norm-closed subalgebras of $\mathfrak{B}(\mathcal{H})$ for some Hilbert space $\mathcal{H}$ closed under taking adjoints (\cite{seg47}, §2). These fall in the class constituted by the abstract Definition~\ref{defi: *} (iii). The Gelfand-Naimark theorem assures every C*-algebra is a C*-subalgebra of $\mathfrak{B} (\mathcal{H})$ for some Hilbert space $\mathcal{H}$ or in other words:
		\begin{thm}[\cite{dav96}, Theorem I.9.12]\label{thm: gn-thm}
			Every C*-algebra $\mathfrak{A}$ has an isometric faithful representation.
		\end{thm}
		The Gelfand-Naimark theorem is a consequence of the Gelfand representation, its commutative manifestation, which relates a commutative Banach algebra $\mathfrak{A}$ with the space $C_{0}(\Omega(\mathfrak{A}))$ of continuous functions vanishing at infinity on its space of characters. By Gelfand theory, every locally compact Hausdorff space $X$ gives rise to the commutative C*-algebra $C_0(X)$ and any commutative C*-algebra $\mathfrak{A}$ gives rise to the locally compact Haudorff space of characters $\Omega(\mathfrak{A})$. Given the correct notion of morphisms this produces an equivalence of categories $\mathbf{LoCoHaus} \simeq \mathbf{CommC^*}^{\mathrm{op}}$, called \emph{Gelfand duality}. This is why the theory of general C*-algebra theory has been termed non-commutative topology. It is thus sometimes helpful to think of C*-algebraical terms as generalized topological notions\footnote{This is a cut down version of the list in §1.11 of \cite{weg93} and can of course be continued.\\}:
		\begin{equation*}
		\begin{gathered}
		\text{topological space} \longleftrightarrow \text{C*-algebra}\\
		\text{open subsets} \longleftrightarrow \text{ideals}\\
		\text{closed subsets} \longleftrightarrow \text{quotients}\\
		\text{clopen subsets} \longleftrightarrow \text{projections}\\
		\text{compact spaces} \longleftrightarrow \text{unital C*-algebras}\\
		\text{proper continous maps} \longleftrightarrow \text{unital *-morphisms}\\
		\text{Alexandroff compactification} \longleftrightarrow \text{unitization}\\
		\text{2nd countable} \longleftrightarrow \text{separable}
		\end{gathered}
		\end{equation*}
		
		We finish with a significant example of C*-algebras, the Cuntz-Krieger algebras. These generalize the Cuntz algebras, which were historically the first examples of separable, simple, infinite C*-algebras:
		
		\begin{defi}[\cite{ck80}, §2]\label{defi: cuntz-krieger algebras}
			Let $n \in \mathbb{N}$ and let $\{A_{i,j}\}_{i,j \in \{1,\dots,n\}} \in \mathbf{M}_{n\times n}(\mathbb{Z}/2\mathbb{Z})$ be a matrix of which every row and column is non-zero and which is not a permutation matrix.	Let $\{S_i\}_{i \in \{1,\dots,n\} }$ be a collection of non-zero partial isometries over some Hilbert space $\mathcal{H}$ with associated support projections $Q_i=S_i^*S_i$ and range projections $P_i=S_iS_i^*$ such that the following hold:
			\begin{enumerate}[(i)]
				\item $P_iP_j=0$ for all $i,j \in \{1,\dots,n\}$ with $i\neq j$.
					
				\item $Q_i=\sum_{j \in \{1,\dots,n\}} A_{i,j} P_j$ for all $i \in \{1,\dots,n\}$.
			\end{enumerate}
			Denote by $C^*(S_1,\dots,S_n)$ be the C*-subalgebra of $\mathcal{B}(\mathcal{H})$ generated by $\{S_i\}_{i \in  \{1,\dots,n\}}$.
			The up to isomorphism unique C*-algebra $\mathcal{O}_A:=C^*(S_1,\dots,S_n)$ is called the \emph{Cuntz-Krieger algebra of $A$}. If every entry of $A$ is $1$, the arising C*-algebra only depends on $n$. It is denoted by $\mathcal{O}_n$ and called the \emph{Cuntz algebra of order $n$}.	
		\end{defi}
		
		\section{Group C*-algebras}\label{app: group c-star-algebras}
		
		For a more thourough treatment of group C*-algebras see Chapter VII of \cite{dav96}. 
		
		\begin{defi}\label{defi: unitaryrepresentations}
			Let $G$ be a locally compact topological group.
			\begin{enumerate}[(i)]
				\item Let $\mathcal{H}$ be a Hilbert space. Denote by $\mathfrak{U}(\mathcal{H})$ the group of unitary elements in $\mathfrak{B}(\mathcal{H})$. A \emph{unitary representation of $G$ on $\mathcal{H}$} is a group-homomorphism $\pi \colon G \to \mathfrak{U}(\mathcal{H})$ such that the map $g \mapsto \pi(g) \cdot x$ is continuous for all $x \in \mathcal{H}$.\footnote{It is continuous with respect to the strong operator topology.}
				Such a representation $\pi$ is called \emph{irreducible} if $\pi(G)$ commutes with no proper projection of $\mathcal{H}$ and it is called \emph{cyclic}, if it has a cyclic vector.
				
				\item Let $\mathcal{H}_1, \mathcal{H}_2$ be Hilbert spaces with unitary representations $\pi_1 \colon G \to \mathfrak{U}(\mathcal{H}_1)$ and $\pi_2 \colon G \to \mathfrak{U}(\mathcal{H}_2)$ of $G$. The representations $\pi_1,\pi_2$ are said to be \emph{unitarily equivalent} if there exists a Hilbert space ismorphism (i.e. a bijection preserving the inner product) $\Phi \colon \mathcal{H}_1 \to \mathcal{H}_2$ such that $\phi \circ \pi_1(g) = \pi_2(g) \circ \phi$ for all $g \in G$.
			\end{enumerate}
		\end{defi}
		
		Locally compact topological groups possess what is called a \emph{left resp. right Haar measure} i.e. every such group $G$ supports a left- resp. right-invariant, non-zero, regular Borel measure $\mu_G$, which is unique up to a scalar.	Let $G$ be locally compact group. As customary in measure theory, every $p \in [1;\infty[$ defines a seminorm on the space $\mathcal{M}(G):=\{f \colon G \to \mathbb{C}| f \text{ is measurable w.r.t. } \mu_G\}$ by $\|f\|_p:=(\int _{G}|f|^{p}\;\mathrm {d}\mu_G)^{1/p}$. Taking the quotient of the semi-normed vectorspaces $\mathcal{L}^p(G):=\{f \in \mathcal{M}_{\mu_G}: \|f\|_p < \infty \}$ by $\ker(\|\cdot\|_p)$ gives -- by abuse of notation -- the normed vector spaces $(\mathrm{L}^p(G), \|\cdot\|_p)$. The space $\mathrm{L}^1(G)$ can be given the structure of a Banach $*$-algebra with approximate unit with the following adjoint and product:
		\begin{equation*}
		\begin{gathered}
		f^{*}(g)=\Delta(g^{-1}) \overline{f(g^{-1})}\\
		(f_1 \ast f_2)(g):=\int\limits_{h \in G} f_1(h)f_2(h^{-1}g)\;\mathrm{d}\mu_G(h)
		\end{gathered}
		\end{equation*}	
		The function $\Delta$ is the \emph{modular function}, which is a continuous group-homomorphism $\Delta \colon G \to \mathbb{R}^+$,  that satisfies $\mu_G(Sg)=\Delta(g)\mu_G(S)$ for all measureable sets $S$. Note, that in general $\mathrm{L}^1(G)$ is not a C*-algebra. It is a fundamental feature of the theory, that there is a close relation between unitary representations of the underlying locally compact group $G$ and $*$-representations of the convolution algebra $\mathrm{L}^1(G)$:
		Every unitary representation $\pi$ of $G$ gives rise to a $*$-algebra-representation $\tilde{\pi}$ of $\mathrm{L}^1(G)$ by integration:
		\begin{equation*}
		\tilde{\pi}(f):=\int\limits_{g \in G} f(g) \pi(g) \mathrm{d}\mu_G(g)
		\end{equation*}
		
		Conversely every non-degenerate $*$-algebra-representation $\tilde{\pi}$ of $\mathrm{L}^1(G)$ induces a unitary representation $\pi$ of $G$ by setting 
		\begin{equation*}
			\pi(g):=\mathrm{\scalebox{0.5}{SOT-}}\lim_{\lambda \in \Lambda} \tilde{\pi} \big((I_\lambda)(g^{-1}(-)\big)
		\end{equation*}
		where $(I_\lambda)_{\lambda \in \Lambda}$ is an approximate unit of $\mathrm{L}^1(G)$ and $\mathrm{\scalebox{0.5}{SOT-}}\lim$ denotes the limit with respect to the strong operator topology. $\mathrm{L}^1(G)$ can be endowed with norm defined by:
		\begin{equation*}
		\|f\|_{C^*}:=\sup \{\|\pi(f)\|: \pi \text{ is a }*\text{-algebra-representation of }\mathrm{L}^1(G) \}
		\end{equation*}
		This supremum is well defined, since $\|\pi(f)\| \leq \|f\|$ for all $*$-algebra-respresentations of $\mathrm{L}^1(G)$ as they are norm-decreasing.  The completion of $\mathrm{L}^1(G)$ with respect to this norm is a C*-algebra, called the \emph{(maximal) group C*-algebra of $G$} denoted by $C^*(G)$. The \emph{left regular representation $\lambda$ of $G$} is the unitary representation on $\mathrm{L}^2(G)$ given by $(\lambda(g)f)(x):=f(g^{-1}x)$.\footnote{Operators $\lambda(g)$ are unitary, since the Haar measure is translation invariant.\\}
		The \emph{reduced group C*-algebra of $G$} denoted by $C_r^*(G)$ is the closure of $\tilde{\lambda}(\mathrm{L}^1(G))$ with respect to the operator norm or equivalently the completion of $\mathrm{L}^1(G)$ with respect to the norm $\|f\|_r:=\|\tilde{\lambda}(f)\|$. If $G$ is a discrete group, $C_c(G)$ densely embeddeds in $\mathrm{L}^1(G)$.
		In direct consequence $C^*(G)$ has the following universal property:
		\begin{prop}
			Let $G$ be a discrete group and let $u \colon G \to \mathfrak{B}(\mathcal{H})$ be a unitary representation. Then there exists a unique $*$-algebra-representation $\pi_u \colon C^*(G) \to \mathfrak{B}(\mathcal{H})$ such that $\pi_u|_G=u$.
		\end{prop}
		There is a 1-1 correspondence between the unitary representations of $G$ and non-degenerate representations of $C^*(G)$. Since consequently the regular representation extends naturally to a surjective $*$-algebra-morphism $\lambda : C^*(G) \to C_r^*(G)$, the reduced group C*-algebra $C_r^*(G)$ is a quotient of $C^*(G)$. In this case it could be defined alternatively by $C_r^*(G):=\lambda(C^*(G))$.

		\section{Hilbert bundles}\label{app: hilbert bundles and hilbert modules}
		
		The following is mostly adopted from \cite{bla06}, §II.7.
		
		\begin{defi}
			Let $X$ be a second countable, locally compact, Hausdorff space endowed with a probability measure $\mu$. Let  $\mathcal{H}:=\{H_x\}_{x \in X}$ be a collection of Hilbert spaces indexed by $X$ and let $\Sigma \subset \prod_{x \in X} H_x$ be a countable set of sections. The triple $(\mathcal{H},\mu,\Sigma)$ is called a \emph{measured Hilbert bundle} if it satisfies the following properties:
			\begin{enumerate}
				\item The set $\{\sigma_x|\sigma \in \Sigma\}$ is a dense subspace of $H_x$ for every $x \in X$.
				
				\item For every $\sigma_1,\sigma_2 \in \Sigma$ the map $x \mapsto \langle \sigma_1 (x),\sigma_2(x) \rangle$ is $\mu$-measurable.
			\end{enumerate}
		\end{defi}
		
		Measured Hilbert bundles induce a Hilbert space by integration:
		\begin{defi}
			Let $(X,\mathcal{H}=\{H_x\}_{x \in X},\mu)$ be a measured Hilbert bundle.
			\begin{enumerate}[(i)]
				\item A section $\xi \in \prod_{x \in X} H_x$ is called \emph{measurable} if for every $\sigma \in \Sigma$ the function $X \to \mathbb{C} \colon x \mapsto \langle \xi(x),\sigma(x)\rangle$ is $\mu$-measurable. Denote by $\Gamma(X,\mathcal{H},\mu)$ the set of all measurable sections.
				
				\item The space $\Gamma(X,\mathcal{H},\mu)$ carries a semi-norm defined by
				\begin{equation*}
				\|\sigma\|_2:=\big(\int_{x \in X} \|\sigma(x)\|_x^2 \; \mathrm{d}\mu\big)^{\frac{1}{2}}.
				\end{equation*} 
				The quotient of $\mu$-integrable measurable sections by the space of null sections i.e.
				\begin{equation*}
				\bigslant{\{\sigma \in \Gamma(X,\mathcal{H},\mu): \|\sigma\|_2< \infty \}}{\{\sigma \in \Gamma(X,\mathcal{H},\mu): \|\sigma\|_2= 0 \}}	 
				\end{equation*}
				is a Hilbert space with the product given by $\langle \sigma_1,\sigma_2 \rangle:=\int_{x \in X} \langle \sigma_1(x),\sigma_2(x)\rangle_x \; \mathrm{d}\mu$ called the \emph{direct integral of $(X,\mathcal{H},\mu)$} and is denoted by $\int_{X}^{\oplus} H_x \;\mathrm{d}\mu$.
			\end{enumerate}
		\end{defi}
		
		\section{A hamfisted portrait of K-groups}\label{app: K-theo}
		
		Operator K-theory is the non-commutative generalization of topological K-theory and is an indispensible tool in the theory of C*-algebras and Banach algebras. It comprises continuous, half exact, homotopy-invariant and stable functors $K_0$ and $K_1$, which associate to a Banach algebra $\mathfrak{A}$ abelian groups $K_0(\mathfrak{A})$ resp. $K_1(\mathfrak{A})$, in particular the functors satisfy $K_0(\mathbb{C})=\mathbb{Z}$ and $K_1(\mathbb{C})=0$.
		For monographs on operator K-theory see e.g. \cite{weg93} or \cite{bla08}, we content ourselves with an adumbration of $K_0$-groups in most handwavy terms:
		
		The invariant $K_0$ is obtained by putting algebraic structure on classes of idempotents, in the case of C*-algebras one can restrict to classes of projections. In C*-algebras there are three types of equivalence relations between projections to consider. For $p,q \in \operatorname{Proj}(\mathfrak{A})$ we get the following list of equivalence relations, where every entry implies the preceding equivalence relations:
		\begin{enumerate}[(i)]
			\item The projections $p,q$ are \emph{algebraically equivalent}, write $p \sim q$, if there exist elements $xy \in \mathfrak{A}$ such that $p=xy$ and $q=yx$.
			\item The projections $p,q$ are \emph{unitarily equivalent}, write $p \sim_u q$, if there exist a unitary element $u \in \mathfrak{A}^+$ such that $p=uqu^*$.
			\item The projections $p,q$ are \emph{homotopic}, write $p \sim_h q$, if there exist a norm continuous path of projections from $p$ to $q$.
		\end{enumerate}
		The sum of two projections $p+q$ is a projection \emph{if and only if} the projections $p,q$ are orthogonal i.e. $ef=0$. Addition of orthogonal projections behaves well with respect to equivalence, but not for every pair of equivalence classes of projections does there exist an orthogonal pair of respective contained representants.
		The idea is to pass over to the C*-algebra $\mathbb{M}_{\infty}(\mathfrak{A} )= \bigcup_{n \in \mathbb{N}} \mathbb{M}_n(\mathfrak{A})$ of infinite matrices with finite non-zero entries in $\mathfrak{A}$. Let $p,q \in \operatorname{Proj}(\mathfrak{A})$, then the following implications hold:
		\begin{equation*}
		p \sim q \Rightarrow
		\begin{pmatrix}
		p & 0\\
		0 & 0
		\end{pmatrix}
		\sim_u
		\begin{pmatrix}
		q & 0\\
		0 & 0
		\end{pmatrix}
		,\quad
		p \sim_u q \Rightarrow
		\begin{pmatrix}
		p & 0\\
		0 & 0
		\end{pmatrix}
		\sim_h
		\begin{pmatrix}
		q & 0\\
		0 & 0
		\end{pmatrix}.
		\end{equation*}
		This means in this setting all the defined equivalence relations coincide. Moreover, for any $p \in \operatorname{Proj}(\mathfrak{A})$ we have:
		\begin{equation*}
		\begin{pmatrix}
		0 & 1\\
		1 & 0
		\end{pmatrix}
		\begin{pmatrix}
		p & 0\\
		0 & 0
		\end{pmatrix}
		\begin{pmatrix}
		0 & 1\\
		1 & 0
		\end{pmatrix}
		=
		\begin{pmatrix}
		0 & 0\\
		0 & p
		\end{pmatrix}
		\end{equation*} 
		By conjugating with a unitary, a projection can be pushed down diagonally, and since the matrices we consider can be arbitrarily large, for any pair $p,q \in \operatorname{Proj}(\mathbb{M}_{\infty} (\mathfrak{A}))$ we can find a orthogonal pair of representants of the respective equivalence classes by pushing one of the involved projections far enough down the diagonal. Thus the set $V(A):=\operatorname{Proj}(\mathbb{M}_{\infty}(\mathfrak{A}))/{\sim}$ becomes a commutative monoid with respect to the direct sum. This construction produces a covariant functor from C*-algebras to commutative monoids. Any *-algebra homomorphism $\alpha \colon \mathfrak{A} \to \mathfrak{B}$ induces a well-defined homomorphism of commutative monoids $V(\alpha) \colon V(\mathfrak{A}) \to V(\mathfrak{B})$.	As in the construction of the integers $\mathbb{Z}$ from the natural numbers $\mathbb{N}$, the commutative monoid  $V(\mathfrak{A})$ gives rise to an abelian group $K_{00}(\mathfrak{A})$ via the \emph{Grothendieck group} construction by considering equivalence classes of formal differences between elements in $V(\mathfrak{A})$. In particular, we have $K_{00}(\mathbb{C})=\mathbb{Z}$. Again, the construction is functorial and comes with a canonical map $\tilde{\iota_\mathfrak{A}} \colon V(\mathfrak{A}) \to K_{00}(\mathfrak{A})$. Any $\alpha \colon \mathfrak{A} \to \mathfrak{B}$ induces a group homomorphism $K_{00}(\alpha) \colon K_{00}(\mathfrak{A}) \to K_{00}(\mathfrak{B})$. This definition, however, is insufficient in the case of non-unital C*-algebras -- in general $K_{00}$ is not half-exact, but it is the exact sequences of K-theory, that make it tick! To have a unified way of definition one turns to $\mathfrak{A}^+$ (see Definition~\ref{defi: unitiz}): The C*-algebra $\mathfrak{A}^{+}$ comes with a canonical surjective *-algebra homomorphism $\pi \colon \mathfrak{A}^+ \to \mathbb{C}$ and by functoriality one can define
		\begin{equation*}
		K_0(\mathfrak{A}):=\ker\big(K_{00}(\pi)\colon K_{00}(\mathfrak{A^+}) \to \mathbb{Z}\big).
		\end{equation*}
		Again, $K_0$ is functorial and there is a canonical map $\iota_\mathfrak{A} \colon V(\mathfrak{A}) \to K_{0}(\mathfrak{A})$.
		If $\mathfrak{A}$ is unital, much of theory becomes simpler e.g. we have $K_0(\mathfrak{A})=K_{00}(\mathfrak{A})$.
		There are three classes of elements in $K_0(\mathfrak{A})$: Elements arising out of projections in $\mathfrak{A}$, elements arising out of adjoining a unit or elements arising out of the matrices over $\mathfrak{A}^+$. The $K_0$-groups carry more information by keeping track of the commutative monoid $V(\mathfrak{A})$ and the projections steming from $\mathfrak{A}$:
		Fix the notations $K_0(\mathfrak{A})^+ : = \Ima(\iota_\mathfrak{A})$ and $\Sigma\mathfrak{A}$ is the set of elements in $K_0(\mathfrak{A})$ that come from elements in $\operatorname{Proj}(\mathfrak{A})$. Under sufficient conditions i.e. when $\mathfrak{A}$ is unital and stably finite,\footnote{Being only unital is not sufficient, the \emph{Cuntz algebra} (see Definition~\ref{defi: cuntz-krieger algebras}) represents a counter-example.\\} the set $K_0(\mathfrak{A})^+ $ is a positive cone for $K_0(\mathfrak{A})$ making it an ordered abelian group. Furthermore, if $1$ is the unit of $\mathfrak{A}$, then $\Sigma\mathfrak{A}=\{[p] \in K_0(\mathfrak{A})^+\big|[p]\leq [1] \}$, thus the scale is uniquely determined by $[1]$. For such a C*-algebra $\mathfrak{A}$ the triple $(K_0(\mathfrak{A}), K_0(\mathfrak{A}^+,[1])$ is an ordered abelian group with order unit e.g. this is the case, if $\mathfrak{A}$ is an AF C*-algebra.
		The group $K_1(\mathcal{A})$ is not defined via the projections but via the unitaries of a C*-algebra $\mathcal{A}$, moreover, it can be given in terms of $K_0$, in that $K_1(\mathfrak{A})$ is the $K_0$-group of the \emph{suspension of $\mathfrak{A}$}. The fundamental feature of $K_0$- and $K_1$-groups is, that any short sequence
		\[0 \longrightarrow \mathfrak{A} \overset{\varphi}{\longrightarrow} \mathfrak{B} \overset{\psi}{\longrightarrow} \mathfrak{C} \longrightarrow 0\]
		of C*-algebras gives rise to the \emph{cyclical 6-term exact sequence of K-groups of C*-algebras}
		\[
		\begin{tikzcd}
		K_0(\mathfrak{A}) \arrow{r}{K_0(\varphi)} & K_0(\mathfrak{B}) \arrow{r}{K_0(\psi)} & K_0(\mathfrak{C}) \arrow{d}{\delta_0}\\
		K_1(\mathfrak{C})  \arrow{u}{\delta_1} & K_1(\mathfrak{B}) \arrow{l}{K_1(\psi)} & K_1(\mathfrak{A})  \arrow{l}{K_1(\varphi)}
		\end{tikzcd}
		\]
		in which $\delta_1$ signifies the \emph{index map}. This designation hints at the ties of K-theory with index theory.
	
		\section{Von Neumann algebras}
		
		Von Neumann algebras form a subclass of C*-algebras. Their theory is very distinct in that it is noncommutative measure theory. In this section we recall the terms we need in Section~\ref{sec: the koopman representation of topological full groups}. A useful reference on von Neumann algebras is Chapter 4 of \cite{mur90} or §I.9 and §III of \cite{bla06}.
		
		\begin{defi}
			Let $\mathcal{H}$ be a Hilbert space.
			\begin{enumerate}[(i)]
				\item Let $\mathcal{S}$ be a subset of $\mathcal{B}(\mathcal{H})$. The set $\mathcal{S}':=\{S' \in \mathcal{B}(\mathcal{H})|S'S=SS' ~ \forall S \in \mathcal{S} \}$ is called the \emph{commutant of $\mathcal{S}$}.
				
				\item  A $*$-subalgebra $\mathfrak{A}$ of the bounded linear operators $\mathcal{B}(\mathcal{H})$ which is closed w.r.t. the strong operator topology\footnote{This is the coarsest topology on $\mathcal{B}(\mathcal{H})$ which makes $f_{(x,y)}\colon \mathcal{B}(\mathcal{H}) \to \mathbb{C}, T \mapsto \langle Tx,y\rangle$ continuous for all $x,y \in \mathcal{H}$.\\} and contains $\operatorname{Id}$ is called \emph{von Neumann algebra on $\mathcal{H}$}.
				
				\item Since intersections of von Neumann Algebras are again von Neumann algebras, we can define for any $\mathcal{S} \subseteq \mathcal{B}(\mathcal{H})$ the \emph{von Neumann algebra $\mathcal{M}_\mathcal{S}$ generated by $\mathcal{S}$} as the intersection of all von Neumann algebras containing $\mathcal{S}$.
				
				\item Let $G$ be a countable group and let $\pi \colon G \to \mathcal{B}(\mathcal{H})$ be a unitary representation of $G$. Denote by $\mathcal{M}_\pi$ the von Neumann Algebra generated by the set of operators $\{\pi(g)|g \in G\}$.
				
				\item An abelian von Neumann algebra $\mathfrak{A}$ is called \emph{maximal} if it is not contained in a strictly larger abelian von Neumann algebra. An abelian von Neumann algebra $\mathfrak{A}$ is maximal \emph{if and only if} $\mathfrak{A}'=\mathfrak{A}$. Such a von Neumann algebra is acronymically termed a \emph{masa}. 
			\end{enumerate}
		\end{defi}
		
		Measure spaces give rise to commutative von Neumann algebras -- the theory of von Neumann algebras is in this sense non-commutative measure theory.
		\begin{rem}\phantomsection\label{rem: vNma}
			\begin{enumerate}[(i)]
				\item By the \emph{von Neumann bicommutant theorem} a $*$-subalgebra $\mathfrak{A}$ of $\mathcal{B}(\mathcal{H})$ is a von-Neumann-algebra \emph{if and only if} $\mathfrak{A}=\mathfrak{A}''$.
				
				\item For any subsets $S,S_1,S_2 \subset \mathcal{B}(\mathcal{H})$ we have $S \subseteq S''$, $S_1 \subseteq S_2$ implies $S_2' \subseteq S_1'$ and thus $S=S'''$.

				\item Let $\pi \colon G \to \mathcal{B}(\mathcal{H})$ be a unitary representation of a discrete group $G$. Then the image is self-adjoint and contains $\operatorname{id}_\mathcal{H}$ and thus $\mathcal{M}_\pi$ is just the double commutant $\pi(G)''$. 
			\end{enumerate}
		\end{rem}
		
	\end{appendices}
	
	\bibliographystyle{alpha}

\end{document}